\documentclass{article}
\usepackage{fancyhdr}
\usepackage{amscd}
\usepackage{graphicx}
\usepackage{amsmath}
\usepackage{amssymb}
\usepackage{amsthm}
\usepackage{mathrsfs}
\newtheorem{dfn}{Definition}[section]
\newtheorem{thm}[dfn]{Theorem}
\newtheorem{prop}[dfn]{Proposition}
\newtheorem{lem}[dfn]{Lemma}
\newtheorem{cor}[dfn]{Corollary}
\newtheorem{rem}[dfn]{Remark}
\newtheorem{car}[dfn]{Computer Assisted Result}
\newtheorem{ass}[dfn]{Assumption}

\newtheorem{nota}[dfn]{Notations}

\newcommand{\Inv}{{\rm Inv}}
\newcommand{\Int}{{\rm int}}

\newcommand{\re}{{\rm Re}}

\newcommand{\diag}{{\rm diag}}
\newcommand{\diam}{{\rm diam}}
\newcommand{\dist}{{\rm dist}}

\newcommand{\exit}{{\rm exit}}
\newcommand{\ent}{{\rm ent}}
\newcommand{\dep}{{\rm dep}}
\newcommand{\arr}{{\rm arr}}

\numberwithin{equation}{section}

\begin{document}

\title{Rigorous numerics for fast-slow systems with one-dimensional slow variable: topological shadowing approach}

\author{Kaname Matsue \thanks{The Institute of Statistical Mathematics, Tachikawa, 190-8562, Tokyo, Japan ({\tt kmatsue@ism.ac.jp})}
}
\lhead{Kaname Matsue}
\rhead{Rigorous Numerics for Fast-Slow Systems}
%\date{}
\maketitle

\begin{abstract}
We provide a rigorous numerical computation method to validate periodic, homoclinic and heteroclinic orbits as the continuation of singular limit orbits for the fast-slow system
\begin{equation*}
\begin{cases}
x' = f(x,y,\epsilon), & \\
y' =\epsilon g(x,y,\epsilon) &
\end{cases}
\end{equation*}
with one-dimensional slow variable $y$. 
Our validation procedure is based on topological tools called isolating blocks, cone condition and covering relations.
Such tools provide us with existence theorems of global orbits which shadow singular orbits in terms of a new concept, the covering-exchange. 
Additional techniques called slow shadowing and $m$-cones are also developed. These techniques give us not only generalized topological verification theorems, but also easy implementations for validating trajectories near slow manifolds in a wide range, via rigorous numerics.
Our procedure is available to validate global orbits not only for sufficiently small $\epsilon > 0$ but all $\epsilon$ in a given half-open interval $(0,\epsilon_0]$. 
Several sample verification examples are shown as a demonstration of applicability.
\end{abstract}



%
%	New Section
%
\section{Introduction}
\label{section-intro}

%
%	New Subsection
%
\subsection{Background of problems and our aims}
In this paper, we consider the dynamical system in $\mathbb{R}^n \times \mathbb{R}$ of the following form:
\begin{equation}
\label{fast-slow}
\begin{cases}
x' = f(x,y,\epsilon), & \\
y' =\epsilon g(x,y,\epsilon), &
\end{cases}
\end{equation}
where $' = d/dt$ is the time derivative and $f,g$ are $C^r$-functions with $r\geq 1$. 
The factor $\epsilon$ is a nonnegative but sufficiently small real number. 
We shall write (\ref{fast-slow}) as (\ref{fast-slow})$_\epsilon$ if we explicitly represent the $\epsilon$-dependence of the system.
The system (\ref{fast-slow}) can be reformulated with a change of time-scale variable as 
\begin{equation}
\label{slow-fast}
\begin{cases}
\epsilon \dot x = f(x,y,\epsilon), & \\
\dot y = g(x,y,\epsilon), &
\end{cases}
\end{equation}
where $\dot{} = d/d\tau$ and $\tau = t/\epsilon$. 
One tries to analyze the dynamics of (\ref{fast-slow}), 
equivalently (\ref{slow-fast}), by suitably combining the dynamics of the {\em layer problem}
\begin{equation}
\label{layer}
\begin{cases}
x' = f(x,y,0), & \\
y' =0, &
\end{cases}
\end{equation}
and the dynamics of the {\em reduced problem}
\begin{equation}
\label{reduced}
\begin{cases}
0 = f(x,y,0), & \\
\dot y = g(x,y,0), &
\end{cases}
\end{equation}
which are the limiting problems for $\epsilon = 0$ on the fast and the slow time scale, respectively. 
Notice that (\ref{reduced}) makes sense {\em only on $f(x,y,0)=0$}, while (\ref{layer}) makes sense in whole $\mathbb{R}^{n+1}$
 as the $y$-parameter family of $x$-systems. 
 The meaning of the \lq\lq $\epsilon\to 0$-limit" is thus different between (\ref{fast-slow}) and (\ref{slow-fast}). 
 This is why (\ref{fast-slow}) or (\ref{slow-fast}) is a kind of {\em singular perturbation problems}. In particular, (\ref{fast-slow}) or (\ref{slow-fast})
  is known as {\em fast-slow systems} (or {\em slow-fast systems}), 
  where $x$ dominates the behavior in the fast time scale and $y$ dominates the behavior in the slow time scale.

When we study the dynamical system of the form (\ref{fast-slow}), 
we often consider limit systems (\ref{layer}) and (\ref{reduced}) independently at first. 
Then ones try to match them in an appropriate way to obtain trajectories for the full system (\ref{fast-slow}).
One of major methods for completely solving singularly perturbed systems like (\ref{fast-slow}) 
is the {\em geometric singular perturbation theory} formulated by Fenichel \cite{F}, Jones-Kopell \cite{JK}, Szmolyan \cite{S} and many researchers. 
A series of theories are established so that formally constructed singular limit orbits 
of (\ref{layer}) and (\ref{reduced}) can perturb to true orbits of (\ref{fast-slow}) for sufficiently small $\epsilon > 0$.
In geometric singular perturbation theory, there are mainly two key points to consider.
One is the description of slow dynamics for sufficiently small $\epsilon$ near the nullcline $\{(x,y)\mid f(x,y,0)=0\}$.
The other is the matching of fast and slow dynamics. 
As for the former, Fenichel \cite{F} provided the {\em Invariant Manifold Theorem} for describing the dynamics 
on and around locally invariant manifolds, called slow manifolds, for sufficiently small $\epsilon > 0$. 
Such manifolds can be realized as the perturbation of normally hyperbolic invariant manifolds at $\epsilon = 0$, 
which are often given by submanifolds of nullcline in (\ref{layer}).
As for the latter, Jones and Kopell \cite{JK} originally formulated the geometric answer 
for the matching problem deriving {\em Exchange Lemma}. This lemma informs that 
the manifold configuration upon exit from the neighborhood of slow manifolds under the assumption of transversal intersection 
between tracking invariant manifolds and the stable manifold of slow manifolds. 
Afterwards, Exchange Lemma has been extended in various directions, e.g. \cite{JKK, L, TKJ}.
Combining these terminologies, one can prove the existence of homoclinic or heteroclinic orbits of invariant sets near singular orbits for sufficiently small $\epsilon > 0$. 
There are also topological ways to prove their existence for sufficiently small $\epsilon$ provided by, say, Carpenter \cite{C}, Gardner-Smoller \cite{GS} and so on,  by using algebraic-topological concepts such as the mapping degree or the Conley index \cite{Con, Mis}.
Such topological approaches also mention the existence of periodic orbits near singular orbits.
\par
On the other hand, all such mathematical results do not give us {\it how large such sufficiently small $\epsilon$ is}. 
In other words, it remains an open problem whether there exist global orbits  given by the continuation of singular orbits
for a given $\epsilon$. This intrinsic problem has been mentioned in many discussions (e.g. \cite{Jones1984}).
From the viewpoint of numerical computations, if $\epsilon > 0$ is sufficiently small, 
(\ref{fast-slow}) becomes the stiff problem and numerically unstable.
Although the effective method for computing slow manifolds is provided by Guchenheimer and Kuehn \cite{GK}, 
computations for extremely small $\epsilon$ (e.g. close to machine epsilon) is still hard to operate correctly. 
These circumstances show that there are gaps between mathematical results (i.e. sufficiently small $\epsilon$)
 and numerical observations (i.e. given $\epsilon$) 
 for completely understanding dynamics of fast-slow systems.
 These are mainly because there are no estimations to measure mathematically rigorous consequences 
 not only quantitatively but also qualitatively. 
 The construction of procedures which bridge mathematical results and numerical observations 
 is necessary to completely understand phenomena in concrete dynamical systems, 
 which is also the case of singular perturbation problems.

\bigskip
Our main aim in this paper is to provide implementations 
for validating the continuation of various global orbits of (\ref{fast-slow}) 
for all $\epsilon \in [0,\epsilon_0]$ rigorously, where $\epsilon_0 > 0$ is a {\em given} number. 
In other words, we provide a method to validate
\begin{description}
\item[{ (Main 1)}] the singular limit orbit $H_0$ for (\ref{fast-slow}) with $\epsilon = 0$, as well as
\item[{ (Main 2)}] global orbits $H_\epsilon$ near $H_0$ for all $\epsilon \in (0,\epsilon_0]$ for a given $\epsilon_0 > 0$.
\end{description}
Singular limit orbit means the union of several heteroclinic orbits in (\ref{layer}) and the submanifolds of the nullcline $\{f(x,y,0) = 0\}$. 
Global orbits mean homoclinic, heteroclinic orbits of invariant sets and periodic orbits.

\bigskip
To this end, we provide the notion of {\em covering-exchange}, which is a topological analogue of Exchange Lemma. 
This concept consists of the following topological notions with suitable assumptions: (i) isolating blocks, (ii) cone conditions and (iii) covering relations. 
These three notions have been already applied to validations of global orbits in dynamical systems very well, as stated in Section \ref{section-preceding}. 
The covering-exchange is constructed by those notions, and informs us
\begin{itemize}
\item the existence of slow manifolds for (\ref{fast-slow}) as well as normally hyperbolic invariant manifolds for (\ref{layer}), and
\item the existence of trajectories not only which converge to invariant sets on slow manifolds but also which exits neighborhoods of slow manifolds after time $T=O(1/\epsilon)$.
\end{itemize}
These properties and the general consequence of covering relations yield the existence of global orbits.
We also generalize the covering-exchange by introducing the additional notion of {\em slow shadowing}, which guarantees the local existence of trajectories which shadow slow manifolds with nonlinear structure. 
This notion enables us to trace trajectories which not only tend to slow manifolds but also stay near slow manifolds for time $O(1/\epsilon)$ from topological viewpoints.
This concept is also very compatible with numerical computations, in particular, for validating trajectories near slow manifolds.

The other main tool to establish our procedure is the assistance of {\em rigorous numerics}, 
namely, computations of enclosures where mathematically correct objects are contained in the phase space. 
All such computations can be realized by {\em interval arithmetics} and mathematical error estimates. 
Combination of the covering-exchange, slow shadowing and rigorous numerics in reasonable processes
provide us with a method proving {\bf (Main 1)} and {\bf (Main 2)} simultaneously.

\par
Note that there are two approaches to consider singular perturbation problems; one is the continuation of structures from the singular limit systems ($\epsilon = 0$) to the full systems (i.e. $\epsilon > 0$), and the other is the consideration of full systems to the singular limit $\epsilon\to 0$.
Our attitude is the former.
In particular, we consider our problems via topological approach on the basis of geometric singular perturbation theory.

\bigskip
This paper is organized as follows. 
In Section \ref{section-preliminary}, we briefly review Fenichel's invariant manifold theorem and topological notions called covering relations and isolating blocks.
A systematic procedure of isolating blocks for validations of invariant manifolds with computer assistance (e.g. \cite{ZM, Mat}) is also discussed.
\par
In Section \ref{section-inv-mfd}, we show how slow manifolds can be validated in given regions with an explicit range $[0,\epsilon_0]$ of $\epsilon$. 
One sees that our fundamental arguments are basically followed by the proof in Jones' article \cite{Jones}.
Such arguments can be validated via the construction of isolating blocks
and singular perturbation problems' version of {\it cone conditions} (cf. \cite{ZCov}) and Lyapunov condition \cite{Mat}. 

In Section \ref{section-slow-dynamics}, we discuss treatments of slow dynamics.
First we introduce the new notion called the {\em covering-exchange} 
for describing the behavior of trajectories around slow manifolds (Section \ref{section-exchange}). 
This concept is a topological analogue of Exchange Lemma 
so that we can reasonably validate tracking invariant manifolds near slow manifolds in a suitable sense. 
This concept also solves the matching problem between fast and slow dynamics.
We also provide a generalization of the covering-exchange; a collection of local behavior near slow manifolds called
{\em slow shadowing}, {\em drop} and {\em jump} (Section \ref{section-show-shadowing}). 
These concepts enable us to construct true trajectories in full system which shadow ones on slow manifolds in reasonable ways via rigorous numerics.
Furthermore, we provide a slight extension of cones, called {\em $m$-cones}, which enables us to sharpen enclosures of stable and unstable manifolds of normally hyperbolic invariant manifolds (Section \ref{section-m-cones}).
The main idea itself is just a slight modification of cone conditions stated in Section \ref{section-inv-mfd}.  
But this technique gives us a lot of benefits in many scenes incorporating with rigorous numerics. 
On the other hand, dynamics on slow manifolds should be considered when slow manifolds exhibit the nontrivial dynamics 
such as fixed points, periodic orbits, homoclinic orbits, etc. 
As an example, we discuss validations of nontrivial fixed points on slow manifolds (Section \ref{section-inv-set-on-mfd}).
In the end of Section \ref{section-slow-dynamics}, we discuss unstable manifolds of invariant sets on slow manifolds (Section \ref{section-fiber-unstable}). 
To deal with these manifolds, we discuss the invariant foliations of slow manifolds, and translate this fiber bundle structure into the terms of cones and covering relations.
This is one of key considerations of heteroclinic orbits in (\ref{fast-slow})$_\epsilon$.

In Section \ref{section-existence}, the existence theorems for periodic and heteroclinic orbits near singular limit orbits with an explicit range $[0,\epsilon_0]$ of $\epsilon$ are presented.

\bigskip
As demonstrations of our proposing implementations, we study the FitzHugh-Nagumo equation: 
\begin{equation}
\label{FN-intro}
\begin{cases}
u' = v & \text{}\\
v' = \delta^{-1}(cv - f(u) + w) & \\
w' = \epsilon c^{-1}(u-\gamma w),
\end{cases}
\end{equation}
where $f(u) = u(u-a)(1-u)$, $a\in (0,1/2)$ and $c,\gamma, \delta > 0$.

The existence of global orbits such as periodic or homoclinic orbits for sufficiently small $\epsilon$ are widely discussed by many authors (e.g. \cite{C, GS}). Note that, as we mentioned before, the existence of those orbits {\em for a given $\epsilon$} remains an open question.
Our proposing ideas lead a road to answer this question.

\begin{car}[Existence of homoclinic orbits]
Consider (\ref{FN-intro}) with $a=0.3$, $\gamma = 10.0$ and $\delta = 9.0$. Then for all $c \in [0.799,0.801]$, there exist the following two kinds of trajectories: 
\begin{enumerate}
\item At $\epsilon = 0$, a singular heteroclinic chain $H_0$ consisting of 
\begin{itemize}
\item heteroclinic orbit from the equilibrium $p_0$ near $(0, 0, 0)\in \mathbb{R}^3$ to the equilibrium $q_1$ near $(1, 0, 0)\in \mathbb{R}^3$,
\item heteroclinic orbit from the equilibrium $p_1$ near $(0.870020061, 0, 0.06362)\in \mathbb{R}^3$ to the equilibrium $q_0$ near $(-0.12966517, 0, 0.06335)\in \mathbb{R}^3$, and
\item two branches of nullcline $\{(u,v,w)\mid v = 0, f(u)=w\}$ connecting two heteroclinic orbits.
\end{itemize}
\item For all $\epsilon \in (0,5.0\times 10^{-5}]$, a homoclinic orbit $H_\epsilon$ of $p^\epsilon \approx p_0$ near $H_0$.
\end{enumerate}
\end{car}
The precise statements and other sample validation results are shown in Section \ref{section-examples}.
Throughout the rest of this paper we make the following assumption, which is essential to our whole discussions herein.
\begin{ass}
\label{ass-paper}
Vector fields $f$ and $g$ have the following form: 
\begin{align*}
f(x,y,\epsilon) &= f_0(x,y) + \sum_{i=1}^m \epsilon^i f_i(x,y) + o(\epsilon^m),\quad f_0\not \equiv 0,\\
g(x,y,\epsilon) &= g_0(x,y) + \sum_{i=1}^m \epsilon^i g_i(x,y) + o(\epsilon^m),\quad g_0\not \equiv 0.
\end{align*}
\end{ass}

General dynamical systems depend on parameters. 
For example, the FitzHugh-Nagumo system (\ref{FN-intro}) contains $a, c, \gamma, \delta$ as parameters. 
Throughout this paper we do not care about parameter dependence of dynamical systems unless otherwise specified.

%
%	New Subsection
%
\subsection{Several preceding works related to global trajectories and singular perturbation problems with rigorous numerics}
\label{section-preceding}
There are many preceding works for the existence of global orbits with rigorous numerics for {\it regular} dynamical systems, namely, $g\equiv 0$ case. For example, Wilczak and Zgliczy\'{n}ski \cite{WZ} apply the topological tool called covering relations to the existence of various type of trajectories in dynamical systems, such as periodic orbits, homoclinic orbits and heteroclinic orbits.
The essence of covering relations is to describe behavior of rectangular-like sets called $h$-sets and apply the mapping degree to the existence of solutions. One of powerful properties of covering relations is that every $h$-sets can be used as a joint of trajectories and that we can validate various complicating behavior of dynamical systems. Indeed, for example, \cite{W} and \cite{W2} by Wilczak validate various type of complex trajectories such as Shi'lnikov homoclinic solutions, heteroclinic solutions and infinitely many periodic solutions in concrete systems (e.g. Michelson system or R\"{o}ssler system).

On the other hand, van der Berg et. al. \cite{BMLM} produce the other approach of rigorous numerical computations of connecting orbits using {\em radii polynomials} and {\em parametrization} technique. Their main idea is to reduce the original problem to a projected boundary value problem in an infinite dimensional functional space via a fixed point argument. Their formulation involves a higher order parametrization of invariant manifolds near equilibria for describing stable and unstable manifolds. Their approach is free from integrations of vector fields.
Hence one can validate various additional properties of invariant manifolds without any knowledges of the existence of trajectories \cite{CL}.

\bigskip
As for rigorous numerics for singular perturbation problems, Gameiro et. al. \cite{GGKKMO} provide a validation method combined with the algebraic-topological singular perturbation analysis. Such analysis is knows as the Conley index theory \cite{Con, Mis}. They actually apply its singular perturbation version \cite{GKMOR, GKMO} to the singularly perturbed predetor-prey model with two slow variables. As a result, they prove the existence of topological horseshoe, in particular, infinite number of periodic orbits with computer assistance. Their results  show, however, the existence of solutions for only sufficiently small $\epsilon$ and hence the bound of $\epsilon > 0$ where the existence result holds is not given. When we apply the Conley index technique, it is necessary to provide an appropriate neighborhood of desiring orbits whose boundary {\em transversally} intersects the vector field defined by the full system (\ref{fast-slow}).
\par
On the other hand, Guckenheimer et. al. \cite{GJM2012} discusses rigorous enclosures of slow manifolds with computer assistance within explicit ranges of $\epsilon$. They introduce a concept of computational slow manifolds related to slow manifolds in geometric singular perturbation theory and succeed validations of slow manifolds in their settings with in explicit ranges of $\epsilon$, while validated ranges of $\epsilon$ are bounded away from $\epsilon = 0$, say, $\epsilon\in [10^{-6}, 10^{-2}]$.
Note that \cite{GJM2012} also discusses validations of tangential bifurcations of slow manifolds for all $\epsilon \in (0,\epsilon_0]$, where $\epsilon_0$ is a given number, say, $10^{-3}$.
One of the other works is a very recent one by Arioli and Koch \cite{AK2015}, which studies the existence and the stability of traveling pulse solutions for the (singularly perturbed) FitzHugh-Nagumo system.
This study, however, focuses only for the parameter range which is {\em not so small}, say, $\epsilon\approx 0.1$ or $0.001$.
In other words, the singular perturbation structure is ignored.

\bigskip
In the case of singular perturbation problems, direct computations of global orbits without any ideas are not practical in various scenes both in the rigorous and in the non-rigorous sense. 

Direct applications of such preceding works without any modifications would yield the failure of operations if we try to cover not only $\epsilon$ which is not so small (e.g. $\epsilon = 10^{-5}$ or $10^{-6}$) but also extremely small $\epsilon$, possibly smaller than machine epsilon.
This failure is because either the stiffness of problems or the effect of fast dynamics. 
Even if it succeeds, there would be huge computation costs due to very slow behavior around slow manifolds. 
Of course, there is a matching problem connecting fast and slow behavior, which is generally arisen in singular perturbation problems. 
If we can overcome all such difficulties as simple as possible, the scope of applications of preceding concepts will dramatically extend.

%
%	New Section
%
\section{Preliminaries}
\label{section-preliminary}
When we consider a fast-slow system (\ref{fast-slow}) from the viewpoint of geometric singular perturbation theory following Fenichel (e.g. \cite{F}), the central issue is {\em normally hyperbolic invariant manifolds}.
Fenichel's theory tells us very rich structure of normally hyperbolic invariant manifolds consisting of equilibria and their small perturbations.
Such structures are fully applied to our arguments. 
In the beginning of this section, we review several results about normally hyperbolic invariant manifolds for fast-slow systems.
\par
Our main methodologies to consider fast-slow systems are well-known topological tools called {\em covering relations}, {\em isolating blocks} and {\em cone conditions}.
These tools well describe behavior of solution sets as well as their asymptotic behavior. 
In successive sections we see that these tools work well even for fast-slow systems.
In this section, we also review two of such tools. 
We also provide a procedure of isolating blocks suitable for fast-slow systems so that they validate slow manifolds, which are available to various systems with computer assistance.
Cone conditions for singular perturbation problems are stated later.
\par
Finally note that readers who are familiar with these topics can skip this section.

%
%	New Subsection
%
\subsection{Fenichel's Invariant Manifold Theorems : review}
\label{section-Fenichel}
Here we briefly review the {\em Fenichel's Invariant Manifold Theorems} (e.g. \cite{F}), following arguments in \cite{Jones}.
The central goal of these results is the description of flow near the set $S_0 = \{(x,y,0)\mid f(x,y,0)=0\}$ with manifold structures: the critical manifold.
The critical manifold $S_0$ can be considered as the $y$-parameter family of equilibria of the layer problem (\ref{layer}). 
Under appropriate hypotheses, $S_0$ can be represented by the graph of a function $x = h(y)$ for $y\in K$, where $K\subset \mathbb{R}^l$ is a compact, simply connected set.

A central assumption among Fenichel's theory is the {\em normal hyperbolicity} and the graph representation of  $S_0$.
\begin{description}
\item[(F)] The set $S_0$ is given by the graph of the $C^\infty$ function $h^0(y)$ for $y\in K$, where the set $K$ is a compact, simply connected domain whose boundary is an $(l-1)$-dimensional $C^\infty$ submanifold. Moreover, assume that $S_0$ is normally hyperbolic. 
\end{description}

\begin{rem}[cf. \cite{Jones}]\rm
Recall that the manifold $S_0$ is {\em normally hyperbolic} if the linearization of (\ref{layer}) at each point in $S_0$ has exactly $l$ eigenvalues on the imaginary axis.

A set $M$ is said to be {\em locally invariant under the flow generated by (\ref{fast-slow})} if it has a neighborhood $V$ of $M$ so that no trajectory can leave $M$ without also leaving $V$. In other words, it is locally invariant if for all $x\in M$, then $\varphi ([0,t],x) \subset V$ implies that $\varphi( [0,t], x)\subset M$, similarly with $[0,t]$ replaced by $[t,0]$ when $t < 0$, where $\varphi$ is a flow.
\end{rem}

Without the loss of generality, we can assume that $h^0(y) = 0$ for all $y\in K$.
Then, there exist the stable and unstable eigenspaces, $S(y) $ and $U(y)$, such that $\dim S(y) = s$ and $\dim U(y) = u$ hold for all $y\in K$. 
With this in mind, we take the transformation $x(\in \mathbb{R}^n) \mapsto (a,b)\in \mathbb{R}^{u + s}$ so that (\ref{fast-slow}) is expressed by 
\begin{equation}
\label{abstract-form0}
\begin{cases}
a' = A(y)a + F_1(x,y,\epsilon) & \\
b' = B(y)b + F_2(x,y,\epsilon) & \\
y' = \epsilon g(x,y,\epsilon) & 
\end{cases}.
\end{equation}
Here $A(y)$ denotes the $u\times u$ matrix which all eigenvalues have positive real part and $B(y)$ denotes the $s\times s$ matrix which all eigenvalues have negative real part. $F_1$ and $F_2$ denotes the higher order term which admit a positive number $\gamma > 0$ satisfying 
\begin{equation*}
|F_i| \leq \gamma(|x| + \epsilon)\quad \text{ as }|x|\to 0.
\end{equation*}
More precise assumptions for $A(y)$ and $B(y)$ are as follows: there exist the quantities $\lambda_A > 0$ and $\mu_B < 0$ such that
\begin{align}
&\lambda_A < {\rm Re}\lambda\quad \text{ for all }\quad \lambda\in {\rm Spec}(A(y))\text{ and }y\in K,\\
&\mu_B > {\rm Re}\lambda\quad \text{ for all }\quad \lambda\in {\rm Spec}(B(y))\text{ and }y\in K.
\end{align}

The key consideration of slow manifolds is the following Fenichel's invariant manifold theorem in terms of graph representations.
\begin{prop}[Persistence of invariant manifolds. cf. \cite{Jones}]
\label{prop-Fen1}
Under the assumption (F), for sufficiently small $\epsilon > 0$, there is a function $x = h^\epsilon(y)$ defined on $K$. 
The graph $S_\epsilon = \{(x,y)\mid x=h^\epsilon(y)\}$ is locally invariant under (\ref{fast-slow}). Moreover $h^\epsilon$ is $C^r$, for any $r < +\infty$, jointly in $y$ and $\epsilon$.
\end{prop}

Proposition \ref{prop-Fen1} is just a consequence of the second Invariant Manifold Theorem as follows.

\begin{prop}[Stable and unstable manifolds. cf. \cite{Jones}]
\label{prop-Fen2}
Under the assumption (F), for sufficiently small $\epsilon > 0$, then for some $\Delta > 0$,
\begin{enumerate}
\item there is a function $a = h_s(b,y,\epsilon)$ defined on $\{(b,y,\epsilon)\mid |b|\leq \Delta, y\in K\}$, such that the graph $W^s(S_\epsilon) = \{(a,b,y)\mid a = h_s(b,y,\epsilon)\}$ is locally invariant under (\ref{fast-slow}). 
Moreover, $a = h_s(b,y,\epsilon)$ is $C^r$, for any $r < +\infty$, jointly in $y$ and $\epsilon$.
\item there is a function $b = h_u(s,y,\epsilon)$ defined on $\{(a,y,\epsilon)\mid |a|\leq \Delta, y\in K\}$, such that the graph $W^u(S_\epsilon) = \{(a,b,y)\mid b = h_u(a,y,\epsilon)\}$ is locally invariant under (\ref{fast-slow}). 
Moreover, $b = h_u(a,y,\epsilon)$ is $C^r$, for any $r < +\infty$, jointly in $y$ and $\epsilon$.
\end{enumerate}
\end{prop}

Fenichel's theorems, Propositions \ref{prop-Fen1} and \ref{prop-Fen2}, insist that normally hyperbolic invariant manifolds as well as their stable and unstable manifolds persist to locally invariant manifolds in the full system (\ref{fast-slow}) for sufficiently small $\epsilon > 0$. 
In other words, slow manifolds can be realized by the $\epsilon$-continuation of normally hyperbolic critical manifolds in the layer problem (\ref{layer}).
The perturbed manifold $S_\epsilon$ for $\epsilon > 0$ is called {\em a slow manifold}.
One of strategies for constructing such manifolds is the construction of a family of {\em isolating blocks} and {\em moving cones} (\cite{Jones}). 
Isolating blocks describe the behavior of vector fields which are transversal to their boundaries, which are well discussed in the Conley index theory \cite{Con, Mis}.
Isolating blocks are reviewed in Section \ref{section-isolatingblock}.

\bigskip
Fenichel's invariant manifold theory informs us not only the existence of perturbed slow manifolds but also invariant foliations of $W^u(S_\epsilon)$ and $W^u(S_\epsilon)$. More precisely, the following result (the third Invariant Manifold Theorem) is known.

\begin{prop}[Fenichel fibering, cf. Theorem 6, 7 in \cite{Jones}]
\label{prop-Fen3}
Under the assumption (F), for sufficiently small $\epsilon > 0$ the following statements hold. For each $v = v_\epsilon = (\hat y, \epsilon)\in S_\epsilon$,
\begin{enumerate}
\item there is a function $(a, y) = h^v_s(b)$ for $|b|\leq \Delta$ sufficiently small so that the graph
\begin{equation*}
W^s(v) = \{(a,b,y,\epsilon)\mid (a, y) = h^v_s(b)\} \subset W^s(S_\epsilon)
\end{equation*}
forms a locally invariant manifold in the sense that
$\varphi_{\epsilon,N}(t, W^s(v)) \subset W^s(\varphi_\epsilon(t,v))$ holds if $\varphi_\epsilon(s,v)\in N$ for all $s\in [0,t]$.
Here the set $\varphi_{\epsilon,N}(t, A)$ denotes the forward evolution of a set $A$ restricted to $N$ given by
\begin{equation*}
\varphi_{\epsilon,N}(t, A) = \{\varphi_\epsilon(t,u)\mid u\in A\text{ and }\varphi_\epsilon([0,t],u)\subset N\}.
\end{equation*}
Moreover, $(a, y) = h^v_s(b)$ is $C^r$ in $v$ and $\epsilon$ jointly for any $r<\infty$.
\item there is a function $(b, y) = h^v_u(a)$ for $|b|\leq \Delta$ sufficiently small so that the graph
\begin{equation*}
W^u(v) = \{(a,b,y,\epsilon)\mid (a, y) = h^v_s(b)\} \subset W^u(S_\epsilon)
\end{equation*}
forms a locally invariant manifold in the sense that
$\varphi_{\epsilon,N}(t, W^u(v)) \subset W^u(\varphi_\epsilon(t,v))$ holds if $\varphi_\epsilon(s,v)\in N$ for all $s\in [t,0]$.
Moreover, $(b, y) = h^v_u(a)$ is $C^r$ in $v$ and $\epsilon$ jointly for any $r<\infty$.
\end{enumerate}
\end{prop}
This invariant foliation is sometimes referred to as {\em Feniciel fibering}. This fibering ensures us the following representations:
\begin{equation*}
W^u(A) = \bigcup_{v\in A}W^u(v),\quad W^s(A) = \bigcup_{v\in A}W^s(v),
\end{equation*}
where $A$ is a subset of slow manifolds.

%
%	New Subsection
%
\subsection{Covering relations : review}
\label{section-cov}

Our main approach to tracking solution orbits is a topological tool called {\em covering relations}.
Covering relations describe topological transversality of rectangular-like domains called {\em $h$-sets} relative to continuous map, and
there are various studies not only from the mathematical viewpoint but also for applications with rigorous numerics (e.g. \cite{CAPD, W, WZ, ZCov, ZG}).
In the present study, we apply this topological methodology to singular perturbation problems. 
In this section, we summarize notions among the theory of covering relations.
For a given norm on $\mathbb{R}^m$, let $B_m(c,r)$ be the open ball of radius $r$ centered at $c\in \mathbb{R}^m$. 
For simplicity, also let $B_m = B_m(0,1)$. We set $\mathbb{R}^0 = \{0\}$, $B_0(0,r) = \{0\}$ and $\partial B_0(0,r) = \emptyset$.

\begin{dfn}[$h$-set, cf. \cite{ZCov, ZG}]\rm
\label{dfn-hset}
An {\em $h$-set} consists of the following set, integers and a map:
\begin{itemize}
\item A compact subset $N\subset \mathbb{R}^m$.
\item Nonnegative integers $u(N)$ and $s(N)$ such that $u(N) + s(N) = n$ with $n\leq m$.
\item A homeomorphism $c_N:\mathbb{R}^n\to \mathbb{R}^{u(N)}\times \mathbb{R}^{s(N)}$ satisfying
\begin{equation*}
c_N(N) = \overline{B_{u(N)}}\times \overline{B_{s(N)}}.
\end{equation*}
\end{itemize}
Finally define the {\em dimension} of an $h$-set $N$ by $\dim N:= n$.
\end{dfn}
We shall write an $h$-set $(N,u(N),s(N),c_N)$ simply by $N$ if no confusion arises. Let
\begin{align*}
N_c &:=  \overline{B_{u(N)}} \times  \overline{B_{s(N)}},\\
N_c^- &:= \partial \overline{B_{u(N)}} \times  \overline{B_{s(N)}},\\
N_c^+ &:= \overline{B_{u(N)}} \times  \partial \overline{B_{s(N)}},\\
N^- &:= c_N^{-1}(N_c^-), \quad N^+:= c_N^{-1}(N_c^+)
\end{align*}

The following notion describes the topological transversality between two $h$-sets relative to continuous maps.

\begin{dfn}[Covering relations, cf. \cite{ZCov, ZG}]\rm
\label{dfn-covrel}
Let $N, M\subset \mathbb{R}^m$ be $h$-sets with $u(N)+s(N), u(M)+s(M)\leq m$ and $u(N)=u(M) = u$. $f: N \to \mathbb{R}^{\dim M}$ denotes a continuous mapping and $f_c:= c_M\circ f\circ c_N^{-1}: N_c\to \mathbb{R}^{u} \times \mathbb{R}^{s(M)}$. We say $N$ {\em $f$-covers} $M$ ($N\overset{f}{\Longrightarrow}M$) if the following statements hold: 
\begin{enumerate}
\item There exists a continuous homotopy $h:[0,1]\times N_c\to \mathbb{R}^{u}\times \mathbb{R}^{s(M)}$ satisfying
\begin{align*}
&h_0 = f_c,\\
&h([0,1],N_c^-)\cap M_c = \emptyset,\\
&h([0,1],N_c)\cap M_c^+ = \emptyset,
\end{align*}
where $h_\lambda = h(\lambda, \cdot)$ ($\lambda \in [0,1]$).
\item There exists a mapping $A:\mathbb{R}^{u}\to \mathbb{R}^{u}$ such that
\begin{equation}
\label{cov-degree}
\begin{cases}
h_1(p,q) = (A(p),0), &\\
A(\partial B_u (0,1)) \subset \mathbb{R}^u \setminus \overline{B_u}(0,1), & \\
\deg(A, \overline{B_u}, 0)\not = 0
\end{cases}
\end{equation}
holds for $p\in \overline{B_u}(0,1), q\in \overline{B_s}(0,1)$.
\end{enumerate}
\end{dfn}

\begin{rem}\rm
In definition of covering relation between $N$ and $M$, the disagreement of $\dim N$ and $\dim M$ is not essential. On the contrary, the equality $u(N) = u(M) = u$ is essential because the mapping degree of $u$-dimensional mapping $A$ should be derived.
\end{rem}

The following propositions gives us useful sufficient conditions for detecting covering relations in practical situations.

\begin{prop}[Finding covering relations, Theorem 15 in \cite{ZG}]
\label{prop-find-CR}
Let $N,M$ be two $h$-sets in $\mathbb{R}^n$ such that $u(N)=u(M) = u$ and $s(N)=s(M)=s$. Let $f:N\to \mathbb{R}^n$ be continuous. Let $f_c = c_M\circ f \circ c_N^{-1} : N_c \to \mathbb{R}^u \times \mathbb{R}^s$. Assume that there exists $q_0 \in \overline{B_s}$ such that following conditions are satisfied:
\begin{enumerate}
\item Setting $S(M)_c^- = \{(p,q)\in \mathbb{R}^u\times \mathbb{R}^s \mid \|p\| > 1\}$,
\begin{align*}
&f_c(\overline{B_u}\times \{q_0\}) \subset \Int(S(M)_c^- \cup M_c),\\
&f_c(N_c^-)\cap M_c = \emptyset,\\
&f_c(N_c)\cap M_c^+ = \emptyset.
\end{align*}
\item Define a map $A_{q_0} : \mathbb{R}^u \to \mathbb{R}^u$ by
\begin{equation*}
A_{q_0}(p) := \pi_u(f_c(p,q_0)),
\end{equation*}
where $\pi_u : \mathbb{R}^u\times \mathbb{R}^s \to \mathbb{R}^u$ be the orthogonal projection onto $\mathbb{R}^u$, $\pi_u(p,q) = p$. Assume that 
\begin{equation*}
A_{q_0}(\partial B_u) \subset \mathbb{R}^u \setminus \overline{B_u},\quad  \deg(A_{q_0}, \overline{B_u},0) \not = 0.
\end{equation*}
\end{enumerate}
Then $N  \overset{f}{\Longrightarrow} M$.
\end{prop}

\begin{prop}[Covering relation in the case $u=1$, Definition 10 in \cite{Z3}]
\label{prop-CR-u1}
Let $N, M$ be $h$-sets with $u(N) = u(M) = 1$. Let $f : N\to \mathbb{R}^{\dim M}$ be continuous. Set
\begin{align*}
&N_c^L := \{-1\}\times \overline{B_{s(N)}},\quad N_c^R := \{+1\}\times \overline{B_{s(N)}},\\
&S(N)_c^L := (-\infty, -1)\times \mathbb{R}^{s(N)},\quad S(N)_c^R := (+1,+\infty)\times \mathbb{R}^{s(N)}.
\end{align*}
Assume that there exists $q_0\in \overline{B_{s(N)}}$ such that
\begin{align*}
f(c_N([-1,1]\times \{q_0\})) & \subset \Int (S(M)^L \cup M \cup S(M)^R),\\
f(N)\cap M^+ = \emptyset
\end{align*}
and either of the following conditions holds:
\begin{align*}
&f(N^L)\subset S(M)^L\quad \text{ and }f(N^L)\subset S(M)^R,\\
&f(N^L)\subset S(M)^R\quad \text{ and }f(N^L)\subset S(M)^L.
\end{align*}
Then $N  \overset{f}{\Longrightarrow} M$. 
\end{prop}

We also consider covering relations with respect to the inverse of continuous maps.

\begin{dfn}[Back-covering relation, Definition 3 and 4 in \cite{ZG}]\rm
Let $N$ be an $h$-set. Define the $h$-set $N^T$ as follows:
\begin{itemize}
\item The compact subset of the quadruple $N^T$ is $N$ itself.
\item $u(N^T) = s(N)$, $s(N^T) = u(N)$.
\item The homeomorphism $c_{N^T} : \mathbb{R}^n \to \mathbb{R}^n = \mathbb{R}^{u(N^T)}\times \mathbb{R}^{s(N^T)}$ is given by
\begin{equation*}
c_{N^T}(x) = j(c_N(x)),
\end{equation*}
where $j : \mathbb{R}^{u(N)}\times \mathbb{R}^{s(N)} \to  \mathbb{R}^{s(N)}\times \mathbb{R}^{u(N)}$ is given by $j(p,q) = (q,p)$.
\end{itemize}
Notice that $N^{T,+} = N^-$ and $N^{T,-} = N^+$.

Next, let $N, M$ be $h$-sets such that $u(N)=u(M)$. 
Let $g:\Omega \subset \mathbb{R}^n\to \mathbb{R}^n$.
Assume that $g^{-1}:M\to \mathbb{R}^n$ is well-defined and continuous. 
Then we say $N$ {\em $g$-back-covers $M$} ($N\overset{g}{\Longleftarrow} M$) if $M^T\overset{g^{-1}}{\Longrightarrow} N^T$ holds.
\end{dfn}

A fundamental result in the theory of covering relations is the following proposition.

\begin{prop}[Theorem 4 in \cite{ZG}]
\label{ZG-periodic}
Let $N_i, i=0,1,\cdots, k$ be $h$-sets such that $u(N_i)= u$ for $i=0,1,\cdots, k$ and let $f_i: N_i \to \mathbb{R}^{\dim (N_{i+1})}$, $i=0,1,\cdots, k-1$, be continuous. 
Assume that, for all $i=0,1,\cdots, k-1$, either of the following holds:
\begin{equation*}
N_i\overset{f_i}{\Longrightarrow}N_{i+1}
\end{equation*}
or
\begin{equation*}
N_i \subset D(f_i^{-1})\quad \text{ and }\quad N_i\overset{f_i}{\Longleftarrow}N_{i+1}.
\end{equation*}
Then there is a point $p\in \Int N_0$ such that
\begin{equation*}
f_i \circ f_{i-1}\circ \cdots f_0(p)\in \Int N_i\quad \text{ for all }i=0,\cdots, k-1.
\end{equation*}
If we additionally assume $N_k\overset{f_k}{\Longrightarrow}N_0$, then the point $p\in \Int N_0$ can be chosen so that 
\begin{equation*}
f_k \circ f_{k-1}\circ \cdots f_0(p)=p.
\end{equation*}
\end{prop}

We sometimes consider an infinite sequence of covering relations.
To deal with such a situation, we define the following concept.
\begin{dfn}[Admissibility, cf. Definition 2.4 in \cite{W}]\rm
\label{dfn-cov-admissibility}
Let $\{M_i\}_{i=1}^k$ be $h$-sets in $\mathbb{R}^m$ and $f : \bigcup_{i=1}^m M_i \to \mathbb{R}^m$ be a continuous map. We say the index sequence $\{i_j\}_{j\in \mathbb{Z}} \subset \{1,\cdots, k\}^\mathbb{Z}$ is {\em admissible with respect to $f$} if $M_{i_j}\overset{f}{\Longrightarrow} M_{i_{j+1}}$ holds of all $j\in \mathbb{Z}$.
Similarly, we say the index sequence $\{i_j\}_{j\in \mathbb{Z}} \subset \{1,\cdots, k\}^\mathbb{Z}$ is {\em back-admissible with respect to $f$} if $M_{i_j}\overset{f}{\Longleftarrow} M_{i_{j+1}}$ holds of all $j\in \mathbb{Z}$. In this case, $f^{-1}$ is assumed to be well-defined in a neighborhood of $\bigcup_{i=1}^m M_i$ and continuous.
\end{dfn}

Recall that the second Invariant Manifold Theorem, Proposition \ref{prop-Fen2}, claims that the stable and unstable manifolds of normally hyperbolic invariant manifolds  can be described by graphs of smooth functions.
The concept of {\em horizontal and vertical disks} are useful to describe asymptotic trajectories in terms of covering relations for describing these situations.

\begin{dfn}[Horizontal and vertical disk, e.g. \cite{W, ZCov}]\rm
\label{dfn-disks}
Let $N$ be an $h$-set. 
Let $b_s: \overline{B_{s(N)}}\to N$ be continuous and let $(b_s)_c = c_N \circ b_s$. We say that $b_s$ is a {\em vertical disk in $N$} if there exists a homotopy $h: [0,1]\times \overline{B_{s(N)}}\to N_c$ such that
\begin{align*}
h_0 &= (b_s)_c,\\
h_1(x) &= (0,x),\quad \text{ for all }x\in \overline{B_{s(N)}},\\
h(t,x) &\in N_c^+,\quad \text{ for all }t\in [0,1]\text{ and }x\in \partial B_{s(N)}.
\end{align*}
Let $b_u: \overline{B_{u(N)}}\to N$ be continuous and let $(b_u)_c = c_N \circ b_u$. We say that $b_u$ is a {\em horizontal disk in $N$} if there exists a homotopy $h: [0,1]\times \overline{B_{u(N)}}\to N_c$ such that
\begin{align*}
h_0 &= (b_u)_c,\\
h_1(x) &= (x,0),\quad \text{ for all }x\in \overline{B_{u(N)}},\\
h(t,x) &\in N_c^-,\quad \text{ for all }t\in [0,1]\text{ and }x\in \partial B_{u(N)}.
\end{align*}
\end{dfn}

Combining these concepts with covering relations, we obtain the following result, 
which is often applied to the existence of homoclinic and heteroclinic orbits.
\begin{prop}[Theorem 3 in \cite{WZ}, Theorem 3.9 in \cite{W}]
\label{WZ-heteroclinic}
Let $N_i, i=0,1,\cdots, k$ be $h$-sets such that $u(N_i)= u$ for $i=0,1,\cdots, k$ and let $f_i: N_i \to \mathbb{R}^{\dim (N_{i+1})}$, $i=0,1,\cdots, k-1$, be continuous. Let $b: \overline{B_u}\to N_0$ be a horizontal disk in $N_0$ and let $v: \overline{B_{s(N_k)}}\to N_k$ be a vertical disk in $N_k$. If $N_i\overset{f_i}{\Longrightarrow}N_{i+1}$ holds for $i=0,1,\cdots, k-1$, then there exists $\tau \in \overline{B_u}$ such that
\begin{align*}
&(f_i\circ f_{i-1}\circ \cdots \circ f_0)(b(\tau)) \in N_{i+1},\quad \text{ for }i=0,1,\cdots, k-2,\\
&(f_{k-1}\circ f_{k-2}\circ \cdots \circ f_0)(b(\tau))\in v(\overline{B_{s(N_k)}}).
\end{align*}
\end{prop}

%
%	New Subsection
%
\subsection{Isolating blocks : review and applications to fast-slow systems}
\label{section-isolatingblock}

A concept of {\em isolating blocks} are typically discussed in the Conley index theory (e.g. \cite{Con, Mis}), which studies the structure of isolated invariant sets from the algebraic-topological viewpoint.
Central notions are {\em isolating neighborhoods} or {\em index pairs} in the Conley index theory, but we concentrate our attentions on {\em isolating blocks} defined as follows. 
In our case, the blocks can be considered very flexible from the viewpoint of not only covering relations but also rigorous numerics.
Moreover, isolating blocks play central roles for the existence of slow manifolds (cf. \cite{Jones} and Section \ref{section-inv-mfd}). 
In this section, we firstly review the definition of isolating blocks. 
Secondly, we show a procedure of isolating blocks around equilibria and their alternatives for fast-slow systems with computer assistance, following \cite{ZM}.

%
%	New Subsubsection
%
\subsubsection{Definition}
\begin{dfn}[Isolating block]\rm
\label{dfn-isolation}
Let $N\subset \mathbb{R}^m$ be a compact set. We say $N$ an {\em isolating neighborhood} if $\Inv(N)\subset \Int (N)$ holds, where 
\begin{equation*}
\Inv(N):= \{x\in N \mid \varphi(\mathbb{R},x)\subset N\}
\end{equation*}
for a flow $\varphi: \mathbb{R}\times \mathbb{R}^m\to \mathbb{R}^m$ on $\mathbb{R}^m$.
Next let $B\subset \mathbb{R}^m$ be a compact set and $x\in \partial B$. We say $x$ an {\em exit} ({\em resp. entrance}) point of $B$, if for every solution $\sigma:[-\delta_1,\delta_2]\to \mathbb{R}^N$ through $x= \sigma(0)$, with $\delta_1\geq 0$ and $\delta_2 > 0$ there are $0\leq \epsilon_1 \leq \delta_1$ and $0 < \epsilon_2 \leq \delta_2$ such that for $0 < t \leq \epsilon_2$,
\begin{equation*}
\sigma(t)\not \in B\ (\text{resp. } \sigma(t)\in \Int(B)),
\end{equation*}
and for $-\epsilon_1 \leq t < 0$,
\begin{equation*}
\sigma(t)\not \in \partial B\ (\text{resp. } \sigma(t)\not \in B)
\end{equation*}
hold. $B^{\exit}$ (resp. $B^{\ent}$) denote the set of all exit (resp. entrance) points of the closed set $B$. We call $B^{\exit}$ and $B^{\ent}$ {\em the exit} and {\em the entrance} of $B$, respectively.
Finally $B$ is called {\em an isolating block} if $\partial B = B^{\exit}\cup B^{\ent}$ holds and $B^{\exit}$ is closed in $\partial B$.
\end{dfn}
Obviously, an isolating block is also an isolating neighborhood.

%
%	New Subsubsection
%
\subsubsection{Construction around equilibria via rigorous numerics : a basic form}
\label{section-block-basic}
There is a preceding work for the systematic construction of isolating blocks around equilibria \cite{ZM}. 
Here we briefly review the method therein keeping the fast-slow system (\ref{fast-slow}) in mind. 
The first part is the review of the preceding work \cite{ZM}.
As the second part, we discuss the analogue of arguments to fast-slow systems.
One will see that such procedures are very suitable for analyzing dynamics around invariant manifolds.

\bigskip
Let $K \subset \mathbb{R}^l$ be a compact, connected and simply connected set. 
Consider first the differential equation of the following abstract form: 
\begin{equation}
\label{abstract-layer}
x' = f(x,\lambda), \quad x\in \mathbb{R}^n,\quad \lambda\in K,\quad  f: \mathbb{R}^n\times K \to \mathbb{R}^n,
\end{equation}
which corresponds to the layer problem (\ref{layer}). 
For simplicity, assume that $f$ is $C^\infty$. 
Our purpose here is to construct an isolating block which contains an equilibrium of (\ref{abstract-layer}).

\bigskip
Let $p_0$ be a numerical equilibrium of (\ref{abstract-layer}) at $\lambda_0 \in K$ and rewrite (\ref{abstract-layer}) as a series around $(p_0,\lambda_0)$:
\begin{equation}
\label{Taylor}
x' = f_x(p_0,\lambda_0)(x-p_0) + \hat f(x,\lambda),
\end{equation}
where $f_x(p_0,\lambda_0)$ is the Fr\'{e}chet differential of $f$ with respect to $x$-variable at $(p_0,\lambda_0)$. $\hat f(x,\lambda)$ denotes the higher order term of $f$ with $O(|x-p_0|^2 + |\lambda-\lambda_0|)$. This term may in general contain an additional term arising from the numerical error $f(p_0,\lambda_0)\approx 0$.

Here assume that the $n\times n$-matrix $f_x(p_0,\lambda_0)$ is nonsingular. Diagonalizing $f_x(p_0,\lambda_0)$, which is generically possible, (\ref{Taylor}) is further rewritten by the following perturbed diagonal system around $(p_0,\lambda_0)$:
\begin{equation}
\label{ode-y-coord}
z_j' = \mu_j z_j + \tilde f_j(z,\lambda),\quad j=1,\cdots, n.
\end{equation}
Here $\mu_j\in \mathbb{C}$, $z = (z_1,\cdots, z_n)$ and $\tilde f (z,\lambda) = (\tilde f_1(z,\lambda), \cdots, \tilde f_n(z,\lambda))^T$ are defined by $x = Pz + p_0$ and $\tilde f(x,\lambda) = P(\hat f(z,\lambda))$, where $P=(P_{ij})_{i,j=1,\cdots, n}$ is a nonsingular matrix diagonalizing $f_x(p_0,\lambda_0)$ and $\ast^T$ is the transpose.

Let $N\subset \mathbb{R}^n$ be a compact set containing $p_0$.
Assume that each $\tilde f_j(z,\lambda)$ has a bound $[\delta_j^-, \delta_j^+]$ in $N\times K$, namely,
\begin{equation*}
\label{error-bounds}
\left\{ \tilde f_j(z,\lambda) \mid x = Pz+p_0\in N, \lambda \in K\right\}\subsetneq [\delta_j^-, \delta_j^+].
\end{equation*}
Then $z_j'$ must satisfy 
\begin{equation*}
\mu_j \left(z_j + \frac{\delta_j^-}{\mu_j}\right) < z_j' < \mu_j \left(z_j + \frac{\delta_j^+}{\mu_j}\right),\quad \forall z \text{ with }x=Pz+p_0\in N,\ \forall \lambda \in K.
\end{equation*}
For simplicity we assume that each $\mu_j$ is real. 
We then obtain the candidate of an isolating block $B$ in $z$-coordinate given by the following:
\begin{align}
\label{block-exit}
B:= \prod_{j=1}^n B_j,\quad B_j = [z_j^-, z_j^+] &:= \left[-\frac{\delta_j^+}{\mu_j}, -\frac{\delta_j^-}{\mu_j} \right]\quad \text{ if }\mu_j > 0,\\
\label{block-entrance}
B_j = [z_j^-, z_j^+] &:= \left[-\frac{\delta_j^-}{\mu_j}, -\frac{\delta_j^+}{\mu_j} \right]\quad \text{ if }\mu_j < 0.
\end{align}

\begin{rem}\rm
In the case that $\mu_j$ is complex-valued for some $i$, $f_x(p_0,\lambda_0)$ contains the complex conjugate of $\mu_j$ as the other eigenvalue. 
Without the loss of generality, we may assume $\mu_j = \alpha_j + \sqrt{-1}\beta_j$, $\mu_{j+1} = \bar \mu_j = \alpha_j - \sqrt{-1}\beta_j$, $\beta_j \not = 0$. 
To be simplified, we further assume that $\mu_j$ and $\mu_{j+1}$ are the only complex pair of eigenvalues of $f_x(p_0,\lambda_0)$. 
The general case can be handled in the same manner.
The dynamics for $z_j$ and $z_{j+1}$ is formally written by
\begin{align*}
z_j' &= \mu_j z_j + \tilde f_j(z,\lambda),\\
z_{j+1}' &= \mu_{j+1} z_{j+1} + \tilde f_{j+1}(z,\lambda).
\end{align*}
Now we would like to consider real dynamical systems. To do this we transform the above form into 
\begin{align*}
w_j' &= \alpha_j w_j + \beta_j w_{j+1} + \bar f_j(w,\lambda),\\
w_{j+1}' &= -\beta_j w_j + \alpha_j w_{j+1} + \bar f_{j+1}(w,\lambda)
\end{align*}
via $Q = \begin{pmatrix} 1 & 1 \\ \sqrt{-1} & -\sqrt{-1} \end{pmatrix}$, $(w_j,w_{j+1})^T=Q(y_j,y_{j+1})^T$ and $(\bar f_j,\bar f_{j+1})^T=Q(\tilde f_j,\tilde f_{j+1})^T$, where $w=(w_1,\cdots, w_n)$ is the new coordinate satisfying $w_i= y_i$ for $i\not = j,j+1$. 
Let $r_j(w,\lambda):= \sqrt{\bar f_j(w,\lambda)^2 + \bar f_{j+1}(w,\lambda)^2}$ and assume that $r_j(w,\lambda)$ is bounded by a positive number $\bar r_j$ uniformly on $N\times K$. 
Our aim here is to construct a candidate of isolating block and hence we assume that
\begin{itemize}
\item the scalar product of the vector field and the coordinate vector
\begin{equation*}
(w_j,w_{j+1}) \cdot (\alpha_j w_j + \beta_j w_{j+1} + \bar f_j(w,\lambda), -\beta_j w_j + \alpha_j w_{j+1} + \bar f_{j+1}(w,\lambda))
\end{equation*}
has the identical sign and the above function never attain $0$ on $\left\{(w_j,w_{j+1}) \mid \sqrt{w_j^2 + w_{j+1}^2} \leq b_j \right\}$ ($b_j > 0$).
\end{itemize}
With this assumption in mind, we set the candidate of isolating block in $i$- and $i+1$-th coordinate
\begin{equation*}
B_{j,j+1}:= \left\{(w_j,w_{j+1}) \mid \sqrt{w_j^2 + w_{j+1}^2} \leq \frac{\bar r_j}{|\alpha_j|} \right\}.
\end{equation*}
Its boundary becomes exit if $\alpha_j > 0$ and entrance if $\alpha_j < 0$. Finally, replace $B_j\times B_{j+1}$ in the definition of $B$ ((\ref{block-exit}) and (\ref{block-entrance})) by $B_{j,j+1}$.
\end{rem}

A series of estimates for error terms involves $N$ and it only makes sense if it is self-consistent, namely, $\{p_0\}+PB\subset N$.
If it is the case, then $B$ is desiring isolating block for (\ref{abstract-layer}). Indeed, if $\mu_j > 0$, then
\begin{equation*}
z_j'\mid_{z_j = z_j^-} < 0\quad \text{ and }\quad z_j'\mid_{z_j = z_j^+} > 0
\end{equation*}
hold. Namely, the set $\{z\in B \mid z_j = z_j^{\pm}\}$ is contained in the exit. Similarly if $\mu_j < 0$ then 
\begin{equation*}
z_j'\mid_{z_j = z_j^-} > 0\quad \text{ and }\quad z_j'\mid_{z_j = z_j^+} < 0
\end{equation*}
hold. Namely, the set $\{z\in B \mid z_j = z_j^{\pm}\}$ is contained in the entrance. Obviously $\partial B$ is the union of the closed exit and the entrance, which shows that $B$ is an isolating block.

Once such an isolating block $B$ is constructed, one obtains an equilibrium in $B$. 
\begin{prop}[cf. \cite{ZM}]\rm
\label{prop-existence-fixpt}
Let $B$ be an isolating block constructed as above. Then $B$ contains an equilibrium of (\ref{abstract-layer}) for all $\lambda \in K$. 
\end{prop}
This proposition is the consequence of general theory of the Conley index (\cite{Mc}). 
Note that the construction of isolating blocks stated in Proposition \ref{prop-existence-fixpt} around points which are {\em not necessarily equilibria} implies the existence of {\em rigorous} equilibria inside blocks.
With an additional property such as uniqueness or hyperbolicity of equilibria, 
this procedure will provide the smooth $\lambda$-parameter family of equilibria, 
which is stated in Section \ref{section-inv-mfd}. 

\bigskip
When we apply these ideas to the fast-slow system (\ref{fast-slow}), we only consider the fast system $x' = f(x,y,\epsilon)$.
Let $K \subset \mathbb{R}^l$ be as above, $\epsilon_0 > 0$ and $x = p_0$ be a numerical zero of $f(x,y_0,0)$ at $y = y_0\in K$. In this case we set $\epsilon = 0$ for computing numerical zeros. 
Via the procedure in the above, the fast system $x' = f(x,y,\epsilon)$ can be generically formulated by the following form: 
\begin{equation}
\label{ode-ab-coord}
a' = Aa + F_1(a,b,y,\epsilon),\quad b' = Bb + F_2(a,b,y,\epsilon).
\end{equation}
Here $A$ is a $u\times u$-dimensional diagonal matrix each of whose eigenvalues has positive real part. 
Similarly, $B$ is an $s\times s$-dimensional ($u+s=n$) diagonal matrix each of whose eigenvalues has negative real part. 
$F_1$ and $F_2$ are higher order terms depending on $p_0$ and $y_0$. 
Equivalently, writing (\ref{ode-ab-coord}) component-wise, 
\begin{align*}
a_j' &= \mu_j^a a_j  + F_{1,j}(x,y,\epsilon),\quad {\rm Re}\mu_j^a > 0,\quad j=1,\cdots, u,\\
b_j' &= \mu_j^b b_j  + F_{2,j}(x,y,\epsilon),\quad {\rm Re}\mu_j^b < 0,\quad j=1,\cdots, s.
\end{align*}
Let $N\subset \mathbb{R}^n$ be a compact set containing $p_0$. 
As before, assume that every $\mu_j^a$ and $\mu_j^b$ is real and that each $F_{i,j_i}$, $i=1,2$, $j_1 = 1,\cdots, u$, $j_2 = 1,\cdots, s$, admits the following enclosure with respect to $N\times K \times [0,\epsilon_0]$: 
\begin{equation}
\label{error-bounds-fast}
\left\{ F_{i,j_i}(x,y,\epsilon) \mid x = Py+p_0\in N, y\in K, \epsilon \in [0,\epsilon_0]\right\}\subsetneq [\delta_{i,j_i}^-, \delta_{i,j_i}^+].
\end{equation}
Define the set $D_c\subset \mathbb{R}^{n+l}$ by the following: 
\begin{equation}
\label{fast-block}
D_c:= \prod_{j=1}^{u} [a_j^-, a_j^+]\times \prod_{j=1}^{s} [b_j^-, b_j^+]\times K,\quad
[a_j^-, a_j^+]:= \left[-\frac{\delta_{1,j}^+}{\mu^a_j},-\frac{\delta_{1,j}^-}{\mu^a_j} \right],\quad
[b_j^-, b_j^+]:= \left[-\frac{\delta_{2,j}^-}{\mu^b_j},-\frac{\delta_{2,j}^+}{\mu^b_j} \right].
\end{equation}
If $\{(p_0,y_0)\}+PD_c\subset N\times K$ holds with an appropriate linear transform $P$ on $\mathbb{R}^{n+l}$, then this procedure is self-consistent. Under this self-consistence, we immediately know that
\begin{align*}
a_j ' > 0&\quad \forall (a,b,y,\epsilon) \in D_c\times [0,\epsilon_0] \text{ with }a_j= a_j^+,\\
a_j ' < 0&\quad \forall (a,b,y,\epsilon) \in D_c\times [0,\epsilon_0] \text{ with }a_j= a_j^-,\\
b_j ' < 0&\quad \forall (a,b,y,\epsilon) \in D_c\times [0,\epsilon_0] \text{ with }b_j= b_j^+,\\
b_j ' > 0&\quad \forall (a,b,y,\epsilon) \in D_c\times [0,\epsilon_0] \text{ with }b_j= b_j^-.
\end{align*}
Remark that these inequalities hold for all $\epsilon \in [0,\epsilon_0]$. This observation is the key point of the construction not only of limiting critical manifolds but of slow manifolds for $\epsilon \in (0,\epsilon_0]$, which is stated in Section \ref{section-inv-mfd}.

\begin{dfn}[Fast-saddle-type block]\rm
Let $D_c\subset \mathbb{R}^{n+l}$ be constructed by (\ref{fast-block}). 
Assume that $K = \overline{B_l} \subset \mathbb{R}^l$. 
We say $D_c$, equivalently $D:= c_D^{-1}(D)$ via a homeomorphism $c_D$, {\em a fast-saddle-type block}. 
Moreover, set
\begin{align*}
D_c^{f,-} &:= \{(a,b,y)\in D_c \mid a_j =a_j^\pm,\ j=1,\cdots, u\},\\
D_c^{f,+} &:= \{(a,b,y)\in D_c \mid b_j =b_j^\pm,\ j=1,\cdots, s\},\\
D_c^s &:= \{(a,b,y)\in D_c \mid y\in \partial B_l\},\\
D^{f,-} &:= c_D^{-1}(D_c^{f,-}),\quad D^{f,+}:= c_D^{-1}(D_c^{f,+}),\quad D^s:= c_D^{-1}(D_c^s).
\end{align*}
We say $D_c^{f,-}$ (equivalently $D^{f,-}$) {\em the fast-exit of $D$} and $D_c^{f,+}$ (equivalently $D^{f,+}$) {\em the fast-entrance of $D$}. 
\end{dfn}
\begin{rem}\rm
Obviously $D$ is an $h$-set, but the integer $u(D)$ and $s(D)$ in Definition \ref{dfn-hset} are not necessarily equal to $u$ and $s$, respectively. We do not have any assumptions for $D^s$ in the above definition. Indeed, $D$ is not necessarily an isolating block in the sense of Definition \ref{dfn-isolation}. 
That is why we omit the word \lq\lq isolating" from the definition of $D$.
\end{rem}

\bigskip
This construction can be slightly extended as follows. Let $\{\eta_{i,j_i}^\pm\}_{j_1=1,\cdots, u, j_2 = 1,\cdots, s}^{i=1,2}$ be a sequence of positive numbers. Defining 
\begin{align}
\notag
&\hat D_c:= \prod_{j=1}^{u} [\hat a_j^-, \hat a_j^+]\times \prod_{j=1}^{s} [\hat b_j^-, \hat b_j^+]\times K,\\
\label{fast-block-2}
&[\hat a_j^-, \hat a_j^+]:= \left[-\frac{\delta_{1,j}^+}{\mu^a_j} - \eta_{1,j_1}^-,-\frac{\delta_{1,j}^-}{\mu^a_j} + \eta_{1,j_1}^+ \right],\quad
[\hat b_j^-, \hat b_j^+]:= \left[-\frac{\delta_{2,j}^-}{\mu^b_j} - \eta_{2,j_2}^-,-\frac{\delta_{2,j}^+}{\mu^b_j} + \eta_{2,j_2}^+ \right],
\end{align}
we can prove that $\hat D_c$ is also a fast-saddle-type block if $P\hat D_c +\{(p_0,y_0)\} \subset N\times K$ holds. We further know
\begin{align*}
a_j ' > 0&\quad \forall (a,b,y,\epsilon) \in \hat D_c\times [0,\epsilon_0] \text{ with }a_j\in [a_j^+, \hat a_j^+],\\
a_j ' < 0&\quad \forall (a,b,y,\epsilon) \in \hat D_c\times [0,\epsilon_0] \text{ with }a_j\in [\hat a_j^-, a_j^-],\\
b_j ' < 0&\quad \forall (a,b,y,\epsilon) \in \hat D_c\times [0,\epsilon_0] \text{ with }b_j\in [b_j^+, \hat b_j^+]\\
b_j ' > 0&\quad \forall (a,b,y,\epsilon) \in \hat D_c\times [0,\epsilon_0] \text{ with }b_j\in [\hat b_j^-, b_j^-].
\end{align*}

This extension leads to the explicit lower bound estimate of distance between $\hat D^{f,\pm}$ and slow manifolds, which is stated in Section \ref{section-inv-mfd}.

%
%	New Subsubsection
%
\subsubsection{Construction around equilibria via rigorous numerics : the predictor-corrector approach}
\label{section-block-pred-corr}

Here we provide another approach for validating fast-saddle-type blocks.
In previous subsection, fast-saddle type blocks are constructed centered at $\{(\bar x, y)\mid y\in K\}$, where $K\subset \mathbb{R}^l$ is a small compact neighborhood of $\bar y \in K$. 
All transformations concerning eigenpairs are done at a point $(\bar x, \bar y)$.
Fig. \ref{fig-prod_corr}-(a) briefly shows this situation.
On the other hand, we can reselect the center of the candidate of blocks so that blocks can be chosen smaller.
As continuations of equilibria with respect to parameters, {\em the predictor-corrector approach} is one of effective approaches.
We now revisit the construction of fast-saddle type blocks with the predictor-corrector approach.

Let $(\bar x, \bar y)$ be a (numerical) equilibrium for (\ref{layer}), i.e., $f(\bar x, \bar y,0)\approx 0$, such that $f_x(\bar x, \bar y,0)$ is invertible.
Let $K$ be a compact neighborhood of $\bar y$.
The central idea is to choose the center as follows instead of $(\bar x, \bar y)$:
\begin{equation}
\label{new-center}
\left(\bar x + \frac{dx}{dy}(\bar y)(y-\bar y),y\right)\equiv \left(\bar x - f_x(\bar x, \bar y)^{-1}f_y(\bar x, \bar y)(y-\bar y),y\right),
\end{equation}
where $x = x(y)$ is the parametrization of $x$ with respect to $y$ such that $\bar x = x(\bar y)$ and that $f(x(y),y, 0) = 0$, which is actually realized in a small neighborhood of $\bar y$ in $\mathbb{R}^l$ since $f_x(\bar x, \bar y)$ is invertible. See Fig. \ref{fig-prod_corr}-(b).
Obviously, the identification in (\ref{new-center}) makes sense thanks to the Implicit Function Theorem.

\begin{figure}[htbp]\em
\begin{minipage}{0.45\hsize}
\centering
\includegraphics[width=8.0cm]{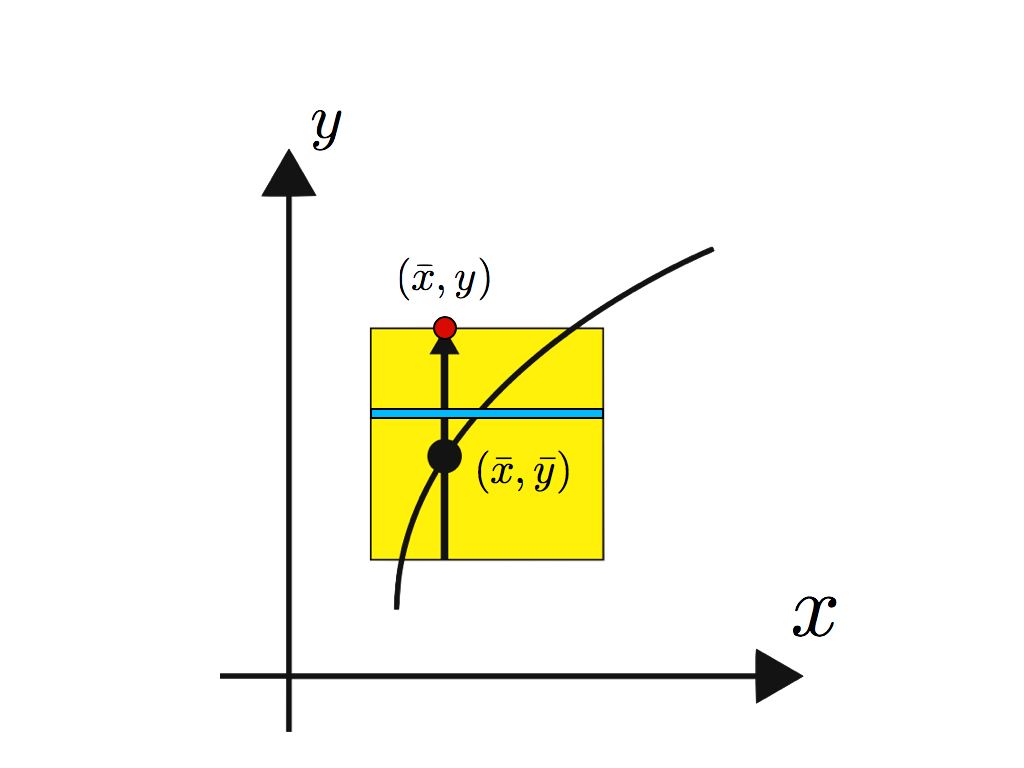}
(a)
\end{minipage}
\begin{minipage}{0.45\hsize}
\centering
\includegraphics[width=8.0cm]{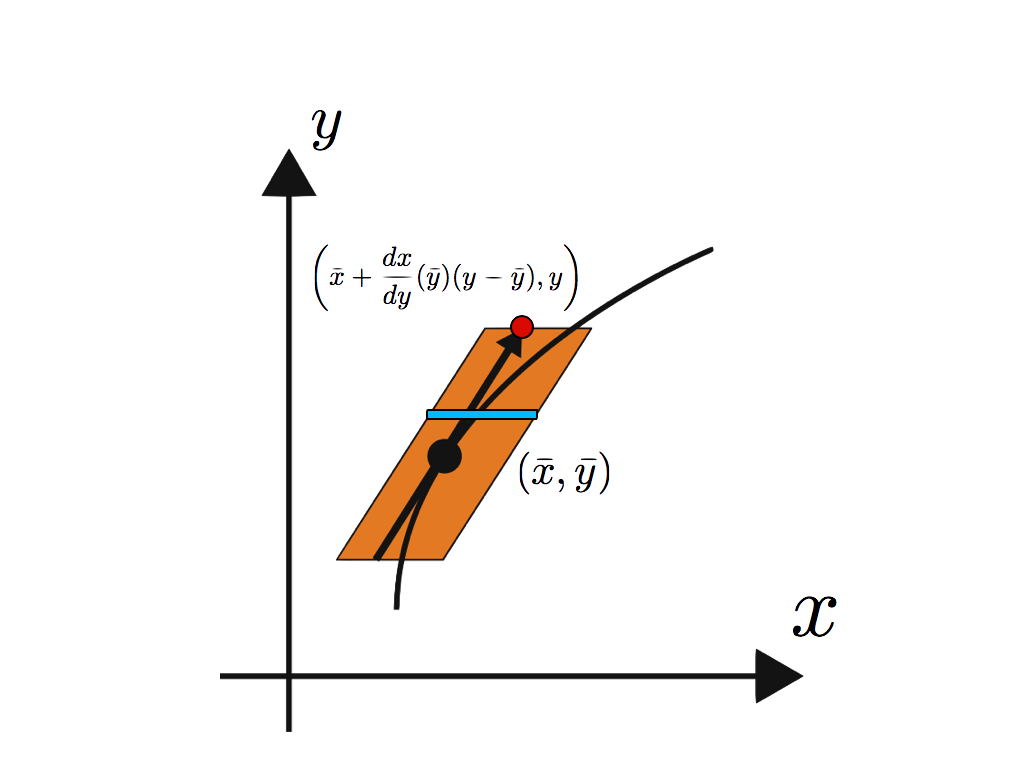}
(b)
\end{minipage}
\caption{Fast-saddle-type blocks with (a) basic form, and (b) predictor-corrector form.}
\label{fig-prod_corr}
(a): The center point $\bar x$ of blocks in $x$-coordinate is fixed.
The blue rectangle shows a fast-saddle-type block at $y\in K$. 
This procedure can be realized just by extending error bounds in (\ref{error-bounds-fast}) to $y$-directions.
However, the \lq\lq higher order" term $F_{i,j_i}(x,y,\epsilon)$ may contain the linear term with respect to $y$. 
This may cause the increase of error bounds.
\par
(b): The center point of blocks is moved along the tangent line (or plane for higher dimensional systems) at $(\bar x, \bar y)$.
The blue rectangle shows a fast-saddle-type block at $y\in K$. 
In principle, size of blocks will be smaller than (a) since the linear term with respect to $y$ in error term is approximately dropped.
\end{figure}

Around the new center, we define the new affine transformation $T : (z,w)\mapsto (x,y)$ as
\begin{equation}
(x,y) = T(z,w) := \left(Pz + \bar x - f_x(\bar x, \bar y)^{-1}f_y(\bar x, \bar y)w, w + \bar y\right).
\end{equation}
where $P$ is a nonsingular matrix diagonalizing $D_x(\bar x, \bar y)$.
Over the new $(z,w)$-coordinate, the fast system $x' = f(x,y,\epsilon)$ is transformed into the following:
\begin{align*}
z' &= P^{-1}\left(x' + \overline{f_x}^{-1}\overline{f_y}w'\right)\\
 	&= P^{-1}\left( f(x,y,\epsilon) + \epsilon\overline{f_x}^{-1}\overline{f_y}g(x,y,\epsilon) \right)\\
	&= P^{-1}\left( \overline{f_x}(Pz - \overline{f_x}^{-1}\overline{f_y} w) + \hat f(z,w,\epsilon)+ \epsilon \overline{f_x}^{-1}\overline{f_y}g(x,y,\epsilon) \right)\\
	&= \Lambda z + P^{-1}\left(-\overline{f_y}w + \hat f(z,w,\epsilon)+ \epsilon \overline{f_x}^{-1}\overline{f_y}g\left(Pz + \bar x - \overline{f_x}^{-1} \overline{f_y} w, w + \bar y, \epsilon\right) \right)\\
	&\equiv  \Lambda z + F(z,w,\epsilon),
\end{align*}
where $\overline{f_x} = f_x(\bar x, \bar y)$ and $\overline{D_y f}=f_y(\bar x, \bar y)$,  and $\Lambda = \diag(\mu^a_1,\cdots, \mu^a_u, \mu^b_1,\cdots, \mu^b_s)$.
The function $\hat f(z,w,\epsilon)$ denotes the higher order term of $f$ with $O(|z|^2 + |w|)$.
Dividing $z$ into $(a,b)$ corresponding to eigenvalues with positive real parts and negative real parts, respectively, as in (\ref{ode-ab-coord}), we can construct a candidate set of fast-saddle-type blocks as in (\ref{fast-block}).
By the similar implementations to Section \ref{section-block-basic}, we can construct fast-saddle-type blocks centered at (\ref{new-center}) for $y\in K$.
\par
Note that the higher order term $\hat f(z,w,\epsilon)$ contains the linear term of $w$ as $\overline{f_y}w$ with small errors in a sufficiently small neighborhood $K$ of $\bar y$.
This fact indicates that, in principle, size of blocks becomes smaller than those in Section \ref{section-block-basic}, as seen in Fig. \ref{fig-prod_corr}-(b).
This benefit is also useful for the following arguments.

%
%	New Section
%
\section{Slow manifold validations}

\label{section-inv-mfd}
In this section, we provide a verification theorem of slow manifolds as well as their stable and unstable manifolds.
Our goal here is to provide sufficient conditions to validate not only the critical manifold $S_0$ but also the perturbed slow manifold $S_\epsilon$ of (\ref{fast-slow})$_\epsilon$ {\em for all $\epsilon \in (0,\epsilon_0]$ in given regions}. 

Recall that Fenichel's results (Propositions \ref{prop-Fen1}, \ref{prop-Fen2}) assume normal hyperbolicity and graph representation of the critical manifold $S_0$.
These assumptions are nontrivial, but very essential to prove the persistence. 
Our verification theorem contains verification of normal hyperbolicity and graph representation of $S_0$.

The main idea is based on discussions in \cite{Jones}. 
For technical reasons, we use a multiple of $\epsilon$ as the new auxiliary variable. 
We set $\epsilon = \eta \sigma$ and $\sigma:= \epsilon_0 > 0$, where $\epsilon_0$ is a given positive number. 
We add the equation $\eta' = 0$ to (\ref{abstract-form0}). 
Furthermore, we consider the following system instead of (\ref{abstract-form0}) for simplicity: 
\begin{equation}
\label{abstract-form}
\begin{cases}
a' = Aa + F_1(x,y,\epsilon) & \\
b' = Bb + F_2(x,y,\epsilon) & \\
y' = \epsilon g(x,y,\epsilon) & \\
\eta' = 0 &
\end{cases}.
\end{equation}
Here $A$ denotes the $u\times u$ matrix which all eigenvalues have positive real part 
and $B$ denotes the $s\times s$ matrix which all eigenvalues have negative real part. 
Note that matrices $A$ and $B$ are (locally) independent of $y$.
This formulation is natural when we the construction of fast-saddle-type blocks stated in Section \ref{section-isolatingblock} is taken into account.

Let $M$ be a fast-saddle type block for (\ref{fast-slow}). 
Section \ref{section-isolatingblock} implies that the coordinate representation, $M_c$, is given by (\ref{fast-block}) (or (\ref{fast-block-2})), which is directly obtained from the system (\ref{abstract-form}). 
A fast-saddle-type block $M$ has the form (\ref{fast-block}), which has $a$-coordinate, $b$-coordinate and $y$-coordinate following (\ref{abstract-form}). 
With this in mind, we put  following notations.

\begin{nota}\rm
Let $\pi_a$, $\pi_b$, $\pi_y$, $\pi_{a,b}$, $\pi_{a,y}$ and $\pi_{b,y}$ be the projection onto the $a$-, $b$-, $y$-, $(a,b)$-, $(a,y)$- and $(b,y)$-coordinate in $M$, respectively. If no confusion arises, we drop the phrase \lq\lq in $M$" in their notations.

We identify nonlinear terms $F_1(x,y,\epsilon)$, $F_2(x,y,\epsilon)$ and $g(x,y,\epsilon)$ with $F_1(a,b,y,\epsilon)$, $F_2(a,b,y,\epsilon)$ and $g(a,b,y,\epsilon)$, respectively, via an affine transform $x(\in \mathbb{R}^n)\mapsto (a,b)\in \mathbb{R}^{u+s}$.

For a squared matrix $A$ with ${\rm Spec}(A) \subset \{\lambda \in \mathbb{C}\mid \re \lambda > 0\}$, $\lambda_A > 0$ denotes a positive number such that
\begin{equation}
\label{bound-unst-ev}
\lambda_A < \re \lambda,\quad \forall \lambda \in {\rm Spec}(A).
\end{equation}
Similarly, for a squared matrix $B$ with ${\rm Spec}(B) \subset \{\lambda \in \mathbb{C}\mid \re \lambda < 0\}$, $\mu_B < 0$ denotes a negative number such that
\begin{equation}
\label{bound-st-ev}
\mu_B > \re \mu,\quad \forall \mu \in {\rm Spec}(B).
\end{equation}

Finally, let $\dist(\cdot, \cdot)$ be the distance between $A$ and $B$, $A,B\subset \mathbb{R}^{n+l}$, given by $\dist(A,B) = \inf_{x\in A, y\in B}|x-y|$.
\end{nota}

The following assumptions are other keys of our verification theorem, so-called {\em cone conditions}.
\begin{ass}
\label{ass-cone-unstable}
Consider (\ref{abstract-form}). Let $N\subset \mathbb{R}^{n+l}$ be a fast-saddle-type block such that the coordinate representation $N_c$ is actually given by (\ref{fast-block}), and $z = (x,y,\epsilon)$. 

Define $\sigma_{\mathbb{A}_1^u} = \sigma_{\mathbb{A}_1^u}(z)$, $\sigma_{\mathbb{A}_2^u} = \sigma_{\mathbb{A}_2^u}(z)$, $\sigma_{\mathbb{B}_1^u} = \sigma_{\mathbb{B}_1^u}(z)$, $\sigma_{\mathbb{B}_2^u} = \sigma_{\mathbb{B}_2^u}(z)$, $\sigma_{g_1^u} = \sigma_{g_1^u}(z)$ and $\sigma_{g_2^u} = \sigma_{g_2^u}(z)$ be maximal singular values of the following matrices at $z$, respectively: 
\begin{align*}
\sigma_{\mathbb{A}_1^u}: \ & \mathbb{A}_1^u(z) = \left(  \frac{\partial F_1}{\partial a}(z) \right) \quad \text{: $u\times u$-matrix},\\
\sigma_{\mathbb{A}_2^u}: \ & \mathbb{A}_2^u(z) =  \left( \frac{\partial F_1}{\partial b}(z) \quad \frac{\partial F_1}{\partial y}(z) \quad \frac{\partial F_1}{\partial \eta}(z)\right)\quad \text{: $u\times (s+l+1)$-matrix},\\
\sigma_{\mathbb{B}_1^u}: \ & \mathbb{B}_1^u(z) = \left( \frac{\partial F_2}{\partial a}(z)\right) \quad \text{: $s\times u$-matrix},\\
\sigma_{\mathbb{B}_2^u}: \ & \mathbb{B}_2^u(z) =  \left( \frac{\partial F_2}{\partial b}(z) \quad \frac{\partial F_2}{\partial y}(z) \quad \frac{\partial F_2}{\partial \eta}(z)\right)\quad \text{: $s\times (s+l+1)$-matrix},\\
\sigma_{g_1^u}: \ & g_1^u(z) = \left(  \frac{\partial g}{\partial a}(z)\right) \quad \text{: $l\times u$-matrix},\\
\sigma_{g_2^u}: \ & g_2^u(z) = \left(\frac{\partial g}{\partial b}(z) \quad \frac{\partial g}{\partial y}(z) \quad \frac{\partial g}{\partial \eta}(z)\right)\quad \text{: $l\times (s+l+1)$-matrix}.
\end{align*}
Assume that the following inequalities hold:
\begin{align}
\label{ineq-graph-unstable}
&\lambda_A - \sup \sigma_{\mathbb{A}_1^u} + \sup \sigma_{\mathbb{A}_2^u} > 0,\\
\label{ineq-cone-unstable}
&\lambda_A + |\mu_B| -  \left\{ \sup \sigma_{\mathbb{A}_1^u} + \sup \sigma_{\mathbb{A}_2^u} + \sup \sigma_{\mathbb{B}_1^u} + \sup \sigma_{\mathbb{B}_2^u} + \sigma \left( \sup \sigma_{g_1^u} + \sup \sigma_{g_2^u} \right) \right\} >0,
\end{align}
where $\lambda_A$ and $\mu_B$ are real numbers satisfying (\ref{bound-unst-ev}) and (\ref{bound-st-ev}), respectively,
and the notation \lq\lq $\ \sup$" means the supremum on $N\times [0,\epsilon_0]$.
\end{ass}

\begin{ass}
\label{ass-cone-stable}
Consider (\ref{abstract-form}). Let $N\subset \mathbb{R}^{n+l}$ be a fast-saddle-type block such that the coordinate representation $N_c$ is actually given by (\ref{fast-block}), and $z = (x,y,\epsilon)$. 

Define $\sigma_{\mathbb{A}_1^s} = \sigma_{\mathbb{A}_1^s}(z)$, $\sigma_{\mathbb{A}_2^s} = \sigma_{\mathbb{A}_2^s}(z)$, $\sigma_{\mathbb{B}_1^s} = \sigma_{\mathbb{B}_1^s}(z)$, $\sigma_{\mathbb{B}_2^s} = \sigma_{\mathbb{B}_2^s}(z)$, $\sigma_{g_1^s} = \sigma_{g_1^s}(z)$ and $\sigma_{g_2^s} = \sigma_{g_2^s}(z)$ be maximal singular values of the following matrices at $z$, respectively: 
\begin{align*}
\sigma_{\mathbb{A}_1^s}: \ & \mathbb{A}_1^s(z) = \left(  \frac{\partial F_1}{\partial b}(z) \right) \quad \text{: $u\times s$-matrix},\\
\sigma_{\mathbb{A}_2^s}: \ & \mathbb{A}_2^s(z) =  \left( \frac{\partial F_1}{\partial a}(z) \quad \frac{\partial F_1}{\partial y}(z) \quad \frac{\partial F_1}{\partial \eta}(z)\right)\quad \text{: $u\times (u+l+1)$-matrix},\\
\sigma_{\mathbb{B}_1^s}: \ & \mathbb{B}_1^s(z) = \left( \frac{\partial F_2}{\partial b}(z)\right) \quad \text{: $s\times s$-matrix},\\
\sigma_{\mathbb{B}_2^s}: \ & \mathbb{B}_2^s(z) =  \left( \frac{\partial F_2}{\partial a}(z) \quad \frac{\partial F_2}{\partial y}(z) \quad \frac{\partial F_2}{\partial \eta}(z)\right)\quad \text{: $s\times (u+l+1)$-matrix},\\
\sigma_{g_1^s}: \ & g_1^s(z) = \left(  \frac{\partial g}{\partial b}(z)\right) \quad \text{: $l\times s$-matrix},\\
\sigma_{g_2^s}: \ & g_2^s(z) = \left(\frac{\partial g}{\partial a}(z) \quad \frac{\partial g}{\partial y}(z) \quad \frac{\partial g}{\partial \eta}(z)\right)\quad \text{: $l\times (u+l+1)$-matrix}.
\end{align*}
Assume that the following inequalities hold:
\begin{align}
\label{ineq-graph-stable}
&|\mu_B| - \sup \sigma_{\mathbb{B}_1^s} + \sup \sigma_{\mathbb{B}_2^s} > 0,\\
\label{ineq-cone-stable}
&\lambda_A + |\mu_B| -  \left\{ \sup \sigma_{\mathbb{A}_1^s} + \sup \sigma_{\mathbb{A}_2^s} + \sup \sigma_{\mathbb{B}_1^s} + \sup \sigma_{\mathbb{B}_2^s} + \sigma \left(\sup \sigma_{g_1^s} + \sup \sigma_{g_2^s} \right) \right\} >0,
\end{align}
where $\lambda_A$ and $\mu_B$ are real numbers satisfying (\ref{bound-unst-ev}) and (\ref{bound-st-ev}), respectively,
and the notation \lq\lq $\ \sup$" means the supremum on $N\times [0,\epsilon_0]$.
\end{ass}

\begin{dfn}[Cone conditions]\rm
\label{dfn-cone}
We say a fast-saddle-type block $N$ satisfies {\em the unstable cone condition} for (\ref{fast-slow})$_\epsilon$ if Assumption \ref{ass-cone-unstable} holds. 
Similarly, we say a fast-saddle-type block $N$ satisfies {\em the stable cone condition} for (\ref{fast-slow})$_\epsilon$ if Assumption \ref{ass-cone-stable} holds.
They make sense only when $\epsilon \in [0,\epsilon_0]$. 
\end{dfn}

By following the proof of Theorem 4 in \cite{Jones}, we obtain the following theorem, which is the main result in this section.
\begin{thm}
\label{thm-inv-mfd-rigorous}
Consider (\ref{abstract-form}). Let $N\subset \mathbb{R}^{n+l}$ be a fast-saddle-type block such that the coordinate representation $N_c$ is actually given by (\ref{fast-block}) with $\pi_y N = K\subset \mathbb{R}^l$. Then
\begin{enumerate}
\item if we assume  
\begin{equation}
\label{Lyapunov-condition}
\lambda_A - \left[ \sup \sigma_{\mathbb{A}_1^u} + \frac{ \sup \sigma_{\mathbb{A}_2^u} +  \sup \sigma_{\mathbb{B}_2^s} }{2} \right] > 0,\quad 
|\mu_B| -\left[ \sup \sigma_{\mathbb{B}_1^s} + \frac{ \sup \sigma_{\mathbb{A}_2^u} +  \sup \sigma_{\mathbb{B}_2^s} }{2} \right] > 0,
\end{equation}
there is an $l$-dimensional normally hyperbolic invariant manifold $S_0 \cap N$ in $\mathbb{R}^{n+l}$ for the limit system (\ref{layer}) such that $S_0 \cap N$ is the graph of a smooth function depending on $y$.
\item if the stable cone condition is satisfied, there exists a Lipschitz function $a = h_s(b,y,\epsilon)$ defined in $\overline{B_s}\times \overline{B_l} \times [0,\epsilon_0]$ such that the graph
\begin{equation*}
W^s(S_\epsilon):= \{(a,b,y,\epsilon)\mid a = h_s(b,y,\epsilon)\}
\end{equation*}
is locally invariant with respect to (\ref{abstract-form}). 
Moreover, $h_s(b,y,\epsilon)$ is $C^r$ for any $r < \infty$ on $\overline{B_s}\times \overline{B_l} \times [0,\epsilon_0]$.
\item if the unstable cone condition is satisfied, there exists a Lipschitz function $b = h_u(a,y,\epsilon)$ defined in $\overline{B_u}\times \overline{B_l} \times [0,\epsilon_0]$ such that the graph
\begin{equation*}
W^u(S_\epsilon):= \{(a,b,y,\epsilon)\mid b = h_u(a,y,\epsilon)\}
\end{equation*}
is locally invariant with respect to (\ref{abstract-form}). Moreover, $h_u(a,y,\epsilon)$ is $C^r$ for any $r < \infty$ on $\overline{B_u}\times \overline{B_l} \times [0,\epsilon_0]$.
\end{enumerate}
As a consequence, under stable and unstable cone conditions, $S_\epsilon:= W^s(S_\epsilon)\cap W^u(S_\epsilon)$ can be defined by the locally invariant $l$-dimensional submanifold of $\mathbb{R}^{n+l}$ inside $N$ for any $\epsilon \in [0,\epsilon_0]$. 
\end{thm}

\begin{proof}
Our proof is based on the slight modification of the proof of Theorem 4 in \cite{Jones} for our current setting. 
The proof consists of three parts: 
(i) normal hyperbolicity and graph representation of $S_0$, 
(ii) existence of the graph representation of $W^s(S_\epsilon)$ and 
(iii) coincidence of the graph of such a derived function with all of $W^s(S_\epsilon)$. 
We shall prove them by tracing discussions in \cite{Jones}. 
Readers who are not familiar with the invariant manifold theorem in singular perturbation problems very well can compare our discussions with Chapter 2 in \cite{Jones}. 
Here we shall derive discussions only for the stable manifold $W^s(S_\epsilon)$, while discussions for the unstable manifold can be also derived replacing matrices in Assumption \ref{ass-cone-stable} by those in Assumption \ref{ass-cone-unstable} and time reversal. 
Remark that an assumption for $N_c$ implies that all discussions are done in terms of (\ref{fast-block}). Without the loss of generality, we may assume $N=N_c$.

\bigskip
\begin{description}
\item[(i): Normal hyperbolicity and the graph representation of $S_0$.] 
\end{description}
First of all, consider the special case $\epsilon = 0$.
In this case, the slow variable $y$ is just a parameter. 
Inequalities in (\ref{Lyapunov-condition}) yield the existence of local Lyapunov functions, which is a consequence of arguments in \cite{Mat} with a little modifications as follows. 
Fix $y\in K$. Define a function $L_y:\mathbb{R}^n\to \mathbb{R}$ by
\begin{equation*}
L_y(a(t),b(t)) := - |a(t) - a_y|^2 + |b(t) - b_y|^2,
\end{equation*}
where $(a_y, b_y, y)$ is an equilibrium for (\ref{abstract-form}) with $\epsilon = 0$ in $N$ at $y$. The existence of $(a_y, b_y, y)$ follows from Proposition \ref{prop-existence-fixpt}.

Then
\begin{align*}
\frac{dL_y}{dt}&(a(t),b(t))\mid_{t=0}= -\frac{d}{dt}(a - a_y)^T \cdot (a - a_y) -(a- a_y)^T \cdot \frac{d}{dt}(a - a_y)\\
	&\quad + \frac{d}{dt}(b - b_y)^T \cdot (b - b_y) + (b - b_y)^T \cdot \frac{d}{dt}(b - b_y) \\
&= - (Aa + F_1(a,b,y,0))^T \cdot (a - a_y) -(a - a_y)^T \cdot (Aa + F_1(a,b,y,0))\\
	&\quad + (Bb + F_1(a,b,y,0))^T \cdot (b - b_y) + (b - b_y)^T \cdot (Bb + F_1(a,b,y,0)). 
\end{align*} 
Let $0_k$ be the $k$-dimensional $k$-vector and $O_{k_1,k_2}$ be the $(k_1,k_2)$-zero matrix. 
Now we obtain
\begin{align*}
- &(Aa + F_1(a,b,y,0))^T \cdot (a - a_y) \\
	&= -\begin{pmatrix}
a-a_y \\ b-b_y \\ 0_l \\ 0
\end{pmatrix}^T \begin{pmatrix}
A + \frac{\partial F_1}{\partial a}(\tilde z_1) & \frac{\partial F_1}{\partial b}(\tilde z_1) & \frac{\partial F_1}{\partial y}(\tilde z_1) & \frac{\partial F_1}{\partial \eta}(\tilde z_1)\\
O_{u,s} & O_{s,s} & O_{l,s} & O_{1,s}\\
O_{u,l} & O_{s,l} & O_{l,l} & O_{1,l}\\
\end{pmatrix}^T
\begin{pmatrix}
a-a_y \\ b-b_y \\ 0_l 
\end{pmatrix}\\
	&\leq -\lambda_A |a-a_y|^2 + \sigma_{\mathbb{A}_1^u}(\tilde z_1) |a-a_y|^2 +  \sigma_{\mathbb{A}_2^u}(\tilde z_1) |a-a_y| |b-b_y|\\
	&\leq -\lambda_A |a-a_y|^2 + \sigma_{\mathbb{A}_1^u}(\tilde z_1) |a-a_y|^2 +  \frac{\sigma_{\mathbb{A}_2^u}(\tilde z_1) }{2}(|a-a_y|^2 + |b-b_y|^2)
\end{align*}
via the Mean Value Theorem, where $\tilde z_1=(\tilde a_1, \tilde b_1, y,0)\in N\times [0,\epsilon_0]$. The same inequalities hold for $(a - a_y)^T \cdot (Aa + F_1(a,b,y,0))$. Similarly, we obtain
\begin{align*}
&(Bb + F_2(a,b,y,0))^T \cdot (b - b_y) \\
	&= \begin{pmatrix}
a-a_y \\ b-b_y \\ 0_l \\ 0
\end{pmatrix}^T \begin{pmatrix}
O_{u,u} & O_{s,u} & O_{l,u} & O_{1,u}\\
 \frac{\partial F_2}{\partial a}(\tilde z_2) & B+ \frac{\partial F_2}{\partial b}(\tilde z_2) & \frac{\partial F_2}{\partial y}(\tilde z_2) & \frac{\partial F_2}{\partial \eta}(\tilde z_2)\\
O_{u,l} & O_{s,l} & O_{l,l} & O_{1,l}\\
\end{pmatrix}^T
\begin{pmatrix}
a-a_y \\ b-b_y \\ 0_l 
\end{pmatrix}\\
	&\leq \mu_B |b-b_y|^2 + \sigma_{\mathbb{B}_1^s}(\tilde z_2) |b-b_y|^2 +  \sigma_{\mathbb{B}_2^s}(\tilde z_2) |a-a_y| |b-b_y|\\
	&\leq \mu_B |b-b_y|^2 + \sigma_{\mathbb{B}_1^s}(\tilde z_2) |b-b_y|^2 +  \frac{\sigma_{\mathbb{B}_2^s}(\tilde z_2) }{2}(|a-a_y|^2 + |b-b_y|^2)
\end{align*}
via the Mean Value Theorem, where $\tilde z_2=(\tilde a_2, \tilde b_2, y,0)\in N\times [0,\epsilon_0]$. The same inequalities hold for $(b - b_y)^T \cdot (Bb + F_2(a,b,y,0))$. Summarizing these inequalities, we obtain
\begin{align}
\notag
&\frac{1}{2}\frac{dL_y}{dt}(a(t),b(t))\mid_{t=0} \\
\notag
	&\leq \left(-\lambda_A + \sigma_{\mathbb{A}_1^u}(\tilde z_1) +  \frac{\sigma_{\mathbb{A}_2^u}(\tilde z_1) + \sigma_{\mathbb{B}_2^s}(\tilde z_2)}{2}\right)|a-a_y|^2 + \left( \mu_B + \sigma_{\mathbb{B}_1^s}(\tilde z_2) + \frac{\sigma_{\mathbb{A}_2^u}(\tilde z_1) + \sigma_{\mathbb{B}_2^s}(\tilde z_2)}{2}\right)|b-b_y|^2\\
\notag
	&\leq \left(-\lambda_A + \sup_{N\times [0,\epsilon_0]}\sigma_{\mathbb{A}_1^u} +  \frac{ \sup_{N\times [0,\epsilon_0]} \sigma_{\mathbb{A}_2^u} + \sup_{N\times [0,\epsilon_0]} \sigma_{\mathbb{B}_2^s}}{2}\right)|a-a_y|^2\\
\notag
	&\quad  + \left( \mu_B + \sup_{N\times [0,\epsilon_0]} \sigma_{\mathbb{B}_1^s} + \frac{ \sup_{N\times [0,\epsilon_0]} \sigma_{\mathbb{A}_2^u} + \sup_{N\times [0,\epsilon_0]} \sigma_{\mathbb{B}_2^s}}{2}\right)|b-b_y|^2.
\end{align} 
The right-hand side is strictly negative unless $(a,b,y) = (a_y, b_y, y)$, which follows from (\ref{ineq-graph-unstable}) and (\ref{ineq-graph-stable}). Obviously, $dL_y/dt = 0$ if and only if $(a,b,y) = (a_y, b_y, y)$. Therefore, $L_y$ is a Lyapunov function. This observation leads to the following facts:
\begin{itemize}
\item $(a,b,y) = (a_y, b_y, y)$ is an equilibrium for (\ref{layer}) which is unique in $N$ with fixed $y$.
\item The linearized matrix
\begin{equation*}
\begin{pmatrix}
-A - \frac{\partial F_1}{\partial a}(\tilde z_1) & -\frac{\partial F_1}{\partial b}(\tilde z_1)\\
 \frac{\partial F_2}{\partial a}(\tilde z_2) & B+ \frac{\partial F_2}{\partial b}(\tilde z_2)
\end{pmatrix}
\end{equation*}
is strictly negative definite for all $\tilde z_1, \tilde z_2\in N$. This implies that $(a,b,y) = (a_y, b_y, y)$ is a hyperbolic fixed point for the layer problem (\ref{layer}).
\end{itemize}
These observations hold for arbitrary $y\in K$. Thanks to the Implicit Function Theorem, we can construct the graph of a smooth $y$-parameter family of such hyperbolic fixed points, which is $S_0$. In particular, $S_0$ has the structure of normally hyperbolic invariant manifold with graph representation.

\begin{description}
\item[(ii): The graph representation of $W^s(S_\epsilon)$.] 
\end{description}
We perform the modification to deal with slow directions. 
We choose a set $\hat K$ whose interior contains $K$ so that its boundary is given by the condition $\hat \nu(y) = 0$ for some $C^\infty$ function $\hat \nu(y)$ and that $\hat \nu(y)$ satisfies $\nabla \hat \nu(y)\not = 0$ for all $y\in \partial \hat K$. 
The function $\hat \nu(y)$ is assumed to be normalized 
so that $\nabla \hat \nu (y) = n_y$ is a unit {\em inward}
% footnote
\footnote{In the lecture note by Jones \cite{Jones}, $n_y$ denotes {\em outward} normal vector. We should remark that it is wrong. 
Indeed, we determine the immediate exit and entrance of a block later. 
Our claim here is that the exit is $\{(a,\zeta \mid a\in \partial (\pi_a N))\}$ and does not contain $\partial \hat K$. 
To this end, $n_y$ should be the inward normal vector.
If we prove the existence of $W^u(S_\epsilon)$, $n_y$ is chosen as the outward normal vector.
} 
% footnote : end
normal to $\partial \hat K$. 
Let $\rho(y)$ be a $C^\infty$ function that has the following values
\begin{equation*}
\rho(y) = \begin{cases}
	1 & \text{if $y\in \hat K^c$,}\\
	0 & \text{if $y\in K$}
	\end{cases}.
\end{equation*}
We now modify our system (\ref{abstract-form}) by adding the term $\delta \rho(y)n_y$, where $\delta$ is a positive number which remains to be chosen.

We add an equation for the small parameter $\epsilon$. 
Following Jones \cite{Jones}, we use a multiple of $\epsilon$ as the new auxiliary variable,  $\epsilon = \eta \sigma$, and append the equation $\eta'=0$ to (\ref{abstract-form}), as noted in the beginning of this section. 
We then obtain the system
% footnote
\footnote{This modification guarantees that the vector field has the {\em inflowing property} with respect to $(\hat N, \partial \hat K)$ (e.g. \cite{CLY, F}). 
Namely, at any point $p\in \partial \hat K$, the vector field at $p$ goes inside $\hat N$.
This property plays an important role to prove the existence of center-stable manifolds via the graph representation.
In the case of $W^u(S_\epsilon)$ or generally center-unstable manifolds, we modify the vector field so that the vector field has the {\em overflowing property} with respect to $(\hat N, \partial \hat K)$.
Namely, at any point $p\in \partial \hat K$, the vector field at $p$ goes outside $\hat N$ or tangent to $\partial \hat N$.
In our case, it is sufficient to choose the unit vector $n_y$ as the unit {\em outward} normal vector to $\partial \hat K$ and to choose $\nu$ in the similar manner to $W^s(S_\epsilon)$.
} 
% footnote : end
\begin{equation}
\label{abstract-form-modify}
\begin{cases}
a' = Aa + F_1(x,y,\epsilon), & \\
b' = Bb + F_2(x,y,\epsilon), & \\
y' = \eta \sigma g(x,y,\epsilon) + \delta \rho(y) n_y, & \\
\eta' = 0. & 
\end{cases}
\end{equation}
As in (\ref{abstract-form}) it is understood that $x$ is a function of $a$ and $b$, which is already realized in Section \ref{section-isolatingblock}, and $\epsilon$ is a function of $\eta$. 
If Theorem \ref{thm-inv-mfd-rigorous} is restated with $\epsilon$ replaced by $\eta$ and proved in that formulation, 
its original version can easily be recaptured by substituting $\epsilon$ back in.

We define the new family of sets $\hat N$ by
\begin{equation*}
\hat N:= \pi_{a,b}N\times \hat K \times [0,1].
\end{equation*}
Define the set 
\begin{equation*}
\Gamma_s:= \{(a,b,y,\eta)\mid \tilde \varphi_\epsilon (t, (a,b,y,\eta)) \in \hat N \text{ for all }t\geq 0\},
\end{equation*}
where $\tilde \varphi_\epsilon$ is the flow generated by (\ref{abstract-form-modify}).
In this part we shall prove that $\Gamma_s$ is the graph of a function of $(b,y,\eta)$, say, $a=h_s(b,y,\eta)$. Set $\zeta = (b,y,\eta)$ and $\hat N_{\hat \zeta}$ denote the crossing section of $\hat N$ at $\zeta = \hat \zeta$: 
\begin{equation*}
\hat N_{\hat \zeta}:= \left\{(a,\hat \zeta)\in \hat N \right\}.
\end{equation*}
We shall show that there is at least one point $(a,\hat \zeta) \in \hat N$ for which $\tilde \varphi_\epsilon (t, (a,\hat \zeta)) \in \hat N$ for all $t\geq 0$. To achieve this, we use the Wazewski Principle. 
If the exit $\hat N^\exit$ is closed in $\partial \hat N$ and $\hat N$ is an isolating block for (\ref{abstract-form-modify}), then the mapping $W:\hat N \to \hat N^\exit$ given by
\begin{equation*}
W(z) = \tilde \varphi_\epsilon(\tau^-(z),z),\quad z=(a,b,y,\eta),\ \tau^-(z) = \sup\{t\geq 0\mid \tilde \varphi_\epsilon([0,t],z)\cap N^\exit = \emptyset\},
\end{equation*}
is continuous. 
It is the general consequence of isolating blocks (see \cite{Smo}, Chapter 22. In \cite{Jones}, such $\hat N$ is called Wazewski set). 
Since $N$ is of saddle-type, the flow is repelling on the fast-exit $\hat N^{f,-}$ for all $\eta\in [0,1]$ and is attracting on the fast-entrance $\hat N^{f,+}$ for all $\eta\in [0,1]$. 
It is actually satisfied by choosing $\hat K\supset K$ sufficiently small, if necessary.
The rest are the $y$-direction and the $\eta$-direction. 
Note that an appropriate choice of $\hat K$ enables us to construct local Lyapunov functions $L_y$ in {\bf (i)} for all $y\in \hat K$.

For points on $\hat N$ with $y\in \partial \hat K$, we know
\begin{equation*}
\langle y', n_y\rangle = \epsilon \langle g(x,y,\epsilon), n_y\rangle + \delta \langle n_y, n_y\rangle
\end{equation*}
since $\rho = 1$ holds on $\partial \hat K$, where $\langle \cdot, \cdot \rangle$ denotes an inner product on appropriate vector spaces. Setting $M_{\hat N}:= \sup_{\hat N}\{|g|, |Dg|\}$, the above can be estimated by
\begin{equation*}
\langle y', n_y\rangle \geq \delta - \epsilon_0 M_{\hat N} > 0,
\end{equation*}
if $\delta > \epsilon_0 M_{\hat N}$. Since we can choose $\delta$ arbitrarily, then it can be achieved. 
Therefore, $\partial \hat K\cap \hat N$ is a part of entrance. 
Finally, in the case $\eta = 0$ or $1$, both of these sets are invariant since $\eta'=0$ holds everywhere and thus render neither the entrance nor the exit. 

\bigskip
The exit $\hat N^\exit$ is then seen to be $\{(a,\zeta)\mid a\in \partial(\pi_a N)\}$. 
By our construction of $N$, for any $\hat \zeta$, the set $\hat N_{\hat \zeta}$ is a ball of dimension $u$. 
Suppose that $\Gamma_s \cap \hat N_{\hat \zeta}=\emptyset$. 
All points in $\hat N_{\hat \zeta}$ then go to $\hat N^\exit$ in finite time and hence, restricting Wazewski map to the crossing section, we obtain a continuous mapping
\begin{equation*}
W: \hat N_{\hat \zeta} \to \hat N^\exit.
\end{equation*}
If we follow this by the projection $\pi_a(a,\zeta) = a$, we see that $\pi_a\circ W$  maps a $u$-dimensional ball onto its boundary, while keeping the boundary fixed. 
This contradicts the well-known No-Retract Theorem. There is thus a point in $\Gamma_s \cap \hat N_{\hat \zeta}$. Since $\hat \zeta$ is arbitrary, this gives at least one value for $a$ as a function of $(b,y,\eta)$.
We shall put it $h_s(b,y,\eta)$.

\bigskip
\begin{description}
\item[(iii): Uniqueness of $W^s(S_\epsilon)$ as a graph of $a=h_s(b,y,\eta)$.] 
\end{description}
In the sequel we show that the graph of the above derived function is all of $\Gamma_s$. At the same time, it will be shown that the function is, in fact, Lipschitz continuous with Lipschitz constant $1$. A comparison between the growth rates in different directions will be derived in the following proposition in our setting. Let $(a_i(t), \zeta_i(t))$, $i=1,2$, be two solutions of (\ref{abstract-form-modify}), set $\Delta a:= a_1(t) - a_2(t)$ and $\Delta \zeta:= \zeta_2(t) - \zeta_1(t)$. Further we define
\begin{equation*}
M(t):= |\Delta a(t)|^2 - |\Delta \zeta(t)|^2.
\end{equation*}

\begin{prop}[cf. Lemma 2 in \cite{Jones}]
\label{prop-cone}
Under assumptions of Theorem \ref{thm-inv-mfd-rigorous}, if $M(t) = 0$ then $M'(t) > 0$ holds as long as the two solutions stay in $\hat N$, unless $\Delta a = 0$.
\end{prop}

\begin{proof}[Proof of Proposition \ref{prop-cone}]
Discussions in our proof are based on the proof of Lemma 2 in \cite{Jones} with arrangements in our setting. We will make on each of the quantities $\langle \Delta a, \Delta a\rangle$ etc. The equation for $\Delta a$ is 
\begin{equation*}
\Delta a' = A(a_2-a_1) + F_1(a_2, b_2, y_2, \eta_2 \sigma) - F_1(a_1, b_1, y_1, \eta_1 \sigma),
\end{equation*}
which we rewrite as
\begin{equation}
\label{differential-difference}
\Delta a' = A\Delta a + \Delta F_1,
\end{equation}
where $\Delta F_1 = F_1(a_2, b_2, y_2, \eta_2 \sigma) - F_1(a_1, b_1, y_1, \eta_1 \sigma)$.
Since $\hat N$ is the product of $h$-sets associated with an affine transformation, thanks to the Mean Value Theorem, 
one can derive the following equality:  
\begin{equation*}
\Delta F_1 = \frac{\partial F_1}{\partial a}(\tilde a_1, \tilde b_1, \tilde y_1, \tilde \eta_1)\Delta a  + \frac{\partial F_1}{\partial b}(\tilde a_1, \tilde b_1, \tilde y_1, \tilde \eta_1)\Delta b + \frac{\partial F_1}{\partial y}(\tilde a_1, \tilde b_1, \tilde y_1, \tilde \eta_1)\Delta y + \frac{\partial F_1}{\partial \eta}(\tilde a_1, \tilde b_1, \tilde y_1, \tilde \eta_1)\Delta \eta
\end{equation*}
for some points $(\tilde a_1, \tilde b_1, \tilde y_1, \tilde \eta_1) =: \tilde z_1\in \hat N$. 
Such an estimate makes sense since $\hat N$ is convex.

We can estimate $\langle \Delta a, \Delta a\rangle' = 2\langle \Delta a', \Delta a\rangle$ by taking the inner product of (\ref{differential-difference}) with $\Delta a$. 
As a result, we obtain
\begin{equation}
\label{eq-a}
\langle \Delta a', \Delta a\rangle = (\Delta a)^T \cdot \left\{ A\Delta a + \frac{\partial F_1}{\partial a}(\tilde z_1)\Delta a  + \frac{\partial F_1}{\partial b}(\tilde z_1)\Delta b + \frac{\partial F_1}{\partial y}(\tilde z_1)\Delta y + \frac{\partial F_1}{\partial \eta}(\tilde z_1)\Delta \eta \right\}.
\end{equation}
We also obtain the following equalities in the same manner: 
\begin{equation}
\label{eq-b}
\langle \Delta b', \Delta b\rangle = (\Delta b)^T \cdot \left\{ B\Delta b + \frac{\partial F_2}{\partial a}(\tilde z_2)\Delta a  + \frac{\partial F_2}{\partial b}(\tilde z_2)\Delta b + \frac{\partial F_2}{\partial y}(\tilde z_2)\Delta y + \frac{\partial F_2}{\partial \eta}(\tilde z_2)\Delta \eta \right\}
\end{equation}
as well as
\begin{equation}
\label{eq-y}
\langle \Delta y', \Delta y\rangle = \sigma (\Delta y)^T \cdot \left\{ \frac{\partial g}{\partial a}(\tilde z_3)\Delta a  + \frac{\partial g}{\partial b}(\tilde z_3)\Delta b + \frac{\partial g}{\partial y}(\tilde z_3)\Delta y + \frac{\partial g}{\partial \eta}(\tilde z_3)\Delta \eta \right\}
\end{equation}
for some $\tilde z_i\in \hat N$, $i=2,3$. Obviously $\langle \Delta \eta', \Delta \eta\rangle = 0$ since $\eta' = 0$.

Recall that our final objective is the estimate of $M'(t) = 2\left\{\langle \Delta a', \Delta a\rangle - (\langle \Delta b', \Delta b\rangle + \langle \Delta y', \Delta y\rangle)\right\}$. 
To this end, we estimate (\ref{eq-a}), (\ref{eq-b}) and (\ref{eq-y}) as the quadratic form associated with non-square matrices. 
For example, the equality (\ref{eq-a}) can be rewritten by the following matrix form: 
\begin{equation*}
\langle \Delta a', \Delta a\rangle = (\Delta a)^T  A \Delta a  + (\Delta a)^T \mathbb{A}^u \begin{pmatrix}\Delta a \\ \Delta b \\ \Delta y \\ \Delta \eta\end{pmatrix} := (\Delta a)^T  A \Delta a  + (\Delta a)^T \mathbb{A}_1^u \Delta a +  (\Delta a)^T \mathbb{A}_2^u \Delta \zeta,
\end{equation*}
where $\mathbb{A}_1^u$ is the $u\times u$-matrix given by $(\partial F_1/\partial a)(\tilde z_1)$ and $\mathbb{A}_2^u$ is the $u \times (s + l + 1)$-matrix given by 
\begin{equation*}
\mathbb{A}_2^u = \mathbb{A}_2^u(\tilde z_1):=\left(\frac{\partial F_1}{\partial b}(\tilde z_1) \quad \frac{\partial F_1}{\partial y}(\tilde z_1) \quad \frac{\partial F_1}{\partial \eta}(\tilde z_1)\right).
\end{equation*}
Denoting $\|\cdot \|_2$ the (matrix) $2$-norm,  general theory of linear algebra yields
\begin{align*}
 (\Delta a)^T \mathbb{A}^u \begin{pmatrix}\Delta a \\ \Delta b \\ \Delta y \\ \Delta \eta\end{pmatrix} &\leq  |\Delta a | (\|\mathbb{A}_1^u\|_2 |\Delta a| + \|\mathbb{A}_2^u\|_2 |\Delta \zeta |)\\
 &= \sigma_{\mathbb{A}_1^u}(\tilde z_1)|\Delta a|^2 + \sigma_{\mathbb{A}_2^u}(\tilde z_1)|\Delta a | |\Delta \zeta |,
\end{align*}
where $\sigma_ {\mathbb{A}_i^u}= \sigma_{\mathbb{A}_i^u}(\tilde z_1)$, $i=1,2$, is the maximal singular value of $\mathbb{A}_i^u$ at $\tilde z_1 \in \hat N$ stated in Assumption \ref{ass-cone-stable}. 
We obtain the following estimate of (\ref{eq-a}):
\begin{equation}
\label{a-estimate}
\langle \Delta a', \Delta a\rangle \geq (\lambda_A - \sup_{\hat N}\sigma_{\mathbb{A}_1^u}) |\Delta a|^2 - \sup_{\hat N}\sigma_{\mathbb{A}_2^u} |\Delta a| |\Delta \zeta |,
\end{equation}
where $\lambda_A$ is a positive number satisfying (\ref{bound-unst-ev}).
In the similar manner we also obtain the estimate of $\langle \Delta b', \Delta b\rangle$ and $\langle \Delta y', \Delta y\rangle$ by
\begin{align*}
\langle \Delta b', \Delta b\rangle &\leq \mu_B  |\Delta b|^2 +  |\Delta b|(\sup_{\hat N}\sigma_{\mathbb{B}_1^u} |\Delta a| + \sup_{\hat N}\sigma_{\mathbb{B}_2^u} |\Delta \zeta |),\\
\langle \Delta y', \Delta y\rangle &\leq \sigma |\Delta y|(\sup_{\hat N}\sigma_{g_1^u} |\Delta a| + \sup_{\hat N}\sigma_{g_2^u} |\Delta \zeta |),
\end{align*}
where $\mu_B$ denotes a negative number satisfying (\ref{bound-st-ev}).
Functions $\sigma_{\mathbb{B}_1^u} = \sigma_{\mathbb{B}_1^u}(\tilde z)$, $\sigma_{\mathbb{B}_2^u} = \sigma_{\mathbb{B}_2^u}(\tilde z)$, $\sigma_{g_1^u} = \sigma_{g_1^u}(\tilde z)$ and $\sigma_{g_2^u} = \sigma_{g_2^u}(\tilde z)$ are maximal singular values of the matrix valued functions defined in Assumption \ref{ass-cone-unstable} at $\tilde z$, respectively.

In particular, at a point satisfying $M(t) = 0$, we further obtain 
\begin{align*}
\langle \Delta a', \Delta a\rangle &\geq (\lambda_A - \sup_{\hat N}\sigma_{\mathbb{A}_1^u} - \sup_{\hat N}\sigma_{\mathbb{A}_2^u}) |\Delta a|^2,\\
\langle \Delta b', \Delta b\rangle &\leq (\mu_B  + \sup_{\hat N}\sigma_{\mathbb{B}_1^u} + \sup_{\hat N}\sigma_{\mathbb{B}_2^u})|\Delta a|^2,\\
\langle \Delta y', \Delta y\rangle &\leq \sigma |\Delta y|(\sup_{\hat N}\sigma_{g_1^u} |\Delta a| + \sup_{\hat N}\sigma_{g_2^u} |\Delta \zeta |)
\end{align*}
to show
\begin{equation*}
M'(t) \geq \left[ \lambda_A + |\mu_B| - \left\{ \sup_{\hat N}\sigma_{\mathbb{A}_1^u} + \sup_{\hat N}\sigma_{\mathbb{A}_2^u} + \sup_{\hat N}\sigma_{\mathbb{B}_1^u} + \sup_{\hat N}\sigma_{\mathbb{B}_2^u} + \sigma \left( \sup_{\hat N}\sigma_{g_1^u} + \sup_{\hat N}\sigma_{g_2^u} \right) \right\} \right] |\Delta a|^2,
\end{equation*}
which proves our statement.
Note that the coefficient of $|\Delta a|^2$ can be positive if $\hat K \supset K$ is sufficiently small, thanks to cone conditions.
\end{proof}

We go back to the proof in {\bf (iii)}. 
We have already shown that the set $\Gamma_s$ contains the graph of a function denoted by 
$a = h_s(b,y,\eta)$. 
Suppose that $\Gamma_s$ contains more than one point with the same value of $b,y,\eta$. 
There would then be $a_1$ and $a_2$ so that both $(a_1,b,y,\eta)$ and $(a_2,b,y,\eta)$ lie in $\Gamma_s$. At $t=0$, we have $|\Delta a| \geq |\Delta \zeta|$. 
By Proposition \ref{prop-cone}, 
\begin{equation*}
|\Delta a(t)| \geq |\Delta \zeta(t)|
\end{equation*}
holds for all $t\geq 0$. In the estimate (\ref{a-estimate}) we can then replace $|\Delta \zeta|$ by $|\Delta a|$ to obtain
\begin{equation*}
\{|\Delta a|^2\}' \geq (\lambda_A - \beta)|\Delta a|^2.
\end{equation*}
A positive number $\beta$ can be chosen as $\sup_{\hat N}\sigma_{\mathbb{A}_1^u} + \sup_{\hat N}\sigma_{\mathbb{A}_2^u}$. 
See (\ref{ineq-graph-unstable}).  
From which it can be easily concluded that $\Delta a$ grows exponentially, which contradicts the hypothesis that both points stay in $\hat N$ for all $t\geq 0$. 

The same argument can be used to show that $h_s$ is also Lipschitz. 
If $(a_1, \zeta_1)$ and $(a_2, \zeta_2)$ are both in $\Gamma_s$ and $|a_2-a_1| \geq |\zeta_2 - \zeta_1|$, then $|a_2-a_1|$ grows exponentially, which contradicts the hypothesis again that both points lie in $\Gamma_s$. 
We then have shown that the set $\Gamma_s$ is the graph of a Lipschitz function. 
In particular, the Lipschitz function $a=h_s(b,y,\epsilon)$ is well-defined.
This manifold is $W^s(S_\epsilon)$ when $y\in \hat K$ is restricted to $K$, in which the modified equation (\ref{abstract-form-modify}) agrees with the original one (\ref{abstract-form}).
Smoothness result follows from the same arguments as \cite{Jones}.
\end{proof}

A direct consequence of graph representations of $W^s(S_\epsilon)$ as well as $W^u(S_\epsilon)$ is that they are homotopic to flat hyperplanes inside $h$-sets.
In terms of the theory of covering relations (Section \ref{section-cov}), $W^s(S_\epsilon)$ is a vertical disc and $W^u(S_\epsilon)$ is a horizontal disc in a given fast-saddle-type block.
\begin{cor}
\label{cor-disk}
Let $N\in \mathbb{R}^{n+l}$ be a fast-saddle-type block for (\ref{abstract-form}) such that the coordinate representation $N_c$ is actually of the form (\ref{fast-block}). 
Assume that $N$ satisfies (\ref{ineq-graph-unstable}) and (\ref{ineq-cone-unstable}).
Let $b = h_u(a,y,\epsilon)$ be a function defining $W^u(S_\epsilon)$ in $N_c$. 
Then there is a homotopy $H: [0,1]\times N_c\times [0,\epsilon_0]\to N_c\times [0,\epsilon_0]$ satisfying 
\begin{equation*}
H_0(a,b,y,\epsilon) = H(0,a,b,y,\epsilon) = (a,h_u(a,y,\epsilon), y,\epsilon),\quad H_1(a,b,y,\epsilon) = (a,0,y,\epsilon).
\end{equation*}

Similarly, if $N_c$ satisfies (\ref{ineq-graph-stable}) and (\ref{ineq-cone-stable}) and $a = h_s(b,y,\epsilon)$ denotes a function defining $W^s(S_\epsilon)$ in $N$, then there is a homotopy $H: [0,1]\times N_c\times [0,\epsilon_0]\to N_c\times [0,\epsilon_0]$ satisfying
\begin{equation*}
H_0(a,b,y,\epsilon) = H(0,a,b,y,\epsilon) = (h_s(b,y,\epsilon),b,y,\epsilon),\quad H_1(a,b,y,\epsilon) = (0,b,y,\epsilon).
\end{equation*}
\end{cor}

\begin{proof}
When $N_c$ satisfies (\ref{ineq-graph-unstable}) and (\ref{ineq-cone-unstable}), define a map $H$ on $[0,1]\times N_c\times [0,\epsilon_0] \to \mathbb{R}^{n+l+1}$ by 
\begin{equation*}
H(\lambda,a,b,y,\epsilon):= (a, (1-\lambda)h_u(a,y,\epsilon), y,\epsilon).
\end{equation*}
By the contraction of $h_u$, $H$ is continuous. Since $N_c$ is convex, then $H$ maps $[0,1]\times N_c\times [0,\epsilon_0]$ into $N_c\times [0,\epsilon_0]$. Thus $H$ is our desiring homotopy. The proof of the case that $N_c$ satisfies (\ref{ineq-graph-stable}) and (\ref{ineq-cone-stable}) is similar. 
\end{proof}

\begin{rem}\rm
We do not need to assume the normal hyperbolicity of the critical manifold $S_0$ and the graph representation of $S_0$ in Theorem \ref{thm-inv-mfd-rigorous}, while the Fenichel's classical theory needs to assume them. 
Indeed, Part {\bf (i)} in the proof of the theorem shows that this assumption can be validated by a computable estimate (\ref{Lyapunov-condition}).
\end{rem}

\begin{rem}\rm
It is well known that {\em $S_\epsilon$ is not uniquely determined in general}.
\end{rem}

Let $D$ and $\hat D$ be fast-saddle type blocks given by (\ref{fast-block}) and (\ref{fast-block-2}), respectively.
Theorem \ref{thm-inv-mfd-rigorous} says that, under cone conditions, the slow manifold $S_\epsilon$ is contained in $D$. Obviously $S_\epsilon$ is also contained in $\hat D$ since $D\subset \hat D$. Moreover, $S_\epsilon$ is uniquely determined in $\hat D$.
If $d_0 \leq \dist(\partial (\pi_{a,b}\hat D), \pi_{a,b}D)$, then our observations imply that the distance between $S_\epsilon$ and $\partial \hat D$ in fast components is greater than $d_0$.
For example, the distance of $\partial \hat D_c$ and $S_\epsilon$ in $a_1$-component is greater than $\min \{\eta_{1,1}^\pm\}$.
Summarizing these arguments we have the following result, which is essential to Section \ref{section-show-shadowing}.

\begin{cor}[Fast-saddle-type blocks with spaces]
Consider (\ref{ode-ab-coord}). Let $D, \hat D\subset \mathbb{R}^{n+l}$ be fast-saddle-type blocks for (\ref{ode-ab-coord}) such that the coordinate representations $D_c$ and $\hat D_c$ are actually given by (\ref{fast-block}) and (\ref{fast-block-2}), respectively, for a given sequence of positive numbers $\{\eta_{i,j_i}^\pm\}_{j_1=1,\cdots, u, j_2 = 1,\cdots, s}^{i=1,2}$. Assume that stable and unstable cone conditions hold in $\hat D_c$. Then the same statements as Theorem \ref{thm-inv-mfd-rigorous} holds in $\hat D_c$. Moreover, the distance between $\partial \hat D_c$ and the validated slow manifold $S_\epsilon$ is estimated by
\begin{align*}
&\dist(\pi_a(\partial \hat D_c), \pi_a(W^s(S_\epsilon))) \geq \min_{j=1,\cdots, u}\{\eta_{1,j}^\pm\},\\
&\dist(\pi_b(\partial \hat D_c), \pi_b(W^u(S_\epsilon))) \geq \min_{j=1,\cdots, s}\{\eta_{2,j}^\pm\}\quad \text{ for }\epsilon \in [0,\epsilon_0].
\end{align*}
\end{cor}

The main feature of our present results is that slow manifolds as well as their stable and unstable manifolds in Fenichel's theorems are validated in given blocks with an explicit range $\epsilon\in [0,\epsilon_0]$.
Our criteria can be explicitly validated with rigorous numerics, as many preceding works (e.g. \cite{W, WZ}).
We end this section noting that all results in this section holds if we replace $y\in \mathbb{R}$ and $\mathbb{R}$-valued function $g$ by $y\in \mathbb{R}^l$ and $\mathbb{R}^l$-valued function $g$, respectively, with $l>1$.

%
%	New Section
%
\section{Covering-Exchange and dynamics around slow manifolds}
\label{section-slow-dynamics}
In Section \ref{section-inv-mfd} we have discussed the validation of slow manifolds.
In this section, we move to the next stage; the behavior near and on validated slow manifolds. 
There are mainly two cases of the exhibition of slow dynamics. 
One is monotone and the other is nontrivial in the sense that  dynamics on slow manifolds admit nontrivial invariant sets such as fixed points or periodic orbits.
\par
First we consider the dynamics near slow manifolds. 
If, for sufficiently small $\epsilon > 0$, a point $q_\epsilon \in \mathbb{R}^{n+1}$ is sufficiently close to slow manifolds, the trajectory through $q_\epsilon $ moves near slow manifolds spending a long time. 
However, the precise description of trajectories off slow manifolds is not easy since they clearly have an influence of fast dynamics $\dot x = f(x,y,\epsilon)$. 
In order to describe such behavior as simple as possible, we introduce an extension of covering relations: the {\em covering-exchange} (Section \ref{section-exchange}), so that it can be appropriately applied to singular perturbation problems.
This extension is a topological analogue of well-known {\em Exchange Lemma}.
In the successive subsections, we provide a generalization of covering-exchange: the {\em slow shadowing} and {\em $m$-cones} (Sections \ref{section-rate} - \ref{section-m-cones}).
These concepts enable us to validate slow manifolds with nonlinear structures as well as their stable and unstable manifolds in reasonable ways via rigorous numerics, keeping the essence of covering-exchange property.
\par
Next, we consider the nontrivial dynamics on slow manifolds, namely, the presence of nontrivial invariant sets. 
We can use ideas discussed in Section \ref{section-isolatingblock} to validate nontrivial invariant sets even on slow manifolds. 
\par
Finally, we discuss the unstable manifold of invariant sets on slow manifolds. 
The invariant foliation structure with respect to slow manifolds is essentially applied to validating unstable manifolds of invariant sets in terms of covering relations.

%
%	New Subsection
%
\subsection{Covering-Exchange with one-dimensional slow variable}
\label{section-exchange}

First we consider the case that the vector field near slow manifolds is monotone. 
In such a case, geometric singular perturbation theory has an answer which describes dynamics around slow manifolds in terms of locally invariant manifolds; {\em Exchange Lemma} (e.g. \cite{JKK, JK, Jones, TKJ}).

Exchange Lemma solves matching problems between fast dynamics and slow dynamics. 
More precisely, if a family of tracking invariant manifolds $\{\Sigma_\epsilon\}_{\epsilon \geq 0}$ have transversal intersection $\Sigma_0 \cap_T W^s(S_0)$ at $\epsilon = 0$, 
then a trajectory $\varphi_\epsilon([0,T_\epsilon], q_\epsilon)$ for a point $q_\epsilon \in \Sigma_\epsilon$
describes the match of fast and slow trajectories near the slow manifold in the full system (\ref{fast-slow})$_\epsilon$ for sufficiently small $\epsilon > 0$.
However, Exchange Lemma requires the assumption of transversality 
between a tracking (locally invariant) manifold and the stable manifold of a slow manifold, which is not easy to validate in rigorous numerics.
Alternatively, we consider the topological analogue of statements in ($C^0$-)Exchange Lemma.

\bigskip
Every trajectories in fast-saddle-type blocks leaves them through exits. Note that our desiring trajectories are the continuation of chains consisting of critical manifolds and heteroclinic orbits in the limit system (\ref{layer}). One thus expects that singularly perturbed global trajectories leave fast-saddle-type blocks through {\em fast-exits}.
To describe our expectations precisely, we define the following notions.
For our purpose, we restrict $u = u(M)$ for $h$-sets to $u=1$ unless otherwise noted.

\begin{dfn}[Fast-exit face]\rm
\label{dfn-face}
Let $M$ be an $(n+1)$-dimensional $h$-set with the following expression via a homeomorphism $c_M: \mathbb{R}^{n+1}\to \mathbb{R}^{u(M)+s(M)}$:
\begin{equation*}
M_c = c_M(M) = \overline{B_u} \times \overline{B_s}\times [0,1]
\end{equation*}
with $u = u(M) =1$ and $s(M) = s + 1$. We say an $h$-set $M^{a}$ {\em a fast-exit face} of $M$ if $u(M^{a}) = u$, $s(M^{a}) = s$ and there exist an element $a\in \partial B_u$ and a compact interval $K_0\subset (0,1)$ such that
\begin{equation*}
c_M(M^{a}) = \{a\}\times \overline{B_s}\times K_0.
\end{equation*}
See also Fig. \ref{fig-CE}.
\end{dfn}

\begin{dfn}[Covering-Exchange]\rm
\label{dfn-CE}
Consider (\ref{fast-slow})$_{\epsilon}$ with fixed $\epsilon \geq 0$. Let $N\subset \mathbb{R}^{n+1}$ be an $h$-set with $u(N)=u= 1$ and $M$ be an $(n+1)$-dimensional $h$-set in $\mathbb{R}^{n+1}$ with $u(M) = 1$. We say $N$ satisfies {\em the covering-exchange property for (\ref{fast-slow})$_\epsilon$ with respect to $M$} if the following statements are satisfied:
\begin{description}
\item[(CE1)] $M$ is a fast-saddle-type block with the coordinate system
\begin{align*}
M_c  = \overline{B_1}\times \overline{B_s} \times [0,1],\quad a\in \overline{B_1},\quad (b_1,\cdots, b_s)\in \overline{B_s}, \quad y\in [0,1].
\end{align*}
\item[(CE2)] Stable and unstable cone conditions (Definition \ref{dfn-cone}) for (\ref{fast-slow})$_\epsilon$ are satisfied in $M$.
\item[(CE3)] For $q\in \{\pm 1\}$, $q\cdot g(x,y,\epsilon) > 0$ holds in $M$. We shall say $q$ {\em the slow direction number}.
\item[(CE4)] Letting $\varphi_\epsilon$ be the flow of (\ref{fast-slow})$_\epsilon$, there exists $T > 0$ such that $N \overset{\varphi_\epsilon(T,\cdot)}{\Longrightarrow} M$.
\item[(CE5)] $M$ possesses a fast-exit face $M^a \subset M^{f,-}$ with the expression
\begin{equation*}
c_M(M^a) = \{a\}\times \overline{B_s} \times [y^-,y^+],\quad a\in \partial B_1,\quad [y^-,y^+]\subset (0,1)
\end{equation*}
such that 
\begin{equation*}
\begin{cases}
\sup (\pi_y \varphi_\epsilon(T, N)\cap M) < y^- & \text{ if $q=+1$}\\
\inf (\pi_y \varphi_\epsilon(T,N)\cap M) > y^+ & \text{ if $q=-1$}
\end{cases}
\end{equation*}
holds.
\end{description}
In this situation we shall call the pair $(N,M)$ {\em the covering-exchange pair} for (\ref{fast-slow})$_{\epsilon}$.
\end{dfn}

A brief illustration of the covering-exchange property is shown in Fig. \ref{fig-CE}. 
(CE2) implies the existence of the slow manifold $S_\epsilon$ as well as a limiting normally hyperbolic critical manifold $S_0$ at $\epsilon = 0$, as shown in Theorem \ref{thm-inv-mfd-rigorous}. 
Combining with (CE1), one sees that $M$ is repelling in $a$-direction and attracting in $b$-direction. 
(CE3) means that slow dynamics in $M$ is monotone. Remark that the monotonicity is assumed not only on $S_\epsilon$ but also in the whole region $M$. 
(CE4) topologically describes the transversality of the stable manifold $W^s(S_\epsilon)$ of the slow manifold $S_\epsilon$ and the $h$-set $\varphi_\epsilon(T,N)$. 

For a fast-saddle-type block $M$ of the form in Definition \ref{dfn-CE} with $q=+1$, 
we say $M_c^{s,-}\equiv \overline{B_1}\times \overline{B_s} \times \{1\}$ and $M_c^{s,+}\equiv \overline{B_1}\times \overline{B_s} \times \{0\}$ (equivalently $M^{s,-}$ and $M^{s,+}$) {\em the slow exit} and {\em the slow entrance} of $M$, respectively. 
Similarly, in the case of $q=-1$, $M_c^{s,-}\equiv \overline{B_1}\times \overline{B_s} \times \{0\}$ and $M_c^{s,+}\equiv \overline{B_1}\times \overline{B_s} \times \{1\}$ be the slow exit and the slow entrance of $M$, respectively.

\begin{figure}[htbp]\em
\begin{minipage}{1\hsize}
\centering
\includegraphics[width=6.0cm]{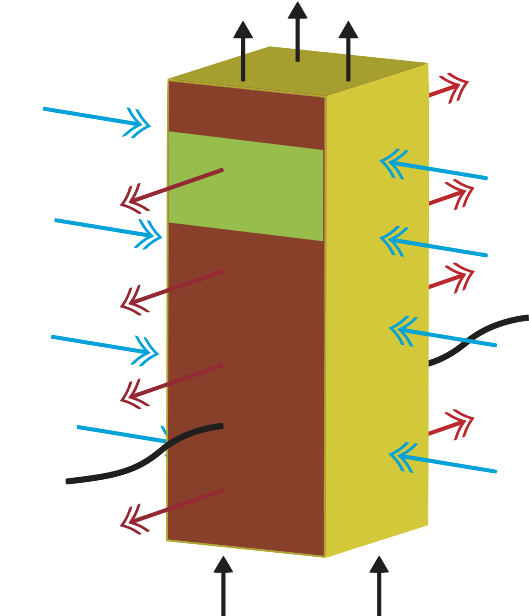}
\end{minipage}
\caption{Illustrations of the Covering-Exchange pair and a fast-exit face.}
\label{fig-CE}
The rectangular parallelepiped $M$ is an $h$-set with $u(M) = 1$ and $s(M)=2$. 
Arrows colored by red, blue and black show the fast unstable direction, fast stable direction and slow direction of (\ref{abstract-form}) at corresponding boundaries, respectively.
The black curve is an image $\varphi_\epsilon(T,N)$ of another $h$-set $N$ for some $T > 0$. 
This figure describes $N \overset{\varphi_\epsilon(T,\cdot )}{\Longrightarrow}M$. 
Since the $y$-component of vector field in $M$ is monotone and the fast component of the vector field on $\partial M$ is already validated concretely, 
then $\varphi_\epsilon(T,N)$ keeps the covering relation $N \overset{\varphi_\epsilon(t,\cdot )}{\Longrightarrow}M$ for all $t \geq T$ until $N$ exits $M$ from its top.
A fast-exit face $M^a$ is colored by green, in which case $u=1, s=1, l=1$.
\end{figure}

\bigskip
Define the Poincar\'{e} map $P^M_\epsilon: M \to \partial M$ in $M$ by
\begin{equation*}
P^M_\epsilon(z):= \varphi_\epsilon(t_\epsilon(z),z),\quad 
t_\epsilon(z):= \sup\{t\mid \varphi_\epsilon([0,t],z)\subset M\}.
\end{equation*}
If no confusion arises, we use this notation for representing the Poincar\'{e} map in a set $M$.

\begin{nota}\rm
For $\hat y \in [0,1]$ and any set $A\subset \mathbb{R}^{n+1}$, let $A_{\hat y} := A\cap \{y= \hat y\}$, $A_{\leq \hat y} := A\cap \{y\leq \hat y\}$ and $A_{\geq \hat y} := A\cap \{y\geq \hat y\}$. 
\end{nota}

\begin{lem}
\label{lem-Poincare-hset}
Fix $\epsilon > 0$. Let $(N,M)$ be a covering-exchange pair for (\ref{fast-slow})$_\epsilon$ with the slow direction number $q=+1$ and the fast-exit face $M^{{\rm exit}}$ with $(M^{{\rm exit}})^+ \subset M^{f,+}$, where
\begin{equation*}
c_M(M^\exit) = \{1\}\times \overline{B_s}\times [y^-,y^+],\quad [y^-,y^+]\subset (0,1).
\end{equation*} 
\par
Then, for each $\hat y\in [0, y^-)$, the restriction of the Poincar\'{e} map $P^M_\epsilon$ to $((P^M_\epsilon)^{-1}(M^{{\rm exit}}))_{\hat y}$ is a homeomorphism.
Moreover, there exists an $h$-set $\tilde M\subset M$ such that $ \tilde M  \overset{P_\epsilon^M}{\Longrightarrow} M^{{\rm exit}}$. 
\end{lem}
\begin{proof}
Via a homeomorphism $c_M$, we may assume that $M = \overline{B_1}\times \overline{B_s} \times [0,1]$.

Since $M$ is an isolating block for $\varphi_\epsilon$, then $P^M_\epsilon$ is continuous on $M$, which is the consequence of standard theory of isolating blocks \cite{Con, Jones, Smo}. 
Consider the backward Poincar\'{e} map $P^{M,-}_\epsilon : M\to \partial M$ given by
\begin{equation*}
P^{M,-}_\epsilon(z):= \varphi_\epsilon(\tau(z),z),\quad 
\tau(z):= \inf\{t \leq 0\mid \varphi_\epsilon([t,0],z)\subset M\}.
\end{equation*}
This is obviously well-defined via the property of flows. 
Since $M$ is also an isolating block for the backward flow, then the exit of $M$ in backward time is closed and hence $P^{M,-}_\epsilon$ is also continuous. 
The correspondence between points on $(P^M_\epsilon)^{-1}(M^{{\rm exit}})\cap \{y = \hat y\}$ and their image is one-to-one and onto under $P^M_\epsilon$.
This property holds since $M\cap \{y = \hat y\}$ is a crossing section from (CE3).
By restricting $M$ to $M_{\geq \hat y}$ and redefining $P^M_\epsilon$ and $P^{M,-}_\epsilon$ for $M_{\geq \hat y}$, their continuity yields that $(P^M_\epsilon)^{-1}(M^{{\rm exit}})\cap \{y = \hat y\}$ is compact. 
Therefore, the restriction of $P^M_\epsilon$ to 
$(P^M_\epsilon)^{-1}(M^{{\rm exit}})\cap \{y = \hat y\}$ is a homeomorphism.

\bigskip
Next consider $P^{M_{\geq \hat y}}_\epsilon$ instead of $P^M_\epsilon$.
Since $M\subset \mathbb{R}^{n+1}$ is a fast-saddle-type block with $u(M) = 1$ satisfying stable and unstable cone conditions, then the stable manifold $W^s(S_\epsilon)$ divides $M$ into two components, where $S_\epsilon$ is the slow manifold in $M$. 
Let $M_+$ be the component of $M\setminus W^s(S_\epsilon)$ containing $M^{{\rm exit}}$. 
Now we can construct an $h$-set with desiring covering relation as follows.

Consider a section $M_{+,\hat b} = \{(a,\hat b,\hat y)\mid (a,\hat b,\hat y)\in N\}$. Let $a_{\hat b}\in \overline{B_1}$ be such that $W^s(S_\epsilon)\cap \{(b,y)=(\hat b, \hat y)\} = \{(a_{\hat b},\hat b,\hat y)\}$.
Since $P^{M_{\geq \hat y}}_\epsilon \mid_{M_+} : M_+\to P^{M_{\geq \hat y}}_\epsilon (M_+)\subset \partial M$ is homeomorphic, $g(x,y,\epsilon) > 0$ holds and 
\begin{equation*}
P^{M_{\geq \hat y}}_\epsilon(1,\hat b, \hat y) = (1,\hat b, \hat y)\in M^{f,-},\quad P^{M_{\geq \hat y}}_\epsilon(a_{\hat b},\hat b, \hat y) \in M^{s,-}
\end{equation*}
hold, Proposition \ref{prop-cone} implies that $z_0, z_1, z_2 \in M$ with $z_0\in W^s(S_\epsilon)$, $z_1, z_2 \in C^u(z_0)$, $|\pi_a(z_1-z_0)| > |\pi_a(z_2-z_0)|$ and $\pi_y(z_1) = \pi_y(z_2)$ satisfy $\pi_y( P^{M_{\geq \hat y}}_\epsilon (z_1)) < \pi_y(P^{M_{\geq \hat y}}_\epsilon (z_2))$.
Here $C^u(z)$ denotes the unstable cone with the vertex $z$:
\begin{equation*}
C^u(z) = \{(a,b,y)\mid |a - \pi_a(z)|^2 \geq  |b - \pi_b(z)|^2 +  |y - \pi_y(z)|^2\}. 
\end{equation*}

Thus there exist points $(a_{\hat b}^\pm,\hat b, \hat y)\in M_{+,\hat b}$ such that  $P^{M_{\geq \hat y}}_\epsilon (a_{\hat b}^\pm,\hat b, \hat y) \in \partial M \cap \{y=y^\pm\}$ with $a_{\hat b} < a^+_{\hat b} < a^-_{\hat b}$.
\par
Let $\psi^\pm : \overline{B_s}\to \overline{B_1}$ be given by $\psi^\pm(\hat b) = a^\pm_{\hat b}$. 
Thanks to the continuity of $P^{M_{\geq \hat y}}_\epsilon$, both $\psi^\pm$ are continuous. 
Finally, define $\tilde M_{\hat y}$ by
\begin{equation*}
\tilde M_{\hat y} := \{(a,b,\hat y)\mid \psi^+(b) \leq a \leq \psi^-(b), b\in\overline{B_s}\}. 
\end{equation*}
It is indeed an $h$-set with
\begin{equation}
\label{new-hset-bd}
{\tilde M_{\hat y}}^- := \{(a,b,\hat y)\mid a = \psi^\pm(b), b\in\overline{B_s}\},\quad {\tilde M_{\hat y}}^+ = \{(a,b,\hat y)\mid \psi^+(b) \leq a \leq \psi^-(b), b\in \partial B_s\}. 
\end{equation}
Slight modifications of $\psi^\pm$ yield the correspondence 
\begin{equation*}
\pi_y(P^{M_{\geq \hat y}}_\epsilon (\psi^-(b),b,\hat y))\in (\hat y,y^-),\quad \pi_y(P^{M_{\geq \hat y}}_\epsilon (\psi^+(b),b,\hat y))\in (y^+,1)\quad \text{ for }b\in \overline{B_s}.
\end{equation*}
We rewrite the corresponding $h$-set as $\tilde M_{\hat y}$ again.

\bigskip
Note that $P^{M_{\geq \hat y}}_\epsilon ({\tilde M_{\hat y}}^-) \cap M^\exit = \emptyset$ by definition of $\psi^\pm$.
Also, $P^{M_{\geq \hat y}}_\epsilon (\tilde M_y)\cap (M^\exit)^+ = \emptyset$ holds since $(M^\exit)^+$ is on $M^{f,+}$. 
The arbitrary choice of $q_0\in \overline{B_s}$ leads to the statement in Proposition \ref{prop-find-CR}. 
In particular, $\tilde M_{\hat y} \overset{P^{M_{\geq \hat y}}_\epsilon }{\Longrightarrow} M^{{\rm exit}}$ holds. 
Of course, $\tilde M_{\hat y} \overset{P^M_\epsilon }{\Longrightarrow} M^{{\rm exit}}$ also holds since $P^{M_{\geq \hat y}}_\epsilon$ is just a restriction of $P^M_\epsilon$ in $M_{\geq \hat y}$.

Moreover, the union $\tilde M := \bigcup_{y\in [0, \hat y]} \tilde M_{\hat y}$ also satisfies $\tilde M \overset{P_\epsilon^M}{\Longrightarrow} M^{{\rm exit}}$.
Proof of the case $q=-1$ is similar.
\end{proof}

The following proposition is the core of the covering-exchange property.
\begin{prop}
\label{prop-CE}
Let $(N,M)$ be a covering-exchange pair for (\ref{fast-slow})$_\epsilon$ with a fast-exit face $M^{{\rm exit}}$. Then there exists an $h$-set $\tilde M\subset M$ such that 
\begin{equation*}
N \overset{ \varphi_\epsilon(T,\cdot) }{\Longrightarrow} \tilde M  \overset{P_\epsilon^M}{\Longrightarrow} M^{{\rm exit}}.
\end{equation*}
\end{prop}
\begin{proof}
Without the loss of generality, we may assume that  $M = \overline{B_1}\times \overline{B_s} \times [0,1]$ and $M^{{\rm exit}} = \{1\}\times \overline{B_s} \times [y^-,y^+]$ via a homomorphism $c_M$, 
where $[y^-,y^+]\subset (0,1)$. Also, let $q=+1$.

By Lemma \ref{lem-Poincare-hset}, we can construct an $h$-set depending on $\hat y\in [0,y^-)$ which $P^M_\epsilon$-covers $M^\exit$. Let $\tilde M_{\hat y}$ be such an $h$-set. 
Note that such an $h$-set can be constructed for all $\hat y\in [0,y^-)$. 
Now choose $\hat y_\pm \in [0,y^-)$ so that
\begin{equation*}
\hat y_- < \inf \pi_y( \varphi(T,N)\cap M) <  \sup \pi_y( \varphi(T,N)\cap M) < \hat y_+,
\end{equation*}
and define a set $\tilde M$ by $\tilde M := \bigcup_{y\in [\hat y_-, \hat y_+]}\tilde M_y$. 
The monotonicity of $g$ and homeomorphic  correspondence of the flow imply that $\tilde M$ is actually an $h$-set with $\tilde M^\pm = \bigcup_{y\in [\hat y_-, \hat y_+]} \tilde M_y^\pm$, where $\tilde M_y^\pm$ are given by (\ref{new-hset-bd}).
The covering relation $\tilde M  \overset{P_\epsilon^M}{\Longrightarrow} M^\exit$ immediately holds from the proof of Lemma \ref{lem-Poincare-hset}. 
Moreover, $N \overset{\varphi_\epsilon(T,\cdot)}{\Longrightarrow} \tilde M$ also holds by the construction of $\tilde M$ and the assumption $N \overset{\varphi_\epsilon(T,\cdot)}{\Longrightarrow} M$.
\end{proof}
Proposition \ref{prop-CE} shows that the covering-exchange property enables us to track solution orbits near slow manifolds. This is just a consequence of Proposition \ref{ZG-periodic}.

\begin{rem}\rm
Our implementations of covering-exchange property automatically solve the matching problem in singular perturbation problems. 
Indeed, let $(N,M)$ be a covering-exchange pair with a fast-exit face $M^{{\rm exit}}$. 
In most cases, $N$ is a family of tracking invariant manifold whose behavior is mainly dominated by  fast dynamics. 
On the other hand, the behavior of points in $\tilde M$ obtained in Proposition \ref{prop-CE} is mainly dominated by the slow dynamics until it leaves $M$. 
If such a point is also on $\varphi_\epsilon(T,N)$, then we can prove the existence of a point $q_\epsilon \in N$ 
which goes to $M$ along the fast dynamics in the first stage, 
stays $M$ dominated by the slow dynamics until it leaves $M$ through $M^{{\rm exit}}$ in the second stage 
and goes outside $M$ along the fast dynamics again in the third stage. 
The trajectory through $q_\epsilon$ is actually dominated by the full system (\ref{fast-slow})$_\epsilon$.
Such a behavior is thus considered as the match of dynamics in different time scales.

One of other key points of the covering-exchange is that the choice of $u$-dimensional variables is changed during time evolutions around slow manifolds. 
Indeed, the covering relation $N \overset{\varphi_\epsilon(T,\cdot)}{\Longrightarrow} \tilde M$ typically corresponds to validation of codimension $p$ connecting orbits between saddle-type equilibria in the limit system.
In such a case, the $u=u(N)=p$-dimensional variable are chosen from $y$-variables, while the $u$-dimensional variables of $\tilde M$ are chosen from $a$-variables: fast-unstable directions, which is natural from the viewpoint of codimension $p$ connecting orbits.
Finally, during the covering relation $\tilde M  \overset{P_\epsilon^M}{\Longrightarrow} M^{{\rm exit}}$, the $u$-dimensional variables return to the part of $y$-variables, since $M^{{\rm exit}}$ is regarded as the initial data of other codimension $p$ connecting orbits in the limit system.
Consequently, we observe the \lq\lq exchange" of the choice of $u$-dimensional variables during the covering relation $\tilde M  \overset{P_\epsilon^M}{\Longrightarrow} M^{{\rm exit}}$ in Proposition \ref{prop-CE}.
In this sense, we call this phenomenon the covering-\lq\lq {\em exchange}".
Needless to say, one can describe such matching property without solving any differential equations in $M$.
\end{rem}

\bigskip
%
%	New Subsection
%
\subsection{Expanding and contracting rate of disks}
\label{section-rate}

Assumptions of the covering-exchange in the previous subsection require that each branch of slow manifolds should be validated by one fast-saddle-type block. 
However, slow manifolds may have nontrivial curvature in general.
It is not thus realistic to validate slow manifolds by one block with computer assistance without any knowledges of vector fields. 
To overcome such computational difficulties, we provide several techniques for generalizing conditions in the covering-exchange.

As a preliminary to next subsections, we consider the expansion and contraction rate of disks in a fast-saddle-type block.
Consider the system (\ref{abstract-form}). Note that $A$ in (\ref{abstract-form}) is the $u\times u$-diagonal matrix such that all eigenvalues of $A$ have positive real part, 
and that $B$ in (\ref{abstract-form}) is the $s\times s$-diagonal matrix such that all eigenvalues of $B$ have negative real part.
Let $N\subset \mathbb{R}^{n+1}$ be a fast-saddle-type block satisfying stable and unstable cone conditions for fixed $\epsilon > 0$. 
Without the loss of generality, via homomorphism $c_N$, we may assume that
\begin{equation*}
N = \overline{B_u}\times \overline{B_s}\times [0,1].
\end{equation*}
Here we also assume that $g(x,y,\epsilon) > 0$ holds in $N$. 
Two cross sections $\overline{B_u}\times \overline{B_s}\times \{0,1\}$ then become the slow entrance $N^{s,+}$ and  the slow exit $N^{s,-}$ of $\varphi_\epsilon$, respectively. 
\begin{dfn}\rm
An {\em $\alpha$-Lipschitz unstable disk} is the graph of a Lipschitz function $\psi : \overline{B_u} \to \overline{B_s}\times [0,1]$ with ${\rm Lip}(\psi) \leq \alpha$. Similarly, an {\em $\alpha$-Lipschitz stable disk} is the graph of a Lipschitz function $\psi : \overline{B_s} \to \overline{B_u}\times [0,1]$ with ${\rm Lip}(\psi) \leq \alpha$. 
\end{dfn}

Next we consider the expanding and contracting rate of stable and unstable disks. 

\begin{lem}
\label{lem-ratio}
Let $N\subset \mathbb{R}^{n+1}$ be a fast-saddle type block satisfying stable and unstable cone conditions. 
Let $D^u(r)$ and $D^s(r)$ be a $1$-Lipschitz unstable disk and a stable disk of the diameter $r$, respectively, contained in $N$. 
Let $T > 0$ be fixed.
\begin{enumerate}
\item Assume that $\varphi_\epsilon([0,T],D^u(r))\subset N$. 
Then for all $z_1, z_2\in D^u(r)$ the following inequality holds: 
\begin{equation*}
|\pi_a(z_1(t)-z_2(t))| \geq e^{\lambda_{\min} t}|\pi_a(z_1-z_2)|\quad \text{ for }t\in [0,T],
\end{equation*}
where $\lambda_{\min} =\lambda_A - (\sup \sigma_{\mathbb{A}_1^u} + \sup \sigma_{\mathbb{A}_2^u}) > 0$ is given by (\ref{ineq-graph-unstable}). Here $\lambda_A$ denotes a positive number satisfying (\ref{bound-unst-ev}).
\item Assume that $\varphi_\epsilon([-T,0],D^s(r))\subset N$. Then for all $z_1, z_2\in D^s(r)$ the following inequality holds: 
\begin{equation*}
|\pi_b(z_1(-t)-z_2(-t))| \geq e^{-\mu_{\min} t}|\pi_b(z_1-z_2)|\quad \text{ for }t\in [0,T],
\end{equation*}
where $\mu_{\min} = \mu_B + (\sup \sigma_{\mathbb{B}_1^s} + \sup \sigma_{\mathbb{B}_2^s}) < 0$ is given by (\ref{ineq-graph-stable}). Here $\mu_B$ denotes a negative number satisfying (\ref{bound-st-ev}).
\end{enumerate}
\end{lem}

The above inequalities imply that the expanding rate of $D^u(r)$ is uniformly bounded below by $\lambda_A$ and that the contracting rate of $D^s(r)$ is uniformly bounded above by $\mu_B$.

\begin{proof}
{\em 1.} Assume first that $z_1, z_2\in D^u(r)$. 
The unstable cone condition with our assumption implies that $z_2 \in C^u(z_1)$ and $z_2(t) \in C^u(z_1(t))$ hold for all $t\in [0,T]$.
In particular, $|\pi_a(z_1(t)-z_2(t))| \geq |\pi_{b,y}(z_1(t)-z_2(t))|$ holds for all $t\in [0,T]$.
Let $\Delta a(t) := \pi_a(z_1(t)-z_2(t))$, $\Delta b(t) := \pi_b(z_1(t)-z_2(t))$ and $\Delta y(t) := \pi_y(z_1(t)-z_2(t))$.
The first variation equation in $a$-direction yields
\begin{equation*}
\langle \Delta a', \Delta a\rangle \geq \lambda_A |\Delta a|^2 - (\sup \sigma_{\mathbb{A}_1^u} |\Delta a|^2 + \sup \sigma_{\mathbb{A}_2^u} |\Delta a|^2),
\end{equation*}
which is obtained by arguments in the proof of Proposition \ref{prop-cone}.
Here we used the fact that $|\Delta a(t)|\geq (|\Delta b(t)|^2 + |\Delta y(t)|)^{1/2}$ holds in $C^u(z_1(t))$ for all $t\in [0,T]$.
This inequality leads to $(|\Delta a|^2)' \geq 2\lambda_{\min}|\Delta a|^2$, and hence $|\Delta a(t)| \geq \exp(\lambda_{\min}t)|\Delta a(0)|$.

\bigskip
{\em 2.}  Next assume that $z_1, z_2\in D^s(r)$. The stable cone condition with our assumption implies that $z_2 \in C^s(z_1)$ and $z_2(-t) \in C^s(z_1(-t))$ hold for all $t\in [0,T]$, where 
\begin{equation*}
C^s(z) = \{(a,b,y)\mid |b - \pi_b(z)|^2 \geq  |a - \pi_a(z)|^2 +  |y - \pi_y(z)|^2\}. 
\end{equation*}
In particular, $|\pi_b(z_1(-t)-z_2(-t))| \geq |\pi_{a,y}(z_1(-t)-z_2(-t))|$ holds for all $t\in [0,T]$.

The backward first variation equation in $b$-direction yields
\begin{align*}
\frac{1}{2}\frac{d}{d\tilde t}|\Delta b|^2 \geq |\mu_B| |\Delta b|^2 - (\sup \sigma_{\mathbb{B}_1^s} |\Delta b|^2 + \sup \sigma_{\mathbb{B}_2^s} |\Delta b|^2),\quad \text{ where }\tilde t = -t.
\end{align*}
Here we used the fact that $|\Delta b(-t)|\geq (|\Delta a(-t)|^2 + |\Delta y(-t)|)^{1/2}$ holds in $C^s(z_1(-t))$ for all $t\in [0,T]$.
This inequality leads to $(|\Delta b|^2)' \geq 2|\mu_{\min}| |\Delta b|^2$, and hence $|\Delta b(-t)| \geq \exp(-\mu_{\min}t)|\Delta b(0)|$.
\end{proof}

Note that Lemma \ref{lem-ratio} holds even when $u\not = 1$.

%
%	New Subsection
%
\subsection{Slow shadowing, drop and jump}
\label{section-show-shadowing}
Under suitable assumptions, slow manifolds as well as limiting critical manifolds are represented by graphs of nonlinear functions. 
In such cases, it is not easy to validate slow manifolds by {\em one} block of fast-saddle-type. 
It is thus natural to construct enclosure of such manifolds by the finite sequence of fast-saddle-type blocks, 
which will be easier than validation by one $h$-set. 
However, the union of finite fast-saddle-type blocks is not generally convex and hence 
we cannot apply arguments of the covering-exchange directly. 
In particular, we have to consider the effect of slow exits and slow entrances of the union. 
Nevertheless, the behavior in $y$-direction is much slower than the behavior in $(a,b)$-direction for sufficiently small $\epsilon$. 
It is thus natural to consider that slow exits and entrances cause no effect on validation of trajectories sufficiently close to slow manifolds for (\ref{fast-slow}).
Here we construct a sufficient condition to explicitly guarantee such expectation. 
The key feature consists of two parts: (i) the comparison of expansion and contraction rate with the speed of slow vector field, and (ii) abstract construction of covering relations around slow manifolds. 
Our proposing concept named {\em slow shadowing} expresses both features.
Slow shadowing can incorporate with the covering relations in Proposition \ref{prop-CE}; the {\em drop} corresponding to
$N \overset{ \varphi_\epsilon(T,\cdot) }{\Longrightarrow} \tilde M$ and the {\em jump} corresponding to $\tilde M  \overset{P_\epsilon^M}{\Longrightarrow} M^{{\rm exit}}$.
As a consequence, the concept of covering-exchange is generalized to the finite union of fast-saddle-type blocks.

The center of our considerations is a sequence of fast-saddle-type blocks $\{N_j\}$ satisfying all assumptions in Theorem \ref{thm-inv-mfd-rigorous}.
In this case, Theorem \ref{thm-inv-mfd-rigorous} indicates that the stable manifold $W^s(S_\epsilon)_j$ is given by the graph of a Lipschitz function $h_s^j$ in each $N_j$.
Cone conditions also yield that these manifolds are patched globally in $\bigcup N_j$.

\begin{lem}
\label{lem-slow-mfd-global}
Let $\{N_j\}_{j=0}^{m_j}$ be a sequence of fast-saddle-type blocks satisfying all assumptions in Theorem \ref{thm-inv-mfd-rigorous}.
Assume that, for all $j$, $N_j\cap N_{j+1}\not = \emptyset$ and that each section $(N_j\cap N_{j+1})_{\bar y}$ contains a unique point of the critical manifold $S_0$.
Then, for all $\epsilon \in [0,\epsilon_0]$, validated stable manifolds $W^s(S_\epsilon)_j$ and $W^s(S_\epsilon)_{j+1}$ in blocks $N_j$ and $N_{j+1}$, respectively, coincide with each other in the intersection $N_j\cap N_{j+1}$ for $j=0,\cdots, m_j-1$.
The similar result holds for $W^u(S_\epsilon)$.
\end{lem}
\begin{proof}
The same arguments as Theorem \ref{thm-inv-mfd-rigorous} for $(a_1, b, y)\in N_j\cap N_{j+1}\cap W^s(S_\epsilon)_j$ and $(a_2, b, y)\in N_j \cap N_{j+1}\cap W^s(S_\epsilon)_{j+1}$ with a fixed $(b,y)$ yield the result.
\end{proof}

\bigskip
Consider two fast-saddle-type blocks $N_1$ and $N_2$.
Assume that each $N_i$ is constructed in the local coordinate $((a_i, b_i), y_i)$ whose origin corresponds to $(\bar x_i, \bar y_i)$ such that (\ref{fast-slow}) locally forms (\ref{abstract-form}).
Two coordinate systems $\{((a_i, b_i), y_i) \mid i=1,2\}$ are related to each other by the following commutative diagram:
\begin{equation}
\label{change-of-coordinate}
\begin{CD}
((a_1,b_1),y_1) @> (T_{12})_c >> ((a_2,b_2),y_2)\\
@V {P_1\times I_1=c_{N_1}^{-1}} VV @VV {P_2\times I_1=c_{N_2}^{-1}} V \\
(x-\bar x_1, \bar y_1) @> \text{translation} >>  (x-\bar x_2, \bar y_2)
\end{CD}
\end{equation}
where $P_1$ and $P_2 : \mathbb{R}^n \to \mathbb{R}^n$ be nonsingular matrices determining the approximate diagonal form (\ref{abstract-form}) around $(\bar x_1, \bar y_1)$ and $(\bar x_2, \bar y_2)$, respectively,
and $I_m : \mathbb{R}^m\to \mathbb{R}^m$ denotes the identity map on $\mathbb{R}^m$. 
The map $T_{12}\equiv c_{N_2}^{-1}\circ (T_{12})_c\circ c_{N_1}$ denotes the composition map given by
\begin{align*}
((a_2,b_2),\bar y_2) &= (T_{12})_c((a_1,b_1),\bar y_1) \equiv ((T_{x,12})_c(a_1,b_1),(T_{y,12})_c\bar y_1) \\
	&:= (P_2\times I_1)^{-1} \{(P_1(a_1,b_1),\bar y_1) + (\bar x_1-\bar x_2, \bar y_1 - \bar y_2)\}.
\end{align*}

We assume
\begin{description}
\item[(SS1)] $u(N_1) = u(N_2) = u$.
\item[(SS2)] $N_1\cap N_2 \not = \emptyset$. For given sequences of positive numbers $\{\eta_{i,j_i}^{k, \pm}\}_{j_1= 1, j_2 = 1,\cdots, s}^{i=1,2}$, $k=1,2$, $(N_k)_c$ is constructed by (\ref{fast-block-2}) with the sequence $\{\eta_{i,j_i}^{k, \pm}\}_{j_1=1,\cdots, u, j_2 = 1,\cdots, s}^{i=1,2}$.
The $y$-component of $N_k$ is given by $\pi_y(N_k) = [y_k^-,y_k^+]$ with 
\begin{align*}
&y_k^- < y_k^+\ (k=1,2),\quad y_2^- \in (y_1^-, y_1^+],\quad y_1^+ \in [y_2^-, y_2^+),
\end{align*}
where $\pi_y$ is the projection onto the slow variable component.
\item[(SS3)] All assumptions in Theorem \ref{thm-inv-mfd-rigorous} are satisfied in $N_1$ and $N_2$. Moreover, $q\cdot \epsilon(x,y,\epsilon) > 0$ holds with the slow direction number $q \in \{\pm 1\}$ in $N_1\cup N_2$.
\item[(SS4)] For $k=1,2$, let $A_k$ and $B_k$ be diagonal matrices representing (\ref{ode-ab-coord}) in $N_k$, and $\alpha_k$ and $\beta_k$ be real numbers satisfying (\ref{bound-unst-ev}) for $A_k$ and (\ref{bound-st-ev}) for $B_k$, respectively.
Also, let
\begin{align}
\notag
r_{a,k} &:= \min_{j=1,\cdots,u}{\eta_{1,j}^{k,\pm}},\quad r_a := \min \{r_{a,1},r_{a,2}\},\quad \bar r_{a,k} := \diam(\pi_a(N_k)),\\
\notag
r_{b,k} &:= \min_{j=1,\cdots,s}{\eta_{2,j}^{k,\pm}},\quad r_b := \min \{r_{b,1},r_{b,2}\},\quad \bar r_{b,k} := \diam(\pi_b(N_k)),\\
\notag
\bar h &\in \left(0, \min_{k=1,2}\{y_k^+ - y_k^-\} \right),\\
\label{setting-shadow}
\epsilon^{up}_k &:= \sup_{N_k\times [0,\epsilon_0]}\epsilon g(x,y,\epsilon),\quad \epsilon^{low}_k := \inf_{N_k\times [0,\epsilon_0]}\epsilon g(x,y,\epsilon),\quad \bar \epsilon_k := \max\{|\epsilon^{up}_k|, |\epsilon^{low}_k|\}\\
\notag
\lambda_k &:= \alpha_k - \left( \sup_{N_k\times [0,\epsilon_0]} \sigma_{\mathbb{A}_1^u} + \sup_{N_k\times [0,\epsilon_0]} 
\sigma_{\mathbb{A}_2^u} \right),\quad
\mu_k := \beta_k +  \left( \sup_{N_k\times [0,\epsilon_0]} \sigma_{\mathbb{B}_1^s} + \sup_{N_k\times [0,\epsilon_0]} \sigma_{\mathbb{B}_2^s} \right).
%\\
\end{align}
See Assumptions \ref{ass-cone-unstable} and \ref{ass-cone-stable} for the definition of $\sigma_{\mathbb{A}_1^u}, \sigma_{\mathbb{A}_2^u}, \sigma_{\mathbb{B}_1^s}$ and $\sigma_{\mathbb{B}_2^s}$.
\item[(SS5)] 
Fix positive numbers $d_a, d_b\in (0,1)$ arbitrarily.
Let $D_1^u\subset N_1$ and $D_2^s\subset N_2$ be families of disks given by
\begin{equation*}
(D_1^u)_c = \left( \bigcup_{q\in W^u(S_\epsilon)_c}\overline{B_s(q,d_b r_b)} \right)\cap (N_1)_c,\quad (D_2^s)_c = \left( \bigcup_{q\in W^s(S_\epsilon)_c} \overline{B_u(q,d_a r_a)} \right)\cap (N_2)_c
\end{equation*}
via $c_{N_1}$ and $c_{N_2}$, respectively.

Then
\begin{equation}
\label{ass-cov-shadow}
(D_1^u)_{y} \overset{T_{x,12}}{\Longrightarrow} (D_2^s)_{y}
\end{equation}
holds for all $y\in \pi_y(N_1)\cap \pi_y(N_2)$.
\end{description}

Since both $P_1$ and $P_2$ are linear maps, the covering relation (\ref{ass-cov-shadow}) is just a transversality of rectangular domains relative to an affine map.

\begin{dfn}[Slow shadowing pair]\rm
\label{dfn-shadow-1}
Let $\chi \in (0,1]$ be a fixed number.
Assume (SS1) - (SS5).
We say the pair $\{N_1, N_2\}$ satisfies the {\em slow shadowing condition (with the ratio $\chi$ and the slow direction number $q$)} if the following inequality holds: 
\begin{equation}
\label{shadow}
\max\left\{\frac{1}{\lambda_k}\log \left(\frac{\bar r_{a,k} - r_a}{d_a r_a}\right), \frac{1}{|\mu_k|}\log \left(\frac{\bar r_{b,k}-r_b}{d_b r_b}\right)\right\} < \chi \frac{\bar h}{\bar \epsilon_k},\quad k=1,2.
\end{equation}
We shall say the pair $\{N_1, N_2\}$ {\em the slow shadowing pair} ({\em with $\chi$ and the slow direction number $q$}) if $\{N_1, N_2\}$ satisfies the slow shadowing condition (with $q$).
We call the number $\chi$ {\em the slow shadowing ratio}.
\end{dfn}
In practical computations, we set sequences of positive numbers $\{\eta_{1,j}^{k,\pm}\}_{j=1}^u$ and $\{\eta_{2,j}^{k,\pm}\}_{j=1}^s$, $k=1,2$, as identical positive numbers, which make settings in practical computations simple.
The assumption (SS5) admits a sufficient condition for validations in terms of cones. 
We revisit the condition in the end of Section \ref{section-m-cones}.

\bigskip
The slow shadowing condition can be generalized to a sequence of fast-saddle-type blocks as follows.
\begin{dfn}[Slow shadowing sequence]\rm
\label{dfn-shadow-2}
Consider a finite sequence of fast-saddle type blocks $\{N_i\}_{i=0}^m$. We shall say the sequence $\{N_i, \chi_i\}_{i=0}^m$ of blocks and positive numbers satisfies the {\em slow shadowing condition (with $\{\chi_i\}_{i=0}^m$ and $q\in \{\pm 1\}$)} if, for $i=0,\cdots, m-1$ and $\bar h_i$ satisfying (SS4), each pair $\{N_i, N_{i+1}\}$ is a slow shadowing pair with an identical slow direction number $q$ in Definition \ref{dfn-shadow-1}, where $\chi_i$ is the slow shadowing ratio for $\{N_i, N_{i+1}\}$. 
We shall say such a sequence $\{N_i\}_{i=0}^m$ {\em the slow shadowing sequence (with $q$)}.
\end{dfn}
The slow shadowing ratio $\chi$ gives us a benefit in practical computations.
We mention this point concretely in Section \ref{section-demo-shadow}.

\bigskip
The core of the slow shadowing is that all disks transversal to stable and unstable manifolds of slow manifolds rapidly expand and contract, respectively, so that covering relation on a crossing section can be derived.
In what follows we only consider the case $q=+1$. The corresponding results to $q=-1$ can be shown in the similar manner.

\begin{prop}[Slow shadowing]
\label{prop-CE-2-1}
Let $\{N_1, N_2\}$ be a slow shadowing pair with the ratio $\chi$. 
Then there exist $h$-sets $M_1\subset (N_1)_{\bar y} = N_1\cap \{y = \bar y\}$ with $\bar y \in [y_1^-, y_1^+ - \bar h]\cap [y_2^- - \chi \bar h, y_2^+ - \bar h]$ and $M_2\subset (N_2)_{\bar y + \chi\bar h} = N_2\cap \{y = \bar y + \chi\bar h\}$ such that
\begin{equation*}
M_1 \overset{P_\epsilon^{(N_1)_{\leq \bar y + \chi\bar h}}}{\Longrightarrow} M_2.
\end{equation*}
\end{prop}

\begin{proof}
For the simplicity, we assume $\chi = 1$.
All arguments below are valid for general $\chi$.
\par
Let $S_\epsilon$ be the slow manifold validated in $N_1\cup N_2$. 
For simplicity, we write $\bar N_1 \equiv (N_1)_{\bar y + \bar h}$, $\bar N_{1,\leq} \equiv (N_1)_{\leq \bar y + \bar h}$ and $\bar P_\epsilon^1 \equiv P_\epsilon^{(N_1)_{\leq \bar y + \bar h}}$.

First note that $\bar \epsilon_k$ is an upper bound of the absolute speed in $y$-direction. 
This implies that any point in $(N_1)_{\bar y}$ which arrives at $\bar N_1$ through the orbit in $N_1$ takes at least time $\bar h/\bar \epsilon_1$. 
Also note that, $W^s(S_\epsilon)$ and $W^u(S_\epsilon)$ can be represented by families of $1$-Lipschitz stable and unstable disks, respectively, which follows from cone conditions.

Let $M_{\bar y}(\delta_1)\subset (N_1)_{\bar y}$ be the $\delta_1$-neighborhood of $W^s(S_\epsilon)\cap\{y=\bar y\}$ in $(N_1)_{\bar y}$, namely,
\begin{equation*}
M_{\bar y}(\delta_1) = \{z = (a,b,\bar y)\in N_1\mid {\rm dist}(z, W^s(S_\epsilon)) < \delta_1\}.
\end{equation*}
Also, let $M_{\bar y;\hat b}(\delta_1)$ be a section of $M_{\bar y}(\delta_1)$ at $(b,y) = (\hat b, \bar y)$, namely, 
\begin{equation*}
M_{\bar y;\hat b}(\delta_1) = \{z = (a,\hat b,\bar y)\in M_{\bar y}(\delta)\}.
\end{equation*}
Obviously, all points in $M_{\bar y;\hat b}$ is contained in the unstable cone $C^u(z_0(\hat b,\bar y))$, where $z_0(\hat b,\bar y)$ is the unique point in $W^s(S_\epsilon)\cap\{(b,y)=(\hat b,\bar y)\}$.
The local positive invariance of the unstable cone implies that $z_1(t)\in C^u(z_0(t;\hat b,\bar y))$, where $\{z_0(t;\hat b,\bar y)\}$ is the solution orbit with $z_0(0;\hat b,\bar y) = z_0(\hat b,\bar y)$ and $z_1(t)$ is the solution orbit with $z_1(0) \in M_{\bar y;\hat b}(\delta_1)$. This relationship holds for all $t\geq 0$ until $z_1(t)$ arrives at $\partial N_1$.
Lemma \ref{lem-ratio}-(1) then yields
\begin{equation*}
|\pi_a(z_1(t) - z_0(t;\hat b,\bar y)) | \geq e^{\lambda_1 t} |\pi_a(z_1(0) - z_0(0;\hat b,\bar y))|,\quad t\geq 0.
\end{equation*}
The slow shadowing condition implies that the unstable disk $M_{\bar y;\hat b}(\delta_1)$ expands through the flow, and all points on the boundary arrive at $N_1^{f,-}$ before the time $\bar T = \bar h / \bar \epsilon_1$ if $\delta_1 \geq d_a r_a$. 
In this case, $P_\epsilon^{N_1}(M_{\bar y}(\delta_1))$ has an intersection with $N_1^{f,-}\cap \{a = \hat a\}$ for all $\hat a \in \partial B_u$. 
This observation holds for arbitrary $\hat b\in \overline{B_s}$. 

The preimage $M_1\equiv (\bar P_\epsilon^1)^{-1}(\bar P_\epsilon^1 (M_{\bar y}(\delta_1))\cap \bar N_{1,\leq}^{s,-})\cap (N_1)_{\bar y}$ thus satisfies
\begin{equation*}
\sup_{z\in M_1} \dist (\pi_a(z),\pi_a W^s(S_\epsilon))< d_a r_a.
\end{equation*}
Note that $\bar N_{1,\leq}^{s,-}$ is equal to $\bar N_1 \equiv (N_1)_{\bar y + \bar h}$.

\begin{lem}
\label{lem-SS-h-set}
The set $M_1$ is an $h$-set.
\end{lem}
\begin{proof}
We may assume that $N_1 = \overline{B_u}\times \overline{B_s}\times [0,1]$ via a homeomorphism $c_{N_1}$. 
$M_1$ is contained in $M_{\bar y}(\delta)$ for some $\delta_1\in (0,d_a r_a)$. 
For each $\hat b \in \overline{B_s}$, consider the section $M_{\bar y;\hat b}(\delta_1) = \{(a,\hat b, \bar y)\mid a\in\overline{B_u}\}$. 
Let $a_{\hat b,\bar y}$ be such that $(a_{\hat b,\bar y}, \hat b, \bar y) = z_0(\hat b,\bar y)$, which is uniquely determined for each $\hat b \in \overline{B_s}$. 
Note that, for each $\hat b \in \overline{B_s}$,  $(a_{\hat b,\bar y}, \hat b, \bar y) = z_0(\hat b,\bar y)$ is contained in $M_1$.
Since $\bar P_\epsilon^1\mid_{(N_1)_{\bar y}} : (N_1)_{\bar y}\to \bar P_\epsilon^1((N_1)_{\bar y}) \subset \partial \bar N_{1,\leq}$ is homeomorphic, $g(x,y,\epsilon)>0$ holds in $N_1$ and
\begin{align*}
&\bar P_\epsilon^1(\hat a,\hat b, \bar y) = (\hat a,\hat b, \bar y)\in \bar N_{1,\leq}^{f,-}\quad \text{ for }\hat a\in \partial B_u,\\
&\bar P_\epsilon^1(a_{\hat b,\bar y},\hat b, \bar y) \in \bar N_1,
\end{align*}
then for all $\theta\in \partial B_u \cong S^{u-1}$, there exists $a_{\hat b,\bar y}(\theta)\in B_u$ 
such that $\bar P_\epsilon^1(a_{\hat b,\bar y}(\theta), \hat b, \bar y)\in \bar N_1\cap \bar N_{1,\leq}^{f,-}\cap \{a=\theta\}$, which is uniquely determined by the property of flows. 
Notice that $|\pi_a(a_{\hat b,\bar y}(\theta), \hat b, \bar y) - \pi_a(a_{\hat b,\bar y}, \hat b, \bar y)|$ monotonically increases along the flow for all $\theta \in \partial B_u$.
Since $\bar P_\epsilon^1$ is continuous, then $a_{\hat b,\bar y}(\theta)$ depends continuously on $\hat b \in \overline{B_s}$ and $\theta\in S^{u-1}$. 
As a result, we have a continuous graph $\psi_{1,\bar y} : \overline{B_s}\times S^{u-1}\to \overline{B_u}$ given by $\psi_{1,\bar y}(b,\theta) = a_{b,\bar y}(\theta)$. 
Then $M_1$ is given by
\begin{equation}
\label{M1-explicit}
M_1 = \{(a,b,\bar y)\mid a = (1-\lambda) a_{b,\bar y} + \lambda \psi_{1,\bar y}(b,\theta),\ b\in \overline{B_s},\ \theta\in S^{u-1}, \lambda\in [0,1]\},
\end{equation}
which is an $h$-set.
\end{proof}

We go back to the proof of Proposition \ref{prop-CE-2-1}.
The next interest is 
\begin{equation*}
\delta_2 = \sup_{z\in \bar P_\epsilon^1(M_1)}\dist(\pi_b(z), \pi_b(W^u(S_\epsilon))).
\end{equation*}
Consider the behavior of sections $\bar M_{1;\hat a} := \{z = (\hat a, b, \bar y + \bar h)\in \bar N_1\mid b\in \overline{B_s} \}\cap \bar P_\epsilon^1(M_1)$ in backward flow. 
Since each $\bar M_{1;\hat a}$ is a stable disk with Lipschitz constant less than $1$, then each point $z_1\in \bar M_{1;\hat a}$ is contained in $C^s(z_0(\hat a, \bar y + \bar h))$, where $z_0(\hat a, \bar y + \bar h)$ is the unique point in $W^u(S_\epsilon)\cap \bar M_{1;\hat a}$. 
Lemma \ref{lem-ratio}-(2) yields
\begin{equation*}
|\pi_b(z_1(t) - z_0(t;\hat a,\bar y + \bar h))| \geq e^{\mu_1 t}|\pi_b(z_1(0) - z_0(0;\hat a,\bar y + \bar h))|\quad \text{ for }t\leq 0.
\end{equation*}
The slow shadowing condition implies that the stable disk $\bar M_{1;\hat a}$ expands in $b$-direction through the backward flow, and its boundary arrives at $N_1^{f,+}$ before the time $\bar T = \bar h / \bar \epsilon_1$ if $\delta_2 \geq dr_b$. 
Note that $M_1\subset (N_1)_{\bar y}$ is an $h$-set given by (\ref{M1-explicit}), which implies that there exists some $t \in [-\bar T,0)$ such that the image $\varphi_\epsilon(t, \bar M_{1,\hat a})$ must have intersections with $N_1\cap \{b= \bar b\}$ for all $\bar b\in \partial B_s$.
This observation holds for arbitrary $\hat a$ with $(\hat a,b,\bar y + \bar h)\in N_1$. 
Therefore, $\delta_2$ should be less than $d_b r_b$.

\bigskip
The same arguments enable us to construct an $h$-set $M_2\subset (N_2)_{\bar y+\bar h}$ in the same way as $M_1$ under assumptions of (\ref{shadow}). 
In this case, our assumptions and the definition of slow shadowing imply $\sup_{z\in M_2}\dist(\pi_a(z), \pi_a W^s(S_\epsilon)) < d_a r_a$.

Boundaries of $M_1$ and $M_2$ are given by
\begin{align*}
M_1 &= \{(a,b,\bar y)\mid a = (1-\lambda) a_{b,\bar y} + \lambda \psi_{1,\bar y}(b,\theta),\ b\in \overline{B_s},\ \theta\in S^{u-1}, \lambda\in [0,1]\},\\
M_2 &= \{(a,b,\bar y + \bar h)\mid a = (1-\lambda) a_{b,\bar y + \bar h} + \lambda \psi_{2,\bar y + \bar h}(b,\theta),\ b\in \overline{B_s},\ \theta\in S^{u-1}, \lambda\in [0,1]\},\\
M_1^- &= \{(a,b,\bar y)\mid a = \psi_{1,\bar y}(b,\theta),\ b\in \overline{B_s}\},\quad M_1^+ = M_1\cap N_1^{f,+},\\
M_2^- &= \{(a,b,\bar y + \bar h)\mid a = \psi_{2,\bar y + \bar h}(b),\ b\in \overline{B_s}\},\quad M_2^+ = M_2\cap N_2^{f,+},
\end{align*}
where $\psi_{2,\bar y + \bar h} : \overline{B_s}\times S^{u-1}\to \overline{B_u}$ is the map associated with $M_2$, which is constructed in the same way as $\psi_{1,\bar y}$.
\par
\bigskip
Now we check if all assumptions in Proposition \ref{prop-find-CR} hold with $f=\bar P_\epsilon^1$.
\par
The estimate $\sup_{z\in M_2}\dist (\pi_a(z),\pi_aW^s(S_\epsilon)) < d_a r_a$ implies 
$M_2\subset D_2^s\cap (N_2)_{\bar y + \bar h}$. 
Similarly, the estimate $\sup_{z\in \bar P_\epsilon^1(M_1)}\dist (\pi_b(z),\pi_bW^u(S_\epsilon)) < d_b r_b$ implies
$\bar P_\epsilon^1(M_1) \subset (T_{x,12}\times I_1)D_1^u\cap (N_1)_{\bar y + \bar h}$. 
From our constructions of $M_1$ and $M_2$ as well as Lemma \ref{lem-SS-h-set}, $\bar P_\epsilon^1(M_1)$ and $M_2$ can be regarded as families of horizontal and vertical disks in $D_1^u$ and $D_2^s$, respectively.
In particular, from (SS5), $M_2 \cap (D_1^u)^{f,-}=\emptyset$, where $(D_1^u)^{f,-} = D_1^u \cap N_1^{f,-}$.
This disjointness yields $\bar P_\epsilon^1(M_1^-)\cap M_2 = \emptyset$.
Similarly, $\bar P_\epsilon^1(M_1) \cap (D_2^s)^{f,+}=\emptyset$, which yields $\bar P_\epsilon^1(M_1)\cap M_2^+ = \emptyset$, where $(D_2^s)^{f,+} = D_2^s \cap N_2^{f,+}$. 
The rest of assumptions obviously holds if we choose $q_0 \in \overline{B_s}$ so that $(a,q_0,\bar y)\in S_\epsilon$ for some $a\in \overline{B_u}$. The property of degree obviously holds since $\bar P_\epsilon : (N_1)_{\bar y} \to (N_1)_{\bar y + \bar h}$ is just a composite of uniformly contracting and expanding maps in corresponding directions.
\end{proof}

Proposition \ref{prop-CE-2-1} can be generalized to a slow shadowing sequence $\{N_i\}_{i=1}^m$, which is straightforward. 

\bigskip
The same arguments as Proposition \ref{prop-CE} yield the following result: the {\em covering-exchange : drop}.
\begin{prop}[Covering-Exchange : Drop]
\label{prop-CE-2-2}
Let $\{N_1, N_2\}$ be a slow shadowing pair with the ratio $\chi$. 
Assume that there is an $h$-set $N$ such that $N \overset{\varphi_\epsilon(T,\cdot)}{\Longrightarrow} (N_1)_{\leq \bar y}$
holds for some $T>0$, where $\bar y \in (y_1^-, y_1^+ - \bar h]\cap [y_2^- - \chi \bar h, y_2^+ - \bar h]$.
Then there are $h$-sets $\tilde M\subset (N_1)_{\leq \bar y}$ and $M_2\subset (N_2)_{\bar y + \chi \bar h}$ such that
\begin{equation*}
N \overset{\varphi_\epsilon(T,\cdot)}{\Longrightarrow} \tilde M \overset{P_\epsilon^{(N_1)_{\leq \bar y + \chi \bar h}}}{\Longrightarrow}  M_2.
\end{equation*}
\end{prop}

\begin{proof}
As in the proof of Proposition \ref{prop-CE-2-1}, we may assume $\chi =1$. 
\par
Let $M_1\subset (N_1)_{\bar y}$ and $M_2\subset (N_2)_{\bar y + \bar h}$ be as in Proposition \ref{prop-CE-2-1} and $\tilde M := (P_\epsilon^{ (N_1)_{\leq \bar y} })^{-1}(M_1)$.
Obviously $N \overset{\varphi_\epsilon(T,\cdot)}{\Longrightarrow} \tilde M$ holds since  $N \overset{\varphi_\epsilon(T,\cdot)}{\Longrightarrow} (N_1)_{\leq \bar y}$. By the construction of $\tilde M$, the covering relation $\tilde M \overset{P_\epsilon^{(N_1)_{\leq \bar y + \bar h}}}{\Longrightarrow}  M_2$ also holds.
\end{proof}

\bigskip
We provide the slow shadowing when a fast-saddle type block admits a fast-exit face, which corresponds to the covering relation $\tilde M  \overset{P_\epsilon^M}{\Longrightarrow} M^{{\rm exit}}$ in Proposition \ref{prop-CE}: the {\em covering-exchange : jump}.
We restrict the unstable dimension $u$ to $1$ in the current considerations.

\begin{prop}[Covering-Exchange : Jump]
\label{prop-CE-2-3}
Let $\{N_1, N_2\}$ be a slow shadowing pair with $u=1$. 
Also, let $N_2^{\exit}\subset N_2^{f,-}$ be the fast-exit face of $N_2$ and $\bar y \in [y_1^-, y_1^+ - \bar h]\cap [y_2^- - \chi \bar h, y_2^+ - \bar h]$.
Assume that $\dist(N_2^{\exit}, \{y= \bar y+\chi \bar h\}) \geq \chi \bar h$.
Then there are $h$-sets $M_1\subset (N_1)_{\bar y}$ and $M_2\subset (N_2)_{\bar y + \chi \bar h}$ such that
\begin{equation*}
M_1 \overset{P_\epsilon^{(N_1)_{\leq \bar y + \chi \bar h}}}{\Longrightarrow} M_2 \overset{P_\epsilon^{N_2}}{\Longrightarrow}N_2^{\exit}.
\end{equation*}
\end{prop}

\begin{proof}
As in the proof of Proposition \ref{prop-CE-2-1}, we may assume $\chi = 1$.
The proof consists of two parts : (i) $M_2 \overset{P_\epsilon^{N_2}}{\Longrightarrow}N_2^{\exit}$ and (ii) $M_1 \overset{P_\epsilon^{(N_1)_{\leq \bar y + \bar h}}}{\Longrightarrow} M_2$. 
Part (i) is a consequence of Lemma \ref{lem-Poincare-hset} with additional property of $M_2\subset (N_2)_{\bar y + \bar h}$. The assumption $\dist(N_2^{\exit}, \{y= \bar y+\bar h\}) > \bar h$ and slow shadowing condition allow $M_2$ to satisfy
\begin{equation*}
\sup_{x\in M_2} \dist(z, W^s(S_\epsilon)\cap \{y = \bar y+\bar h\}) < d_a r_a.
\end{equation*}
Therefore, the same arguments as Proposition \ref{prop-CE-2-1} yield $M_1 \overset{P_\epsilon^{(N_1)_{\leq \bar y + \bar h}}}{\Longrightarrow} M_2$, the statement (ii). Note that $W^s(S_\epsilon) \cap M_2 = \emptyset$. See also Fig. \ref{fig-shadow}-(e).
\end{proof}

\begin{figure}[htbp]\em
\begin{minipage}{0.32\hsize}
\centering
\includegraphics[width=4.0cm]{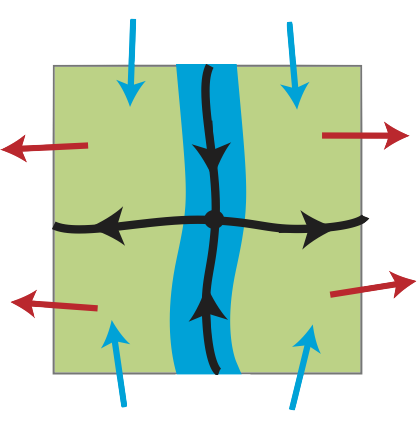}
(a)
\end{minipage}
\begin{minipage}{0.32\hsize}
\centering
\includegraphics[width=4.0cm]{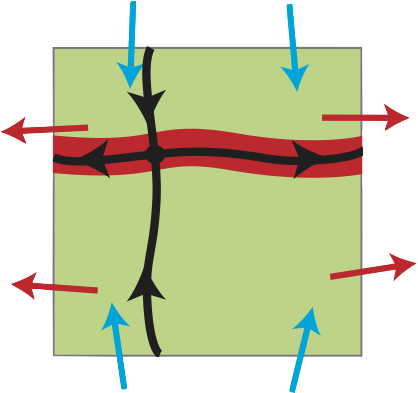}
(b)
\end{minipage}
\begin{minipage}{0.32\hsize}
\centering
\includegraphics[width=4.0cm]{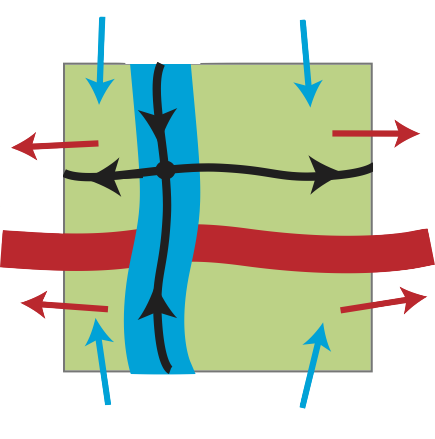}
(c)
\end{minipage}

\bigskip
\begin{minipage}{0.5\hsize}
\centering
\includegraphics[width=6.0cm]{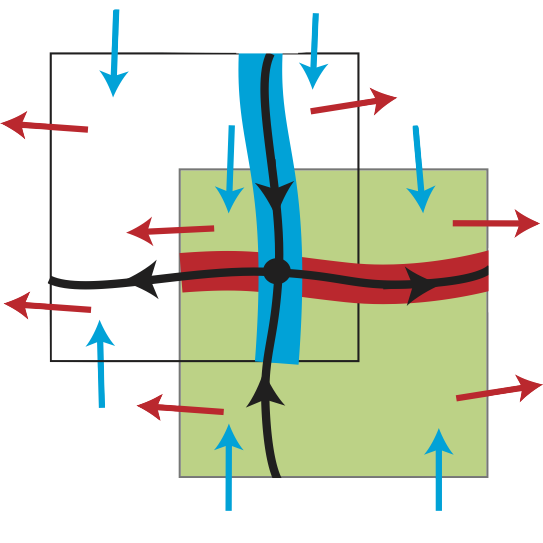}
(d)
\end{minipage}
\begin{minipage}{0.5\hsize}
\centering
\includegraphics[width=6.0cm]{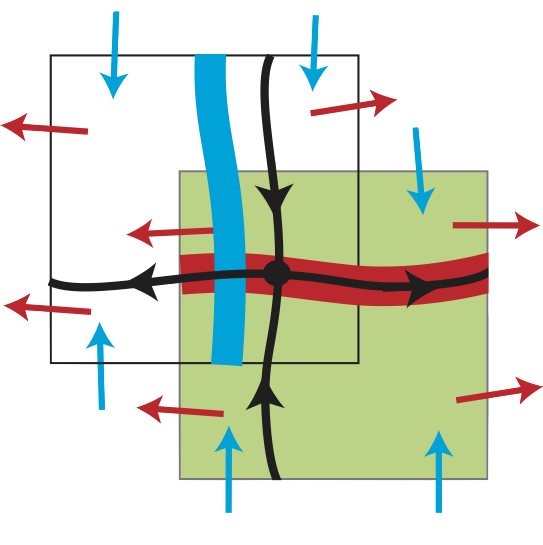}
(e)
\end{minipage}
\caption{Slow shadowing, drop and jump : schematic illustrations.}
\label{fig-shadow}
Note that all figures here show the projection onto the $(a,b)$-space.
\par
{\bf Slow shadowing} : Covering relation $M_1 \overset{P_\epsilon^{(N_1)_{\leq \bar y + \bar h}}}{\Longrightarrow} M_2$ in Proposition \ref{prop-CE-2-1} is described by the process \lq\lq (a)$\to$(b)$\to$(d)", where $P_\epsilon^{N} : N \to \partial N$ is the Poincar\'{e} map in $N$. 
A set colored by blue in (a) denotes $M_1$. 
The Poincar\'{e} map $P_\epsilon^{(N_1)_{\leq \bar y + \bar h}}$ maps $M_1$ to $P_\epsilon^{(N_1)_{\leq \bar y + \bar h}}(M_1)\subset (N_1)_{\bar y + \bar h}$ described by the red set in (b) and (d).  
The slow shadowing condition admits the choice of $M_2$ in $(N_2)_{\bar y + \bar h}$ drawn by the blue set in (d).
\par
{\bf Covering-Exchange : Drop} : Covering relation $N \overset{\varphi_\epsilon(T,\cdot)}{\Longrightarrow} \tilde M \overset{P_\epsilon^{(N_1)_{\leq \bar y + \bar h}}}{\Longrightarrow}  M_2$ in Proposition \ref{prop-CE-2-2} is described by the process \lq\lq (c)$\to$(b)$\to$(d)". 
In (c), the set colored by red denotes $\varphi_\epsilon(T,N)$ and the blue one denotes $\tilde M$. 
The Poincar\'{e} map $P_\epsilon^{(N_1)_{\leq \bar y + \bar h}}$ maps $\tilde M$ to $P_\epsilon^{(N_1)_{\leq \bar y + \bar h}}(\tilde M)\subset (N_1)_{\bar y + \bar h}$ described by the red set in (b) and (d).  
The slow shadowing condition admits the choice of $M_2$ in $(N_2)_{\bar y + \bar h}$ drawn by the blue set in (d).
\par
{\bf Covering-Exchange : Jump} : Covering relation $M_1 \overset{P_\epsilon^{(N_1)_{\leq \bar y + \bar h}}}{\Longrightarrow} M_2 \overset{P_\epsilon^{N_2}}{\Longrightarrow}N_2^{\exit}$ in Proposition \ref{prop-CE-2-3} is described by the process \lq\lq (a)$\to$(b)$\to$(e)".
A set colored by blue in (a) denotes $M_1$. 
The Poincar\'{e} map $P_\epsilon^{(N_1)_{\leq \bar y + \bar h}}$ maps $M_1$ to $P_\epsilon^{(N_1)_{\leq \bar y + \bar h}}(M_1)\subset (N_1)_{\bar y + \bar h}$ described by the red set in (b) and (e).  
The slow shadowing condition admits the choice of $M_2$ in $(N_2)_{\bar y + \bar h}$ drawn by the blue set in (e). The $h$-set $M_2$ $P_\epsilon^{N_2}$-covers $N_2^{\exit}\subset N_2^{f,-}$. The fast exit $N_2^{f,-}$ is drawn by the edge of white rectangle which admits horizontal red arrows.
\end{figure}

For the convenience and correspondence to Definition \ref{dfn-CE}, we introduce the following notion.

\begin{dfn}[Covering-exchange sequence]\rm
\label{dfn-CE-seq}
Let $N\subset \mathbb{R}^{n+1}$ be an $h$-set and $\{N_\epsilon^j\}_{j=0}^{j_M}$ be a sequence of fast-saddle-type blocks. Assume that
\begin{description}
\item[(Slow shadowing)] $\{N_\epsilon^j\}_{j=0}^{j_M-1}$ is a slow shadowing sequence with $u=1$.
\item[(Drop)] $N \overset{\varphi_\epsilon(T,\cdot)}{\Longrightarrow} (N_\epsilon^0)_{\leq \bar y}$ holds for some $T>0$, where $\bar y$ is given in Proposition \ref{prop-CE-2-2}.
\item[(Jump)] $\{N_\epsilon^{j_N-1}, N_\epsilon^{j_N}\}$ is a slow shadowing pair with a fast-exit face $N_\epsilon^{\rm exit}$ of $N_\epsilon^{j_N}$ satisfying assumptions in Proposition \ref{prop-CE-2-3}.
\end{description}
Then we call the triple $(N, \{N_\epsilon^j\}_{j=0}^{j_N}, N_\epsilon^{\rm exit})$ {\em the covering-exchange sequence}.
\end{dfn}
Obviously, the case $j_N = 1$ is nothing but the notion of covering-exchange pair.
Remark that, in the current setting, covering-exchange sequences are always assumed to be defined with $u=1$.

\bigskip
Throughout the rest of this paper, the bold-style phrases {\bf Drop} and {\bf Jump} denote the corresponding descriptions stated in Definition \ref{dfn-CE-seq}.

%
%	New Subsection
%
\subsection{$m$-cones}
\label{section-m-cones}
We have discussed  in Section \ref{section-inv-mfd} that it is systematically possible to construct fast-saddle-type blocks as well as cone conditions. 
However, such blocks are generally too small compared with validation enclosures of trajectories, if we try to validate covering-exchange sequences. 
Moreover, when we solve differential equations with a fast-exit face as initial data in this situation, 
solution orbits will hardly move in the early stage because the vector field is close to zero. 
This phenomenon causes accumulation of computation errors (e.g. wrapping effect) and extra computation costs (e.g. memory or time). 
In particular, there is little hope to validate covering-exchange sequences.
Such difficulties can be avoided if we find a large fast-saddle type block directly. 
A direct approach would be finding crossing sections which form boundaries of a large fast-saddle type block. 
It is not realistic to find such sections via interval arithmetics because vector fields are nonlinear and we have to consider the effect of slow dynamics. 
In many cases, direct search of blocks would be based on trial and error, which is not systematic. 
Our aim in this subsection is to provide an appropriate method to overcome difficulties with respect to solving differential equations.

\bigskip
In Section \ref{section-inv-mfd}, we have constructed cones of the form $C^u = \{|a-a_0| > |\zeta - \zeta_0|\}$ and $C^s = \{|b-b_0| > |\nu - \nu_0|\}$ with a vertex $(a, b, y) = (a_0, b_0, y_0)$, where $\zeta = (b, y)$ and $\nu = (a,y)$. 
See Fig. \ref{fig-cone}-(a) for the illustration of cones and the unstable manifold of a saddle fixed point. 
One expects that, in a neighborhood of $C^u$, for example, $C^u$ can be extended to $\tilde C$ like Fig. \ref{fig-cone}-(b).
 
On the other hand, if the unstable manifold is sufficiently regular, one also expects that such an extended cone can be sharper. 
More precisely, $C^u$ can be extended to the union of $C^u$ and a collection of {\em $m$-cones} $C_m^u = \{|a-a_0| > m|\zeta - \zeta_0|\}$ for some $m > 1$ keeping isolation, which is shown in Fig. \ref{fig-cone}-(c). The same expectation is valid for $C^s$.

Proposition \ref{prop-cone} means that the difference of two solutions in a fast-saddle-type block $\hat N$ along dynamics 
is restricted by the moving cone $C^u = \{(a,b,y,\eta)\mid  M(t) > 0\}$. 
The analogue of this argument in the case of sharper cones is derived below.

\begin{ass}
\label{ass-m-cone-unstable}
Consider (\ref{abstract-form}). 
Let $N\subset \mathbb{R}^{n+1}$ be an $h$-set, 
$z = (x,y,\epsilon)$ and fix $m > 1$. 

Define $\sigma_{\mathbb{A}_1^u}^m = \sigma_{\mathbb{A}_1^u}^m(z)$, $\sigma_{\mathbb{A}_2^u}^m = \sigma_{\mathbb{A}_2^u}^m(z)$, $\sigma_{\mathbb{B}_1^u}^m = \sigma_{\mathbb{B}_1^u}^m(z)$, $\sigma_{\mathbb{B}_2^u}^m = \sigma_{\mathbb{B}_2^u}^m(z)$, $\sigma_{g_1^u}^m = \sigma_{g_1^u}^m(z)$ and $\sigma_{g_2^u}^m = \sigma_{g_2^u}^m(z)$ be maximal singular values of the following matrices at $z$, respectively: 
\begin{align*}
\sigma_{\mathbb{A}_1^u}^m: \ & \mathbb{A}_1(z) = \left(  \frac{\partial F_1}{\partial a}(z) \right) \quad \text{: $u\times u$-matrix},\\
\sigma_{\mathbb{A}_2^u}^m: \ & \mathbb{A}_2(z) =  m^{-1}\left( \frac{\partial F_1}{\partial b}(z) \quad \frac{\partial F_1}{\partial y}(z) \quad \frac{\partial F_1}{\partial \eta}(z)\right)\quad \text{: $u\times (s+1+1)$-matrix},\\
\sigma_{\mathbb{B}_1^u}^m: \ & \mathbb{B}_1(z) = m \left( \frac{\partial F_2}{\partial a}(z)\right) \quad \text{: $s\times u$-matrix},\\
\sigma_{\mathbb{B}_2^u}^m: \ & \mathbb{B}_2(z) =  \left( \frac{\partial F_2}{\partial b}(z) \quad \frac{\partial F_2}{\partial y}(z) \quad \frac{\partial F_2}{\partial \eta}(z)\right)\quad \text{: $s\times (s+1+1)$-matrix},\\
\sigma_{g_1^u}^m: \ & g_1(z) = m \left(  \frac{\partial g}{\partial a}(z)\right) \quad \text{: $1\times u$-matrix},\\
\sigma_{g_2^u}^m: \ & g_2(z) = \left(\frac{\partial g}{\partial b}(z) \quad \frac{\partial g}{\partial y}(z) \quad \frac{\partial g}{\partial \eta}(z)\right)\quad \text{: $1\times (s+1+1)$-matrix}.
\end{align*}
Assume that the following inequalities hold:
\begin{align}
\label{ineq-m-graph-unstable}
&\lambda_A - \left( \sup \sigma_{\mathbb{A}_1^u}^m + \sup \sigma_{\mathbb{A}_2^u}^m \right) > 0,\\
\label{ineq-m-cone-unstable}
&\lambda_A + |\mu_B| -  \left\{ \sup \sigma_{\mathbb{A}_1^u}^m + \sup \sigma_{\mathbb{A}_2^u}^m + \sup \sigma_{\mathbb{B}_1^u}^m + \sup \sigma_{\mathbb{B}_2^u}^m + \sigma \left( \sup \sigma_{g_1^u}^m + \sup \sigma_{g_2^u}^m\right) \right\} >0,
\end{align}
where $\lambda_A$ and $\mu_B$ are real numbers satisfying (\ref{bound-unst-ev}) and (\ref{bound-st-ev}), respectively,
and the notation \lq\lq $\ \sup$" means the supremum on $N\times [0,\epsilon_0]$.
\end{ass}

\begin{prop}[Unstable $m$-cone]
\label{prop-unst-m-cone}
Consider (\ref{abstract-form}). 
Let $N\subset \mathbb{R}^{n+1}$ be an $h$-set 
 and fix $m > 1$. 
Assume that inequalities (\ref{ineq-m-graph-unstable}) and (\ref{ineq-m-cone-unstable}) in Assumption \ref{ass-m-cone-unstable} hold. 
 Then, letting a function $M^{u,m}(t):= |\Delta a(t)|^2 - m^2|\Delta \zeta(t)|^2$, ${M^{u,m}}'(t) > 0$ holds for all points in $N$ satisfying $M^{u,m}(t) = 0$ unless $\Delta a = 0$, where $\zeta = (b,y,\eta)$.
\end{prop}

\begin{proof}
Do the same arguments as the proof of Proposition \ref{prop-cone}, replacing $M(t)$ by $M^{u,m}(t)$.
\end{proof}

\begin{ass}
\label{ass-m-cone-stable}
Consider (\ref{abstract-form}). Let $N\subset \mathbb{R}^{n+1}$ be an $h$-set, 
$z = (x,y,\epsilon)$ and fix $m > 1$. 

Define $\sigma_{\mathbb{A}_1^s}^m = \sigma_{\mathbb{A}_1^s}^m(z)$, $\sigma_{\mathbb{A}_2^s}^m = \sigma_{\mathbb{A}_2^s}^m(z)$, $\sigma_{\mathbb{B}_1^s}^m = \sigma_{\mathbb{B}_1^s}^m(z)$, $\sigma_{\mathbb{B}_2^s}^m = \sigma_{\mathbb{B}_2^s}^m(z)$, $\sigma_{g_1^s}^m = \sigma_{g_1^s}^m(z)$ and $\sigma_{g_1^s}^m = \sigma_{g_1^s}^m(z)$ be maximal singular values of the following matrices at $z$, respectively: 
\begin{align*}
\sigma_{\mathbb{A}_1^s}^m: \ & \mathbb{A}_1(z) = m\left(  \frac{\partial F_1}{\partial b}(z) \right) \quad \text{: $u\times s$-matrix},\\
\sigma_{\mathbb{A}_2^s}^m: \ & \mathbb{A}_2(z) =  \left( \frac{\partial F_1}{\partial a}(z) \quad \frac{\partial F_1}{\partial y}(z) \quad \frac{\partial F_1}{\partial \eta}(z)\right)\quad \text{: $u\times (u+1+1)$-matrix},\\
\sigma_{\mathbb{B}_1^s}^m: \ & \mathbb{B}_1(z) = \left( \frac{\partial F_2}{\partial b}(z)\right) \quad \text{: $s\times s$-matrix},\\
\sigma_{\mathbb{B}_2^s}^m: \ & \mathbb{B}_2(z) =  m^{-1}\left( \frac{\partial F_2}{\partial a}(z) \quad \frac{\partial F_2}{\partial y}(z) \quad \frac{\partial F_2}{\partial \eta}(z)\right)\quad \text{: $s\times (u+1+1)$-matrix},\\
\sigma_{g_1^s}^m: \ & g_1(z) = m\left(  \frac{\partial g}{\partial b}(z)\right) \quad \text{: $1\times s$-matrix},\\
\sigma_{g_2^s}^m: \ & g_2(z) = \left(\frac{\partial g}{\partial a}(z) \quad \frac{\partial g}{\partial y}(z) \quad \frac{\partial g}{\partial \eta}(z)\right)\quad \text{: $1\times (u+1+1)$-matrix}.
\end{align*}
Assume that the following inequalities hold:
\begin{align}
\label{ineq-m-graph-stable}
&|\mu_B| - \left( \sup \sigma_{\mathbb{B}_1^s}^m + \sup \sigma_{\mathbb{B}_2^s}^m \right) > 0,\\
\label{ineq-m-cone-stable}
&\lambda_A + |\mu_B| -  \left( \sup \sigma_{\mathbb{A}_1^s}^m + \sup \sigma_{\mathbb{A}_2^s}^m + \sup \sigma_{\mathbb{B}_1^s}^m + \sup \sigma_{\mathbb{B}_2^s}^m + \sup \sigma_{g_1^s}^m + \sup \sigma_{g_2^s}^m \right) >0,
\end{align}
where $\lambda_A$ and $\mu_B$ are real numbers satisfying (\ref{bound-unst-ev}) and (\ref{bound-st-ev}), respectively,
and the notation \lq\lq $\ \sup$" means the supremum on $N\times [0,\epsilon_0]$.
\end{ass}

\begin{prop}[Stable $m$-cone]
\label{prop-st-m-cone}
Consider (\ref{abstract-form}). Let $N\in \mathbb{R}^{n+1}$ be an $h$-set 
  and fix $m > 1$.  Assume that inequalities (\ref{ineq-m-graph-stable}) and (\ref{ineq-m-cone-stable}) in Assumption \ref{ass-m-cone-unstable} hold. Then, defining a function $M^{s,m}(t):= |\Delta b(t)|^2 - m^2|\Delta \nu(t)|^2$, ${M^{s,m}}'(t) < 0$ holds for all points on $N$ satisfying $M^{s,m}(t) = 0$ unless $\Delta b = 0$, where $\nu = (a,y,\eta)$.
\end{prop}

\begin{proof}
Do the same arguments as the proof of Proposition \ref{prop-cone}, replacing $M(t)$ by $M^{s,m}(t)$.
\end{proof}

\begin{dfn}[$m$-cone conditions]\rm
We shall call inequalities (\ref{ineq-m-graph-unstable}) and (\ref{ineq-m-cone-unstable}) in Assumption \ref{ass-m-cone-unstable} the {\em unstable $m$-cone condition in $N$}.
Similarly, we shall call inequalities (\ref{ineq-m-graph-stable}) and (\ref{ineq-m-cone-stable}) in Assumption \ref{ass-m-cone-stable} the {\em stable $m$-cone condition in $N$}. 
When these conditions are satisfied, {\em the unstable $m$-cone and the stable $m$-cone with the vertex $z=(a_0, b_0, y_0)$} (in the $(a,b,y)$-coordinate) are given as follows, respectively:
\begin{align*}
C_m^u(z) := \{(a,b,y)\mid |a - a_0|^2 \geq  m^2(|b - b_0|^2 +  |y - y_0|^2)\},\\
C_m^s(z) := \{(a,b,y)\mid |b - b_0|^2 \geq  m^2(|a - a_0|^2 +  |y - y_0|^2)\}. 
\end{align*}
\end{dfn}

Validations of $m$-cones themselves are in fact independent of the construction of fast-saddle-type blocks discussed in Section \ref{section-isolatingblock}. Moreover, the choice of $m$ can be arbitrary as long as corresponding cone conditions hold.

\bigskip
An implementation of $m$-cones in unstable direction is the following.
\begin{enumerate}
\item Prepare a fast-saddle-type block $N$ such that $N_c$ is given by (\ref{fast-block-2}) satisfying the unstable cone condition. 
Via a homeomorphism $c_N$ we may assume that $N$ is represented by
\begin{equation*}
N = \prod_{j=1}^u [a_j^-, a_j^+]\times \prod_{j=1}^s [b_j^-, b_j^+] \times [0,1].
\end{equation*}

\item Choose a fast-exit face $N^\exit$. For example, set 
\begin{equation*}
N^\exit = \prod_{j=1}^{j_0-1} [a_j^-, a_j^+]\times \{a_{j_0}^\ast\}\times \prod_{j=j_0-1}^u [a_j^-, a_j^+]\times \prod_{j=1}^s [b_j^-, b_j^+] \times [y^-, y^+],\quad \ast \in \{\pm\},
\end{equation*}
where $[y^-, y^+]\subset [0,1]$. Let $\ell > 0$ is a given number and $V_m^{u,j_0}$ be an $h$-set given by
\begin{align}
\label{set-verify-unstable-m-cone}
V_m^{u,j_0} &= \prod_{j=1}^{u} [\tilde a_j^-, \tilde a_j^+] \times \prod_{j=1}^s \left[b_j^- - \frac{\ell}{m}, b_j^+ + \frac{\ell}{m} \right] \times \left[y^- - \frac{\ell}{m}, y^+ + \frac{\ell}{m} \right],\\
\notag
[\tilde a_j^-, \tilde a_j^+] &= \begin{cases}
[a_j^-, a_j^+ + \ell]  & \text{ if $j=j_0$ and $\ast = +$,}\\
[a_j^- - \ell, a_j^+] & \text{ if $j=j_0$ and $\ast = -$,}\\
[a_j^-, a_j^+] & \text{ otherwize.}
\end{cases}
\end{align}
See Fig. \ref{fig-cone}-(c).
\item Verify the unstable $m$-cone condition, (\ref{ineq-m-graph-unstable}) and (\ref{ineq-m-cone-unstable}), in $V_m^{u,j_0}$.
\end{enumerate}

The following lemma is a consequence of discussions in Lemma \ref{prop-unst-m-cone} and arguments in Theorem \ref{thm-inv-mfd-rigorous}.
\begin{lem}
\label{lem-exit-cone}
Let $N$ be a fast-saddle type block satisfying the stable cone condition, and $V_m^{u,j_0}$ be given by (\ref{set-verify-unstable-m-cone}). Assume that the $m$-unstable cone condition is satisfied in $V_m^{u,j_0}$. 
Then any points on $N^\exit$ leaves the set
\begin{equation*}
N\cup C_m^{u,j_0},\quad C_m^{u,j_0}:= \{(a,\zeta)=(a,b,y) \in V_m^{u,j_0} \mid |a - a_0| \geq m|\zeta - \zeta_0|, (a_0, \zeta_0)\in N^\exit\}
\end{equation*}
under the flow through
$(C_m^{u,j_0})^\ast:= C_m^{u,j_0}\cap \{a= \tilde a_{j_0}^\ast\},\ \ast\in \{\pm\}$. 
The sign $\ast$ is exactly the location of $N^\exit$.
\end{lem}

\bigskip
$m$-cones in stable direction is constructed in the similar manner.
\begin{enumerate}
\item Prepare a fast-saddle-type block $N$ such that $N_c$ is given by (\ref{fast-block-2}) satisfying the stable cone condition. Via a homeomorphism $c_N$ we may assume that $N$ is represented by
\begin{equation*}
N = \prod_{j=1}^u [a_j^-, a_j^+]\times \prod_{j=1}^s [b_j^-, b_j^+] \times [0,1].
\end{equation*}

\item Choose a face of the fast-entrance $N^\ent$. As an example, set 
\begin{equation*}
N^\ent =  \prod_{j=1}^u [a_j^-, a_j^+] \times \prod_{j=1}^{j_0-1} [b_j^-, b_j^+]\times \{b_{j_0}^\ast\}\times \prod_{j=j_0-1}^s [b_j^-, b_j^+]\times [y^-, y^+],\quad \ast\in \{\pm\},
\end{equation*}
where $[y^-,y^+]\subset [0,1]$.
Let $\ell > 0$ is a given number and $V_m^{s,j_0}$ be an $h$-set given by
\begin{align}
\label{set-verify-stable-m-cone}
V_m^{s,j_0} &= \prod_{j=1}^u \left[a_j^- - \frac{\ell}{m}, a_j^+ + \frac{\ell}{m} \right] \times \prod_{j=1}^{s} [\tilde b_j^-, \tilde b_j^+]\times \left[y^- - \frac{\ell}{m}, y^+ + \frac{\ell}{m} \right],\\
\notag
[\tilde b_j^-, \tilde b_j^+] &= \begin{cases}
[b_j^-, b_j^+ + \ell] & \text{ if $j=j_0$ and $\ast = +$,}\\
[b_j^- - \ell, b_j^+] & \text{ if $j=j_0$ and $\ast = -$,}\\
[b_j^-, b_j^+] & \text{ otherwize.}
\end{cases}
\end{align}
\item Verify the stable $m$-cone condition, (\ref{ineq-m-graph-stable}) and (\ref{ineq-m-cone-stable}), in $V_m^{s,j_0}$.
\end{enumerate}

The following lemma is a consequence of similar arguments to Lemma \ref{lem-exit-cone}.
\begin{lem}
\label{lem-ent-cone}
Let $N$ be a fast-saddle type block satisfying the stable cone condition, and $V_m^{s,j_0}$ be given by (\ref{set-verify-stable-m-cone}). Assume that the $m$-stable cone condition is satisfied in $N\cup V_m^{s,j_0}$. 
Then any points on $N^\ent$ leaves
\begin{equation*}
N\cup C_m^{s,j_0},\quad C_m^{s,j_0}:= \{(b,\nu) \in V_m^{s,j_0} \mid  \nu = (a,y), |b - b_0| \geq m|\nu - \nu_0|, (b_0, \nu_0)\in N^\ent\}
\end{equation*}
under the backward flow through
$(C_m^{s,j_0})^\ast:= C_m^{s,j_0}\cap \{b=\tilde b_{j_0}^\ast\},\ \ast \in \{\pm\}$.
The sign $\ast$ is exactly the location of $N^\ent$.
\end{lem}

\begin{dfn}[$m$-cones]\rm
We call the set of the form $C_m^{u,j_0}$ an {\em unstable $m$-cone of $N$ with the length $\ell$}. 
Similarly, we call the set of the form $C_m^{s,j_0}$ a {\em stable $m$-cone of $N$ with the length $\ell$}.
\end{dfn}

\begin{figure}[htbp]\em
\begin{minipage}{0.5\hsize}
\centering
\includegraphics[width=5.0cm]{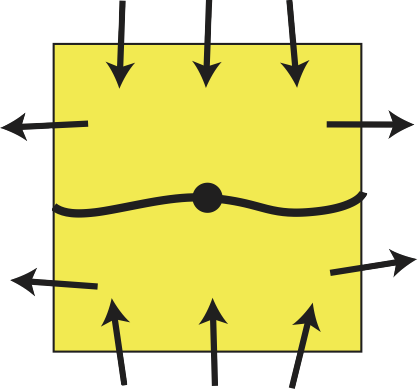}
(a)
\end{minipage}
\begin{minipage}{0.5\hsize}
\centering
\includegraphics[width=5.0cm]{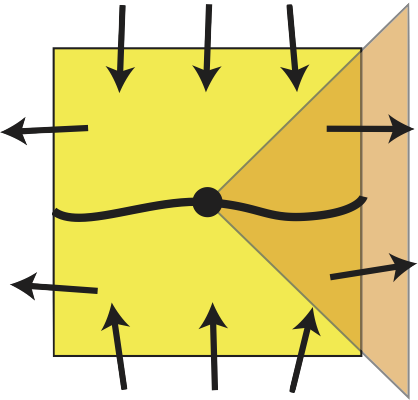}
(b)
\end{minipage}
\begin{minipage}{0.5\hsize}
\centering
\includegraphics[width=7.0cm]{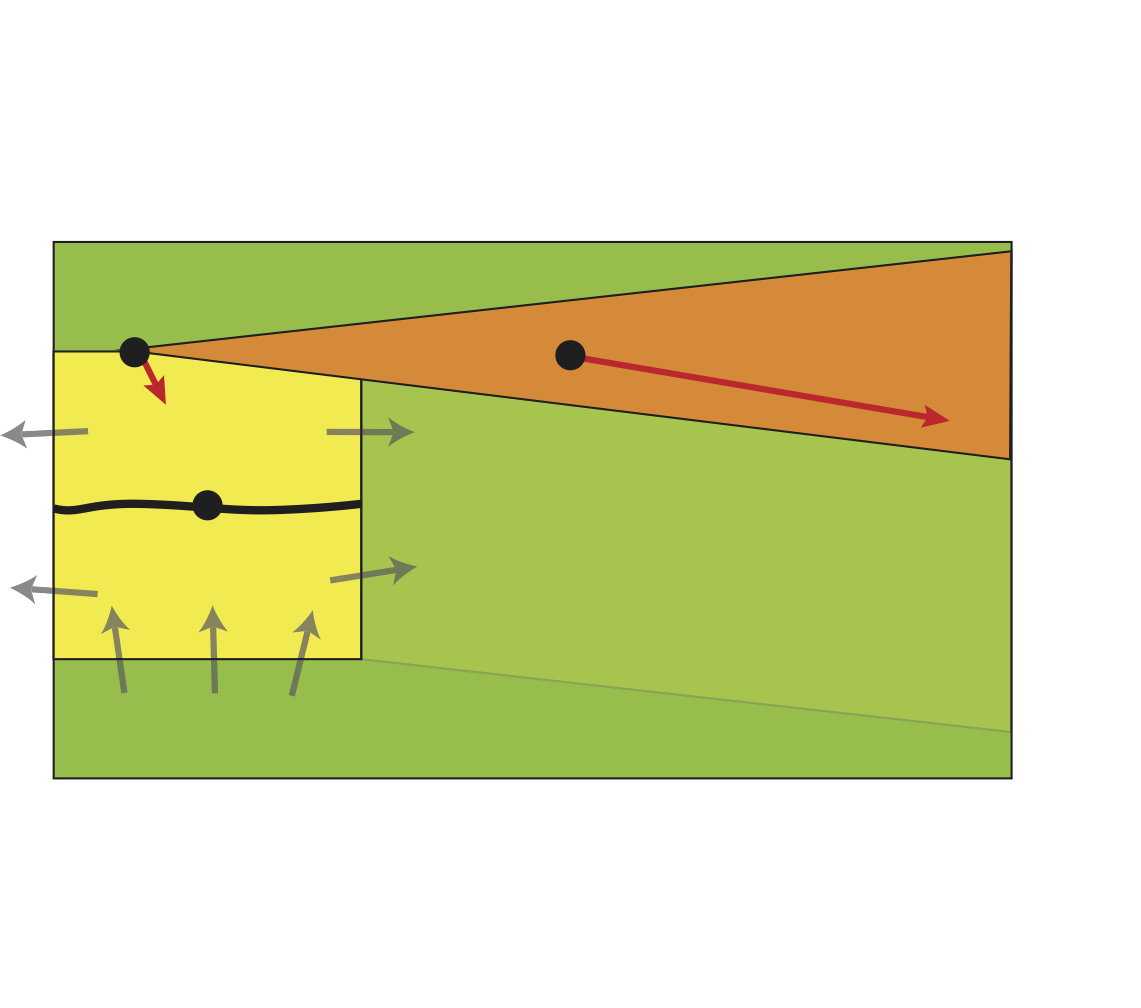}
(c)
\end{minipage}
\begin{minipage}{0.5\hsize}
\centering
\includegraphics[width=7.0cm]{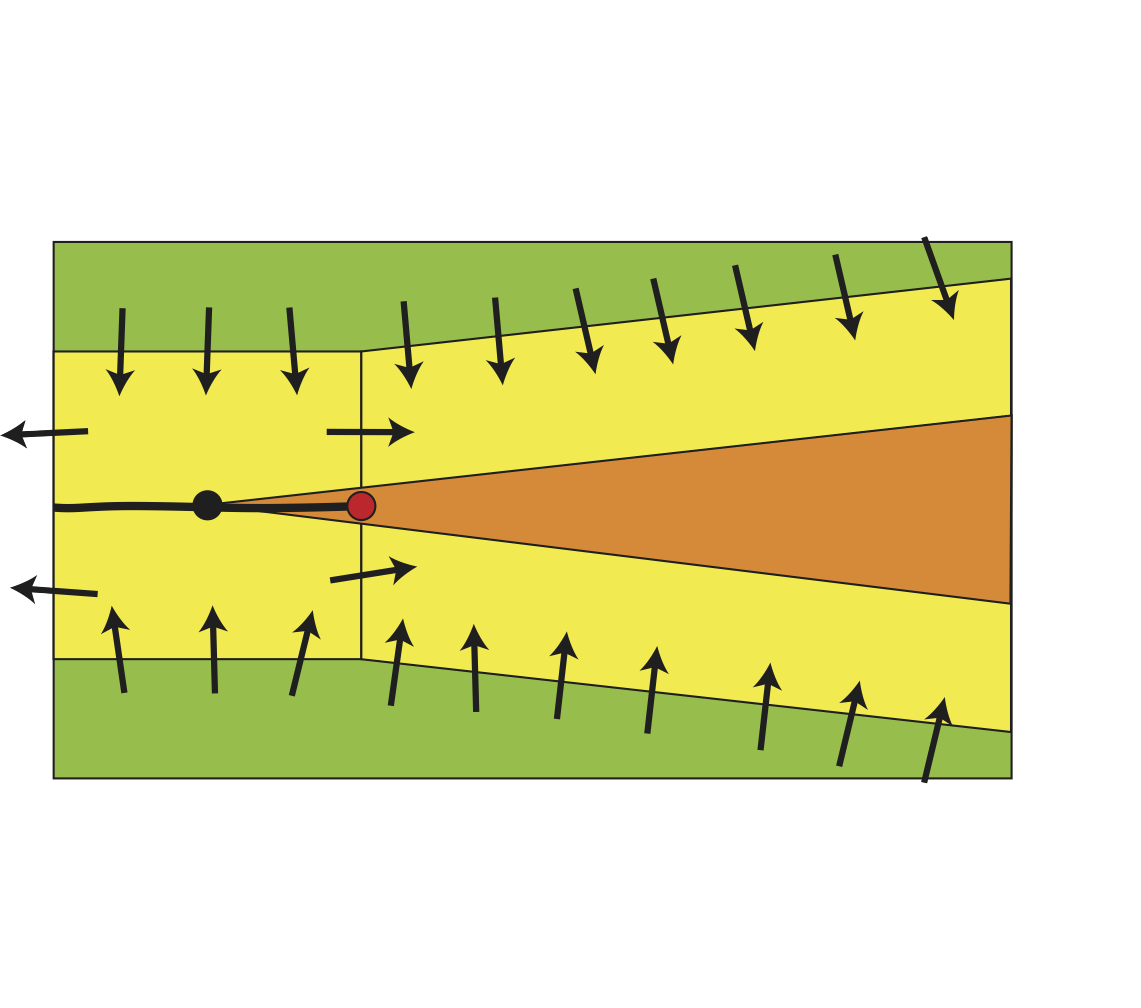}
(d)
\end{minipage}
\caption{Unstable $m$-cones.}
\label{fig-cone}
(a). A fast-saddle-type block $N$ containing the slow manifold $S_\epsilon$ validated in Theorem \ref{thm-inv-mfd-rigorous}. 
In this figure, a black ball corresponds to slow manifold $S_\epsilon$ and a black curve corresponds to its unstable manifold $W^u(S_\epsilon)$. 
The accuracy of the stable (resp. unstable) manifold $W^s(S_\epsilon)$ (resp. $W^u(S_\epsilon$) is measured by the size of the fast-entrance $N^{f,+}$ (resp. the fast-exit $N^{f,-}$). 
In general, blocks are small and the flow stay near those blocks for small $t$, which cause the accumulation of various computational errors. 

\bigskip
(b). A candidate of extended cones. One expect that our validated cones stated in Theorem \ref{thm-inv-mfd-rigorous}. can be locally extended. 
However, one can imagine that it is quite too large for the enclosure of $W^u(S_\epsilon)$ if $W^u(S_\epsilon)$ is sufficiently smooth. 
The extension of cones drawn here is thus quite coarse for smooth manifolds. 

\bigskip
(c). Validation of $m$-cone condition. 
It is sufficient to verify $m$-cone condition on a rectangular domain $V_m^u$ given by (\ref{set-verify-unstable-m-cone}). The set $V_m^u$ is colored by green. 
The orange region is the $m$-cone with a vertex in $N$. 
$m$-cone conditions imply that all points in an $m$-cone stay the $m$-cone until they leave $V_m^\ast$.

\bigskip
(d). An unstable $m$-cone $C_m^u$ of $N$, which is the union of $m$-cones with vertices on $N^\exit$. 
The union $N\cup C_m^u$ is colored by yellow.
All trajectories through $N^\exit$ leaves $V_m^u$ through $C_m^u\cap \partial V_m^u$. 
In general, $C_m^u\cap \partial V_m^u$ is far from slow manifolds. 
This fact helps us with validations with reasonable computation steps and accuracy. 
See also Section \ref{section-demo-cones}.
\end{figure}

Using $m$-cones, we can extend an $h$-set after construction of a small fast-saddle type block. 
Such an $h$-set keeps isolation in radial direction, thanks to $m$-cone conditions.
Moreover, we can obtain a priori estimates of trajectories {\em far from equilibria}, 
keeping their accuracy as well as possible by adjustments of $m$ and $\ell$. 
Of course, we do not need to solve differential equations to obtain such estimates.
This technique reduces extra computation costs mentioned at the beginning of this subsection.
Furthermore, in layer problems (\ref{layer}), $m$-cones immediately yield continuation of stable and unstable manifolds of critical invariant manifolds. Indeed, the unions $N\cup C_m^{u,j_0}$ and $N\cup C_m^{s,j_0}$ are $h$-sets satisfying $m$-(un)stable cone conditions.  The unstable (resp. stable) manifold is hence represented by a horizontal (resp. vertical) disk in $N\cup C_m^{u,j_0}$ (resp. $N\cup C_m^{s,j_0}$). See \cite{ZCov} for details. 

\begin{rem}\rm
Arguments involving Propositions \ref{prop-unst-m-cone} and \ref{prop-st-m-cone} still holds replacing $y\in \mathbb{R}$ and $\mathbb{R}$-valued function $g$ by $y\in \mathbb{R}^l$ and $\mathbb{R}^l$-valued function $g$, respectively.
\end{rem}

\bigskip
In what follows we show applications of $m$-cones to fast-slow systems. 
In the case of fast-slow systems, we have to care about movements of trajectories in slow direction. 
With the help of slow shadowing condition, we can construct a sequence of covering relations for $m$-cones.
For technical reasons we restrict $u$ and $s$ to $1$, in particular, $n=3$. 
This is exactly the case in our verification examples, Section \ref{section-examples}. 
Consider (\ref{abstract-form}) again.

\begin{dfn}[Depature time, Arrival time]\rm
Let $N$ be a fast-saddle-type block such $N_c$ is given by (\ref{fast-block-2}) with $u=s=1$.
Also, let $V_m^{u}$ be a set of the form (\ref{set-verify-unstable-m-cone}) with $u=s=1$ satisfying the unstable $m$-cone condition. 
Define the {\em departure time} $T_{\dep}$ in $V_m^u$ by
\begin{equation}
\label{departure}
T_{\dep} =T_{\dep}(V_m^{u}) := \frac{1}{2} \int_{(d_a r_a)^2}^{(\bar r_a+\ell)^2} \frac{1}{\lambda_{\min}(a)} \frac{d(a^2)}{a^2} + \delta,
\end{equation}
where $d_a$ is a given number in (SS5), $\bar r_a = \diam(\pi_a(N))$, $r_a$ is given in (\ref{setting-shadow}), $\delta > 0$ is a sufficiently small number,
\begin{equation*}
\lambda_{\min}(\tilde a) := \lambda_A - \left( \sup_{z\in V_m^u\cap \{a\leq \tilde a\}} \sigma_{\mathbb{A}_1^u}^m(z) + \sup_{z\in V_m^u\cap \{a\leq \tilde a\}} \sigma_{\mathbb{A}_2^u}^m(z) \right)
\end{equation*}
and $\lambda_A> 0$ is a real number satisfying (\ref{bound-unst-ev}).
Note that the unstable $m$-cone condition implies that $\lambda_{\min}(\tilde a) > 0$ for $\tilde a\in (d_a r_a, \bar r_a + \ell)$.
\par
Similarly, let $V^s_m$ be a set of the form (\ref{set-verify-stable-m-cone}) with $u=s=1$ satisfying the stable $m$-cone condition. Define the {\em arrival time} $T_{\arr}$ in $V_m^s$ by
\begin{equation}
\label{arrival}
T_{\arr} =T_{\arr}(V_m^s) := -\frac{1}{2}\int_{(d_b r_b)^2}^{(\bar t_b +\ell)^2} \frac{1}{\mu_{\min}(b)} \frac{d(b^2)}{b^2} + \delta,
\end{equation}
where $d_b$ is a given number in (SS5), $\bar r_b = \diam(\pi_b(N))$, $r_b$ is given in (\ref{setting-shadow}), $\delta > 0$ is a sufficiently small number,
\begin{equation*}
\mu_{\min}(\tilde b) := \mu_B + \left( \sup_{z\in V_m^s\cap \{b\leq \tilde b\}} \sigma_{\mathbb{B}_1^s}^m(z) + \sup_{z\in V_m^s\cap \{b\leq \tilde b\}} \sigma_{\mathbb{B}_2^s}^m(z) \right)
\end{equation*}
and $\mu_B < 0$ is a real number satisfying (\ref{bound-st-ev}). 
Note that the stable $m$-cone condition implies that $\mu_{\min}(\tilde b) < 0$ for $\tilde b\in (d_b r_b, \bar r_b + \ell)$.
\end{dfn}

\begin{rem}\rm
In practical computations, we use the following upper bounds of departure and arrival times:
\begin{align*}
\frac{1}{2} \int_{(d_a r_a)^2}^{(\bar r_a+\ell)^2} \frac{1}{\lambda_{\min}(a)} \frac{d(a^2)}{a^2} &\leq \frac{1}{2} \sum_{j=0}^{T-1} \lambda_{\min}\left(d_a r_a + \frac{(j+1) L_a}{T}\right)^{-1} \int_{(d_a r_a + \frac{j L_a}{T})^2}^{(d_a r_a + \frac{(j+1) L_a}{T})^2}  \frac{d(a^2)}{a^2} \\
	&= \sum_{j=0}^{T-1} \lambda_{\min}\left(d_a r_a + \frac{(j+1) L_a}{T}\right)^{-1}  \log \left( \frac{Td_a r_a + (j+1) L_a}{Td_a r_a + j L_a}\right) ,\\
-\frac{1}{2}\int_{(d_b r_b)^2}^{(\bar r_b+\ell)^2} \frac{1}{\mu_{\min}(b)} \frac{d(b^2)}{b^2} &\leq  -\frac{1}{2} \sum_{j=0}^{T-1} \mu_{\min}\left(d_b r_b + \frac{(j+1) L_b}{T}\right)^{-1} \int_{(d_b r_b + \frac{j L_b}{T})^2}^{(d_b r_b + \frac{(j+1) L_b}{T})^2}  \frac{d(b^2)}{b^2} \\
	&= -\sum_{j=0}^{T-1} \mu_{\min}\left(d_b r_b + \frac{(j+1) L_b}{T}\right)^{-1}  \log \left( \frac{Td_b r_b + (j+1) L_b}{Td_b r_b + j L_b}\right),
\end{align*}
where $L_a = \bar r_a + \ell - d_a r_a$ and $L_b = \bar r_b + \ell - d_b r_b$. Usually, the bigger the number of partitions $T$ is, the smaller right-hand sides of above inequalities are. We may set the difference between summations with different $T$'s as $\delta$ in (\ref{departure}) and (\ref{arrival}).
\end{rem}

The following results show that $m$-cones generate additional covering relations to validate global orbits for (\ref{fast-slow})$_\epsilon$.

\begin{prop}
\label{prop-cov-unstable}
Let $N$ be a fast-saddle-type block for (\ref{abstract-form}) in $\mathbb{R}^3$ with $u=s=1$ which forms
\begin{equation*}
N = [a^-, a^+] \times [b^-, b^+]\times [y_0^-, y_0^+]
\end{equation*}
and actually given by (\ref{fast-block-2}).
Also, let $V_m^u$ be a set of the form (\ref{set-verify-unstable-m-cone}) with $u=s=1$ which contains $N$. 
Assume that the unstable $m$-cone condition is satisfied in $V_m^u$. 
Let $T_{\dep}=T_{\dep}(V_m^u)$ be the departure time in $V_m^u$. Define 
\begin{align}
\notag
\epsilon^+ &:= \sup_{V_m^u\times [0,\epsilon_0]} \epsilon g(x,y,\epsilon),\quad \epsilon^- := \inf_{V_m^u\times [0,\epsilon_0]} \epsilon g(x,y,\epsilon)\\
\notag
N^{\exit} &:= \{a^\ast\}\times [b^-, b^+] \times [y^-, y^+]\quad \text{ with }[y^-,y^+]\cup [y^- + \epsilon^- T_{\dep}, y^+ + \epsilon^+ T_{\dep}]\subset [y_0^-,y_0^+],\\
\label{exit-cone}
(C_m^u)^\exit &:= C_m^u \cap \left( \{\tilde a^\ast\}  \times \left[b^- - \frac{\ell}{m}, b^+ + \frac{\ell}{m} \right] \times \left[y^- +\epsilon^+ T_{\dep}, y^+ + \epsilon^- T_{\dep} \right] \right),
\end{align}
where $\ast \in \{\pm\}$ is chosen so that either $\tilde a^- = a^- - \ell$ or $\tilde a^+ = a^+ + \ell$ holds.
Then $N^{\exit}  \overset{P_\epsilon^{C_m^u}}{\Longrightarrow} (C_m^u)^\exit$ holds. 
\end{prop}

\begin{proof}
Let $z_1=(a_1,\zeta_1) \in W^s(S_\epsilon) \subset N$ and $z_2 = (a_2,\zeta_2)\in N^{\exit}$ be such that $\zeta_1 = \zeta_2$ and that $\pi_y(z_1) = \pi_y(z_2) \in [y^-,y^+]$. 
The unstable $m$-cone condition and Lemma \ref{prop-unst-m-cone} imply that $z_2(t) \in C_m^u(z_1(t))$ for all $t\geq 0$ until $z_2(t)$ arrives at $(C_m^u)^\ast$, where $(C_m^u)^\ast$ is given in Lemma \ref{lem-exit-cone}. 
The difference $\Delta a \equiv a_2(t) - a_1(t) = \pi_a(z_2(t)) - \pi_a(z_1(t))$ satisfies
\begin{equation*}
(|\Delta a|^2)' \geq 2\lambda_{\min}(a_2)|\Delta a|^2,
\end{equation*}
which follows from the same argument as the proof of Lemma \ref{lem-ratio}. 
This implies that the solution orbit $z_2(t)$ with $z_2(0) = z_2$ arrives at $(C_m^u)^\ast$ at time $t < T_{\dep}$ and that 
\begin{align*}
\pi_y(P_\epsilon (N^{\exit}\cap \{y=y^+\})) &\subsetneq (y^+ + \epsilon^- T_{\dep}, y^+ + \epsilon^+ T_{\dep}),\\
\pi_y(P_\epsilon (N^{\exit}\cap \{y=y^-\})) &\subsetneq (y^- + \epsilon^- T_{\dep}, y^- + \epsilon^+ T_{\dep}). 
\end{align*}
Thus $P_\epsilon (N^{\exit}\cap \{y=y^\pm\}) \cap (C_m^u)^\exit = \emptyset$ holds. 
Note that $\partial (C_m^u)^\exit$ consists of
\begin{align*}
(C_m^u)^{\exit,-} &= C_m^u \cap  \left( \{\tilde a^\ast\}  \times \left[b^- - \frac{\ell}{m}, b^+ + \frac{\ell}{m} \right] \times \left\{ y^\pm +\epsilon^\mp T_{\dep}\right \} \right),\\
 (C_m^u)^{\exit,+} &= \overline{\partial C_m^u \setminus \{a=a^\ast, \tilde a^\ast\} } \cap \left( \{\tilde a^\ast\}  \times \left[b^- - \frac{\ell}{m}, b^+ + \frac{\ell}{m} \right] \times \left[y^- +\epsilon^+ T_{\dep}, y^+ + \epsilon^- T_{\dep} \right] \right).
\end{align*}
Moreover,
$(C_m^u)^{\exit,+} \cap P_\epsilon (N^{\exit}) = \emptyset$ holds by Proposition \ref{prop-unst-m-cone}. 
Then $N^{\exit}  \overset{P_\epsilon^{C^u}}{\Longrightarrow} (C_m^u)^\exit$ holds by Proposition \ref{prop-find-CR} with $q_0 \in N^{\exit}\cap W^u(S_\epsilon)$.
\end{proof}

\begin{dfn}\rm
We shall call $(C_m^u)^\exit$ given in (\ref{exit-cone}) {\em the fast-exit face of $C_m^u$}.
\end{dfn}
The following statements also hold from the same arguments under the backward flow.

\begin{prop}
\label{prop-cov-stable}
Let $N_1$ and $N_2$ be fast-saddle-type blocks for (\ref{abstract-form}) in $\mathbb{R}^3$ with $u=s=1$ which form
\begin{equation*}
N_1 = [a_1^-, a_1^+] \times [b_1^-, b_1^+]\times [y_1^-, y_1^+],\quad N_2 = [a_2^-, a_2^+] \times [b_2^-, b_2^+]\times [y_2^-, y_2^+]
\end{equation*}
such that $\{N_1, N_2\}$ is a slow shadowing pair with the slow direction number $q=+1$ and $\bar h$ given below. 
Also, let $V_m^s$ be a set of the form (\ref{set-verify-stable-m-cone}) with $u=s=1$ containing $N_1$. 
Assume that the stable $m$-cone condition is satisfied in $V_m^s$.
Let $T_{\arr}=T_{\arr}(V_m^s)$ be the arrival time in $V_m^s$. Define 
\begin{align*}
\epsilon^+ &:= \sup_{V_m^s\times [0,\epsilon_0]} \epsilon g(x,y,\epsilon),\quad \epsilon^- := \inf_{V_m^u\times [0,\epsilon_0]} \epsilon g(x,y,\epsilon),\\
N_1^{\ent} &:= [a_1^-,a_1^+]\times  \{b_1^+\} \times [y^-, \bar y].\quad  \text{ with }[y^-,\bar y]\subset [y_1^-, y_1^+].
\end{align*}
Assume that there is an $h$-set $N_0$ and $T_0>0$ such that $N_0 \overset{\varphi_\epsilon(T_0,\cdot)}{\Longrightarrow} C_m^s \cap \{y\in [y^-,\bar y]\}$ 
and that
\begin{equation}
\label{ineq-y-cov}
y^- < \inf \pi_y(\varphi_\epsilon(T_0,N)) + \epsilon^- T_{\arr},\quad \sup \pi_y(\varphi_\epsilon(T_0,N)) + \epsilon^+ T_{\arr} < \bar y.
\end{equation}
Let  $\bar h > 0$ be such that $\bar y + \bar h < y_1^+$.

Then there exist $h$-sets $\tilde M_1 \subset (N_1\cup C_m^s)_{\leq \bar y}$ and $M_2\subset (N_2)_{\bar y + \bar h}$ such that 
\begin{equation*}
N_0 \overset{\varphi_\epsilon(T_0,\cdot)}{\Longrightarrow} \tilde M_1\overset{P_\epsilon^{(N_1\cup C_m^s)_{\bar y + \bar h}}}{\Longrightarrow}  M_2.
\end{equation*}
\end{prop}

\begin{proof}
First let $S_\epsilon$ is the slow manifold validated in $N_1\cup N_2$.
Also, let 
$M_1\subset (N_1)_{\bar y}$ and $M_2\subset (N_2)_{\bar y + \bar h}$ be $h$-sets such that $M_1\overset{P_\epsilon^{(N_1)_{\leq \bar y + \bar h}}}{\Longrightarrow} M_2$, which are constructed in Proposition \ref{prop-CE-2-1}. 
Note that $M_1$ contains $W^s(S_\epsilon)\cap (N_1)_{\bar y}$. 
Let $\tilde M := (P_\epsilon^{(N_1)_{\leq \bar y }})^{-1}(M_1) \subset  N_1$ and $\tilde M^\ent := \tilde M_1\cap N_1^\ent$. Clearly the set $\tilde M_1$ contains $W^s(S_\epsilon)\cap (N_1)_{\leq \bar y}$.

\bigskip
Next, let $w_1=(b_1,\nu_1) \in \tilde M^\ent$ and $w_2\in  W^u(S_\epsilon)$ be such that $w_2 \in C_m^s(w_1)$. 
The stable $m$-cone condition and Lemma \ref{lem-ent-cone} imply that $w_2(-t) \in C_m^s(w_1(-t))$ for all $t\geq 0$ until $w_2(-t)$ arrives at $\partial C_m^s$ under the backward flow. 
The difference $\Delta b \equiv b_2(-t) - b_1(-t) = \pi_b(w_2(-t)) - \pi_b(w_1(-t))$ satisfies
\begin{equation*}
\frac{d}{d\tilde t}(|\Delta b|^2) \geq 2\mu_{\min}(b_2)|\Delta b|^2
\end{equation*}
by the same argument as the proof of Lemma \ref{lem-ratio}, where $\tilde t = -t$. 
This inequality implies that the backward solution orbit $w_2(-t)$ with $w_2(0) = w_2$ arrives at $\partial C_m^s$ at time $t < T_{\arr}$ and that 
\begin{align*}
\pi_y((P_\epsilon^{C_m^s})^{-1} (\tilde M^{\ent}\cap \{y=\bar y\})) &\subsetneq (\bar y - \epsilon^+ T_{\arr}, \bar y - \epsilon^- T_{\arr}),\\
\pi_y((P_\epsilon^{C_m^s})^{-1} (\tilde M^{\ent}\cap \{y=y^-\})) &\subsetneq (y^- - \epsilon^+ T_{\arr}, y^- - \epsilon^- T_{\arr}). 
\end{align*}
Define an $h$-set $\tilde M_1 := \tilde M \cup (P_\epsilon^{C_m^s})^{-1} (\tilde M^\ent)$, where
\begin{align*}
\tilde M_1^+ &= (\tilde M_1\cap \{b=\tilde b^\pm\}) \cup (\tilde M \cap \{y=\bar y, y^-\}) \cup (P_\epsilon^{C_m^s})^{-1} (\tilde M^{\ent}\cap \{y=\bar y, y^-\}),\\
\tilde M_1^- &= \overline{\partial \tilde M_1\setminus \tilde M_1^+}.
\end{align*}
Note that $\varphi_\epsilon(T_0,N_0)\cap \tilde M_1^+ = \emptyset$ follows from (\ref{ineq-y-cov}) and $N_0 \overset{\varphi_\epsilon(T_0,\cdot)}{\Longrightarrow} C_m^s \cap \{y\in [y^-,\bar y]\}$. 
The relationship $\varphi_\epsilon(T_0,N_0^-)\cap \tilde M_1 = \emptyset$ immediately follows from $N_0 \overset{\varphi_\epsilon(T_0,\cdot)}{\Longrightarrow} C_m^s \cap \{y\in [y^-,\bar y]\}$.
Then the covering relation $N_0 \overset{\varphi_\epsilon(T_0,\cdot)}{\Longrightarrow} \tilde M_1$ also follows from $N_0 \overset{\varphi_\epsilon(T_0,\cdot)}{\Longrightarrow} C_m^s \cap \{y\in [y^-,\bar y]\}$.

\end{proof}

Slow shadowing sequence with extended $m$-cones determine the generalized sequence of covering-exchange sequence in Definition \ref{dfn-CE-seq}.

\begin{dfn}[Covering-exchange sequence with extended cones]\rm
\label{dfn-CE-seq-cone}
Consider (\ref{fast-slow})$_\epsilon$ in $\mathbb{R}^3$.
Let $N\subset \mathbb{R}^3$ be an $h$-set with $u(N)=1$ and $\{N_\epsilon^j\}_{j=0}^{j_N}$ be a sequence such that 
\begin{itemize}
\item $\{N_\epsilon^j\}_{j=0}^{j_N}$ is a slow shadowing sequence for (\ref{fast-slow})$_\epsilon$ with $u=s=1$;
\item $\{N_\epsilon^{j_N-1}, N_\epsilon^{j_N}\}$ is a slow shadowing pair for (\ref{fast-slow})$_\epsilon$ with the exit $N_\epsilon^{\rm exit} \subset (N_\epsilon^{j_M})^{f,-}$.
\end{itemize}

Also, let $C_{m^u}^u$ be the unstable $m^u$-cone of $N_\epsilon^{j_N}$ and $C_{m^s}^s$ be the stable $m^s$-cone of $N_\epsilon^0$. Assume that
\begin{itemize}
\item all assumptions in Proposition \ref{prop-cov-unstable} are satisfied with $N_\epsilon^\exit $ and $C_m^u = C_{m^u}^u$;
\item all assumptions in Proposition \ref{prop-cov-stable} are satisfied with $N_0 = N$, $N_i = N_\epsilon^{i-1}$ and $C_m^s = C_{m^s}^s$.
\end{itemize}
Then we call the collection $(N, \{N_\epsilon^j\}_{j=1}^{j_M}, N_\epsilon^{\rm exit}, C_{m^u}^u, C_{m^s}^s)$ {\em the covering-exchange sequence with extended cones} $C_{m^u}^u, C_{m^s}^s$.
\end{dfn}

Another benefit of $m$-cones is that we can validate the assumption (SS5) in slow shadowing in terms of $m$-cones.
In the following proposition, we do not any restrictions on $u$ and $s$.

\begin{prop}[A sufficient condition of (SS5)]
\label{prop-SS5}
Consider two fast-saddle-type blocks $N_1$ and $N_2$ constructed in the local coordinate $((a_i,b_i),y_i)$, where (\ref{fast-slow}) locally forms (\ref{abstract-form}) in each coordinate with the commutative diagram (\ref{change-of-coordinate}).
Assume that both $N_1$ and $N_2$ satisfy the unstable $m_u$-cone condition and the stable $m_s$-cone condition for $m_u, m_s \geq 1$,  and that (SS1)-(SS4) are satisfied.
\par
Independently, consider two $h$-sets $\tilde D_1^u$ and $\tilde D_2^s$ given by
\begin{align*}
(\tilde D_1^u)_c &:= \overline{B_u(0,r_a)} \times \overline{B_s\left(0,d_b r_b + \frac{r_a}{m_u}\right)},\quad 
(\tilde D_2^s)_c := \overline{B_u\left(0,d_a r_a + \frac{r_b}{m_s}\right)} \times  \overline{B_s(0,r_b)}
\end{align*}
via $c_{N_1}$ and $c_{N_2}$, respectively,
where $r_a, r_b, d_a, d_b$ denote positive numbers given in (SS4)-(SS5).
Write the nonsingular matrix $P\equiv P_2^{-1}P_1$ in (\ref{change-of-coordinate}) into the block form
\begin{equation*}
P = \begin{pmatrix}
p_{11} & p_{12}\\
p_{21} & p_{22}
\end{pmatrix},
\end{equation*}
where the first row $(p_{11}\ p_{12})$ acts on the $u$-dimensional vector $a_1$ and the second row $(p_{21}\  p_{22})$ acts on the $s$-dimensional vector $b_1$.
Finally assume that 
\begin{equation*}
\tilde D_1^u \overset{T_{x,12}}{\Longrightarrow} \tilde D_2^s,
\end{equation*}
and that 
\begin{equation}
\label{ineq-cone-SS5}
m_s^{-2} < \frac{m_u^2\cdot \min\left(1, \underline{\mathfrak{p}_2}^2-\sigma_{21} \overline{\mathfrak{p}_2} - \frac{2\sigma_{12}^2}{m_u^2} \right) }{2 \overline{\mathfrak{p}_1}^2+m_u^2 \sigma_{21} \overline{\mathfrak{p}_2}},
\end{equation}
where $\overline{\mathfrak{p_i}}$ and $\underline{\mathfrak{p_i}}$ are the maximal and the minimal eigenvalue of the matrix $P_{ii}$, respectively, and $\sigma_{ij}$ denotes the maximal singular value of $P_{ij}$.
Then (\ref{ass-cov-shadow}) holds.
\end{prop}

\begin{proof}
First we set
\begin{align*}
(\tilde D_1^u)_c^- &= \partial B_u(0,r_a) \times \overline{B_s\left(0,d_b r_b + \frac{r_a}{m_u}\right)},\quad 
(\tilde D_1^u)_c^+ = \overline{B_u(0,r_a)} \times \partial B_s\left(0,d_b r_b + \frac{r_a}{m_u}\right)\\
(\tilde D_2^s)_c^- &:= \partial B_u\left(0,d_a r_a + \frac{r_b}{m_s}\right) \times \overline{B_s(0,r_b)},\quad 
(\tilde D_2^s)_c^+ := \overline{B_u\left(0,d_a r_a + \frac{r_b}{m_s}\right)} \times \partial B_s(0,r_b)
\end{align*}
corresponding to the $h$-set structure.

We embed $T_{x,12}(\tilde D_1^u)$ and $\tilde D_2^s$ on $(N_1)_{\bar y}$ and $(N_2)_{\bar y}$, respectively, so that the origin of both $T_{x,12}(\tilde D_1^u)$ and $\tilde D_2^s$ is the point $q\in S_\epsilon \cap (N_1\cap N_2)_{\bar y}$. 
We write the embedded sets as $T_{x,12}(\tilde D_1^u)$ and $\tilde D_2^s$ again, respectively.
Without the loss of generality, we may assume that the representation of $q$ in $(a_2,b_2,y_2)$-coordinate is given by $q = (0,0,\bar y)$.
Note that, in this case, $D_1^u\cap \tilde D_1^u$ is a family of horizontal disks in $\tilde D_1^u$ and $D_2^s\cap \tilde D_2^s$ is a family of vertical disks in $\tilde D_2^s$.

Define two sets $\hat D_1^u$ and $\hat D_2^s$ as follows. See also Fig. \ref{fig-SS5}:
\begin{align*}
(\hat D_1^u)_c &:= (\tilde D_1^u)_c \cup \bigcup_{z\in (\tilde D_1^u)_c^- } (C_{m_u}^u(z)\cap \{z=(a_1,b_1,y_1)\mid |a_1|\geq r_a\}),\\
(\hat D_2^s)_c &:= (\tilde D_2^s)_c \cup \bigcup_{z\in (\tilde D_2^s)_c^+ } (C_{m_s}^s(z)\cap \{z=(a_2,b_2,y_2)\mid |b_2|\geq r_b\}). 
\end{align*}
\begin{figure}[htbp]\em
\begin{minipage}{0.5\hsize}
\centering
\includegraphics[width=5.0cm]{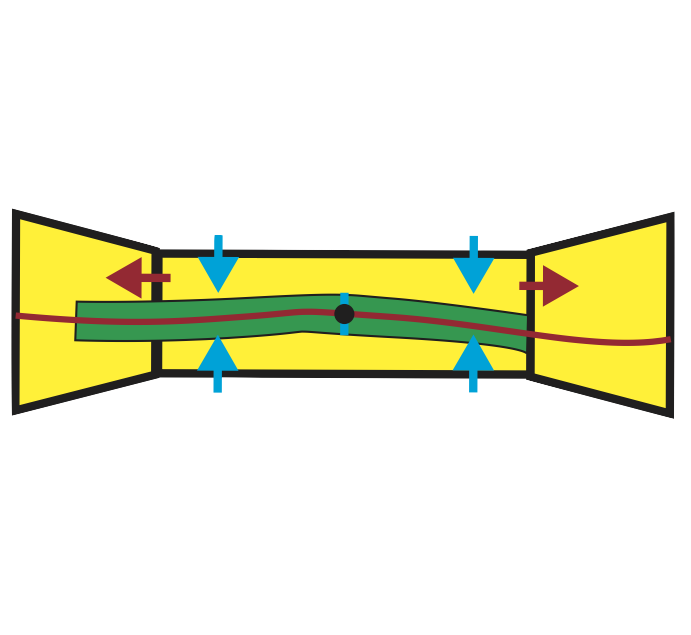}
(a)
\end{minipage}
\begin{minipage}{0.5\hsize}
\centering
\includegraphics[width=5.0cm]{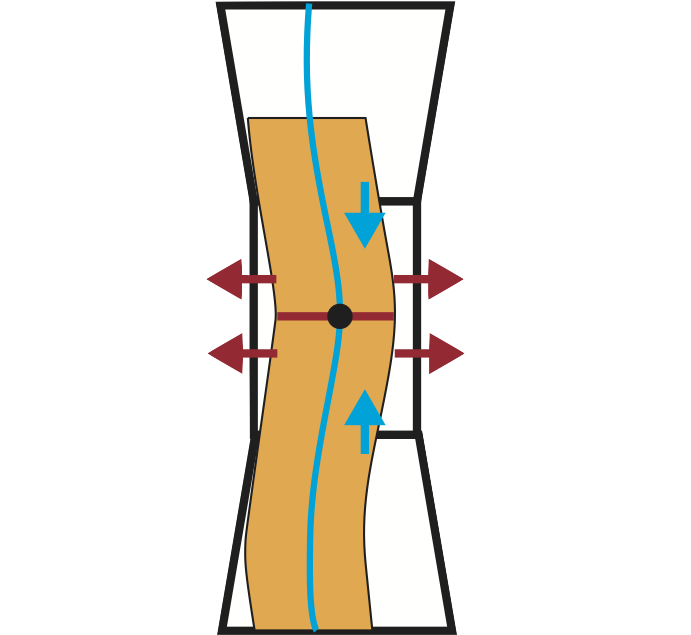}
(b)
\end{minipage}\par
\begin{minipage}{0.5\hsize}
\centering
\includegraphics[width=5.0cm]{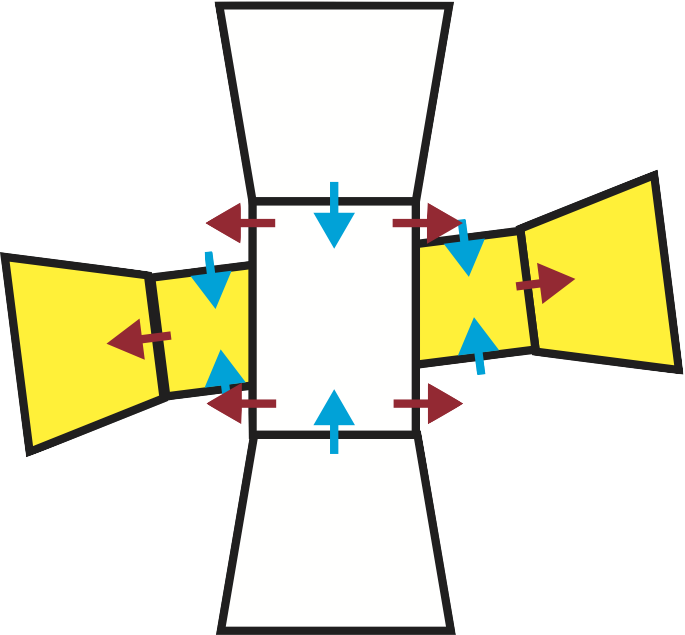}
(c)
\end{minipage}
\begin{minipage}{0.5\hsize}
\centering
\includegraphics[width=5.0cm]{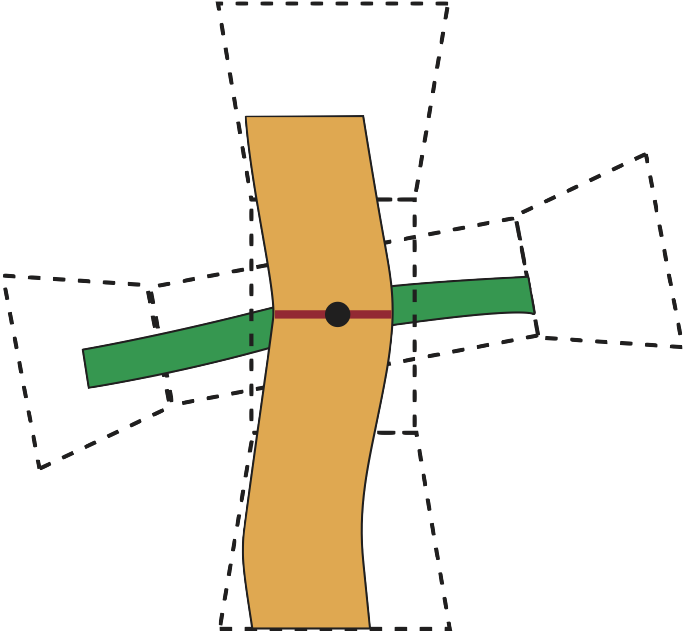}
(d)
\end{minipage}
\caption{Illustration of Proposition \ref{prop-SS5}.}
\label{fig-SS5}
(a) : The set $\hat D_1^u\cap \Gamma$ with $u=s=1$, where $\Gamma = \{y = \bar y\}$. 
The red curve represents the slice of the unstable manifold $W^u(S_\epsilon)$ by $\Gamma$.
The green tube represents the set $D_1^u\cap \Gamma$.
Definition of $\hat D_1^u$ and the unstable $m_u$-cone conditions indicate the inclusion $D_1^u\cap \Gamma \subset \hat D_1^u$.
Note that $D_1^u\cap \tilde D_1^u$ consists of a collection of horizontal disks.
\par
(b) : The set $\hat D_2^s\cap \Gamma$ with $u=s=1$, where $\Gamma = \{y = \bar y\}$. 
The blue curve represents the slice of the stable manifold $W^s(S_\epsilon)$ by $\Gamma$.
The orange tube represents the set $D_2^s\cap \Gamma$.
Definition of $\hat D_2^s$ and the stable $m_s$-cone conditions indicate the inclusion $D_2^s\cap \Gamma \subset \hat D_2^s$.
Note that $D_2^s\cap \tilde D_2^s$ consists of a collection of vertical disks.
\par
(c) : Covering relation $\tilde D_1^u \overset{T_{x,12}}{\Longrightarrow} \tilde D_2^s$ and (\ref{ineq-cone-SS5}).
The inequality (\ref{ineq-cone-SS5}) indicates that two cones outside $\tilde D_1^u$ and $\tilde D_2^s$ are sufficiently sharp satisfying  $ (\pi_{a_2,b_2}\circ (T_{x,12}\times I_1))\overline{(\hat D_1^u\setminus \tilde D_1^u)} \cap \pi_{a_2,b_2}(\overline{\hat D_2^s\setminus \tilde D_2^s}) = \emptyset$.
\par
(d) : Resulting covering relation (\ref{ass-cov-shadow}).
\end{figure}
The unstable $m_u$-cone and the stable $m_s$-cone conditions imply that $D_1^u\cap \Gamma \subset \hat D_1^u$ and $D_2^s\cap \Gamma \subset \hat D_2^s$, where $D_1^u$ and $D_2^s$ are given in (SS5) and $\Gamma = \{y=\bar y\}$.
Indeed, $D_1^u$ is given by a tube $B_s(W^u(S_\epsilon)_c,d_b r_b)$ in $(N_1)_c$. 
The unstable manifold $W^u(q)$ is contained in the unstable $m_u$-cone $C_{m_u}^u(q)$. 
If $W^u(S_\epsilon)_c$ runs on the range $\{|a|\leq r_a\}$ in $a$-direction, the enclosure of $W^u(S_\epsilon)_c$ in $b$-direction is bounded by $\{|b|\leq r_a/m_u\}$. 
The similar arguments yield that the enclosure of the tube $B_s(W^u(S_\epsilon)_c,d_b r_b)\cap \{|a|\leq r_a\}$ in $b$-direction is bounded by $\{|b|\leq d_b r_b + r_a/m_u\}$.
Relations between unstable manifolds through $q\in S_\epsilon$ and $C_{m_u}^u(q)$ thus yield that $D_1^u\cap \Gamma \subset \hat D_1^u$.
The similar arguments to $D_2^s$ yield the second.
In particular, the covering relation $\pi_{a_1,b_1} \hat D_1^u \overset{T_{x,12}}{\Longrightarrow} \pi_{a_2,b_2} \hat D_2^s$ implies (SS5).
Note that this argument is independent of $q\in S_\epsilon\cap N_1\cap N_2$, in particular, $\pi_y(q)$.
\par
\bigskip
By the construction of $\hat D_1^u$ and $\hat D_2^s$, it is sufficient to prove the following two conditions to verify $\pi_{a_1,b_1} \hat D_1^u \overset{T_{x,12}}{\Longrightarrow} \pi_{a_2,b_2} \hat D_2^s$:
\begin{itemize}
\item $\tilde D_1^u \overset{T_{x,12}}{\Longrightarrow} \tilde D_2^s$.
\item $ (\pi_{a_2,b_2}\circ (T_{x,12}\times I_1))\overline{(\hat D_1^u\setminus \tilde D_1^u)} \cap \pi_{a_2,b_2}\overline{(\hat D_2^s\setminus \tilde D_2^s)} = \emptyset$.
\end{itemize}
The former is just one of our assumptions. 
The latter concerns with intersections between subsets of cones. 
Since an $m$-cone $C_m^u(z)$ is a collection of lines inside $C_m^u(z)$ through $z$, it is sufficient to consider the location of lines through base sets.

Now we consider the image of the unstable $m_u$-cone $\mathcal{C}_c \equiv \bigcup_{z\in (\tilde D_1^u)_c^-}C_{m_u}^u(z)$ under $\tilde T_c\equiv (T_{x,12})_c\times I_1$. 
Since $\tilde T_c$ is the nonsingular affine map, any lines are mapped into lines via $\tilde T_c$.
\par
Choose a point $(a_{12}, b_{12}, y_2)\in (\tilde D_1^u)_c^-$.
The boundary $\partial \mathcal{C}_c$ is a subset of $\bigcup_{(a_{12}, b_{12}, y_2)\in (\tilde D_1^u)_c^-} \{|a_{11}-a_{12}|^2 = m_u^2(|b_{11}-b_{12}|^2 + |y_1-y_2|^2)\}$.
The transformation $\tilde T_c : (a_{1i}, b_{1i}, y_i)\mapsto  (a_{2i}, b_{2i}, y_i)$ yields
\begin{align*}
|a_{21} - a_{22}|^2 &= |p_{11}(a_{11}-a_{12}) + p_{12}(b_{11}-b_{12})|^2 \\
	&= m_u^2(|p_{21}(a_{11}-a_{12}) + p_{22}(b_{11}-b_{12})|^2 + |y_1-y_2|^2).
\end{align*}
The triangular inequality yields
\begin{align*}
&|p_{11}(a_{11}-a_{12}) + p_{12}(b_{11}-b_{12})|^2\\
	&\quad\leq |p_{11}(a_{11}-a_{12})|^2 + 2|p_{11}(a_{11}-a_{12})| |p_{12}(b_{11}-b_{12})| + |p_{12}(b_{11}-b_{12})|^2\\
	&\quad\leq 2(|p_{11}(a_{11}-a_{12})|^2 + |p_{12}(b_{11}-b_{12})|^2)\\
	&\quad\leq 2(  \overline{\mathfrak{p}_1}^2 |a_{11}-a_{12}|^2 + \sigma_{12}^2|b_{11}-b_{12}|^2),\\
&|p_{21}(a_{11}-a_{12}) + p_{22}(b_{11}-b_{12})|^2 + |y_1-y_2|^2\\
	&\quad\geq (|p_{21}(a_{11}-a_{12})| - |p_{22}(b_{11}-b_{12})|)^2 + |y_1-y_2|^2\\
	&\quad= |p_{21}(a_{11}-a_{12})|^2 -2|p_{21}(a_{11}-a_{12})| |p_{22}(b_{11}-b_{12})| + |p_{22}(b_{11}-b_{12})|^2 + |y_1-y_2|^2\\
	&\quad\geq |p_{21}(a_{11}-a_{12})|^2 -2\sigma_{21} \overline{\mathfrak{p}_2} |a_{11}-a_{12}| |b_{11}-b_{12}| + \underline{\mathfrak{p}_2}^2 |b_{11}-b_{12}|^2 + |y_1-y_2|^2\\
	&\quad\geq  -\sigma_{21} \overline{\mathfrak{p}_2} (|a_{11}-a_{12}|^2 + |b_{11}-b_{12}|^2) + \underline{\mathfrak{p}_2}^2 |b_{11}-b_{12}|^2 + |y_1-y_2|^2.
\end{align*}
Thus
\begin{align*}
&2(  \overline{\mathfrak{p}_1}^2 |a_{11}-a_{12}|^2 + \sigma_{12}^2|b_{11}-b_{12}|^2) \\
 &\quad \geq m_u^2\left\{-\sigma_{21} \overline{\mathfrak{p}_2} (|a_{11}-a_{12}|^2 + |b_{11}-b_{12}|^2) + \underline{\mathfrak{p}_2}^2 |b_{11}-b_{12}|^2 + |y_1-y_2|^2\right\},
\end{align*}
further,
\begin{align*}
& (2 \overline{\mathfrak{p}_1}^2+m_u \sigma_{21} \overline{\mathfrak{p}_2} )    \\
 &\quad \geq  m_u^2\left\{  \left(\underline{\mathfrak{p}_2}^2-\sigma_{21} \overline{\mathfrak{p}_2} - \frac{2\sigma_{12}^2}{m_u^2} \right) |b_{11}-b_{12}|^2 + |y_1-y_2|^2\right\}\\
 &\quad \geq m_u^2\cdot \min  \left(1, \underline{\mathfrak{p}_2}^2-\sigma_{21} \overline{\mathfrak{p}_2} - \frac{2\sigma_{12}^2}{m_u^2} \right) \left\{ |b_{11}-b_{12}|^2 + |y_1-y_2|^2\right\},
\end{align*}
equivalently, 
\begin{equation*}
|a_{11}-a_{12}|^2\geq \frac{m_u^2\cdot \min\left(1, \underline{\mathfrak{p}_2}^2-\sigma_{21} \overline{\mathfrak{p}_2} - \frac{2\sigma_{12}^2}{m_u^2} \right) }{2 \overline{\mathfrak{p}_1}^2+m_u^2 \sigma_{21} \overline{\mathfrak{p}_2}} \left\{ |b_{11}-b_{12}|^2 + |y_1-y_2|^2\right\}.
\end{equation*}
Let 
\begin{equation*}
\bar m^2 :=  \frac{m_u^2\cdot \min\left(1, \underline{\mathfrak{p}_2}^2-\sigma_{21} \overline{\mathfrak{p}_2} - \frac{2\sigma_{12}^2}{m_u^2} \right) }{2 \overline{\mathfrak{p}_1}^2+m_u^2 \sigma_{21} \overline{\mathfrak{p}_2}}.
\end{equation*}
The above inequality indicates that any line in $\hat  D_1^u$ is mapped into another one included in $C_{\bar m}^u(z)$ for some $z$.
By our construction, we know that any (half) lines in $\tilde T\overline{\hat  D_1^u\setminus \tilde D_1^u}$ lie on ones with vertices in $T_{x,12}\tilde D_1^u \cap (\tilde D_2^s\setminus (\tilde D_2^s)^+)$. 
Such vertices can be also chosen as points on $T_{x,12}(\tilde D_1^u)^-$, which are disjoint from $\tilde D_2^s$.
Remark that $T_{x,12}\tilde D_1^u\cap (\tilde D_2^s)^+ = \emptyset$ follows from the covering relation $\tilde D_1^u \overset{T_{x,12}}{\Longrightarrow} \tilde D_2^s$.
\par
For any point $z\in T_{x,12}\tilde D_1^u \cap \tilde D_2^s$, we easily know that $C_{\bar m}^u(z)\cap C_{m_s}^s(z)=\{z\}$ if $(\bar m m_s)^2 > 1$.
This implies that any lines in $\tilde T\overline{\hat  D_1^u\setminus \tilde D_1^u}$ are disjoint from $C_{m_s}^s(z)$ with $z\in T_{x,12}\tilde D_1^u \cap \tilde D_2^s$.
This fact holds for any $z\in T_{x,12}\tilde D_1^u \cap \tilde D_2^s$.

On the other hand, for any point $w\in (\tilde D_2^s)^+$, there is a point $z\in T_{x,12}\tilde D_1^u \cap \tilde D_2^s$ such that $C_{m_s}^s(w)\subset C_{m_s}^s(z)$,
which follows from the property of cones and the structure of $T_{x,12}\tilde D_1^u \cap \tilde D_2^s$ from the covering relation $\tilde D_1^u \overset{T_{x,12}}{\Longrightarrow} \tilde D_2^s$.
Consequently, we know that the set $\overline{\hat D_2^s\setminus \tilde D_2^s}$ is contained in the union of $C_{m_s}^s(z)$ with $z\in T_{x,12}\tilde D_1^u \cap \tilde D_2^s$. 
\par
Combining these observations, we obtain $\tilde T(\overline{\hat D_1^u\setminus \tilde D_1^u}) \cap \overline{\hat D_2^s\setminus \tilde D_2^s} = \emptyset$.
\end{proof}

%
%	New Subsection
%
\subsection{Invariant sets on slow manifolds}
\label{section-inv-set-on-mfd}
Next we consider invariant sets on slow manifolds. When $\epsilon > 0$, the dynamics on slow manifolds can generally exhibit non-trivial dynamics. In other words, slow manifolds can have nontrivial invariant sets, such as equilibria, periodic orbits and so on, for slow dynamics. Such sets play a key role for describing nontrivial global solutions such as singularly perturbed homoclinic or heteroclinic orbits.

\bigskip
Consider the slow manifold $S_\epsilon$ in a fast-saddle type block $M$ with stable and unstable cone conditions.
Theorem \ref{thm-inv-mfd-rigorous} implies that $S_\epsilon$ is represented by the graph of a function $x = h^\epsilon(y)$ smoothly depending on $y$ and $\epsilon$ including $\epsilon = 0$. Substituting $h^\epsilon(y)$ into (\ref{fast-slow}), one sees that the $y$-equation will decouple from the $x$-equation. Since $S_\epsilon$ is parameterized by $y$, the resulting decoupled equation
\begin{equation*}
y' = \epsilon g(h^\epsilon(y), y,\epsilon)
\end{equation*}
describes the dynamics on $S_\epsilon$. After time rescaling $\tau = t/\epsilon$ we obtain
\begin{equation}
\label{eq-slow-mfd}
\dot y = g(h^\epsilon(y), y,\epsilon),\quad \dot {} = \frac{d}{d\tau},
\end{equation}
to show that the dynamics on slow manifolds is reduced to the {\em regular perturbation problem}. Taking the limit $\epsilon \to 0$ in (\ref{eq-slow-mfd}) we obtain the dynamics on the critical manifold $S_0 \subset \{f(x,y,0) = 0\}$: 
\begin{equation}
\label{eq-critical-mfd}
\dot y = g(h^0(y), y,0).
\end{equation}
When we want to study dynamics on slow manifolds, such as the existence of fixed points and their stability, (\ref{eq-slow-mfd}) as well as (\ref{eq-critical-mfd}) are in the center of our considerations.

\bigskip
In general, what we know about $S_\epsilon$ is just the fact that it is the graph of a function $x = h^\epsilon(y)$, and its concrete description is quite difficult to obtain. 
As for the normally hyperbolic critical manifold $S_0 = \{x = h^0(y)\}$, we know that it is the subset of the nullcline $\{f(x,y,0) = 0\}$. 
From such a fact, we can study the detail of $S_0$ such as the differential of $h^0$ as a $y$-function via implicit function differential equation $f_x(h^0(y),y,0)(\partial h^0/\partial y) + f_y(h^0(y), y, 0) = 0$. 
On the contrary, the function $h^\epsilon(y)$ does not possess such simple and useful properties. For example, $f(h^\epsilon(y),y,\epsilon) \equiv 0$ does not necessarily hold. 
In general, $h^\epsilon(y)$ is the solution of the following nonlinear partial differential equation: 
\begin{equation}
\label{rigorous-slow-mfd}
\epsilon \frac{\partial h^\epsilon}{\partial y}(y) g(h^\epsilon(y),y,\epsilon) = f(h^\epsilon(y),y,\epsilon).
\end{equation}
In order to calculate the differential $\partial h^\epsilon/\partial y$ for studying the stability of fixed points on slow manifolds, for example, we have to solve the equation (\ref{rigorous-slow-mfd}) rigorously. Nevertheless, from the viewpoint of rigorous numerics, it will be more reasonable regarding $h^\epsilon$ as $h^0$ with small errors (in $C^r$-sense) than solving (\ref{rigorous-slow-mfd}) directly. With this in mind, we consider (\ref{eq-slow-mfd}) as the differential inclusion
\begin{equation*}
\dot y \in \{g(h^\epsilon(y),y,\epsilon) \mid \epsilon\in [0,\epsilon_0], (h^\epsilon(y),y)\in M\},
\end{equation*}
where $M$ is a fast-saddle type block containing $S_\epsilon$.
The right-hand side possesses $h^0(y)$ as the representative of the enclosure.

\bigskip
As an example of dynamics on $S_\epsilon$ via rigorous numerics, we consider the validation of fixed points for (\ref{eq-slow-mfd}) on $S_\epsilon$ as well as their stability for $\epsilon \in (0,\epsilon_0]$. 
Thanks to regular perturbation relationship between (\ref{eq-slow-mfd}) and (\ref{eq-critical-mfd}), it is natural to consider that the fixed point of (\ref{eq-slow-mfd}) on $S_\epsilon$ will be close to that of (\ref{eq-critical-mfd}) on $S_0$.
We then rewrite (\ref{eq-slow-mfd}) as
\begin{equation*}
\dot y =  g(h^0(y),y,0) + \left\{  g(h^\epsilon(y),y,\epsilon) -  g(h^0(y),y,0) \right\}. 
\end{equation*}
Using the Mean Value Theorem, the error term possesses the following expression: 
\begin{align*}
g(h^\epsilon(y),y,\epsilon) -  g(h^0(y),y,0) &= g(h^\epsilon(y),y,\epsilon) -  g(h^0(y),y,\epsilon) + g(h^0(y),y,\epsilon) -  g(h^0(y),y,0),\\
g(h^\epsilon(y),y,\epsilon) -  g(h^0(y),y,\epsilon) &= \frac{\partial g}{\partial x}(\tilde x, y, \epsilon)(h^\epsilon(y)- h^0(y))\quad \text{ for some }\tilde x \in M,\\
g(h^0(y),y,\epsilon) -  g(h^0(y),y,0) &= \frac{\partial g}{\partial \epsilon}(h^0(y), y, \tilde \epsilon)\epsilon\quad \text{ for some }\tilde \epsilon \in [0,\epsilon_0].
\end{align*}
The difference between $h^\epsilon(y)$ and $h^0(y)$ can be estimated by $|(h^\epsilon(y) - h^0(y))_{x_i}| \leq \beta_{x_i}$.
Here $\beta_{x_i}$ corresponds to the size of the $i$-th component in $x$-coordinate of a fast-saddle-type block $M$ obtained by rigorous numerics. 
It is precisely computable. 
Such an estimate can be done because both $h^\epsilon(y)$ and $h^0(y)$ belong to an identical $h$-set.
Finally we obtain the following estimate: 
\begin{equation}
\label{estimate-error-slow}
\left\{  g(h^\epsilon(y),y,\epsilon) -  g(h^0(y),y,0) \right\} \in  \frac{\partial g}{\partial x}(\tilde x, y, \epsilon) \beta_{x_i}[-1,1] + \frac{\partial g}{\partial \epsilon}(h^0(y), y, \tilde \epsilon)[0,\epsilon_0].
\end{equation}

Let $y_0$ be the numerical zero of $g(h^0(y),y,0) = 0$. The principal part $g(h^0(y),y,0)$ can be then expanded by
\begin{equation*}
g(h^0(y),y,0) = g(h^0(y_0),y_0,0) + \left(g_x h^0_y + g_y \right)_{y=y_0}(y-y_0) + h.o.t.
\end{equation*}
The Jacobian matrix $J_0 \equiv g_x h^0_y + g_y$ at $\epsilon = 0$ can be explicitly calculated via the Implicit Function Theorem to obtain
\begin{equation*}
g_x h^0_y + g_y = - g_x(f_x)^{-1}f_y + g_y.
\end{equation*}
One can explicitly calculate eigenvalues of $J_0$, which will be close to those of the Jacobian matrix of $J_\epsilon \equiv \partial g / \partial y$ for $\epsilon > 0$. 
On the other hand, the chain rule yields
\begin{equation*}
J_\epsilon = \frac{\partial g}{\partial x}(h^\epsilon(y),y,\epsilon) \frac{\partial h^\epsilon}{\partial y}(y) + \frac{\partial g}{\partial y}(h^\epsilon(y),y,\epsilon)
\end{equation*}
and it contains $\partial h^\epsilon / \partial y$ arisen in (\ref{rigorous-slow-mfd}). 
We additionally need to the Fr\'{e}chet differential of error terms of the form $g(h^\epsilon(y),y,\epsilon) - (\partial h^\epsilon /\partial y)\cdot (y-y_0)$,
when we try to validate the unique existence or hyperbolicity of equilibria via cone conditions \cite{ZCov} or Lyapunov conditions \cite{Mat}, for example.
Instead, we apply a topological tool with $J_0$ to validation of invariant sets on slow manifolds.

\bigskip
Here we construct an isolating block on slow manifolds.
Using enclosures like (\ref{estimate-error-slow}) and eigenvalues of $J_0$, we can compute the {\em rigorous} enclosure of vector fields on crossing sections on slow manifolds. 

Consider the case $l=1$, which is our original setting. 
Assume that we have an $\epsilon$-parameter family of slow manifolds $\{S_\epsilon\}_{\epsilon \in (0,\epsilon_0]}$ as well as the critical manifold $S_0$, which can be validated by the method in Section \ref{section-inv-mfd}. 
Then one would have obtained a fast-saddle-type block $M$ surrounding $\{S_\epsilon\}_{\epsilon \in [0,\epsilon_0]}$. 
The schematic illustration of such a procedure is shown in Fig. \ref{fig-slow-dynamics}. 
Let $y_0 \in S_0$ be a numerical zero of $g(x,y,0)$. 
One expects that there exists an equilibrium $y_\epsilon$ on $S_\epsilon$ for sufficiently small $\epsilon > 0$. 
We calculate the enclosure of $g(h^\epsilon(y),y,\epsilon)$ on $\{y = y_0\pm \delta\}$ for small $\delta$ in such a situation. 
It can be done by
\begin{align*}
g(h^\epsilon(y),y,\epsilon)\mid_{y=y_0\pm \delta} &= g(h^0(y),y,0) + \left\{  g(h^\epsilon(y),y,\epsilon) -  g(h^0(y),y,0) \right\}\\
&= g(h^0(y_0),y_0,0) + \left(g_x h^0_y + g_y \right)_{y=y_0}(y-y_0) + h.o.t.\\
&\in \{g(h^0(y_0),y_0,0) \pm J_0 \delta  + h.o.t.\} + \frac{\partial g}{\partial x}(\tilde x, y, \epsilon) \beta_{x}[-1,1] + \frac{\partial g}{\partial \epsilon}(h^0(y), y, \tilde \epsilon)\epsilon,
\end{align*}
where $\beta_x > 0$ can be determined by the size of $M$. 
Remark that the enclosure of all terms in the right-hand side are rigorously computable 
and that they make sense for all $\epsilon \in [0,\epsilon_0]$. 
The principal term is $\pm J_0 \delta$. $g(h^0(y_0),y_0,0)$ is very close to $0$ since $y_0$ is assumed to be the numerical zero of $g(h^0(y_0),y_0,0)$. 
If
\begin{equation}
\label{att-slow}
g(h^\epsilon(y),y,\epsilon)\mid_{y=y_0- \delta}\ > 0\quad \text{ and }\quad g(h^\epsilon(y),y,\epsilon)\mid_{y=y_0+ \delta}\ < 0
\end{equation}
are validated on $M \times [0,\epsilon_0]$, then the set $B_\epsilon:= \{(h^\epsilon(y), y, \epsilon)\in M\times [0,\epsilon_0] \mid y\in [y_0-\delta, y_0+\delta]\}$ is an isolating block on $S_\epsilon$ for $\dot y = g(h^\epsilon(y), y, \epsilon)$ with the attracting boundary. 
Similarly, if
\begin{equation}
\label{rep-slow}
g(h^\epsilon(y),y,\epsilon)\mid_{y=y_0- \delta}\ < 0\quad \text{ and }\quad g(h^\epsilon(y),y,\epsilon)\mid_{y=y_0+ \delta}\ > 0
\end{equation}
are validated on $M \times [0,\epsilon_0]$, then the set $B_\epsilon = \{(h^\epsilon(y), y, \epsilon)\in M\times [0,\epsilon_0] \mid y\in [y_0-\delta, y_0+\delta]\}$ is an isolating block on $S_\epsilon$ for $\dot y = g(h^\epsilon(y), y, \epsilon)$ with the repelling boundary. 
In both cases, the general Conley index theory \cite{Con} yields that $\Inv(B_\epsilon)\not = \emptyset$ for all $\epsilon \in (0,\epsilon]$. 
$\Inv(B_\epsilon)$ is actually the $\epsilon$-parameter family of invariant sets on slow manifolds. 
The direction of vector fields corresponds to the stability of $\Inv(B_\epsilon)$. 
Moreover, Proposition \ref{prop-existence-fixpt} shows the existence of an equilibria in $\Inv(B_\epsilon)$ for (\ref{eq-slow-mfd}), which leads to an equilibrium in the full system (\ref{fast-slow}).
Note that isolating blocks can be constructed by the same implementations as Section \ref{section-isolatingblock}. 

\begin{rem}\rm
In terms of the Conley index theory, constructed sets $\{B_\epsilon\}_{\epsilon\in [0,\epsilon_0]}$ are {\em singular isolating neighborhoods}. 
Indeed, for $\epsilon = 0$, $N_0$ consists of the curve of hyperbolic equilibria and hence it is not an isolating neighborhood. On the other hand, for $\epsilon \in (0,\epsilon]$ it is an isolating neighborhood for $\dot y = g(h^\epsilon(y), y, \epsilon)$.

The general theory of the Conley index can be referred in e.g. \cite{Con, Mis, Smo}. In particular, \cite{Mis} gives the explanation of singular perturbation version of the Conley index. More advanced topics for singular perturbation problems are shown in \cite{GGKKMO, GKMOR, GKMO} for example.
\end{rem}

\begin{rem}\rm
All validation of vector fields such as (\ref{att-slow}) and (\ref{rep-slow}) make sense {\em only on slow manifolds} $\{S_\epsilon\}_{\epsilon \in [0,\epsilon_0]}$, since the decoupled vector field (\ref{eq-slow-mfd}) makes sense only on slow manifolds. 
In other words, estimates of vector fields discussed here does not always give reasonable information of vector fields off slow manifolds. 
Calculations of vector fields discussed here and Section \ref{section-exchange} should be thus considered independently.

A simpler way to study the dynamics on slow manifolds will be to focus on the dynamics (\ref{eq-critical-mfd}), namely, neglect the error term $g(h^\epsilon(y),y,\epsilon) -  g(h^0(y),y,0)$.
We then study the dynamics (\ref{eq-slow-mfd}) on slow manifolds via the regular perturbation theory. 
The essence of this idea is the same as ours, but such results make sense only for sufficiently small $\epsilon > 0$. 
One of our main purpose here is to study the slow dynamics for all $\epsilon \in (0,\epsilon_0]$ and hence such a simpler idea is not sufficient to our study. 
\end{rem}

\begin{figure}[htbp]\em
\begin{minipage}{0.33\hsize}
\centering
\includegraphics[width=4.0cm]{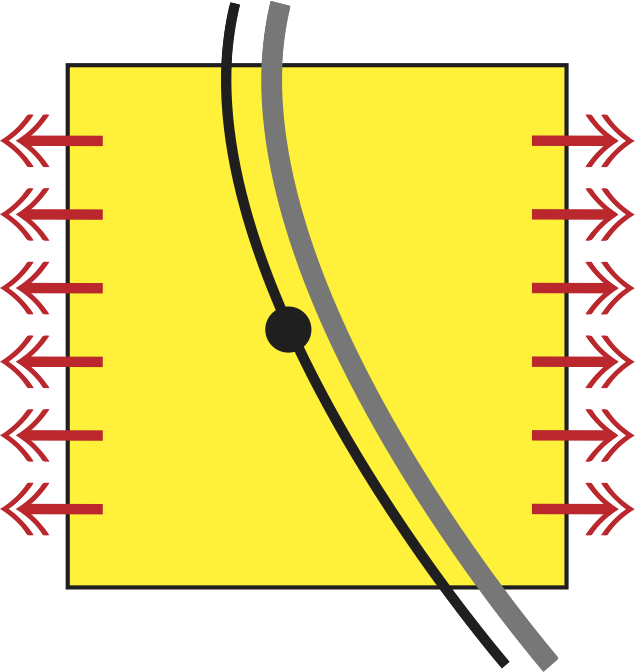}
(a)
\end{minipage}
\begin{minipage}{0.33\hsize}
\centering
\includegraphics[width=4.0cm]{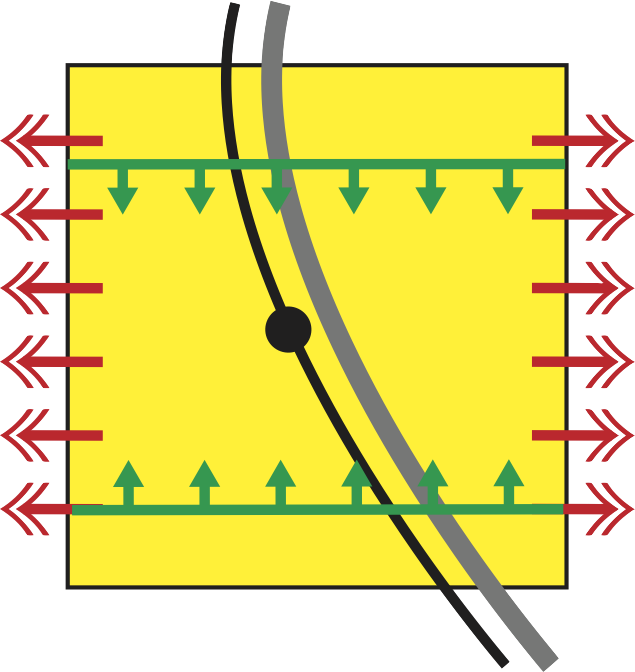}
(b)
\end{minipage}
\begin{minipage}{0.33\hsize}
\centering
\includegraphics[width=4.0cm]{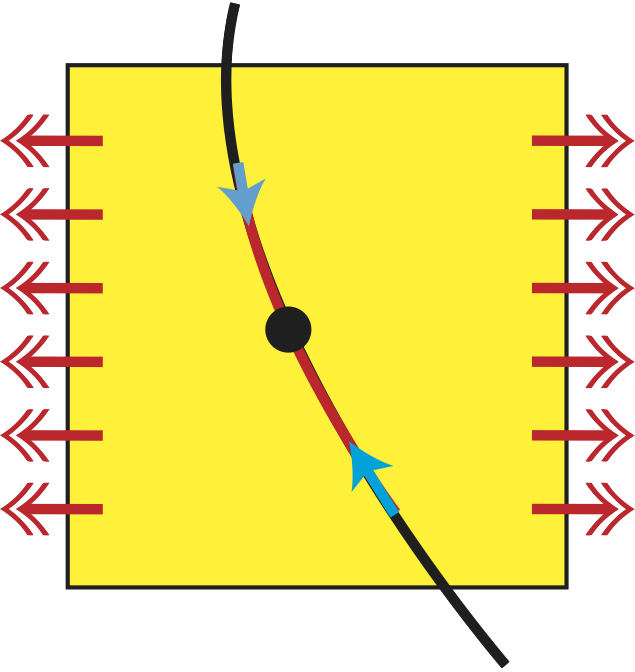}
(c)
\end{minipage}
\caption{Invariant sets on slow manifolds.}
\label{fig-slow-dynamics}
(a). Critical manifold $S_0$ ($\epsilon = 0$, colored by ash), and perturbed slow manifold $S_\epsilon$ ($\epsilon \in (0,\epsilon_0]$, colored by black) via Theorem \ref{thm-inv-mfd-rigorous}. The yellow rectangle with red arrows describes a fast-saddle-type block with the fast-exit. The critical manifold $S_0$ is a subset of nullcline $f(x,y,0)=0$. 
In general, $S_\epsilon$ exhibits a nontrivial structure such as equilibria (drawn by a black ball).

\bigskip
(b). Validation of vector fields on $S_\epsilon\ (\epsilon \in [0,\epsilon_0])$. 
One cannot detect the position of the slow manifold $S_\epsilon$ in advance. In order to detect vector field on such invariant manifolds, one put a section $S$ (colored green) having an intersection with $S_\epsilon$ and validate the vector field on $S$ via rigorous numerics. 
This contains information of the vector field on $S\cap S_\epsilon$.

\bigskip
(c). An illustration of isolating block containing an equilibrium on $S_\epsilon$. Analysis in (b) contains the information of rigorous vector field on $S_\epsilon$, which gives an isolating neighborhood colored by red. General topological arguments such as the mapping degree or the Conley index guarantee the existence of nontrivial invariant sets. In the case of illustration described here, there exists a fixed point on $S_\epsilon$, namely, an equilibrium of the full system (\ref{fast-slow}).
\end{figure}

%
%	New Subsection
%
\subsection{Unstable manifolds of invariant sets on slow manifolds}
\label{section-fiber-unstable}

We discuss the unstable manifold of invariant sets in the full system. 
Combining covering-exchange property with graph representations of locally invariant manifolds, the standard consequence of covering relations, Proposition \ref{WZ-heteroclinic}, yields the existence of connecting orbits between slow manifolds for (\ref{fast-slow})$_\epsilon$. 
However, it does not always mean the existence of connecting orbits between invariant sets for (\ref{fast-slow})$_\epsilon$. 
This difference comes from the fact that the meaning of slow variable $y$ changes between $\epsilon = 0$ and $\epsilon > 0$. 
When we apply covering relations to connecting orbits between equilibria for (\ref{fast-slow})$_0$, namely (\ref{layer}), a natural approach is to prove the existence of orbits connecting points on the stable and unstable manifolds of equilibria at a certain $\bar y$. 
This immediately yields validation of a connecting orbit at $y = \bar y$.
On the other hand, if $\epsilon > 0$, a parameter $y$ becomes a time dependent variable and slow manifolds do not always consist of a family of equilibria or general invariant sets.
In fact, if unstable manifolds of normally hyperbolic invariant manifolds are perturbed for $\epsilon > 0$, then the negatively invariant manifolds (e.g. unstable manifolds) of {\em invariant sets} are drastically collapsed in general. 
In particular, dimensions of $h$-sets corresponding to unstable manifolds of invariant sets change. 
The standard arguments of covering relation is not thus sufficient to validate connecting orbits in (\ref{fast-slow})$_\epsilon$.

To overcome this difficulty, we apply further structure of slow manifolds, called Fenichel fibering, to validate (un)stable manifolds of invariant sets on slow manifolds.

A direct consequence of discussions in the previous subsection is stated as follows. 
This result can be generalized in arbitrary dimensions, but we omit the detail because we do not make such generalized arguments in this paper.

\begin{lem}
\label{lem-adm-cov}
Let $N$ be a fast-saddle type block for (\ref{fast-slow})$_{\epsilon}$ satisfying stable and unstable cone conditions,
and $I_\epsilon \subset \Int M$ be a subset of $1$-dimensional slow manifold in $N$. 
Assume that $I_\epsilon$ is an isolating block on the slow manifold whose boundaries are repelling (cf. (\ref{rep-slow})). 
Then a trivial collection $\{S_{\epsilon}\}$ is admissible with respect to $\varphi_{\epsilon}(t,\cdot)$ for all sufficiently small $t>0$ (cf. Definition \ref{dfn-cov-admissibility}). 

Similarly, assume that $I_\epsilon$ is an isolating block on the slow manifold whose boundaries are attracting (cf. (\ref{att-slow})). 
Then a trivial collection $\{S_{\epsilon}\}$ is back-admissible with respect to $\varphi_{\epsilon}(t,\cdot)$ for all sufficiently small $t>0$.

Here we regard the flow $\varphi_\epsilon$ in the full system as $\varphi_\epsilon$ restricted to $I_\epsilon$ in both cases. This restriction makes sense since slow manifolds are locally invariant.
\end{lem}

Since $I_\epsilon$ is $1$-dimensional, if $I_\epsilon$ is repelling, the covering relation $I_\epsilon\overset{\varphi_\epsilon(t,\cdot)}{\Longrightarrow} I_\epsilon$ is nothing but $I_\epsilon \subset \varphi_\epsilon(t,I_\epsilon)$.
Similarly, if $I_\epsilon$ is attracting, the back-covering relation $I_\epsilon\overset{\varphi_\epsilon(t,\cdot)}{\Longleftarrow} I_\epsilon$ is nothing but $I_\epsilon \subset \varphi_\epsilon(-t,I_\epsilon)$.
For simplicity, we write $k$-covering relations $N\overset{f}{\Longrightarrow} N$ for an $h$-set $N$ and a continuous map $f$ by $\left\{N\overset{f}{\Longrightarrow} N \right\}_k$.
Let $\left\{N\overset{f}{\Longleftarrow} N \right\}_k$ be defined by the similar manner.

Lemma \ref{lem-adm-cov} gives the other proof of existence of an invariant set in $I_\epsilon$. 
Note that this invariant set is actually invariant for the full system (\ref{fast-slow})$_\epsilon$.

\bigskip
A key essence of Fenichel fibering is cone conditions, which is the same as invariant manifold theorem. 
On the other hand, our statements in Theorem \ref{thm-inv-mfd-rigorous} involve cone condition which are computable via rigorous numerics. 
Theorem \ref{thm-inv-mfd-rigorous} thus yields the following fiber validations.

\begin{cor}
\label{cor-fiber-rigorous}
Consider (\ref{abstract-form}). Let $N\subset \mathbb{R}^{n+l}$ be a fast-saddle-type block such that the coordinate representation $N_c$ is actually given by (\ref{fast-block}) with $\pi_y N = K\subset \mathbb{R}^l$. 
Assume that stable and unstable cone conditions are satisfied in $N$ and let $S_\epsilon$ be the validated slow manifold.
Then for each $v\in \Int S_\epsilon$, the following statements hold.
\begin{enumerate}
\item There exists a Lipschitz function $(a,y) = h_s(b, \epsilon)$ defined in a neighborhood of $\pi_b(v)$ in $B_s$ and $\epsilon\in  [0,\epsilon_0]$ such that the graph
\begin{equation*}
W^s(v):= \{(a,b,y,\epsilon)\mid (a,y) = h_s(b,\epsilon)\}
\end{equation*}
is locally invariant with respect to (\ref{abstract-form}). 
\item There exists a Lipschitz function $(b,y) = h_u(a,\epsilon)$ defined in a neighborhood of $\pi_a(v)$ in $B_u$ and $\epsilon\in  [0,\epsilon_0]$ such that the graph
\begin{equation*}
W^u(v):= \{(a,b,y,\epsilon)\mid (b,y) = h_u(a,\epsilon)\}
\end{equation*}
is locally invariant with respect to (\ref{abstract-form}). 
\end{enumerate}
These functions are well-defined in a neighborhood of each $v$, which means that, for example, $W^u(v)$ may have an intersection with $\partial N\setminus N^{f,-}$. In other words, $h_u$ is not always a horizontal disk in $N$ in the sense of Definition \ref{dfn-disks}.
\end{cor}

In the case of fast-slow systems, the property that $h_u$ becomes a horizontal disk in $N$ is nontrivial, which is due to dynamics in slow direction. 

From now on we relate the fiber bundle structure of slow manifolds to covering relations.
Fiber bundle structure of $W^u(I_\epsilon)$ for $I_\epsilon \subset S_\epsilon$ leads to construct a deformation retract $r_\lambda$ of unstable manifolds so that the fiber bundle $W^u(I_\epsilon)$ $r_\lambda$-covers a fast-exit face. 
Namely, the following lemma holds. 

\begin{lem}
\label{lem-fiber-cov}
Consider (\ref{abstract-form}). Let $N\subset \mathbb{R}^{n+1}$ be a fast-saddle-type block with $u=1$ such that the coordinate representation $N_c$ is actually given by $N_c = \overline{B_1}\times \overline{B_s} \times [0,1]$.
Also, let $N^\exit$ be a fast-exit face of $N$ with 
\begin{equation*}
\pi_a N^\exit = \{1\}\subset \partial B_1,\quad (N^\exit)^- = \{y^\pm\}\subset [0,1]\ \text{ with }0 < y_- < y_+ < 1
\end{equation*}
and $I_\epsilon \subset \Int N$ be a connected subset of $1$-dimensional slow manifold in $N$. 

Assume that, for any $z\in I_\epsilon$, the Lipschitz function $(b,y) = b_{u,\epsilon}^z(a)$ is a horizontal disk in $N$. 
We also assume that, at $z_\pm \in \partial I_\epsilon$ with $z_+ > z_-$, $\pi_y  b_{u,\epsilon}^{z_+}(1) > y_+$ and $\pi_y  b_{u,\epsilon}^{z_-}(1) < y_-$ hold. 

Then $I_\epsilon \overset{r_1\circ  b_{u,\epsilon}^\ast}{\Longrightarrow} N^\exit$, where $r_\lambda\circ  b_{u,\epsilon}^z$ is the deformation retract given by
\begin{equation}
\label{retract-plus}
(r_\lambda\circ  b_{u,\epsilon}^z)(a) := (\lambda + (1-\lambda)a, b_{u,\epsilon}^z(\lambda + (1-\lambda)a)),\quad \lambda\in [0,1],\ z\in I_\epsilon.
\end{equation}
If $\pi_a N^\exit = \{-1\}$, then the same statement holds by replacing $r_\lambda$ in (\ref{retract-plus}) by
\begin{equation}
\label{retract-minus}
(r_\lambda\circ  b_{u,\epsilon}^z)(a) := (-\lambda + (1-\lambda)a, b_{u,\epsilon}^z(-\lambda + (1-\lambda)a)),\quad \lambda\in [0,1],\ z\in I_\epsilon.
\end{equation}
\end{lem}

\begin{proof}
Our assumptions imply that $r_\lambda\circ b_{u,\epsilon}^{z_\pm}(\overline{B_1})\cap N^\exit = \emptyset$. 
Note that graphs of $b_{u,\epsilon}^{z_\pm}$ lie on the unstable manifold $W^u(I_\epsilon)$, which implies $r_\lambda\circ b_{u,\epsilon}^{z}(\overline{B_1})\cap (N^\exit)^+ = \emptyset$ for all $z\in I_\epsilon$. 
The rest of assumptions in Proposition \ref{prop-find-CR} or \ref{prop-CR-u1} is easy to check.
\end{proof}

If we additionally assume that $I_\epsilon$ is a singular isolating neighborhood, then we can describe unstable manifolds of invariant sets.
\begin{prop}
\label{prop-fiber}
In addition to assumptions in Lemma \ref{lem-fiber-cov}, we assume that $I_\epsilon$ is an isolating block on the slow manifold. Then there exists a point $z\in \Inv(I_\epsilon)$ such that $W^u(z)\cap N^\exit \not = \emptyset$. 
\end{prop}

\begin{proof}
Since $I_\epsilon$ is isolated, Lemma \ref{lem-adm-cov} yields the covering relation $\left\{I_\epsilon \overset{\varphi_\epsilon(t,\cdot)}{\Longrightarrow} I_\epsilon \right\}_k$ for arbitrary $k\geq 1$ and all sufficiently small $t > 0$ if $I_\epsilon$ is repelling. 
Similarly, if $I_\epsilon$ is attracting, the covering relation $\left\{I_\epsilon \overset{\varphi_\epsilon(t,\cdot)}{\Longleftarrow} I_\epsilon \right\}_k$ holds for arbitrary $k\geq 1$.

We only deal with the attracting case.
By Proposition \ref{ZG-periodic} there exists a point $z_k \in I_\epsilon$ such that 
\begin{equation}
\label{inv-discrete}
\varphi_\epsilon(-k' t,z_k)\in I_\epsilon\text{ for all }k'=0,1,\cdots, k,\quad W^u(z_k) \cap N^\exit \not = \emptyset.
\end{equation}
Since $I_\epsilon$ is compact, then the sequence $\{z_k\}_{k\geq 1}$ contains a convergence subsequence. Let $z_\infty$ be the limit. 
We can choose a sufficiently small $t_\epsilon > 0$ such that $\varphi_\epsilon([-t_\epsilon,0),z)\cap I_\epsilon = \emptyset$ holds for all $z\in \partial I_\epsilon$, since $I_\epsilon$ is an isolating block.
Note that $\partial I_\epsilon$ means the boundary with respect to the slow manifold in $N$.
Obviously we can choose a limit point $z_\infty$ from (\ref{inv-discrete}) with $t=t_\epsilon$. 
The property of $t_\epsilon, z_\infty$ and (\ref{inv-discrete}) thus imply $\varphi_\epsilon( [-k t_\epsilon, 0], z_\infty) \subset I_\epsilon$ for all $k\in \mathbb{N}$, which yields $\varphi_\epsilon ((-\infty,0],z_\infty)\subset I_\epsilon$. 
The inclusion $\varphi_\epsilon ([0, \infty),z_\infty)\subset I_\epsilon$ is obvious since $I_\epsilon$ is attracting. 
As a consequence, $z_\infty \in \Inv(I_\epsilon)$. 
By continuity of the graph $b_{u,\epsilon}^z$ with respect to $z$, then $W^u(z_\infty)$ is well-defined and $W^u(z_\infty) \cap N^\exit \not = \emptyset$.
\end{proof}

The validation of $I_\epsilon$ being an isolating block on slow manifolds can be done by arguments in Section \ref{section-inv-set-on-mfd}. 
On the other hand, if we apply Lemma \ref{lem-fiber-cov}, we have to know the location of fibers $W^u(z)$ for all $z\in I_\epsilon$. 
The unstable ($m$-)cone condition indicates that, for $z\in I_\epsilon$, the fiber $W^u(z)$ is included in the unstable cone $C_m^u(z)$. 
We can apply this cone to validate assumptions in Lemma \ref{lem-fiber-cov}. 
Namely, the following lemma holds, which can be validated by rigorous numerics.

\begin{lem}
\label{lem-fiber-cone}
Consider (\ref{abstract-form}). Let $N\subset \mathbb{R}^{n+1}$ be a fast-saddle-type block with $u=1$ such that the coordinate representation $N_c$ is actually given by $N_c = \overline{B_1}\times \overline{B_s} \times [0,1]$.
Also, let $N^\exit$ be a fast-exit face of $N$ with 
\begin{equation*}
\pi_a N^\exit = \{1\},\quad (N^\exit)^- = \{y^\pm\}\ \text{ with }0 < y_- < y_+ < 1
\end{equation*}
and $I_\epsilon \subset \Int N$ be a connected subset of $1$-dimensional slow manifold in $N$. 
Assume that, for some $m\geq 1$, the unstable $m$-cone condition is satisfied in $N$ and that, at $z_\pm \in \partial I_\epsilon$, the unstable $m$-cones $C_m^u(z_{\pm})$ satisfy the following properties:
\begin{align}
\label{ass-fiber1}
&(C_m^u(z_{\pm})) \cap \partial N \subset N^{f,-},\\
\label{ass-fiber2}
&\inf (\pi_y  (C_m^u(z_+)) \cap \{a=1\}) > y_+,\quad \sup (\pi_y  (C_m^u(z_-)) \cap \{a=1\}) < y_-.
\end{align}
Then all assumptions in Lemma \ref{lem-fiber-cov} hold.
\end{lem}

\begin{proof}
The unstable cone condition indicates, as mentioned, that the graph of the Lipschitz function $b_{u,\epsilon}^z$ is contained in $C_m^u(z)$ for all $z\in I_\epsilon$.
Since cones have the identical shape for $z\in I_\epsilon$ and $I_\epsilon$ is $1$-dimensional, then the condition $\{(a,b^z_{u,\epsilon})\mid a\in \overline{B_1}\}\subset N$ for all $z\in I_\epsilon$ obviously follows from (\ref{ass-fiber1}). 
Inequalities $\pi_y  b_{u,\epsilon}^{z_+}(1) > y_+$ and $\pi_y  b_{u,\epsilon}^{z_-}(1) < y_-$ are direct consequences of (\ref{ass-fiber2}).
\end{proof}

%
%	New Section
%
\section{Topological existence theorems of singularly perturbed trajectories}
\label{section-existence}

Now we are ready to state our verification theorems. In this section we derive existence theorems for periodic and heteroclinic orbits which can be applied to (\ref{fast-slow}). 

%
%	New Subsection
%
\subsection{Existence of periodic orbits}
First we state the existence theorems for periodic orbits near singular orbits. 
\begin{thm}[Existence of periodic orbits.]
\label{thm-periodic-1}
Consider (\ref{fast-slow}). Assume that there exist $\rho \in \mathbb{N}$, $\epsilon_0 > 0$ such that (\ref{fast-slow}) admits the following $\epsilon\ (\in [0,\epsilon_0])$-parameter family of sets in $\mathbb{R}^{n+1}$ (see also Fig. \ref{fig-theorem-main}-(a)):
\begin{description}
\item[$\mathcal{S}_\epsilon^j$: ] 
$j = 0,\cdots, \rho$: a fast-saddle-type block satisfying the covering-exchange property with respect to $\mathcal{F}_\epsilon^{j-1}$ defined below and for $T^{j-1} > 0$.
\item[$\mathcal{F}_\epsilon^j$: ] $j = 0,\cdots, \rho$: a fast-exit face of $\mathcal{S}_\epsilon^j$. Identify $\mathcal{F}_\epsilon^{-1}$ with $\mathcal{F}_\epsilon^\rho$.
\end{description}

Then the following statements hold.
\begin{enumerate}
\item When $\epsilon \in (0,\epsilon_0]$, (\ref{fast-slow}) admits a periodic orbit $\{\Gamma_\epsilon(t)\mid t\in \mathbb{R}\}$ with a sequence of positive numbers
\begin{equation*}
0 < t^0_f < t^1_s < t^1_f < \cdots < t^\rho_f < t^0_s \equiv T_\epsilon < \infty
\end{equation*}
such that 
\begin{align*}
&\Gamma_\epsilon(T_\epsilon) = \Gamma_\epsilon(0) \in \mathcal{F}_\epsilon^0,\\
&\Gamma_\epsilon([t^j_f,t^{j+1}_s])\subset \mathcal{S}_\epsilon^{j+1},\quad \Gamma_\epsilon(t^{j+1}_s)\in \mathcal{F}_\epsilon^{j+1}\quad (j = 0,\cdots, \rho-1)\quad \text{ and }\\
&\Gamma_\epsilon([t^\rho_f,T_\epsilon])\subset \mathcal{S}_\epsilon^0.
\end{align*}
\item When $\epsilon = 0$, (\ref{fast-slow}) admits the collection of heteroclinic orbits $\{\Gamma^j = \{x_{y_j}(t)\}_{t\in \mathbb{R}}\}_{j=0}^\rho$ with
\begin{align*}
&\lim_{t\to -\infty} x_{y_j}(t) = p_j \in \mathcal{S}_0^j\text{ with }f(p_j, y_j, 0) = 0,\\
&\lim_{t\to +\infty} x_{y_j}(t) = q_j \in \mathcal{S}_0^{j+1} \text{ with }f(q_j, y_j, 0) = 0 \quad (j = 0,\cdots, \rho-1),\\
&\lim_{t\to -\infty} x_{y_\rho}(t) = p_\rho \in \mathcal{S}_0^\rho\text{ with }f(p_\rho, y_\rho, 0) = 0,\\
&\lim_{t\to +\infty} x_{y_\rho}(t) = q_\rho \in \mathcal{S}_0^0 \text{ with }f(q_\rho, y_\rho, 0) = 0 \quad \text{ and }\\
&\Gamma^j \cap \mathcal{F}_\epsilon^j \not = \emptyset \quad (j = 0,\cdots, \rho)
\end{align*}
bridging $1$-dimensional normally hyperbolic invariant manifolds $S_0^j\subset \mathcal{S}^j_0$ and $S_0^{j+1}\subset \mathcal{S}^{j+1}_0$ for $j=0,\cdots, \rho-1$. In case that $j=\rho$ such a heteroclinic orbit connects $S_0^\rho\subset \mathcal{S}^\rho_0$ and $S_0^0\subset \mathcal{S}^0_0$. 
\end{enumerate}
\end{thm}

\begin{proof}
First consider the case $\epsilon \in (0,\epsilon_0]$. 
By Proposition \ref{prop-CE} and our assumption, there exist $h$-sets $\tilde{\mathcal{S}}_\epsilon^j\subset \mathcal{S}_\epsilon^j$ such that we get the sequence of covering relations
\begin{equation*}
\mathcal{F}_\epsilon^\rho \overset{ \varphi_\epsilon(T_\rho,\cdot)}{\Longrightarrow} \tilde{\mathcal{S}}_\epsilon^0 \overset{ P^0_\epsilon}{\Longrightarrow} \mathcal{F}_\epsilon^0 \overset{ \varphi_\epsilon(T_0,\cdot) }{\Longrightarrow} \tilde{\mathcal{S}}_\epsilon^1 \overset{P^1_\epsilon}{\Longrightarrow} \cdots \overset{ \varphi_\epsilon(T_{\rho-1},\cdot) }{\Longrightarrow} \tilde{\mathcal{S}}_\epsilon^\rho  \overset{ P^\rho_\epsilon}{\Longrightarrow} \mathcal{F}_\epsilon^\rho, 
\end{equation*}
where 
$P^j_\epsilon : \mathcal{S}_\epsilon^j\to \partial \mathcal{S}_\epsilon^j$ is the Poincar\'{e} map of $\mathcal{S}_\epsilon^j$. 
Then our statement is just the consequence of Proposition \ref{ZG-periodic}.
The case $\epsilon = 0$ is just the consequence of Proposition \ref{WZ-heteroclinic}.
All arguments with respect to horizontal and vertical disks are valid from Corollary \ref{cor-disk}.
\end{proof}

\begin{rem}\rm
Standard Exchange Lemma says that the period when trajectories stays near slow manifolds in $O(\epsilon)$-distance is $O(1/\epsilon)$ (actually $O(\exp(-c/\epsilon))$ for some $c > 0$, discussed in \cite{JKK}).
The period $T_\epsilon$ is thus expected to be $O(1/\epsilon)$ for $\epsilon \in (0,\epsilon_0]$. 
\end{rem}

\begin{figure}[htbp]\em
\begin{minipage}{0.5\hsize}
\centering
\includegraphics[width=6.0cm]{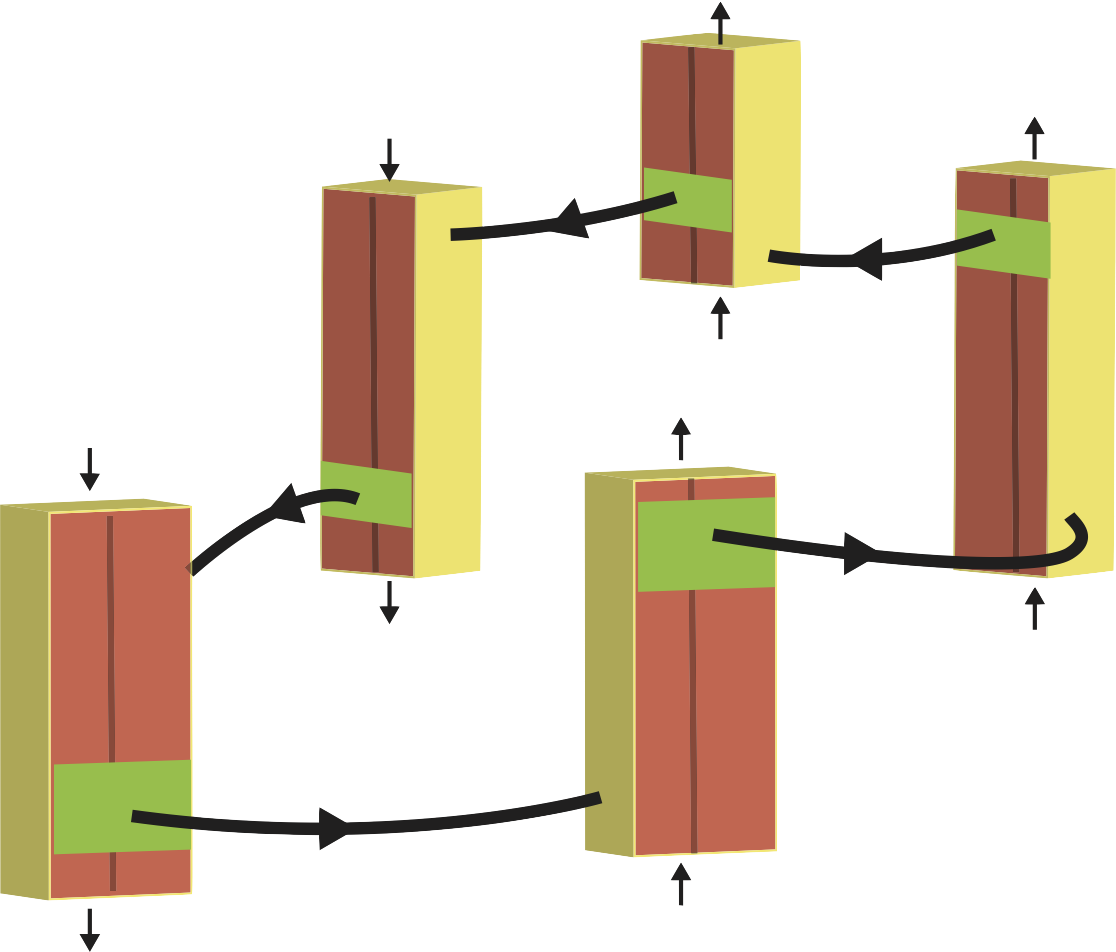}
(a)
\end{minipage}
\begin{minipage}{0.5\hsize}
\centering
\includegraphics[width=6.0cm]{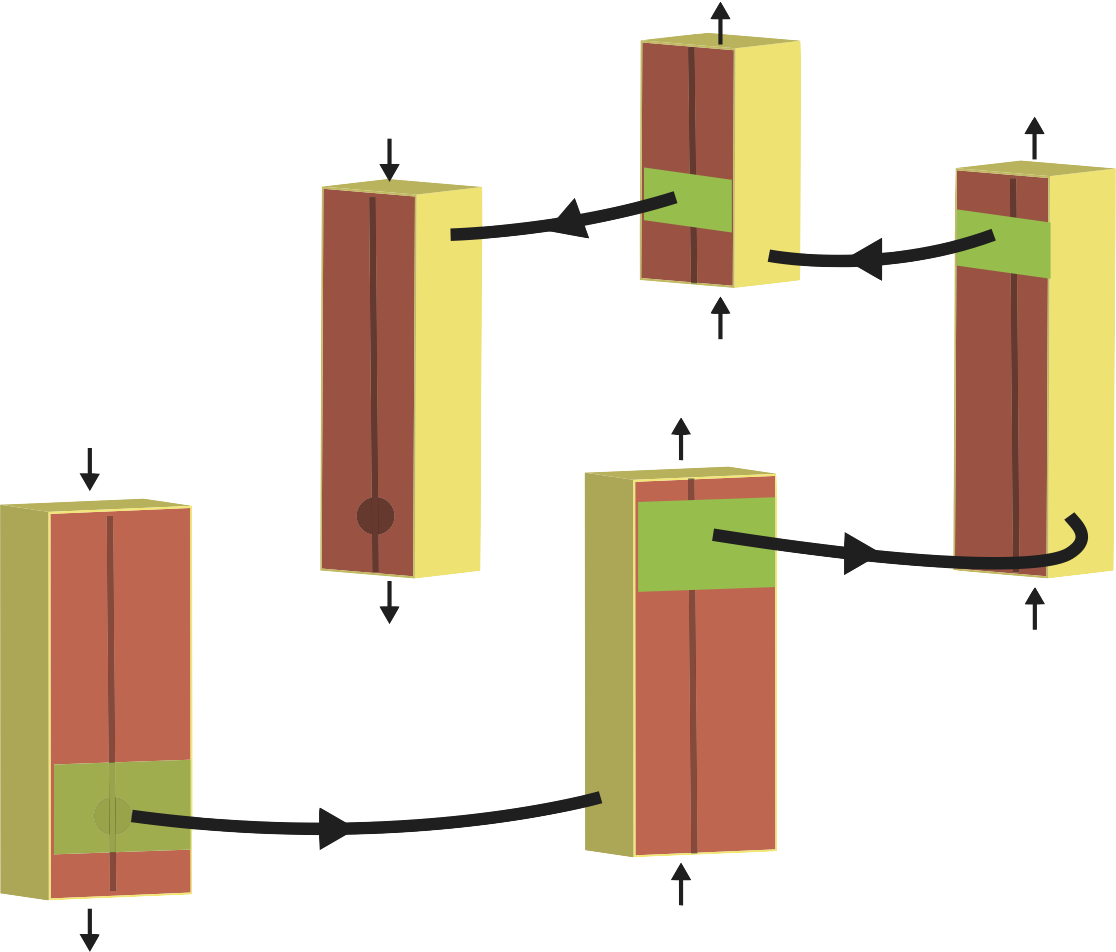}
(b)
\end{minipage}
\caption{Illustration of $\{\mathcal{S}^j_\epsilon, \mathcal{F}^j_\epsilon\}_{j=0}^\rho$ in Theorems \ref{thm-periodic-1} and \ref{thm-heteroclinic-1}, $\rho = 4$.}
\label{fig-theorem-main}
Rectangular parallelepipeds (colored red and yellow) are fast-saddle-type blocks with fast-exit faces (colored by green) satisfying the covering-exchange property. Vertical black arrows show the slow vector field. Bold black curves outside blocks describe heteroclinic orbits $\{\Gamma^j\}_{j=0}^\rho$. Black lines inside blocks correspond to limit critical manifolds. Their union generates a singular limit orbit $\Gamma_0$. Theorem \ref{thm-periodic-1} claims that there exists a family of periodic orbits $\{\Gamma_\epsilon\}_{\epsilon \in (0,\epsilon_0]}$ near $\Gamma_0$.
\end{figure}

% Theorem with slow shadowing sequence

Theorem \ref{thm-periodic-1} can be generalized as stated below, replacing covering-exchange pairs by covering-exchange sequences.

\begin{cor}[Validation of periodic orbits with slow shadowing]
\label{cor-periodic-2}
Consider (\ref{fast-slow}). Assume that there exist $\rho \in \mathbb{N}$, $\epsilon_0 > 0$ such that (\ref{fast-slow}) admits the following $\epsilon\ (\in [0,\epsilon_0])$-parameter family of sets in $\mathbb{R}^{n+1}$ (see also Fig. \ref{fig-thm-per-shadow}):
\begin{description}
\item[$\{N_\epsilon^{j,i}\}_{i=0,\cdots, m_j}^{j=0,\cdots, \rho}$: ] a sequence of $h$-sets which forms a covering-exchange sequence with $\mathcal{F}_\epsilon^{j-1}$ defined below.
\item[$\mathcal{F}_\epsilon^j$: ] $j = 0,\cdots, \rho$: a fast-exit face of $N_\epsilon^{j,m_j}$. Identify $\mathcal{F}_\epsilon^{-1}$ with $\mathcal{F}_\epsilon^\rho$.
\end{description}
Then all statements in Theorem \ref{thm-periodic-1} hold with $\mathcal{S}_\epsilon^j = \bigcup_{i=0}^{m_j}N_\epsilon^{j,i}$.
\end{cor}

\begin{proof}
We only consider the case $\epsilon \in (0,\epsilon_0]$.
By Propositions \ref{prop-CE-2-1}, \ref{prop-CE-2-2}, \ref{prop-CE-2-3} and our assumptions, there is a sequence of $h$-sets $\{M_\epsilon^{j,i}\}_{i=0,\cdots, m_j}^{j=0,\cdots, \rho}$ with $M_\epsilon^{j,i} \subset N_\epsilon^{j,i}$ which admits a sequence of covering relations
\begin{align*}
\mathcal{F}_\epsilon^\rho &\overset{\varphi_\epsilon(T_\rho,\cdot)}{\Longrightarrow}
 M_\epsilon^{0,0} \overset{ P^{0,0}_\epsilon}{\Longrightarrow}
 M_\epsilon^{0,1} \overset{ P^{0,1}_\epsilon}{\Longrightarrow}
 \cdots \overset{ P^{0,m_0-1}_\epsilon}{\Longrightarrow}
 M_\epsilon^{0,m_0}\\
 &\quad \overset{ P^{0,m_0}_\epsilon}{\Longrightarrow}
  \mathcal{F}_\epsilon^0 \overset{ \varphi_\epsilon(T_0,\cdot)}{\Longrightarrow}
 M_\epsilon^{1,0} \overset{ P^{1,0}_\epsilon}{\Longrightarrow}
    \cdots  \cdots \overset{ P^{\rho,m_\rho-1}_\epsilon}{\Longrightarrow}
 M_\epsilon^{\rho,m_\rho} \overset{ P^{\rho,m_\rho}_\epsilon}{\Longrightarrow}
  \mathcal{F}_\epsilon^\rho,
\end{align*}
where $P^{j,i}_\epsilon : (N^{j,i}_\epsilon)_{\leq \bar y_{j,i} + \bar h_{j,i}} \to \partial (N^{j,i}_\epsilon)_{\bar y_{j,i} + \bar h_{j,i}}$ is the Poincar\'{e} map in $(N^{j,i}_\epsilon)_{\leq \bar y_{j,i} + \bar h_{j,i}}$ with appropriate choices of $\bar y = \bar y_{j,i}$ and $\bar h = \bar h_{j,i}$ following Propositions \ref{prop-CE-2-1}, \ref{prop-CE-2-2} and  \ref{prop-CE-2-3}. 
Then our statement is just the consequence of Proposition \ref{ZG-periodic}.
\end{proof}

\begin{figure}[htbp]\em
\begin{minipage}{1\hsize}
\centering
\includegraphics[width=6.0cm]{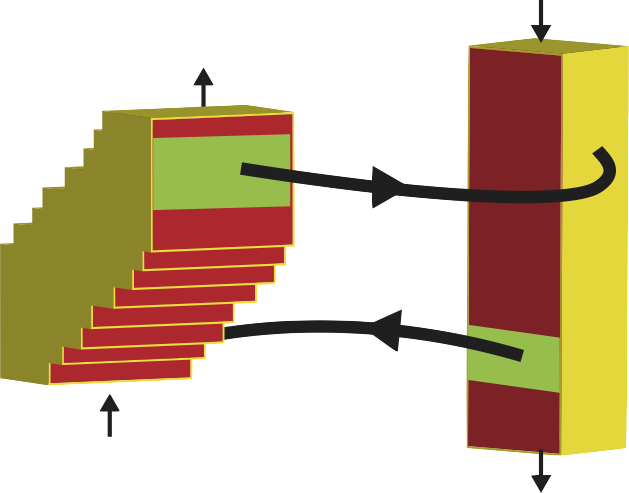}
\end{minipage}
\caption{Illustration of $\{\mathcal{S}^j_\epsilon, \mathcal{F}^j_\epsilon\}_{j=0}^\rho$ in Theorem \ref{cor-periodic-2}, $\rho = 2$.}
\label{fig-thm-per-shadow}
Rectangular parallelepipeds (colored red and yellow) are fast-saddle-type blocks with fast-exit faces (colored by green) satisfying the covering-exchange property and slow shadowing. 
Vertical black arrows show the slow vector field. Bold black curves outside blocks describe heteroclinic orbits $\{\Gamma^j\}_{j=0}^\rho$ and curves or surfaces inside blocks correspond to limit critical manifolds. Their union generates a singular limit orbit $\Gamma_0$. Theorem \ref{thm-periodic-1} claims that there exists a family of periodic orbits $\{\Gamma_\epsilon\}_{\epsilon \in (0,\epsilon_0]}$ near $\Gamma_0$.
\end{figure}

% Theorem with slow shadow sequence and m-cones

Note that points on the slow manifold $S_\epsilon^j$ are contained in all $h$-sets $\{M^{i,j}_\epsilon\}$ in the proof of Corollary \ref{cor-periodic-2}. This observation implies that true orbits in the full system (\ref{fast-slow}) {\em shadow} trajectories on $S_\epsilon^j$ via covering relations (cf. \cite{Rob}). This fact gives us a suggestion to describe true trajectories for fast-slow systems with multi-dimensional slow variables from the viewpoint of shadowing.

\par
Corollary \ref{cor-periodic-2} can be further generalized as stated below, at least in the case $u=s=1$, replacing covering-exchange sequences by covering-exchange sequences with extended cones.

\begin{cor}[Validation of periodic orbits with slow shadowing and $m$-cones]
\label{cor-periodic-3}
Consider (\ref{fast-slow}). Assume that there exist $\rho \in \mathbb{N}$, $\epsilon_0 > 0$ such that (\ref{fast-slow}) admits the following $\epsilon\ (\in [0,\epsilon_0])$-parameter family of sets in $\mathbb{R}^3$ (see also Fig. \ref{fig-theorem-main}):
\begin{description}
\item[$\{N_\epsilon^{j,i}\}_{i=0,\cdots, m_j}^{j=0,\cdots, \rho}$: ] a sequence of fast-saddle type blocks which admits a sequence of faxt-exit faces $\{N_\epsilon^{j, \exit}\}^{j=0,\cdots, \rho}$ with $N_\epsilon^{j, \exit}\subset N_\epsilon^{j, m_j}$. 

The collection $(\mathcal{F}_\epsilon^{j-1}, \{N_\epsilon^{j,i}\}_{i=0,\cdots, m_j}, N_\epsilon^{j, \exit}, C_{m^u_j}^u, C_{m^s_j}^s)$ forms a covering-exchange sequence with extended cones $C_{m^u_j}^u, C_{m^s_j}^s$.
Associating sets $C_{m^u_j}^u$, $C_{m^s_j}^s$ and $\mathcal{F}_\epsilon^j$ are defined below. 
\item[$C_{m^u_j}^u\ (j=0,\cdots, \rho)$:] the unstable $m^u_j$-cones of $N_\epsilon^{j,m_j}$. Identify $C_{m^u_{-1}}^u$ with $C_{m^u_\rho}^u$.
\item[$C_{m^s_j}^s\ (j=0,\cdots, \rho)$:] the stable $m^s_j$-cone of $N_\epsilon^{j,0}$.
\item[$\mathcal{F}_\epsilon^j\ (j = 0,\cdots, \rho)$:] a fast-exit face $(C_{m^u_j}^u)^\exit$ of $C_{m^u_j}^u$. Identify $\mathcal{F}_\epsilon^{-1}$ with $\mathcal{F}_\epsilon^\rho$.

\end{description}
Then all statements in Theorem \ref{thm-periodic-1} hold with $\mathcal{S}_\epsilon^j = C_{m^s_j}^s \cup \bigcup_{i=0}^{m_j}N_\epsilon^{j,i} \cup C_{m^u_j}^u$.
\end{cor}

\begin{proof}
Replace $\mathcal{F}_\epsilon^j$ and $M_\epsilon^{j,0}$ in the proof of Corollary \ref{cor-periodic-2} by $N_\epsilon^{j, \exit} \overset{P_\epsilon^{m^u_j}}{\Longrightarrow} \mathcal{F}_\epsilon^j $ and $\tilde M_\epsilon^{j,0}$, respectively. 
 Here $P_\epsilon^{m^u_j} : C_{m^u_j}^u \to \partial C_{m^u_j}^u$ is the Poincar\'{e} map and $\tilde M^{j,0}_\epsilon \subset (N_\epsilon^{j,0}\cup C_{m^s_j}^s)$ is the $h$-set corresponding to $\tilde M_1$ in Proposition \ref{prop-cov-stable}.
\end{proof}

\begin{rem}\rm
The key essence of our validation near slow manifolds consists of the following three pieces:
\begin{itemize}
\item Construction of fast-saddle-type blocks;
\item Stable and unstable cone conditions;
\item Slow shadowing condition (if necessary).
\end{itemize}
As long as validations of the above procedures pass, we can extend slow manifolds {\em in an arbitrary range} keeping the existence of points near slow manifolds which exits their neighborhoods after time $T=O(\epsilon)$, as stated in the Exchange Lemma.
\end{rem}

%
%	New Subsection
%
\subsection{Existence of connecting orbits}
Similar settings and arguments to the previous subsection yield the existence of heteroclinic orbits near singular orbits. 

\begin{thm}[Existence of heteroclinic orbits]
\label{thm-heteroclinic-1}
Consider (\ref{fast-slow}). Assume that there exist $\rho \in \mathbb{N}$ and $\epsilon_0 > 0$ such that (\ref{fast-slow}) admits the following $\epsilon\ (\in [0,\epsilon_0])$-parameter family of sets in $\mathbb{R}^{n+1}$ (see also Fig. \ref{fig-theorem-main}-(b)):
\begin{description}
\item[$\mathcal{S}_\epsilon^j$: ] $j = 0$: a fast-saddle-type block satisfying stable and unstable cone conditions. Moreover, it contains the nonempty maximal invariant sets $S_{\epsilon,u}$.\\
$j = \rho$: a fast-saddle-type block satisfying the covering-exchange property with respect to $\mathcal{F}_\epsilon^{\rho-1}$ (defined below) except condition 3. 
Moreover, it contains the nonempty maximal invariant set $S_{\epsilon,s}$. 
The set $S_{\epsilon,s}$ is contained in an attracting isolating block on the slow manifold in $\mathcal{S}_\epsilon^\rho$. 
\par
$j = 1,\cdots, \rho-1$: a fast-saddle-type block satisfying the covering-exchange property with respect to $\mathcal{F}_\epsilon^{j-1}$ defined below and for $T^{j-1} > 0$.
\item[$\mathcal{F}_\epsilon^j$: ] $j = 0,\cdots, \rho-1$: a fast-exit face of $\mathcal{S}_\epsilon^j$.\par
If $j=0$, the invariant set $S_{\epsilon,u}$ admits an isolating block $B(S_{\epsilon,u})$ on the slow manifold in $\mathcal{S}_\epsilon^0$ such that all assumptions in Lemma \ref{lem-fiber-cone} holds.
\end{description}

Then the following statements hold.
\begin{enumerate}
\item When $\epsilon \in (0,\epsilon_0]$, (\ref{fast-slow}) admits a heteroclinic orbit $\{\Gamma_\epsilon(t)\mid t\in \mathbb{R}\}$ with a sequence of positive numbers
\begin{equation*}
0 < t^0_f < t^1_s < t^1_f < \cdots < t^{\rho-1}_s < t^{\rho-1}_f < \infty
\end{equation*}
such that 
\begin{align*}
&\dist(\Gamma_\epsilon(t),S_{\epsilon.u}) \to 0\quad \text{ as }t\to -\infty,\\
&\Gamma_\epsilon((-\infty, 0]) \subset W^u(S_{\epsilon,u}) \subset \mathcal{S}_\epsilon^0,\quad \Gamma_\epsilon(0)\in \mathcal{F}_\epsilon^0,\\
&\Gamma_\epsilon([t^{j-1}_f,t^{j}_s])\subset \mathcal{S}_\epsilon^{j},\quad \Gamma_\epsilon(t^{j}_s)\in \mathcal{F}_\epsilon^j\quad (j = 1,\cdots, \rho-1),\\
&\Gamma_\epsilon([t^{\rho-1}_f,\infty))\subset W^s(S_{\epsilon,s}) \subset \mathcal{S}_\epsilon^\rho\quad \text{ and }\\
&\dist(\Gamma_\epsilon(t),S_{\epsilon.s})\to 0\quad \text{ as }t\to +\infty.
\end{align*}
\item When $\epsilon = 0$, (\ref{fast-slow}) admits the collection of heteroclinic orbits $\{\Gamma^j = \{x_{y_j}(t)\}_{t\in \mathbb{R}}\}_{j=0}^\rho$ with
\begin{align*}
&\lim_{t\to -\infty} x_{y_j}(t) = p_j \in \mathcal{S}_0^j\text{ with }f(p_j, y_j, 0) = 0,\\
&\lim_{t\to +\infty} x_{y_j}(t) = q_j \in \mathcal{S}_0^{j+1} \text{ with }f(q_j, y_j, 0) = 0,\\
&\Gamma^j\cap \mathcal{F}_\epsilon^j\not = \emptyset \quad (j = 0,\cdots, \rho-1)
\end{align*}
bridging $1$-dimensional normally hyperbolic invariant manifolds $S_0^j\subset \mathcal{S}^j_0$ and $S_0^{j+1}\subset \mathcal{S}^{j+1}_0$ for $j=0,\cdots, \rho-1$. 
\end{enumerate}
\end{thm}

\begin{proof}
Consider the case $\epsilon \in (0,\epsilon_0]$. 
We deal with the case that $B(S_{\epsilon,u})$ is attracting.
By Propositions \ref{prop-CE}, \ref{prop-fiber} and our assumption we get the sequence of covering relations
\begin{equation*}
\left\{B(S_{\epsilon,u}) \overset{\varphi_{\epsilon(t,\cdot)}}{\Longleftarrow} B(S_{\epsilon,u})  \right\}_k \overset{ r_1\circ b_{u,\epsilon}^\ast }{\Longrightarrow} \mathcal{F}^0_\epsilon \overset{ \varphi_\epsilon(T_0,\cdot) }{\Longrightarrow} \tilde{\mathcal{S}}_\epsilon^1 \overset{ P^1_\epsilon}{\Longrightarrow} \mathcal{F}_\epsilon^1 \overset{ \varphi_\epsilon(T_1,\cdot) }{\Longrightarrow} \tilde{\mathcal{S}}_\epsilon^2 \overset{P^2_\epsilon}{\Longrightarrow} \cdots \overset{ \varphi_\epsilon(T_{\rho-1},\cdot) }{\Longrightarrow} \mathcal{S}_\epsilon^\rho, 
\end{equation*}
for all sufficiently small $t>0$ and arbitrary $k\in \mathbb{N}$.
Here $r_\lambda$ is the deformation retract given by (\ref{retract-plus}) or (\ref{retract-minus}), $b_{u,\epsilon}^z$ is the horizontal disk at $z\in B(S_{\epsilon,u})$ whose graph is the unstable manifold $W^u(z)$ and $P^j_\epsilon : \mathcal{S}_\epsilon^j\to \partial \mathcal{S}_\epsilon^j$ is the Poincar\'{e} map of $\mathcal{S}_\epsilon^j$. 
The stable manifold $W^s(S_{\epsilon,s})$ is given by a vertical disk $b^s_\epsilon$ in $\mathcal{S}_\epsilon^\rho$ with $(b,y)$-coordinates as the $s(\mathcal{S}_\epsilon^\rho)$-dimensional direction of the $h$-set $\mathcal{S}_\epsilon^\rho$.
Our statement is just the consequence of Propositions \ref{prop-fiber} and \ref{WZ-heteroclinic}.
See Definition \ref{dfn-disks} about horizontal and vertical disks.
All arguments with respect to horizontal and vertical disks are valid from Corollaries \ref{cor-disk} and \ref{cor-fiber-rigorous}.
In the case that $B(S_{\epsilon,u})$ is repelling, the same arguments valid replacing the covering relation $\left\{B(S_{\epsilon,u}) \overset{\varphi_{\epsilon(t,\cdot)}}{\Longleftarrow} B(S_{\epsilon,u})  \right\}_k$ by $\left\{B(S_{\epsilon,u}) \overset{\varphi_{\epsilon(t,\cdot)}}{\Longrightarrow} B(S_{\epsilon,u})  \right\}_k$.
\par
The case $\epsilon = 0$ is the same as that of Theorem \ref{thm-periodic-1}.
Remark that we do not need the sequence of covering relations
\begin{equation*}
\left\{B(S_{\epsilon,u}) \overset{\varphi_{\epsilon(t,\cdot)}}{\Longleftarrow} B(S_{\epsilon,u})  \right\}_k \overset{ r_1\circ b_{u,\epsilon}^\ast }{\Longrightarrow} \mathcal{F}^0_\epsilon
\end{equation*}
to prove our statements in this case.
\end{proof}

% Theorem with slow shadowing sequence

Theorem \ref{thm-heteroclinic-1} can be generalized as stated below, replacing covering-exchange pairs by covering-exchange sequences.

\begin{cor}[Validation of heteroclinic orbits with slow shadowing]
\label{cor-heteroclinic-2}
Consider (\ref{fast-slow}). Assume that there exist $\rho \in \mathbb{N}$ and $\epsilon_0 > 0$ such that (\ref{fast-slow}) admits the following $\epsilon\ (\in [0,\epsilon_0])$-parameter family of sets in $\mathbb{R}^{n+1}$:
\begin{description}
\item[$\{N_\epsilon^{j,i}\}_{i=0,\cdots, m_j}^{j=0,\cdots, \rho}$ {\rm with }$m_0 = 0$.] $j=1,\cdots, \rho-1$ : a sequence of $h$-sets which forms a covering-exchange sequence with $\mathcal{F}_\epsilon^{j-1}$ defined below. 
\par
$j=\rho$ : a sequence of $h$-sets which forms a covering-exchange sequence with $\mathcal{F}_\epsilon^{\rho-1}$ defined below except the last assumption in Definition \ref{dfn-CE-seq}.
\par
Blocks $N_\epsilon^{0,0}$ and $N_\epsilon^{\rho,m_\rho}$ contain nonempty maximal invariant sets $S_{\epsilon,u}$ and $S_{\epsilon,s}$, respectively.
The invariant set $S_{\epsilon,s}$ is contained in an attracting isolating block on the slow manifold in $N_\epsilon^{\rho,m_\rho}$. 

\par
\item[$\mathcal{F}_\epsilon^j$: ] $j = 0,\cdots, \rho-1$: a fast-exit face of $N_\epsilon^{j,m_j}$. 

If $j=0$, the invariant set $S_{\epsilon,u}$ admits an isolating block $B(S_{\epsilon,u})$ on the slow manifold in $N_\epsilon^{0,0}$ such that all assumptions in Lemma \ref{lem-fiber-cone} holds.
\end{description}
Then all statements in Theorem \ref{thm-heteroclinic-1} hold with $\mathcal{S}_\epsilon^j = \bigcup_{i=0}^{m_j}N_\epsilon^{j,i}$.
\end{cor}

\begin{proof}
We only consider the case $\epsilon \in (0,\epsilon_0]$ and $B(S_{\epsilon,u})$ is attracting.
By Propositions \ref{prop-fiber}, \ref{prop-CE-2-1}, \ref{prop-CE-2-2}, \ref{prop-CE-2-3} and our assumptions, there is a sequence of $h$-sets $\{M_\epsilon^{j,i}\}_{i=0,\cdots, m_j}^{j=0,\cdots, \rho}$  with $M_\epsilon^{j,i} \subset N_\epsilon^{j,i}$ which admits a sequence of covering relations
\begin{align*}
\left\{B(S_{\epsilon,u}) \overset{\varphi_{\epsilon(t,\cdot)}}{\Longleftarrow} B(S_{\epsilon,u})  \right\}_k \overset{ r_1\circ b_{u,\epsilon}^\ast }{\Longrightarrow}  \mathcal{F}_\epsilon^0 &\overset{ \varphi_\epsilon(T_0,\cdot)}{\Longrightarrow}
 M_\epsilon^{1,0} \overset{ P^{1,0}_\epsilon}{\Longrightarrow}
 M_\epsilon^{1,1} \overset{ P^{1,1}_\epsilon}{\Longrightarrow}
 \cdots \overset{ P^{1,m_1-1}_\epsilon}{\Longrightarrow}
 M_\epsilon^{1,m_1}\\
 &\quad \overset{ P^{1,m_1}_\epsilon}{\Longrightarrow}
  \mathcal{F}_\epsilon^1 \overset{\varphi_\epsilon(T_1,\cdot)}{\Longrightarrow}
 M_\epsilon^{2,0} \overset{ P^{2,0}_\epsilon}{\Longrightarrow}
    \cdots  \cdots \overset{ P^{\rho,m_\rho-1}_\epsilon}{\Longrightarrow}
 M_\epsilon^{\rho,m_\rho}
\end{align*}
for arbitrary $k\in \mathbb{N}$ and sufficiently small $t>0$.
Here $r_\lambda$ is the deformation retract given by (\ref{retract-plus}) or (\ref{retract-minus}), $b_{u,\epsilon}^z$ is the horizontal disk at $z\in B(S_{\epsilon,u})$ whose graph is the unstable manifold $W^u(z)$.
Also, $P^{j,i}_\epsilon : (N^{j,i}_\epsilon)_{\leq \bar y_{j,i} + \bar h_{j,i}} \to \partial (N^{j,i}_\epsilon)_{\bar y_{j,i} + \bar h_{j,i}}$ is the Poincar\'{e} map in $(N^{j,i}_\epsilon)_{\leq \bar y_{j,i} + \bar h_{j,i}}$ with appropriate choices of $\bar y = \bar y_{j,i}$ and $\bar h = \bar h_{j,i}$ following Propositions \ref{prop-CE-2-1}, \ref{prop-CE-2-2} and  \ref{prop-CE-2-3}.
The stable manifold $W^s(S_{\epsilon,s})$ is given by a vertical disk $b^s_\epsilon$ in $N_\epsilon^{\rho,m_\rho}$ with $(b,y)$-coordinates as the $s(N_\epsilon^{\rho,m_\rho})$-dimensional direction of the $h$-set $N_\epsilon^{\rho,m_\rho}$.
Our statement is just the consequence of Propositions \ref{prop-fiber} and \ref{WZ-heteroclinic}.
All arguments with respect to horizontal and vertical disks are valid from Corollaries \ref{cor-disk} and \ref{cor-fiber-rigorous}.
\end{proof}

% Theorem with slow shadowing sequence and m-cones

Corollary \ref{cor-heteroclinic-2} can be further generalized as stated below, at least in the case $u=s=1$, replacing covering-exchange sequences by covering-exchange sequences with extended cones.

\begin{cor}[Validation of heteroclinic orbits with slow shadowing and $m$-cones]
\label{cor-heteroclinic-3}
Consider (\ref{fast-slow}). Assume that there exist $\rho \in \mathbb{N}$ and $\epsilon_0 > 0$ such that (\ref{fast-slow}) admits the following $\epsilon\ (\in [0,\epsilon_0])$-parameter family of sets in $\mathbb{R}^3$:
\begin{description}
\item[$\{N_\epsilon^{j,i}\}_{i=0,\cdots, m_j}^{j=0,\cdots, \rho}$ \text{\rm with }$m_0 = 0$.] $j=1,\cdots, \rho-1$ : a sequence of fast-saddle type blocks which admits a sequence of faxt-exit faces $\{N_\epsilon^{j, \exit}\}^{j=0,\cdots, \rho}$ with $N_\epsilon^{j, \exit}\subset N_\epsilon^{j, m_j}$.
Moreover, the collection $(\mathcal{F}_\epsilon^{j-1}, \{N_\epsilon^{j,i}\}_{i=0,\cdots, m_j}, N_\epsilon^{j, {\rm exit}}, C_{m^u_j}^u, C_{m^s_j}^s)$ forms a covering-exchange sequence with extended cones $C_{m^u_j}^u, C_{m^s_j}^s$.
Associating sets $C_{m^u_j}^u$, $C_{m^s_j}^s$ and $\mathcal{F}_\epsilon^j$ are defined below. 
\par
$j=\rho$ : a sequence of $h$-sets which forms a covering-exchange sequence with cones with $\mathcal{F}_\epsilon^{\rho-1}$ and $C_{m^s_\rho}^s$ defined below, except the last assumption in Definition \ref{dfn-CE-seq}.
\par
Blocks $N_\epsilon^{0,0}$ and $N_\epsilon^{\rho,m_\rho}$ contain nonempty maximal invariant sets $S_{\epsilon,u}$ and $S_{\epsilon,s}$, respectively.
The invariant set $S_{\epsilon,s}$ is contained in an attracting isolating block on the slow manifold in $N_\epsilon^{\rho,m_\rho}$. 

\item[$C_{m^u_j}^u\ (j=1,\cdots, \rho-1)$:] the unstable $m^u_j$-cones of $N_\epsilon^{j,m_j}$. 
\item[$C_{m^s_j}^s\ (j=1,\cdots, \rho)$:] the stable $m^s_j$-cone of $N_\epsilon^{j,0}$.
\item[$\mathcal{F}_\epsilon^j\ (j = 0,\cdots, \rho-1)$:] a fast-exit face $(C_{m^u_j}^u)^\exit$ of $C_{m^u_j}^u$. 

If $j=0$, the invariant set $S_{\epsilon,u}$ admits an isolating block $B(S_{\epsilon,u})$ on the slow manifold in $N_\epsilon^{0,0}$ such that all assumptions in Lemma \ref{lem-fiber-cone} holds with the fast-exit face $N_\epsilon^{0, \exit}$ of $N_\epsilon^{0,0}$. 

\end{description}
Then all statements in Theorem \ref{thm-heteroclinic-1} hold with 
\begin{equation*}
\mathcal{S}_\epsilon^j = 
\begin{cases}
\bigcup_{i=0}^{m_j}N_\epsilon^{j,i} \cup C_{m^u_j}^u = N_\epsilon^{0,0} \cup C_{m^u_0}^u & \text{ if $j=0$,}\\
C_{m^s_j}^s \cup \bigcup_{i=0}^{m_j}N_\epsilon^{j,i} & \text{ if $j=\rho$,}\\
C_{m^s_j}^s \cup \bigcup_{i=0}^{m_j}N_\epsilon^{j,i} \cup C_{m^u_j}^u & \text{ otherwise.}
\end{cases}
\end{equation*}
\end{cor}

\begin{proof}
Replace $\mathcal{F}_\epsilon^j$ and $M_\epsilon^{j,0}$ in the proof of Theorem \ref{cor-heteroclinic-2} by $N_\epsilon^{j, \exit} \overset{P_\epsilon^{m^u_j}}{\Longrightarrow} \mathcal{F}_\epsilon^j$ and $\tilde M_\epsilon^{j,0}$, respectively. 
Here $P_\epsilon^{m^u_j} : C_{m^u_j}^u \to \partial C_{m^u_j}^u$ is the Poincar\'{e} map and $\tilde M^{j,0}_\epsilon \subset (N_\epsilon^{j,0}\cup C_{m^s_j}^s)$ is the $h$-set corresponding to $\tilde M_1$ in Proposition \ref{prop-cov-stable}.
\end{proof}

\begin{rem}\rm
Changing the choice of covering-exchange sequences and fast-saddle-type blocks containing nontrivial invariant sets, we can obtain various type of singularly perturbed global orbits near singular orbits. 
For example, in Theorem \ref{thm-heteroclinic-1}, further assuming $\mathcal{S}_\epsilon^0 = \mathcal{S}_\epsilon^\rho$ and $S_{\epsilon,u} = S_{\epsilon,s} = \{p_\epsilon\}$, an equilibrium, then there exists an $\epsilon$-family of homoclinic orbits $\{H_\epsilon\}_{\epsilon\in (0,\epsilon_0]}$ of $p_\epsilon$. 

We can replace Condition 4 in the covering-exchange property, $\mathcal{F}_\epsilon^j \overset{\varphi_\epsilon(T^j,\cdot )}{\Longrightarrow}\tilde{\mathcal{S}}_\epsilon^{j+1}$, by a sequence of covering relations
\begin{equation*}
\mathcal{F}_\epsilon^j \overset{\varphi_\epsilon(T^j,\cdot )}{\Longrightarrow} M^j_1 \overset{\varphi_\epsilon(T^j_1,\cdot )}{\Longrightarrow} M^j_2 \overset{\varphi_\epsilon(T^j_2,\cdot )}{\Longrightarrow}  \cdots  \overset{\varphi_\epsilon(T^j_{k-1},\cdot )}{\Longrightarrow} M^j_k \overset{\varphi_\epsilon(T^j_k,\cdot )}{\Longrightarrow} \tilde{\mathcal{S}}_\epsilon^{j+1}
\end{equation*}
for $h$-sets $\{M^j_i\}_{i=1}^k$ and positive numbers $\{T^j_i\}_{i=1}^k$ to prove the same statements as Theorem \ref{thm-periodic-1} and \ref{thm-heteroclinic-1}. 
These are applications to Proposition \ref{ZG-periodic} or \ref{WZ-heteroclinic} and such extensions are useful for validating trajectories with complex behavior.
\end{rem}

%
%	New Section
%
\section{Sample validation results for FitzHugh-Nagumo system}
\label{section-examples}

In this section we provide several examples of singularly perturbed orbits with computer assistance. 
Our sample system is the FitzHugh-Nagumo system given by
\begin{equation}
\label{FN}
\begin{cases}
u' = v & \text{}\\
v' = \delta^{-1}(cv - f(u) + w) & \\
w' = \epsilon c^{-1}(u-\gamma w),
\end{cases}
\end{equation}
where $a\in (0,1/2)$, $c,\gamma$ and $\delta$ are positive parameters and $f(u) = u(u-a)(1-u)$. (\ref{FN}) is well-known as the system of traveling wave solutions $(U,W) = (\psi_U(x-ct), \psi_W(x-ct))$ of the following PDE: 
\begin{equation}
\label{FN-PDE}
\begin{cases}
U_t = U_{xx} + f(U) - W &\\
W_t = \epsilon(U-\gamma W) &
\end{cases}.
\end{equation}

%
%	New Subsection
%
\subsection{Strategy and parameters}
\label{section-strategy}
By following arguments in previous sections, we validate global orbits. 
The following implementation is basically common in our computations.
In particular, we concentrate on construction of covering-exchange sequences associated with slow shadowing sequences.

\begin{description}
\item[Step 1.] Fix $\epsilon_0 > 0$ and several parameters.

\item[Step 2.] For constructing the $j$-th slow shadowing sequence $\{N_\epsilon^{j,i}\}_{i=0}^{m_j}$ ($j=0,\cdots, \rho > 0$) with $\pi_y(N_\epsilon^{j,i}) = [y_{j,i}^-, y_{j,i}^+]$, we set identical positive numbers $a_0, b_0$ and $\bar h$ in (SS4) in advance. 
Before constructing each $N_\epsilon^{j,i}$, we compute approximate equilibria $\{(u_{j,i}, 0, w_{j,i})\}_{i=0}^{m_j}$ for (\ref{FN})$_0$ (i.e., $\epsilon = 0$) along a (normally hyperbolic) branch of the nullcline. 
For simplicity, compute them with $|w_{j, i+1}-w_{j,i}|\equiv \bar h$ for all $i$ so that $\{w_{j,i}\}_{i=0}^{m_j}$ is monotonously increasing (resp. decreasing) in the case $q=+1$ (resp. $q=-1$).

Also, let $|y_{j,i}^+ - w_{j,i}| = |w_{j,i} - y_{j,i}^-|\equiv H/2$ for some $H>0$ for further simplicity.
Set  $\pi_y (N_\epsilon^{j,\exit}) = [y_{j,m_j}^- + \bar h, y_{j,m_j}^+ - \delta]$ in the case of $q=+1$
Similarly, set $\pi_y (N_\epsilon^\exit) = [y_{j,m_j}^- +\delta, y_{j,m_j}^+ - \bar h]$ in the case of $q=-1$.
Here $\delta > 0$ denotes an arbitrarily small number.
For validation of (SS5), we apply Proposition \ref{prop-SS5}.
Under these settings, verify the slow shadowing condition (\ref{shadow}).
\par
In the case of validations of heteroclinic orbits, we additionally need to validate $W^u(I_\epsilon)$ for an isolating block $I_\epsilon$ on slow manifolds. To this end, verify all assumptions in Lemma \ref{lem-fiber-cone}. 
We revisit this verification later in Lemma \ref{lem-validation-unstable}.

\item[Step 3.] If necessary, verify the unstable $m$-cone condition for an appropriate $m > 1$ in a block of the form (\ref{set-verify-unstable-m-cone}) (Section \ref{section-m-cones}). Also, 
construct the fast-exit face $(C_m^u)^\exit$ of unstable $m$-cones following (\ref{exit-cone}).
Similarly, verify the stable $m$-cone condition of at the fast-entrance of blocks for an appropriate $m > 1$ if necessary. 
When we apply the predictor-corrector method for constructing fast-saddle-type block (Section \ref{section-block-pred-corr}), vertices of $m$-cones are slid in general.
In this case, we cut the cone so that the fast-exit face is parallel to the $y$-axis. 
As a consequence, the length $\ell_u$ or $\ell_s$ of extended cones shorten at most $|f_x(\bar x, \bar y)^{-1}f_y(\bar x, \bar y)H|$, where $H$ is the height of the block in the $y$-direction.
See Fig. \ref{fig-cut-cones}.
\begin{figure}[htbp]\em
\begin{minipage}{0.5\hsize}
\centering
\includegraphics[width=5.0cm]{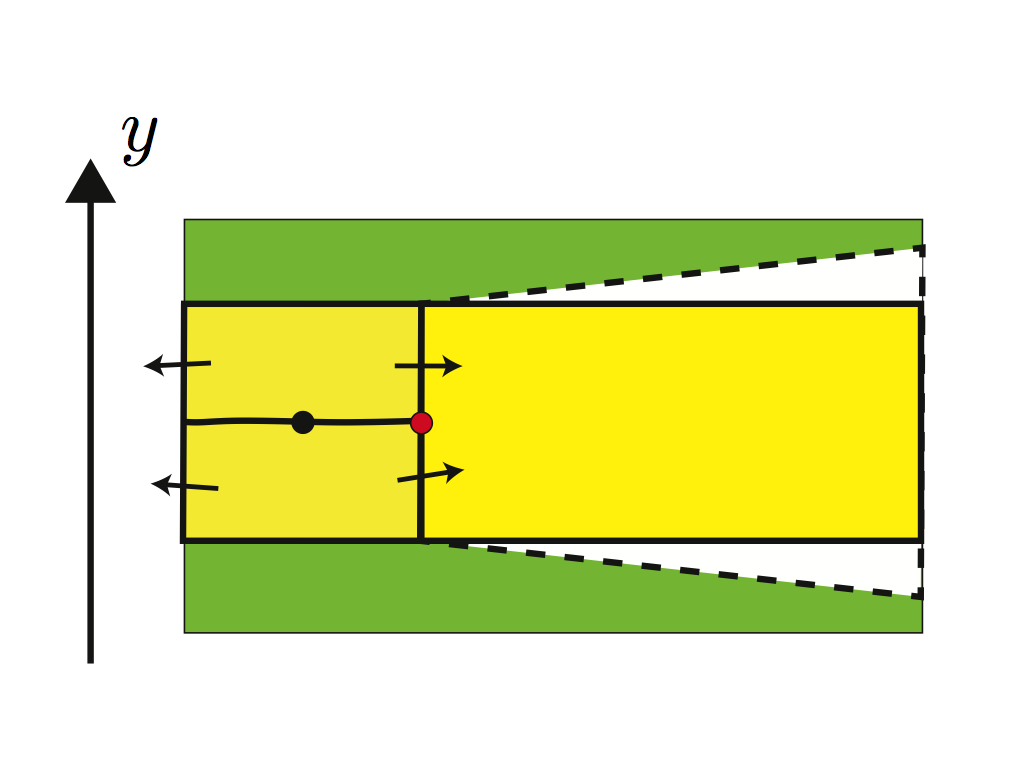}
(a)
\end{minipage}
\begin{minipage}{0.5\hsize}
\centering
\includegraphics[width=5.0cm]{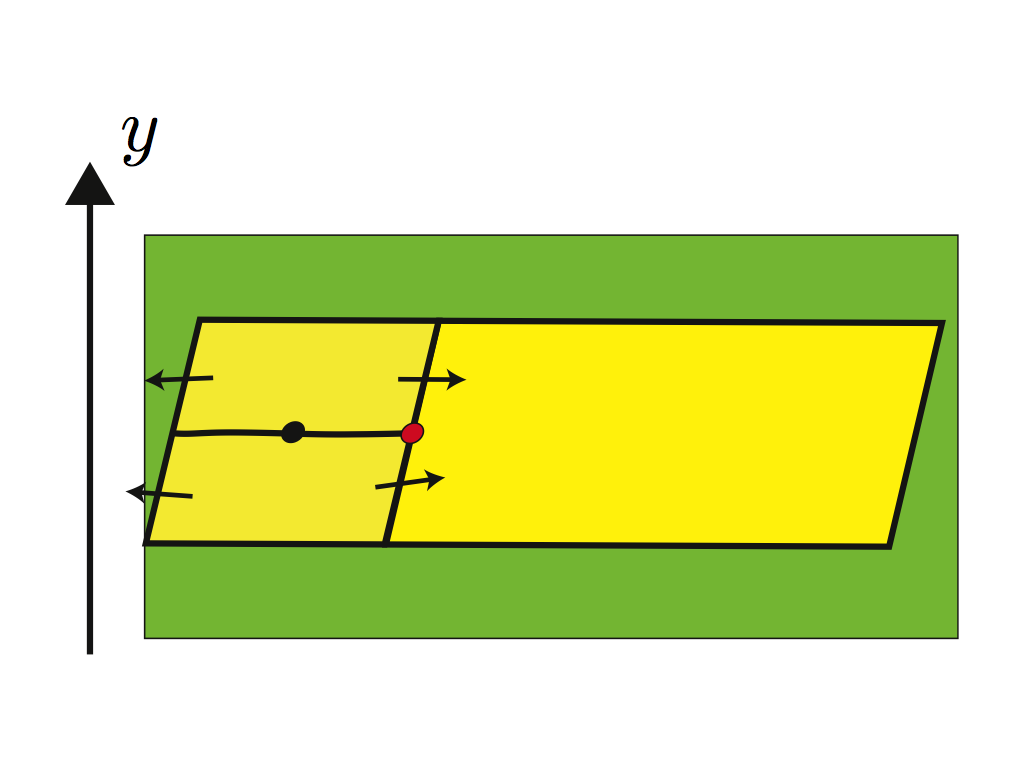}
(b)
\end{minipage}\par
\begin{minipage}{1\hsize}
\centering
\includegraphics[width=7.0cm]{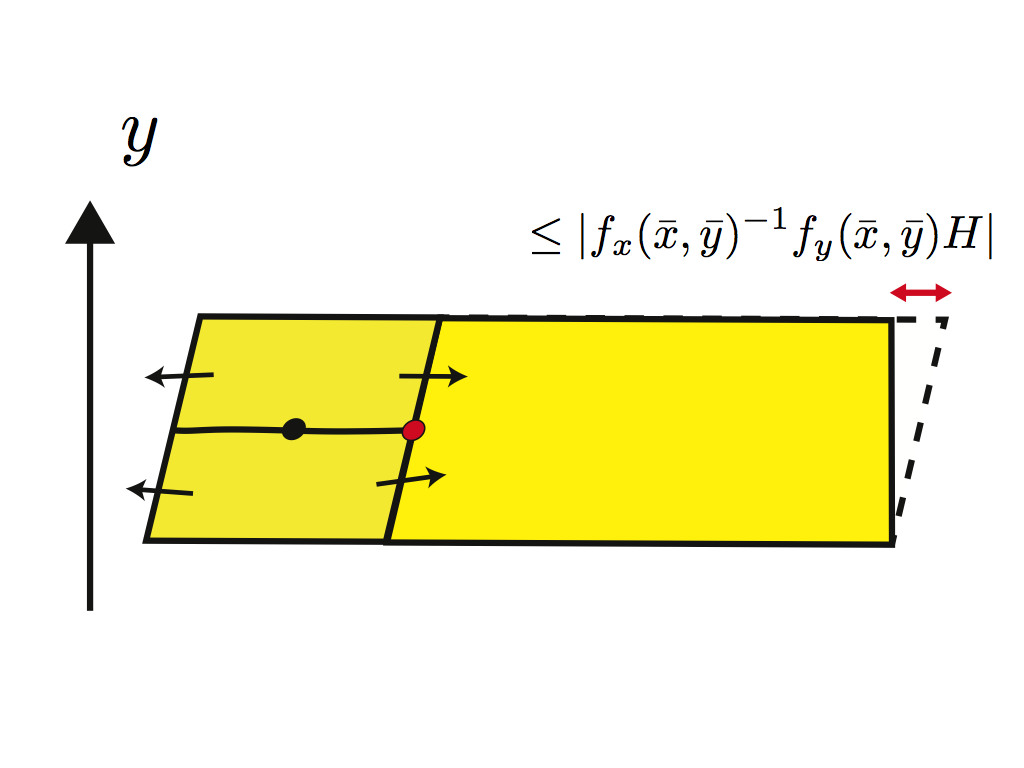}
(c)
\end{minipage}
\caption{Cutting $m$-cones}
\label{fig-cut-cones}
(a) : A fast-saddle-type block $N$ with the basic form (Section \ref{section-block-basic}) and associated unstable $m$-cone.
The union of yellow and white regions represents $N\cup C_m^u$. 
The yellow region represents $(N\cup C_m^u)\cap \{y\in \pi_y(N)\}$.
The $x$-components of block $N$ are identical for $y$.
\par
(b) : A fast-saddle-type block with the predictor-correctorc form (Section \ref{section-block-pred-corr}) and associated unstable $m$-cone.
The $x$-components of the block $N$ as well as the cone $C_m^u$ are slid, since the $x$-component of the center point also depends on $y$; $(x,y) = (\bar x + (dx/dy)(\bar y)\cdot (y-\bar y), y)$.
The fast-exit face thus has non-trivial angle to the $y$-axis.
\par
(c) : Cut the edge of unstable $m$-cones.
Resulting fast-exit face is parallel to the $y$-axis. The length $\ell_u$ of the extended cone shortens at most $|f_x(\bar x, \bar y)^{-1}f_y(\bar x, \bar y)H|$, where $H$ is the height of the block in the $y$-direction.
\end{figure}

\item[Step 4.] Solve initial value problems of ODEs for setting a fast-exit-face of each block as an initial data and verify (CE4) in Definition \ref{dfn-CE} or {\bf Drop} $N_\epsilon^{j-1,\exit} \overset{\varphi_\epsilon(T^j_k,\cdot )}{\Longrightarrow} (N_\epsilon^{j,0})_{\leq y_{j,0}^+-\bar h}$. 
This operation consists of direct applications of interval arithmetics and ODE solver libraries such as CAPD \cite{CAPD}. 
Although our computations here are operated in full systems, slow dynamics can be regarded as the small error since our interest in this step is mainly the behavior of fast dynamics.
\end{description}

\bigskip
Validation of assumptions in Lemma \ref{lem-fiber-cone} in Step 2 can be done as follows. 
\begin{lem}[Validation of assumptions in Lemma \ref{lem-fiber-cone}]
\label{lem-validation-unstable}
Let $N$ be a fast-saddle-type block with $\pi_y(N) = [y_N^-,y_N^+]$, and $I$ be an isolating block $I$ containing an equilibrium in the validated slow manifold $S_\epsilon$ with $0 < {\rm dist}(S_\epsilon, N^{f,-}) < {\rm diam}(\pi_a(N))-r_a$, where $r_a > 0$, and ${\rm dist}(S_\epsilon, N^{f,+}) \geq r_b > 0$.
Assume that $N$ satisfies the unstable $m_u$-cone condition with
\begin{equation}
\label{practical-stable}
\frac{{\rm diam}(\pi_a(N))-r_a}{m_u} < r_b.
\end{equation}
Also assume that $I$ can be chosen in 
\begin{equation}
\label{practical-block}
N\cap \left\{y_N^- + \frac{{\rm diam}(\pi_a(N))-r_a}{m_u}\leq y \leq y_N^+ - \frac{{\rm diam}(\pi_a(N))-r_a}{m_u} \right\}
\end{equation}
and that any subset $\tilde I\subset S_\epsilon$ containing $I$ is also an isolating block with the same isolating information as $I$. Then we can choose an isolating block $\bar I_\epsilon \subset S_\epsilon$ and the fast-exit face $N^{\rm exit}$ with $\pi_y(N^{\rm exit}) = [y^-,y^+]$ and 
\begin{equation}
\label{practical-exit}
y_N^- + 2\frac{{\rm diam}(\pi_a(N))-r_a}{m_u}\leq y^- < y^+  \leq y_N^+ - 2\frac{{\rm diam}(\pi_a(N))-r_a}{m_u}
\end{equation}
so that assumptions in Lemma \ref{lem-fiber-cone} holds. 
\end{lem}
\begin{proof}
For each $p\in I$, any point $q\in W^u(p)$ satisfies $|\pi_b(q)-\pi_b(p)| < ({\rm diam}(\pi_a(N))-r_a) / m_u$, which is the consequence of properties of unstable cones.
The same property holds for $y$-components.
Therefore assumptions concerning with (\ref{practical-stable}) and (\ref{practical-block}) implies (\ref{ass-fiber1}). 
Similarly, it immediately holds that we can choose $\bar I_\epsilon$ and $N^{\rm exit}$ with (\ref{practical-exit}) so that (\ref{ass-fiber2}) is satisfied.
\end{proof}

\bigskip
In practical computations, the most difficult part is Step 4.
In general, the larger both the fast-exit face $N_\epsilon^{j,\exit}$ and the target $h$-set in {\bf Drop}, i.e., $(N_\epsilon^{j,0})_{\leq y_{j,0}^+-\bar h}$, are, the easier validations of covering relations are.
However, if we validate covering relations described in Step 4, the resulting $h$-sets $M_\epsilon^{j,i}$ is very close to the slow exit, i.e., $(N_\epsilon^{j,i})_{y_{j,i}^+}$ (resp. $(N_\epsilon^{j,i})_{y_{j,i}^-}$) in case that $q=+1$ (resp. $q=-1$). 
In particular, the next fast-exit face $N_\epsilon^{j,\exit}$ becomes too thin to validate the next {\bf Drop}.
Procedures in Step 2 as well as Step 4 thus looks incompatible each other.
Nevertheless, an appropriate choice of the slow shadowing ratio $\chi$ avoids this inconsistency.

\begin{lem}[Validity of Step 2]
\label{lem-shadow-ratio}
For each $j=1,\cdots, \rho$, consider the $j$-th slow shadowing sequence $\{N_\epsilon^{j,i}\}_{i=0}^{m_j}$ in Step 2.
Assume that a covering relation $N_\epsilon^{j-1,\exit} \overset{\varphi_\epsilon(T_j,\cdot )}{\Longrightarrow} (N_\epsilon^{j,0})_{\leq y_{j,0}^+-\bar h}$ holds.

Let $\chi_j$ be the slow shadowing ratio satisfying
\begin{equation}
\label{choice-ratio}
\chi_j \leq 1- \frac{H}{\bar h}\left( \frac{|w_{j,m_j}-w_{j,0}|-H}{\bar h} \right)^{-1} = 1- \frac{H}{\bar h}\left( \frac{m_j \bar h-H}{\bar h} \right)^{-1}
\end{equation}
with $m_j \bar h > H$.
\par
Finally assume that $\{N_\epsilon^{j,i}\}_{i=0}^{m_j}$ is the slow shadowing sequence with the identical ratio $\chi_j$.
Then we can choose a fast-exit face $N_\epsilon^{j,\exit}$ so that 
$\pi_y (N_\epsilon^{j,\exit}) = [y_{j,m_j}^- + \bar h, y_{j,m_j}^+ - \delta]$ (resp. $\pi_y (N_\epsilon^\exit) = [y_{j,m_j}^- + \delta, y_{j,m_j}^+ - \bar h]$) in the case of $q=+1$ (resp. $q=-1$), where $\delta >0$ is an arbitrarily small number.
As a consequence, Step 2 and Step 4 are valid simultaneously.
\end{lem}

\begin{proof}
We only prove the case $q=+1$. The case $q=-1$ is similar.
\par
By Proposition \ref{prop-CE-2-2} with $\chi_j$, we can construct a covering relation $M_\epsilon^{j,0} \overset{P_\epsilon^{j,0}}{\Longrightarrow} M_\epsilon^{j,1}$, where
\begin{equation*}
M_\epsilon^{j,0}\subset (N_\epsilon^{j,0})_{y_{j,0}^+-\bar h}\quad \text{ and }\quad M_\epsilon^{j,1}\subset (N_\epsilon^{j,1})_{y_{j,0}^+-(1-\chi_j)\bar h} = (N_\epsilon^{j,1})_{y_{j,1}^+-(2-\chi_j)\bar h}. 
\end{equation*}
The last equality follows from the choice of $y_{j,i}^\pm$ and $w_{j,i}$ in Step 2.
Repeating this argument, two $h$-sets describing the $i$-th covering relation $M_\epsilon^{j,i-1}\ \overset{P_\epsilon^{j,i-1}}{\Longrightarrow} M_\epsilon^{j,i}$ is located on $(N_\epsilon^{j,i-1})_{I_{i-1}}$ and $(N_\epsilon^{j,i})_{I_{i}}$, respectively, where
\begin{equation*}
I_i = y_{j,i}^+ - \{(i+1)-i\chi_j\}\bar h.
\end{equation*}
We thus obtain $I_{m_j} = y_{j,m_j}^+ - \{(m_j+1)-m_j\chi_j\}\bar h$.
If the ratio $\chi_j$ can be chosen satisfying (\ref{choice-ratio}), we obtain
\begin{align*}
\{(m_j+1)-m_j\chi_j\}\bar h &\leq \left\{(m_j+1)-m_j \left\{1- \frac{H}{\bar h}\left( \frac{|w_{j,m_j}-w_{j,0}|-H}{\bar h} \right)^{-1} \right\} \right\} \bar h\\
	&= \left\{1 + \frac{m_j H}{\bar h}\left( \frac{m_j \bar h - H}{\bar h}\right)^{-1} \right\}\bar h.
\end{align*}
Consequently,
\begin{align*}
y_{j,m_j}^+ - \{(m_j+1)-m_j\chi_j\}\bar h &\leq y_{j,m_j}^+ - \bar h - \frac{m_j H}{\bar h} \left( \frac{m_j \bar h-H}{\bar h}\right)^{-1} \bar h\\
	&\leq y_{j,m_j}^+ - \bar h -  \left( \frac{m_j\bar h H}{m_j \bar h-H}\right)
	\leq y_{j,m_j}^+ - \bar h -  H = y_{j,m_j}^- - \bar h.
\end{align*}
Obviously, the slow shadowing pair with the ratio $\chi_j$ satisfies the slow shadowing condition with the ratio $\chi'$ for all $\chi'\in [\chi_j, 1]$.
Therefore, arranging several $\chi_{j,i}$'s in the slow shadowing pair $\{N_\epsilon^{j,i-1}, N_\epsilon^{j,i}\}$, we can take the $h$-set $M_\epsilon^{j,m_j}$ on $(N_\epsilon^{j,m_j})_{y_{j,i}^-}$.
Statements of the lemma follow from the same arguments as the proof of Proposition \ref{prop-CE-2-3}.
Schematic pictures of the proof are shown in Fig. \ref{fig-shadow-ratio}.
\end{proof}

\begin{figure}[htbp]\em
\begin{minipage}{0.48\hsize}
\centering
\includegraphics[width=6.0cm]{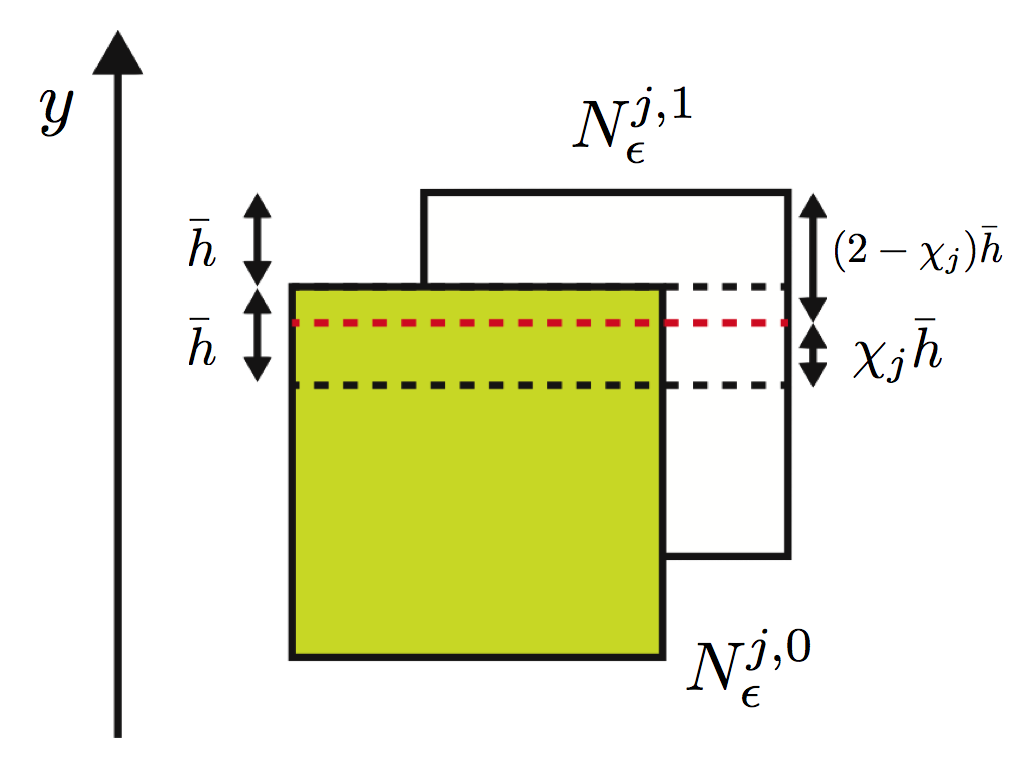}
(a)
\end{minipage}
\begin{minipage}{0.48\hsize}
\centering
\includegraphics[width=6.0cm]{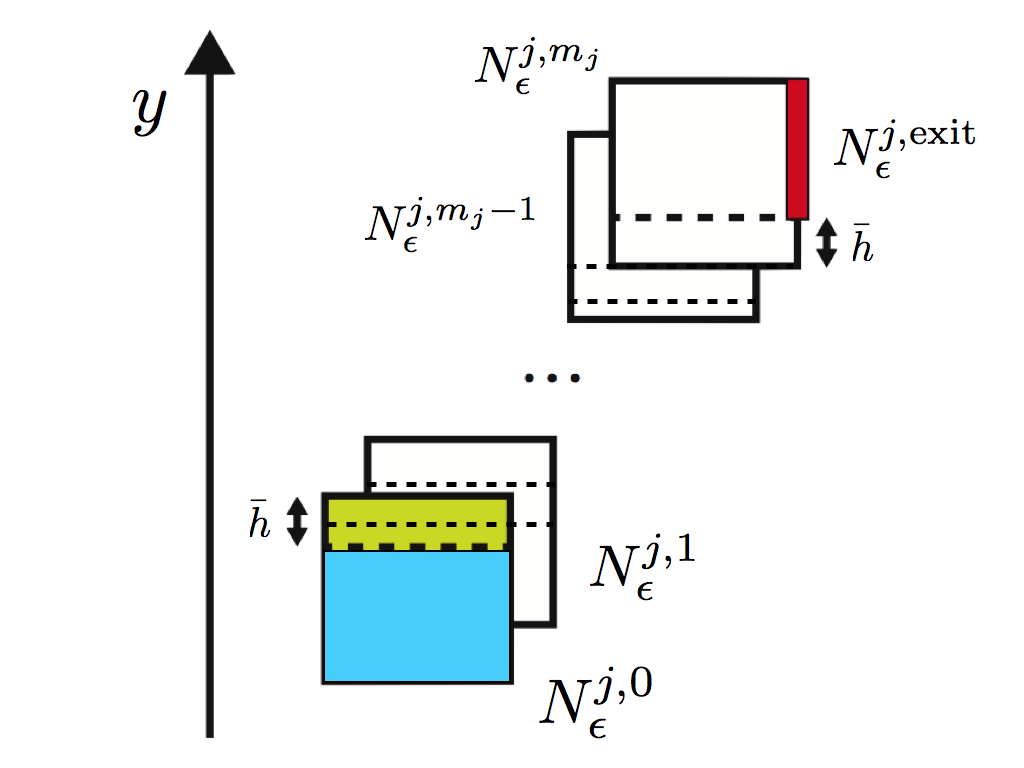}
(b)
\end{minipage}
\caption{Schematic pictures of Lemma \ref{lem-shadow-ratio} ($q=+1$)}
\label{fig-shadow-ratio}
Dotted lines denote sections $(N_\epsilon^{j,i})_y$ for each $i$. \\
(a) : If the sequence of fast-saddle-type blocks $\{N_\epsilon^{j,i}\}_{i=0}^{m_j}$ is a slow shadowing sequence with the ratio $\chi_j$, 
the section $(N_\epsilon^{j,i+1})_y$ where the covering relation $M_\epsilon^{j,i} \overset{P_\epsilon^{j,i}}{\Longrightarrow} M_\epsilon^{j,i+1} \subset (N_\epsilon^{j,i+1})_y$ holds gets lower than the usual version ($\chi_j = 1$).\\
(b) : Repeating the procedure in (a) sufficiently many times, we can take the section $(N_\epsilon^{j,m_j})_y$ before {\bf Jump} at the bottom of $N_\epsilon^{j,m_j}$; namely, $(N_\epsilon^{j,m_j})_{y_{j,m_j}^-}$. 
As a consequence, we can take the fast-exit face $N_\epsilon^{j,{\rm exit}}$ large keeping the target $h$-set $(N_\epsilon^{j,0})_{\leq y}$ in {\bf Drop} large.
\end{figure}

Thanks to Lemma \ref{lem-shadow-ratio}, we replace Step 2 by the following, which enables us to verify assumptions of results in Section \ref{section-existence} with large fast-exit faces and large target $h$-sets in {\bf Drop}:
\begin{description}
\item[Step 2'] Replace the statement \lq\lq verify the slow shadowing condition" in Step 2 by \lq\lq verify the slow shadowing condition with the slow shadowing ratio $\chi_j$", where $\chi_j < 1$ is a given number satisfying (\ref{choice-ratio}). 
\end{description}

If {\bf Steps 1, 2', 3 and 4} pass, then all assumptions of either Theorem \ref{thm-periodic-1}, Corollaries \ref{cor-periodic-2} or  \ref{cor-periodic-3} are satisfied in the case of periodic orbits for all $\epsilon \in (0,\epsilon_0]$. Similarly, all assumptions of either Theorem \ref{thm-heteroclinic-1}, Corollaries \ref{cor-heteroclinic-2} or \ref{cor-heteroclinic-3} are satisfied in the case of heteroclinic or homoclinic orbits.

\bigskip
In Step 2, the predictor-corrector approach discussed in Section \ref{section-block-pred-corr} is used for choosing local coordinates around (normally hyperbolic) invariant manifolds. 
A concrete form for validations is shown in Appendix \ref{appendix-concrete-FN}.

Note that all our examples below are cases $u = 1$. In such cases, one can directly verify covering relation $\mathcal{F}_\epsilon^{j-1} \overset{\varphi_\epsilon(T_j,\cdot )}{\Longrightarrow} M_\epsilon^{j,0}$ by checking all assumptions in Proposition \ref{prop-CR-u1}. 
In Step 4, we apply this proposition to validating covering relations.

\bigskip
We gather parameters we deal with computations in our settings except ones which arise in (\ref{FN}), before moving to practical computation examples.
These parameters are set for each branch of slow manifolds which we try to find.
Let $j\in \{0,1,\cdots, \rho\}$ be the number of branches.
\begin{itemize}
\item $\bar h$ : The height in the $w$-direction for slow shadowing condition (\ref{shadow}).
A sequence of equilibria $\{(u_i, v_i, w_i)\}$ for (\ref{layer}) are set so that $|w_{i+1}-w_i|\equiv \bar h$ for all $i$.
\item $H$ : The height of fast-saddle-type blocks in $w$-direction. Each $w$-interval of the length $H$ corresponds to the set $K$ in Section \ref{section-block-basic}.
\item $\bar w_0$ : A given number such that $|\bar w_0 - w_0| < \bar h$, where $w_0$ is the $w$-component of an equilibrium $(u_0, v_0, w_0)$. 
The equilibrium $(u_0, v_0, w_0)$ is computed numerically and becomes the center of the fast-saddle-type block $N_\epsilon^{j,0}$; the target $h$-set of {\bf Drop}.
\item $\bar w_{m_j}$ : A given number such that $|\bar w_{m_j} - w_{m_j}| < \bar h$, where $w_{m_j}$ is the $w$-component of an equilibrium $(u_{m_j}, v_{m_j}, w_{m_j})$. 
The equilibrium $(u_{m_j}, v_{m_j}, w_{m_j})$ is computed numerically and becomes the center of the fast-saddle-type block $N_\epsilon^{j,_{m_j}}$; the last $h$-set of the $j$-th slow shadowing sequence containing a fast-exit face.
\item $\chi$ : The slow shadowing ratio given by
\begin{equation*}
\chi = 1- \frac{H}{\bar h}\left( \frac{|\bar w_{m_j}- \bar w_{0}|-H-2\bar h}{\bar h} \right)^{-1}.
\end{equation*}
One can easily check that $\chi$ satisfies (\ref{choice-ratio}).
\item $r_a, r_b$ : The length of spaces in fast-saddle-type blocks introduced in (\ref{setting-shadow}). 
For simplicity, these numbers are identical for all blocks.
Moreover, they are assumed to be identical each other.
\item $d_a, d_b$ : Positive numbers less than $1$ introduced in (SS5).
\item $m_u, m_s$ : Positive numbers determining the sharpness of unstable and stable $m$-cones, respectively.
\item $\ell_u, \ell_s$ : Positive numbers determining the length $\ell$ of unstable and stable $m$-cones in (\ref{set-verify-unstable-m-cone}) and (\ref{set-verify-stable-m-cone}), respectively.
By using these numbers, we compute bounds of the departure time $T_{\rm dep}$ in (\ref{departure}) to construct the fast-exit face $(C_m^u)^{\rm exit}$ of the unstable $m$-cone in (\ref{exit-cone}), and the arrival time $T_{\rm arr}$ in (\ref{arrival}).
The arrival time $T_{\rm arr}$ is used to validate {\bf Drop} to the target $h$-set $(N_\epsilon^{j,0}\cup C_m^s)_{y_0^+-\bar h}$ in the $w$-direction corresponding to (\ref{ineq-y-cov}), as stated in Proposition \ref{prop-cov-stable}.
\par
If we apply the predictor-corrector form (Section \ref{section-block-pred-corr}) to constructing fast-saddle-type blocks, the practical length $\ell_u$ is set as $\ell_u - |f_u(\bar u)^{-1}|H$, following Step 3 and Fig. \ref{fig-cut-cones}. 
The factor $f_u(\bar u)^{-1}$ is the differential of $u$ by $w$ at the center point $(\bar u, \bar v, \bar w)$ via the implicit function differential for (\ref{FN}).
Details are shown in Appendix \ref{appendix-concrete-FN}.
\end{itemize}

\bigskip
All computations are done by MacBook Air 2011 model (1.6 GHz, Intel Core i5 Processor, 4GB Memory), GCC version 4.2.1 with {\tt -O2} option and CAPD library \cite{CAPD} version 3.0.

%
%	New Subsection
%
\subsection{Demonstration 1 : slow shadowing sequences with the ratio $\chi$}
\label{section-demo-shadow}
Validations of not only slow manifolds near critical manifolds consisting of equilibria for (\ref{layer}) but also the existence of trajectories which shadow slow manifolds are our starting points of whole considerations.
Slow manifolds for (\ref{FN}) is now expected to be near the nullcline $\{v=0, f(u) = w\}$. 
The aim of this subsection is to test how large slow manifolds can be validated in terms of slow shadowing sequence.
Following Step 2 at the beginning of this section, we validate slow shadowing sequences.
Note that validations in this section also verify {\bf Jump} in Proposition \ref{prop-CE-2-3}.

As a demonstration, we fix $a=0.3$, $\gamma = 10.0$, $\delta = 9.0$ and $c\in [0.799,0.801]$. 
These parameters are also used in Sections \ref{section-demo-homo} and \ref{section-homo-continuation}.

\begin{car}
\label{car-shadow}
Consider (\ref{FN}) with $a=0.3$, $\gamma = 10.0$ and $\delta = 9.0$. 
Then for all $c \in [0.799,0.801]$ and $\epsilon \in [0,5.0\times 10^{-5}]$, the branch of slow manifolds near the nullcline $\{v=0, f(u) = w\}$ in $\{-1.765629966434\times 10^{-1}\leq u \leq 2.017612584956\times 10^{-3}, -6.0\times 10^{-4}\leq w\leq 0.099\}$ is validated. 
In particular, the slow shadowing condition with $q=-1$ between blocks around this slow manifold is validated with parameters listed \lq\lq First branch" in Table \ref{table-shadowing}.

Similarly, the branch of slow manifolds near the nullcline $\{v=0, f(u) = w\}$ in $\{0.8504842978868\leq u \leq 1.021440903396, -1.58\times 10^{-2}\leq w\leq 0.07\}$ is validated. 
In particular, the slow shadowing condition with $q=+1$ between blocks around this slow manifold is validated with parameters listed \lq\lq Third branch" in Table \ref{table-shadowing}.
\end{car} 

\begin{center}
  \begin{table}[h]
    \begin{center}
      \begin{tabular}{|c|c|c|} \hline
	Parameters & First branch & Third branch \\ 
	\hline
	$\chi$ & $0.8807339449541285$ & $0.8786764705882354$ \\
	$\bar h$ & $0.003$ & $0.003$ \\
	$H$ & $0.0065$ & $0.0066$ \\
	$d_a$ & $0.75$ & $0.75$ \\
	$d_b$ & $0.7$ & $0.75$ \\
	$r_a$ & $0.008$ & $0.008$ \\
	$r_b$ & $0.0085$ & $0.0078$ \\
	$m_u$ & $100$ & $100$ \\
	$m_s$ & $100$ & $100$ \\
	computation time & $0.566$ sec. & $0.467$ sec. \\
        \hline
      \end{tabular}
    \end{center}
    \caption{Validation parameters of slow shadowing in Computer Assisted Result \ref{car-shadow}.}
    \label{table-shadowing}
  \end{table}
\end{center}

This validation result implies that we have already validated trajectories with appropriately chosen initial data, say $h$-sets, which shadow slow manifolds with an arbitrary length {\em for all $\epsilon \in (0,5.0\times 10^{-5}]$}, as long as slow shadowing are validated.
Moreover, {\bf Jump} has been also validated for any fast-exit face with an appropriate height from the bottom (for $q=+1$).
The height $H$ can be explicitly determined from (\ref{choice-ratio}) in our setting.
Remark that the range of our validating slow manifolds is not the limit of our verifications. 
Validations of slow shadowing sequences are just iterations of Step 2 and can be validated very fast, if we have fast solver of linear algebra.
Notice that slow shadowing sequences for $\epsilon\in (0,\epsilon_0]$ validate trajectories which shadow slow manifolds {\em without solving any differential equations for all $\epsilon\in (0,\epsilon_0]$}.

On the other hand, there is a trade-off for validating slow shadowing sequences. 
For example, if we raise the value of $\epsilon$, say $6.0\times 10^{-5}$, the slow shadowing condition (\ref{shadow}) violates, since the slow speed becomes faster than expansion and contraction of $h$-sets in hyperbolic directions around slow manifolds.
Factors determining (\ref{shadow}) are $\bar h$, $H$, $d_a$, $d_b$, $r_a$, $r_b$ as well as eigenvalues and size of fast-saddle-type blocks.
One expects that, the larger parameters, say $\bar h$, $r_a$, $r_b$ are, the easier the validation of (\ref{shadow}) will be.
However, in such a case, the covering relation in Proposition \ref{prop-SS5} is often violated.
In particular, Assumption (SS5) is violated.
This is mainly because the distance between two centers $(u_i, v_i, w_i)$ and $(u_{i+1}, v_{i+1}, w_{i+1})$ become larger and hence the affine map $T_{x,12}$ moves $h$-sets larger, if we increase $\bar h$, $r_a$, $r_b$. 
This is also the case if we increase parameters $d_a$ and $d_b$.

%
%	New Subsection
%
\subsection{Demonstration 2 : $m$-cones}
\label{section-demo-cones}
Next, we show a demonstration of $m$-cones. 
When we want to construct a covering-exchange sequence, we need to verify covering relation $\mathcal{F}_\epsilon^1 \overset{\varphi_\epsilon(T_0,\cdot )}{\Longrightarrow}\mathcal{S}_\epsilon$, 
where $\mathcal{F}_\epsilon^1$ is a fast-exit face of a fast-saddle type block and $\mathcal{S}_\epsilon$ is an other block. 
In order to verify this covering relation, we solve differential equation with the initial data $\mathcal{F}_\epsilon^1$. 
On the other hand, we can replace $\mathcal{F}_\epsilon^1$ by a fast-exit face $\mathcal{F}_\epsilon^2$ of an extended cone, thanks to discussions in Section \ref{section-m-cones}.
Here we solve (\ref{FN}) with two initial data, $\mathcal{F}_\epsilon^1$ and $\mathcal{F}_\epsilon^2$, to see the following two points:
\begin{enumerate}
\item accuracy of solution orbits, and
\item verification of covering relations.
\end{enumerate}
As an example, we set $a=0.01$, $\gamma = 0.0$, $\delta = 5.0$, $c\in [0.495, 0.505]$ and $\epsilon \in [0.0, 5.0\times 10^{-6}]$. These are parameter values used in Section \ref{section-demo-per}. 
In demonstrating computations, we used the ODE solver in CAPD library based on Lohner's method discussed in \cite{ZLoh}. 
The order of Taylor expansion is set $p=6$ and time step size is set $\Delta t = 0.0001$.
Computation result is shown in Fig. \ref{fig-demo-cones}.

\begin{figure}[htbp]\em
\begin{minipage}{0.48\hsize}
\centering
\includegraphics[width=6.0cm]{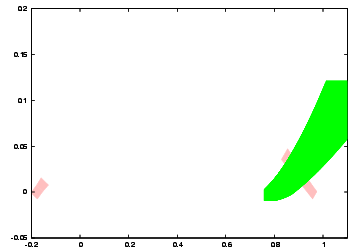}
(a)
\end{minipage}
\begin{minipage}{0.48\hsize}
\centering
\includegraphics[width=6.0cm]{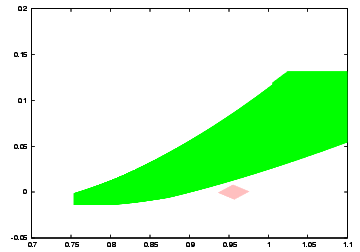}
(b)
\end{minipage}
\par

\bigskip
\begin{minipage}{0.32\hsize}
\centering
\includegraphics[width=5.0cm]{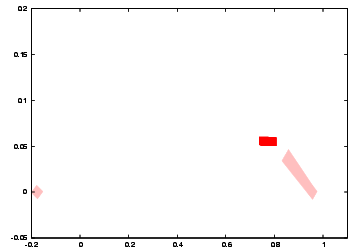}
(c)
\end{minipage}
\begin{minipage}{0.32\hsize}
\centering
\includegraphics[width=5.0cm]{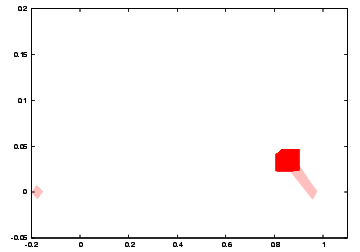}
(d)
\end{minipage}
\begin{minipage}{0.32\hsize}
\centering
\includegraphics[width=5.0cm]{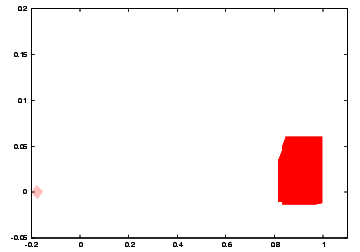}
(e)
\end{minipage}
\caption{Comparison of solution enclosures with and without $m$-cones.}
\label{fig-demo-cones}
Horizontal axis: $u$. Vertical axis: $v$. Each figure represents the projection of trajectories and cones on $(u,v)$-plane.
\par

\bigskip
(a) Validation of solution orbits (green) with initial data $\mathcal{F}_\epsilon^2$. 
Pink regions are the union of fast-saddle-type blocks and extended cones. Computation step of ODEs is $590$.
\par

\bigskip
(b) The same computation result as (a). 
In this case the stable $m$-cone is not validated. The tiny pink region around $(u,v) = (0.956721, 0)$ is the validated fast-saddle-type block. 
Readers see that enclosures of trajectories are much bigger than the fast-saddle-type block.
\par

\bigskip
(c) Validation of solution orbits (red) with initial data $\mathcal{F}_\epsilon^1$ (i.e. without unstable $m$-cones) and the same time steps as (a). Validated trajectories do not arrive at target region $\mathcal{S}_\epsilon$ yet.
\par

\bigskip
(d) Validated trajectories with additional time step computations to (c). More precisely, computation step of ODEs is $620$. Enclosures become bigger and bigger.
\par

\bigskip
(e) Validated trajectories with additional time step computations to (d). More precisely, computation step of ODEs is $650$. Enclosures are already bigger than stable cones, which implies that validation of covering relations can be never done.
\end{figure}

\bigskip
First we compare the case we use stable $m$-cones with the case we do not use stable $m$-cones. Fig. \ref{fig-demo-cones}-(a) and (b) show the same computational results. 
The only difference is whether or not a stable $m$-cone is validated. 
In this example, we validate the stable $29$-cone with the length $\ell_s = 0.108232$ of a fast-saddle-type block around $(0.956575, 0, 0.0392)\in \mathbb{R}^3$ (see Fig. \ref{fig-demo-cones}-(a)).
We solved ODE with initial data 
\begin{equation*}
\mathcal{F}_\epsilon^2 \cap \{w\in [0.039249948844,0.03926994850296]\}
\end{equation*}
after dividing it into uniform $30$ small pieces. 
Here $\mathcal{F}_\epsilon^2$ denotes the fast-exit face of unstable $21$-cone with the length $\ell_u = 0.0247787$ of a fast-saddle type block around the origin in $\mathbb{R}^3$. 
In this example, we solved ODE in $590$ steps.

In general, validated fast-saddle type blocks corresponding to $\mathcal{S}_\epsilon$ are very small, as shown in Fig. \ref{fig-demo-cones}-(b) (colored by pink). 
On the other hand, validated trajectories are quite bigger than blocks. 
In our example, validated trajectories are already bigger than the block, which implies that we can {\em never} validate $\mathcal{F}_\epsilon^1 \overset{\varphi_\epsilon(T_0,\cdot )}{\Longrightarrow}\mathcal{S}_\epsilon$ in this setting. 
A direct settlement of this problem would be a refinement of initial data, which leads to huge computational costs in many cases and is not realistic. Instead, we consider the problem with the help of stable $m$-cones, which is shown in Fig. \ref{fig-demo-cones}-(a).
In this case, the target block corresponding to $\mathcal{S}_\epsilon$ becomes big enough to validate covering relations. 
Thanks to Section \ref{section-m-cones}, we can discuss validation of trajectories with extended cones, which is much easier than verifications without cones.

\bigskip
Next, we compare the case which we use unstable $m$-cones with the case which we do not use unstable $m$-cones. 
Computation result with unstable $m$-cone is Fig. \ref{fig-demo-cones}-(a). 
Figs. \ref{fig-demo-cones}-(c), (d) and (e) show enclosure of trajectories with initial data $\mathcal{F}_\epsilon^1$, namely, a fast-exit face of a small fast-saddle-type block.
If we do not use unstable $m$-cones, we need to solve ODEs for longer time steps than the case we use unstable $m$-cones (Fig. \ref{fig-demo-cones}-(c)). Such extra computations cause additional computational errors and there is little hope to validate covering relations, as indicated in Figs. \ref{fig-demo-cones}-(d), (e).

%
%	New Subsection
%
\subsection{Periodic orbits}
\label{section-demo-per}
We go to validations of global orbits for (\ref{FN}).
Our first example is validation of periodic orbits. 
As a demonstration we set $a=0.01$, $\gamma = 0.0$ and $\delta = 5.0$ and $c\approx 0.5$. 
All validation of covering-exchange sequences with extended cones yield the following computer assisted result.

\begin{car}
\label{car-periodic}
Consider (\ref{FN}) with $a=0.01$, $\gamma = 0.0$ and $\delta = 5.0$. Then for all $c \in [0.495,0.505]$ the following trajectories are validated.  
\begin{enumerate}
\item At $\epsilon = 0$, there is a singular heteroclinic chain $H_0$ consisting of
\begin{itemize}
\item heteroclinic orbits from $p_0$ to $q_0$, and from $p_1$ to $q_1$, 
\item branches $M^0$, $M^1$ of nullcline $\{v=0, f(u) = w\}$. $M^0$ contains $p_0$ and $q_0$. Similarly, $M^1$ contains $p_1$ and $q_1$.
\end{itemize}
Equilibria $p_0$, $q_0$, $p_1$ and $q_1$ are validated by
\begin{align*}
&|\pi_{u,v}(p_0) - (-0.177098234,2.18166218\times 10^{-6})| < 2.49103628\times 10^{-2},\\
&|\pi_y (p_0) - 0.0395| < 3.25\times 10^{-3},\\
&|\pi_{u,v}(q_1) - (0.956125336,-5.98704406\times 10^{-6})| < 2.10571434\times 10^{-2},\\
&|\pi_y (q_1) - 0.03932| < 3.3\times 10^{-3},\\
&|\pi_{u,v}(p_1) - (0.850811351,-9.47162333\times 10^{-6})| < 2.44899980\times 10^{-2},\\
&|\pi_y (p_1)-0.10602| < 3.3\times 10^{-3},\\
&|\pi_{u,v}(q_0) - (-0.282970990, 1.85607735\times 10^{-6})| < 2.12552591\times 10^{-2},\\
&|\pi_y (q_0)-0.10675| < 3.25\times 10^{-3}.
\end{align*}
\item For all $\epsilon \in (0,5.0\times 10^{-6}]$, there exists a periodic orbit $H_\epsilon$ near $H_0$.
\end{enumerate}
Parameters for validations are listed in Table \ref{table-per}.
\end{car} 

We omit computation times for slow shadowing and {\bf Jump}, since they take only a few seconds as stated in Section \ref{section-demo-shadow}.

\begin{center}
  \begin{table}[h]
    \begin{center}
      \begin{tabular}{|c|c|c|} \hline
	Parameters & On $M^0$ ($q=-1$) & On $M^1$ ($q=+1$) \\ 
	\hline
	$\chi$ & $0.8922056384742952$ & $0.8895212587880817$ \\
	$\bar h$ & $0.0025$ & $0.0023$ \\
	$H$ & $0.0065$ & $0.0066$ \\
	$d_a$ & $0.8$ & $0.8$ \\
	$d_b$ & $0.7$ & $0.8$ \\
	$r_a$ & $0.008$ & $0.008$ \\
	$r_b$ & $0.0085$ & $0.008$ \\
	$m_u$  & $21$ (around $p_0$) & $55$ (around $p_1$)\\
	$m_s$ & $28$  (around $q_0$) & $29$ (around $q_1$)\\
	$\ell_u$ & $0.0247787$  (around $p_0$) & $0.0186724$ (around $p_1$)\\
	$\ell_s$ & $0.113375$  (around $q_0$) & $0.108232$ (around $q_1$)\\\hline
	 & $N_\epsilon^{0,\exit} \overset{\varphi_\epsilon(T_0,\cdot )}{\Longrightarrow} N_\epsilon^{1,0}$
	 & $N_\epsilon^{1,\exit} \overset{\varphi_\epsilon(T_1,\cdot )}{\Longrightarrow} N_\epsilon^{0,0}$ \\ \hline
	 $T_i$ & $0.059$ (with $\Delta t = 1.0\times 10^{-4}$) & $0.0266$ (with $\Delta t = 2.0\times 10^{-5}$) \\
	computation time & $70$ min. $6$ sec. & $78$ min. $53$ sec. \\
        \hline
      \end{tabular}
    \end{center}
    \caption{Validation parameters of slow shadowing in Computer Assisted Result \ref{car-periodic}.}
    \label{table-per}
  \end{table}
\end{center}

\begin{figure}[htbp]\em
\begin{minipage}{0.32\hsize}
\centering
\includegraphics[width=5.0cm]{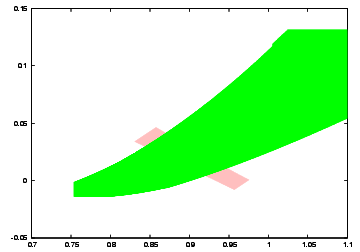}
(a-1)
\end{minipage}
\begin{minipage}{0.32\hsize}
\centering
\includegraphics[width=5.0cm]{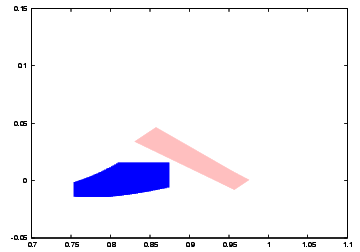}
(a-2)
\end{minipage}
\begin{minipage}{0.32\hsize}
\centering
\includegraphics[width=5.0cm]{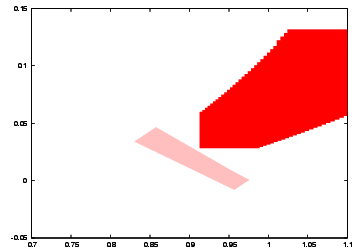}
(a-3)
\end{minipage}\par
\begin{minipage}{0.32\hsize}
\centering
\includegraphics[width=5.0cm]{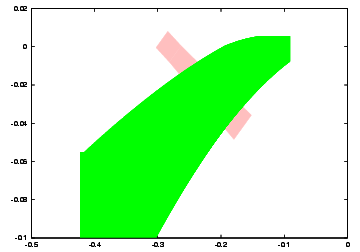}
(b-1)
\end{minipage}
\begin{minipage}{0.32\hsize}
\centering
\includegraphics[width=5.0cm]{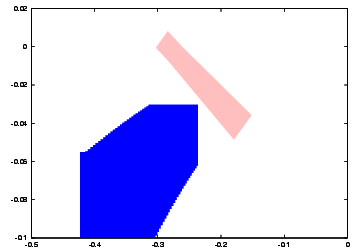}
(b-2)
\end{minipage}
\begin{minipage}{0.32\hsize}
\centering
\includegraphics[width=5.0cm]{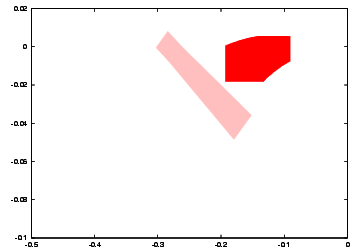}
(b-3)
\end{minipage}
\caption{Validation of covering-exchange sequences in Computer Assisted Result \ref{car-periodic}. }
\label{fig-periodic}
Horizontal axis: $u$. Vertical axis: $v$. Each figure represents the projection of trajectories and cones on $(u,v)$-plane. 
\par
(a-1) : Validation of $\varphi_\epsilon(T_0, (C_{m_u}^u)^\exit) \cap \{w \in [0.0365499948149, 0.0422098969893467]\}$.
\par
(a-2) : Validation of $\varphi_\epsilon(T_0, (C_{m_u}^u)^\exit) \cap \{w \in [0.0365499948149, 0.0365699944692267]\}$.
\par
(a-3) : Validation of $\varphi_\epsilon(T_0, (C_{m_u}^u)^\exit) \cap \{w \in [0.04218989733502, 0.0422098969893467] \}$.
\par
(b-1) : Validation of $\varphi_\epsilon(T_1, (C_{m_u}^u)^\exit) \cap \{w \in [0.1037942789937, 0.10913580929979]\}$.
\par
(b-2) : Validation of $\varphi_\epsilon(T_1, (C_{m_u}^u)^\exit) \cap \{w \in [0.1037942789937, 0.103814744627057]\} $.
\par
(b-3) : Validation of $\varphi_\epsilon(T_1, (C_{m_u}^u)^\exit) \cap \{w \in [0.109115343666433, 0.10913580929979] \}$.
\end{figure}

%
%	New Subsection
%
\subsection{Heteroclinic cycles}
\label{section-demo-cycle}
The second example is a family of heteroclinic cycles. 
Consider the cubic curve $w = f(u)$. One can see that it is symmetric with respect to the inflection point 
\begin{equation*}
(u_{\inf}, w_{\inf}) := \left( \frac{1+a}{3}, \frac{(1+a)(1-2a)(2-a)}{27}\right)
\end{equation*}
and we can choose $(u_3, w_3)$ from the curve $w=f(u)$ to be the point that is symmetric to the origin $(u,w) =(0,0)$ with respect to the inflection point. Deng \cite{D2} shows that $\gamma_0 := 9/(2-a)(1-2a)$ and $c_0:= (1-2a)/\sqrt{2}$ with $\delta = 1.0$ admit a heteroclinic loop of $(0,0)$ and $(u_3, w_3)$ for sufficiently small $\epsilon$. 
Symmetry of the cubic curve $w=f(u)$ with respect to $(u_{\inf}, w_{\inf})$ implies that the vector field (\ref{FN}) with $(\gamma, c) = (\gamma_0, c_0)$ is symmetric under the following transformation:
\begin{equation*}
\tilde u = \frac{2(1+a)}{3} - u,\quad \tilde v = -v,\quad \tilde w = \frac{2(1+a)(1-2a)(2-a)}{27} - w.
\end{equation*}
It is thus sufficient to validate a heteroclinic orbit from the origin to $(u_3, 0, v_3)$
if we want to validate heteroclinic cycles.
Here we validate heteroclinic cycles for concrete $\epsilon$ and specific $a$ and $\delta$.
All validations of covering-exchange sequence with extended cones yield the following computer assisted result.

\begin{car}
\label{car-cycle}
Consider (\ref{FN}) with $a=0.3$, $\gamma = \gamma_0$, $c =c_0$. Then, for $\delta =1.0$, the following trajectories are validated.  
\begin{enumerate}
\item At $\epsilon = 0$, there is a singular heteroclinic chain $H_0$ consisting of
\begin{itemize}
\item heteroclinic orbits from $p_0$ to $q_1$, and from $p_1$ to $q_0$, 
\item branches $M^0$, $M^1$ of nullcline $\{v=0, f(u) = w\}$. $M^0$ contains $p_0$ and $q_0$. Similarly, $M^1$ contains $p_1$ and $q_1$. 
\end{itemize}
Equilibria $p_0$ and $q_1$ are validated by
\begin{align*}
&|\pi_{u,v}(p_0) - (3.34681997\times 10^{-4}, -1.97213258\times 10^{-9})| < 1.35013586\times 10^{-2},\\
&|\pi_y (p_0) | < 1.5\times 10^{-3},\\
&|\pi_{u,v}(q_1) - (0.999441149,-1.32078036\times 10^{-5})| < 1.89700763\times 10^{-2},\\
&|\pi_y (q_1) -2.0\times 10^{-4}| < 2.5\times 10^{-3}.
\end{align*}
The equilibrium $p_0$ admits an attracting isolating block on $S_\epsilon$ contained in $\{(u,v,w)\in S_\epsilon  \cap \{u \leq 1.80738697\times 10^{-2}\} \mid w\in [-1.200\times 10^{-5}, 6.010\times 10^{-4}]\}$.

Note that $p_1$ is symmetric to $p_0$ with respect to the inflection point $(u_{\inf}, 0, w_{\inf})$. Similarly, $q_0$ is symmetric to $q_1$ with respect to the inflection point $(u_{\inf}, 0, w_{\inf})$.
\item For all $\epsilon \in (0,5.0\times 10^{-6}]$, there exists a heteroclinic cycle $H_\epsilon$ near $H_0$. 
The cycle $H_\epsilon$ consists of heteroclinic orbits from $p_0$ to $q_1$ and from $p_1$ to $q_0$.
\end{enumerate}
Parameters for validations are listed in Table \ref{table-cycle}.
Finally, assumptions of Lemma \ref{lem-validation-unstable} are validated with the attracting isolating block $\{w\in [-1.208036756757\times 10^{-3}, 1.208036756757\times 10^{-3}]\}$ on the slow manifold containing $p_0$ and the unstable $37$-cone condition. 
\end{car}

\begin{center}
  \begin{table}[h]
    \begin{center}
      \begin{tabular}{|c|c|c|} \hline
	Parameters & On $M^0$ ($q=-1$) & On $M^1$ ($q=+1$) \\ 
	\hline
	$\chi$ & $0.9358974358974359$ & $0.9168053244592347$ \\
	$\bar h$ & $0.0001$ & $0.0001$ \\
	$H$ & $0.003$ & $0.005$ \\
	$d_a$ & $0.8$ & $0.8$ \\
	$d_b$ & $0.7$ & $0.8$ \\
	$r_a$ & $0.003$ & $0.006$ \\
	$r_b$ & $0.0045$ & $0.008$ \\
	$m_u$  & $14$ (around $p_0$) & $-$ \\
	$m_s$ & $-$ & $2.2$ (around $q_1$)\\
	$2({\rm diam}(\pi_a(N))-r_a)/m_u$ & $4.271697297297\times 10^{-4}$ & $-$\\
	$\ell_u$ & $0.0151974$  (around $p_0$) & $-$ \\
	$\ell_s$ & $-$ & $0.0949713$ (around $q_1$)\\\hline
	 & $N_\epsilon^{0,\exit} \overset{\varphi_\epsilon(T_0,\cdot )}{\Longrightarrow} N_\epsilon^{1,0}$
	 & $N_\epsilon^{1,\exit} \overset{\varphi_\epsilon(T_1,\cdot )}{\Longrightarrow} N_\epsilon^{0,0}$ \\ \hline
	 $T_i$ & $0.045$ (with $\Delta t = 1.0\times 10^{-4}$) & $-$ \\
	 computation time & $43$ min. $51$ sec. & $-$ \\
        \hline
      \end{tabular}
    \end{center}
    \caption{Validation parameters of slow shadowing and isolating blocks in Computer Assisted Result \ref{car-cycle}.}
    \label{table-cycle}
  \end{table}
\end{center}

\begin{figure}[htbp]\em
\begin{minipage}{0.32\hsize}
\centering
\includegraphics[width=5.0cm]{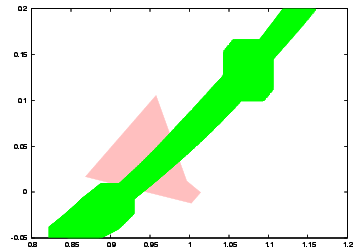}
(a-1)
\end{minipage}
\begin{minipage}{0.32\hsize}
\centering
\includegraphics[width=5.0cm]{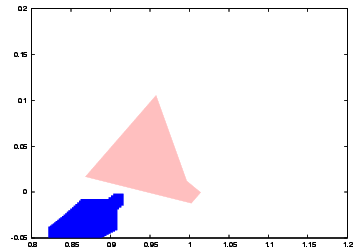}
(a-2)
\end{minipage}
\begin{minipage}{0.32\hsize}
\centering
\includegraphics[width=5.0cm]{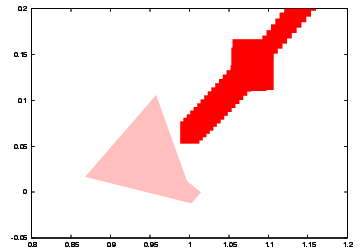}
(a-3)
\end{minipage}\par
\caption{Validation of covering-exchange sequences in Computer Assisted Result \ref{car-cycle}. }
\label{fig-computation-cycle}
Horizontal axis: $u$. Vertical axis: $v$. Each figure represents the projection of trajectories and cones on $(u,v)$-plane.
\par
(a-1) : Validation of $\varphi_\epsilon(T_0, (C_{m_u}^u)^\exit) \cap \{w \in [-0.00103331528021667, 9.35989145343001\times 10^{-4}]\}$.
\par
(a-2) : Validation of $\varphi_\epsilon(T_0, (C_{m_u}^u)^\exit) \cap \{w \in [-0.00103331528021667, -0.001023982083887]\}$.
\par
(a-3) : Validation of $\varphi_\epsilon(T_0, (C_{m_u}^u)^\exit) \cap \{w \in [9.26655949013333\times 10^{-4}, 9.35989145343001\times 10^{-4}] \}$.
\end{figure}

%
%	New Subsection
%
\subsection{Heteroclinic orbits}
\label{section-demo-hetero}
The third example is a family of heteroclinic orbits which are different from heteroclinic cycles.

In this example we apply the method in Section \ref{section-inv-set-on-mfd} to isolating blocks on slow manifolds, which validates equilibria on slow manifolds. 

Let $(\bar u, 0, \bar w) \not \approx (0,0,0)$ be a point in $\mathbb{R}^3$ such that $f(\bar u) = \bar w$ and that $\bar u = \gamma \bar w$. Assume that $f_u(\bar u)\not = 0$. Then there is a function $u = h(w)$ which is unique in a small neighborhood of $(\bar u,0, \bar w)$ such that $\bar u = h(\bar w)$, $f(h(\bar w)) = \bar w$ and $f(h(w)) = w$ hold in such a neighborhood. This is due to the Implicit Function Theorem.
Assuming that $f_u(u)\not = 0$ holds for all $u$ in a given neighborhood of $\bar u$, the above implicit function representation holds in the given neighborhood.
By using implicit function differential, the vector field $g(u,v,w) = c^{-1}(u-\gamma w)$ is rewritten by
\begin{equation*}
g(u,v,w) = \frac{1}{c}\left(\frac{1}{f_u(\bar u)}-\gamma\right)(w-\bar w) + \frac{1}{c}\left((u-\bar u)-\frac{w-\bar w}{f_u(\bar u)}\right)
\end{equation*}
near $(h(\bar w), 0, \bar w)\in \mathbb{R}^3$. Remark that $\bar u = \gamma \bar w$. This expression leads to an effective estimate of vector fields on slow manifolds, as stated in Section \ref{section-inv-set-on-mfd}.

All validation of slow shadowing sequences with extended cones and isolating blocks on slow manifolds yield the following computer assisted result.

\begin{car}
\label{car-heteroclinic}
Consider (\ref{FN}) with $a=0.2$, $\gamma =15.0$, $\delta = 5.0$. 
\begin{enumerate}
\item For each $c \in [0.947, 0.948]$, the following two kinds of trajectories are validated: 
\begin{enumerate}
\item At $\epsilon = 0$, there is a singular heteroclinic chain $H^0$ consisting of a heteroclinic orbit from $p_0 = (0,0,0)$ to $q_1$ and a branch $M^1$ of nullcline $\{v=0, f(u) = w\}$ connecting $q_1$ and $p_1 = (u_{m_1}^1,0,w_{m_1}^1)$. Equilibria have the following estimates:
\begin{align*}
&|\pi_{u,v}(p_0) - (8.00733430\times 10^{-4}, 2.18164458\times 10^{-7})| < 2.77714227\times 10^{-2},\\
&|\pi_y (p_0)-3.3\times 10^{-4}| < 2.75\times 10^{-3},\\
&|\pi_{u,v}(q_1) - (1.00015243, -1.21520324\times 10^{-5})| < 2.11148288\times 10^{-2},\\
&|\pi_y (q_1) + 4.6\times 10^{-4}| < 3.0\times 10^{-3},\\
&|\pi_{u,v}(p_1) - (0.882676299, 6.72075379\times 10^{-7})| < 2.64847328\times 10^{-2},\\
&|\pi_y (p_1)-0.0596| < 3.0\times 10^{-3}.
\end{align*}
\item For all $\epsilon \in (0,1.0\times 10^{-5}]$, there exists a heteroclinic orbit $H^1_\epsilon$ from  $p_0$ to $p_1$ near $H^1_0$. 
The equilibrium $p_1$ admits an attracting isolating block on $S_\epsilon$ contained in $\{(u,v,w)\in S_\epsilon  \cap \{u \geq 0.848998902\} \mid w\in [5.9080\times 10^{-2}, 5.9988\times 10^{-2}]\}$.
\end{enumerate}
Parameters for validations are listed in Table \ref{table-hetero-ptoq}.
Finally, assumptions of Lemma \ref{lem-validation-unstable} are validated with the attracting isolating block $\{w\in [-2.11442375\times 10^{-3}, 2.77442375\times 10^{-3}]\}$ on the slow manifold containing $p_0$ and the unstable $80$-cone condition. 
\item For each $c \in [0.1995, 0.2005]$, there exist the following two kinds of trajectories: 
\begin{enumerate}
\item At $\epsilon = 0$, there is a singular heteroclinic chain $H^2_0$ consisting of a heteroclinic orbit from $p_1$ to $q_0$ and a branch $M^0$ of nullcline $\{v=0, f(u) = w\}$ connecting $q_0$ and $p_0$. Equilibria have the following estimates:
\begin{align*}
&|\pi_{u,v}(p_1) - (0.876694738, 2.76020695\times 10^{-6})| < 2.69669311\times 10^{-2},\\
&|\pi_y (p_1)-0.0618| < 4.0\times 10^{-3},\\
&|\pi_{u,v}(q_0) - (-0.128335342,1.50674261\times 10^{-5})| < 1.52471761\times 10^{-2},\\
&|\pi_y (q_0) - 0.0622| < 4.0\times 10^{-3},\\
&|\pi_{u,v}(p_0) - (-2.67285871\times 10^{-3}, 1.20922437\times 10^{-6})| < 2.32761788\times 10^{-2},\\
&|\pi_y (p_0) + 1.2\times 10^{-3}| < 4.0\times 10^{-3}.
\end{align*}
\item For all $\epsilon \in (0,1.0\times 10^{-5}]$, there exists a heteroclinic orbit $H^2_\epsilon$ from  $p_1$ to $p_0$ near $H^2_0$.  The equilibrium $p_0$ admits an attracting isolating block on $S_\epsilon$ contained in $\{(u,v,w)\in S_\epsilon  \cap \{u \leq 1.06035231\times 10^{-2}\} \mid w\in [-1.040\times 10^{-4}, 1.264\times 10^{-3}]\}$.
\end{enumerate}
Parameters for validations are listed in Table \ref{table-hetero-qtop}.
Finally, assumptions of Lemma \ref{lem-validation-unstable} are validated with the attracting isolating block $\{w\in [5.800831058824 \times 10^{-2}, 6.51916894117 \times 10^{-2}]\}$ on the slow manifold containing $p_1$ and the unstable $85$-cone condition. 
\end{enumerate}
\end{car}

In this case we have to care about the existence of an equilibrium on the slow manifold in $S_\epsilon^1$. Easy calculations yield that $S_\epsilon^1$ possesses at most one equilibrium.

\begin{center}
  \begin{table}[h]
    \begin{center}
      \begin{tabular}{|c|c|c|} \hline
	Parameters & On $M^0$ ($q=-1$) & On $M^1$ ($q=+1$)\\ 
	\hline
	$\chi$ & $0.8949178448605274$ & $0.8877665544332212$ \\
	$\bar h$ & $0.00023$ & $0.00022$ \\
	$H$ & $0.0055$ & $0.006$ \\
	$d_a$ & $0.8$ & $0.75$ \\
	$d_b$ & $0.8$ & $0.8$ \\
	$r_a$ & $0.008$ & $0.008$ \\
	$r_b$ & $0.0085$ & $0.008$ \\
	$m_u$  & $50$ (around $p_0$) & $-$ \\
	$m_s$ & $-$ & $25.5$ (around $q_1$)\\
	$2({\rm diam}(\pi_a(N))-r_a)/m_u$ & $4.169025\times 10^{-4}$ & $-$\\
	$\ell_u$ & $0.0136582$  (around $p_0$) & $-$ \\
	$\ell_s$ & $-$ & $0.117888$ (around $q_1$)\\\hline
	 & $N_\epsilon^{0,\exit} \overset{\varphi_\epsilon(T_0,\cdot )}{\Longrightarrow} N_\epsilon^{1,0}$
	 & $N_\epsilon^{1,\exit} \overset{\varphi_\epsilon(T_1,\cdot )}{\Longrightarrow} N_\epsilon^{0,0}$ \\ \hline
	 $T_i$ & $0.064$ (with $\Delta t = 1.0\times 10^{-4}$) & $-$ \\
	 computation time & $38$ min. $21$ sec. & $-$ \\
        \hline
      \end{tabular}
    \end{center}
    \caption{Validation parameters of slow shadowing and isolating blocks in Computer Assisted Result \ref{car-heteroclinic} with $c\in [0.947, 0.948]$.}
    \label{table-hetero-ptoq}
  \end{table}
\end{center}

\begin{center}
  \begin{table}[h]
    \begin{center}
      \begin{tabular}{|c|c|c|} \hline
	Parameters & On $M^0$ ($q=-1$) & On $M^1$ ($q=+1$)\\ 
	\hline
	$\chi$ & $0.8484848484848486$ & $0.8540145985401461$ \\
	$\bar h$ & $0.0001$ & $0.0002$ \\
	$H$ & $0.008$ & $0.008$ \\
	$d_a$ & $0.8$ & $0.75$ \\
	$d_b$ & $0.8$ & $0.8$ \\
	$r_a$ & $0.0042$ & $0.006$ \\
	$r_b$ & $0.0045$ & $0.0068$ \\
	$m_u$ & $-$ & $20$ (around $p_1$)  \\
	$m_s$  & $26.5$ (around $q_0$) & $-$\\
	$2({\rm diam}(\pi_a(N))-r_a)/m_u$ & $-$ & $4.083105882353\times 10^{-4}$ \\
	$\ell_u$ & $-$ & $0.0539878$  (around $p_1$)\\
	$\ell_s$  & $0.113106$ (around $q_0$) & $-$\\\hline
	 & $N_\epsilon^{0,\exit} \overset{\varphi_\epsilon(T_0,\cdot )}{\Longrightarrow} N_\epsilon^{1,0}$
	 & $N_\epsilon^{1,\exit} \overset{\varphi_\epsilon(T_1,\cdot )}{\Longrightarrow} N_\epsilon^{0,0}$ \\ \hline
	 $T_i$ & $-$ & $0.055$ (with $\Delta t = 1.0\times 10^{-4}$) \\
	 computation time & $-$ & $39$ min. $41$ sec. \\
        \hline
      \end{tabular}
    \end{center}
    \caption{Validation parameters of slow shadowing and isolating blocks in Computer Assisted Result \ref{car-heteroclinic} with $c\in [0.1995, 0.2005]$.}
    \label{table-hetero-qtop}
  \end{table}
\end{center}

\begin{figure}[htbp]\em
\begin{minipage}{0.32\hsize}
\centering
\includegraphics[width=5.0cm]{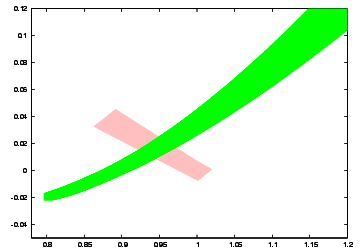}
(a-1)
\end{minipage}
\begin{minipage}{0.32\hsize}
\centering
\includegraphics[width=5.0cm]{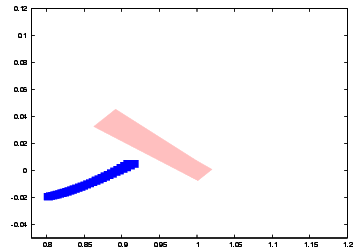}
(a-2)
\end{minipage}
\begin{minipage}{0.32\hsize}
\centering
\includegraphics[width=5.0cm]{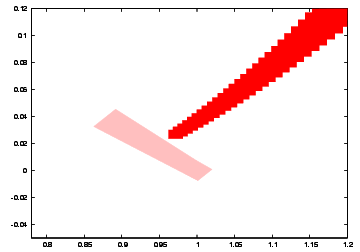}
(a-3)
\end{minipage}\par
\begin{minipage}{0.32\hsize}
\centering
\includegraphics[width=5.0cm]{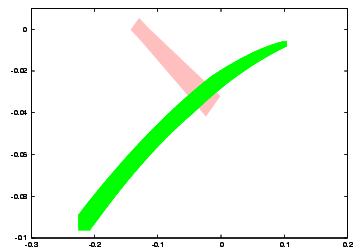}
(b-1)
\end{minipage}
\begin{minipage}{0.32\hsize}
\centering
\includegraphics[width=5.0cm]{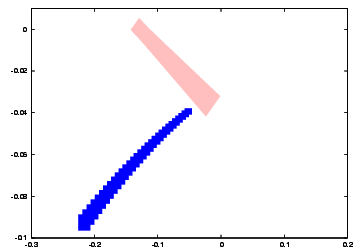}
(b-2)
\end{minipage}
\begin{minipage}{0.32\hsize}
\centering
\includegraphics[width=5.0cm]{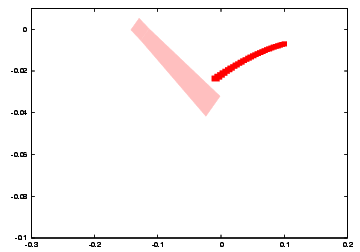}
(b-3)
\end{minipage}
\caption{Validation of covering-exchange sequences in Computer Assisted Result \ref{car-heteroclinic}. }
\label{fig-computation-heteroclinic}
Horizontal axis: $u$. Vertical axis: $v$. Each figure represents the projection of trajectories and cones on $(u,v)$-plane.
\par
(a-1) : Validation of $\varphi_\epsilon(T_0, (C_{m_u}^u)^\exit) \cap \{w \in [-0.00142875124866067, 0.00177996171276667]\}$.
\par
(a-2) : Validation of $\varphi_\epsilon(T_0, (C_{m_u}^u)^\exit) \cap \{w \in [-0.00142875124866067, -0.00141195170436]\}$.
\par
(a-3) : Validation of $\varphi_\epsilon(T_0, (C_{m_u}^u)^\exit) \cap \{w \in [0.001763162168466, 0.00177996171276667]\}$.
\par
(b-1) : Validation of $\varphi_\epsilon(T_1, (C_{m_u}^u)^\exit) \cap \{w \in [0.058532027216316, 0.06509312870908]\}$.
\par
(b-2) : Validation of $\varphi_\epsilon(T_1, (C_{m_u}^u)^\exit) \cap \{w \in [0.058532027216316, 0.058557359654512]\}$.
\par
(b-3) : Validation of $\varphi_\epsilon(T_1, (C_{m_u}^u)^\exit) \cap \{w \in [0.065067796270884, 0.06509312870908]\}$.
\end{figure}

%
%	New Subsection
%
\subsection{Homoclinic orbits}
\label{section-demo-homo}
The final example is a family of homoclinic orbits. Set $a=0.3$, $\gamma = 10.0$, $\delta = 9.0$ and $c \approx 0.8$. 
All validation of covering-exchange sequences with extended cones yield the following computer assisted result.

\begin{car}
\label{car-homoclinic}
Consider (\ref{FN}) with $a=0.3$, $\gamma = 10.0$ and $\delta = 9.0$. Then for all $c \in [0.799,0.801]$ the following trajectories are validated.  
\begin{enumerate}
\item At $\epsilon = 0$, there is a singular heteroclinic chain $H_0$ consisting of
\begin{itemize}
\item heteroclinic orbits from $p_0$ to $q_1$, and from $p_1$ to $q_0$, 
\item branches $M^0$, $M^1$ of nullcline $\{v=0, f(u) = w\}$. $M^0$ contains $p_0$ and $q_0$. Similarly, $M^1$ contains $q_1$ and $p_1$.
\end{itemize}
Equilibria $p_0$, $q_1$, $p_1$ and $q_0$ are validated by
\begin{align*}
&|\pi_{u,v}(p_0) - (-4.22605959\times 10^{-3}, 4.67157371\times 10^{-7})| < 2.90338686\times 10^{-2},\\
&|\pi_y (p_0) - 1.94\times 10^{-3}| < 3.25\times 10^{-3},\\
&|\pi_{u,v}(q_1) - (0.996532931,-4.66925590\times 10^{-6})| < 2.13488628\times 10^{-2},\\
&|\pi_y (q_1) - 2.02\times 10^{-3}| < 3.3\times 10^{-3},\\
&|\pi_{u,v}(p_1) - (0.870020061,-2.91627199\times 10^{-6})| < 2.85183938\times 10^{-2},\\
&|\pi_y (p_1)-0.06362| < 3.3\times 10^{-3},\\
&|\pi_{u,v}(q_0) - (-0.129665170,3.70642116\times 10^{-6})| < 2.20129903\times 10^{-2},\\
&|\pi_y (q_0)-0.06335| < 3.25\times 10^{-3}.
\end{align*}
\item For all $\epsilon \in (0,5.0\times 10^{-6}]$, there exists a homoclinic orbit $H_\epsilon$ of $p_0$ near $H_0$. 
The equilibrium $p_0$ admits an attracting isolating block on $S_\epsilon$ contained in $\{(u,v,w)\in S_\epsilon \cap \{u \leq 2.10010357\times 10^{-2}\} \mid w\in [-3.1550\times 10^{-4}, 2.1545\times 10^{-3}]\}$.
\end{enumerate}
Parameters for validations are listed in Table \ref{table-homo}.
Finally, assumptions of Lemma \ref{lem-validation-unstable} are validated with the attracting isolating block $\{w\in [-1.09722\times 10^{-3}, 4.97722\times 10^{-3}]\}$ on the slow manifold and the unstable $120$-cone condition. 
\end{car}

\begin{center}
  \begin{table}[h]
    \begin{center}
      \begin{tabular}{|c|c|c|} \hline
	Parameters & On $M^0$ ($q=-1$) & On $M^1$ ($q=+1$)\\ 
	\hline
	$\chi$ & $0.8812568505663135$ & $0.8792535675082328$ \\
	$\bar h$ & $0.00023$ & $0.00022$ \\
	$H$ & $0.0065$ & $0.0066$ \\
	$d_a$ & $0.8$ & $0.8$ \\
	$d_b$ & $0.75$ & $0.8$ \\
	$r_a$ & $0.008$ & $0.008$ \\
	$r_b$ & $0.0085$ & $0.0078$ \\
	$m_u$  & $50$ (around $p_0$) & $52$ (around $p_1$) \\
	$m_s$ & $45$ (around $q_0$) & $55$ (around $q_1$)\\
	$2({\rm diam}(\pi_a(N))-r_a)/m_u$ & $2.9606\times 10^{-4}$ & $-$\\
	$\ell_u$ & $0.0441793$  (around $p_0$) & $0.0370798$ (around $p_1$) \\
	$\ell_s$ & $0.12217$ (around $q_0$) & $0.0984568$ (around $q_1$)\\\hline
	 & $N_\epsilon^{0,\exit} \overset{\varphi_\epsilon(T_0,\cdot )}{\Longrightarrow} N_\epsilon^{1,0}$
	 & $N_\epsilon^{1,\exit} \overset{\varphi_\epsilon(T_1,\cdot )}{\Longrightarrow} N_\epsilon^{0,0}$ \\ \hline
	 $T_i$ & $0.067$ (with $\Delta t = 1.0\times 10^{-4}$) & $0.0288$ (with $\Delta t = 2.0\times 10^{-5}$) \\
	 computation time & $61$ min. $59$ sec. & $73$ min. $57$ sec. \\
        \hline
      \end{tabular}
    \end{center}
    \caption{Validation parameters of slow shadowing and isolating blocks in Computer Assisted Result \ref{car-homoclinic}.}
    \label{table-homo}
  \end{table}
\end{center}

\begin{rem}\rm
In this case we easily know that the only equilibrium of (\ref{FN}) for $\epsilon > 0$ is $(u,v,w) = (0,0,0)$ and the dynamics around $(0,0,0)$ can be easily understood to show that $(0,0,0)$ is attracting on the slow manifold in $\mathcal{S}_\epsilon^0$. 

Note that appropriate arrangements of computation schemes can improve the accuracy of validated trajectories and the validation range of parameters including $\epsilon$.

In many articles, the uniqueness of global orbits is also discussed. 
From the geometrical viewpoint, typical arguments for the uniqueness require the transversality of locally invariant manifolds. 
In our cases, however, transversality in the sense of differential manifolds is not mentioned. 
In other words, we cannot prove the uniqueness of validated orbits in the current setting. 
\end{rem}

\begin{figure}[htbp]\em
\begin{minipage}{0.32\hsize}
\centering
\includegraphics[width=5.0cm]{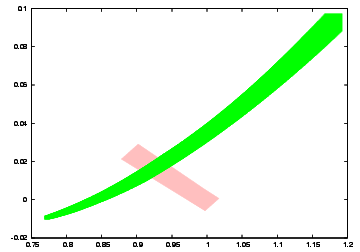}
(a-1)
\end{minipage}
\begin{minipage}{0.32\hsize}
\centering
\includegraphics[width=5.0cm]{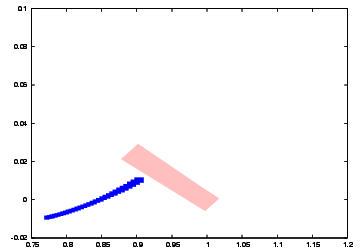}
(a-2)
\end{minipage}
\begin{minipage}{0.32\hsize}
\centering
\includegraphics[width=5.0cm]{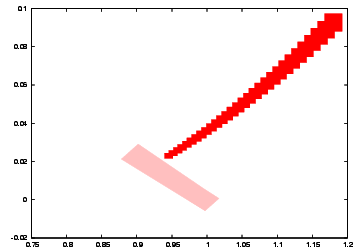}
(a-3)
\end{minipage}\par
\begin{minipage}{0.32\hsize}
\centering
\includegraphics[width=5.0cm]{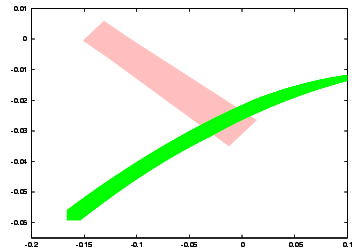}
(b-1)
\end{minipage}
\begin{minipage}{0.32\hsize}
\centering
\includegraphics[width=5.0cm]{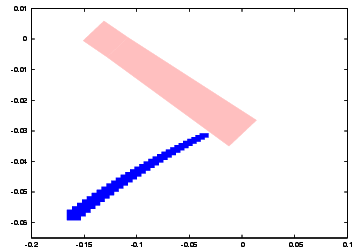}
(b-2)
\end{minipage}
\begin{minipage}{0.32\hsize}
\centering
\includegraphics[width=5.0cm]{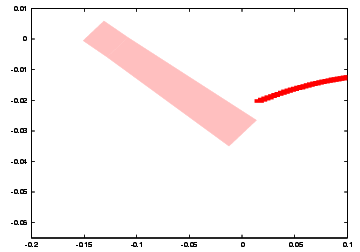}
(b-3)
\end{minipage}
\caption{Validation of covering-exchange sequences in Computer Assisted Result \ref{car-homoclinic}. }
\label{fig-computation-homoclinic}
Horizontal axis: $u$. Vertical axis: $v$. Each figure represents the projection of trajectories and cones on $(u,v)$-plane.
\par
(a-1) : Validation of $\varphi_\epsilon(T_0, (C_{m_u}^u)^\exit) \cap \{w \in 
[-0.001128757897204, 0.00412596711158]\}$.
\par
(a-2) : Validation of $\varphi_\epsilon(T_0, (C_{m_u}^u)^\exit) \cap \{w \in 
[-0.001128757897204, -0.00110862485119333]\}$.
\par
(a-3) : Validation of $\varphi_\epsilon(T_0, (C_{m_u}^u)^\exit) \cap \{w \in 
[0.00410583406556933, 0.00412596711158]\}$.
\par
(b-1) : Validation of $\varphi_\epsilon(T_1, (C_{m_u}^u)^\exit) \cap \{w \in 
[0.06094493968776, 0.0668173357334667]\}$.
\par
(b-2) : Validation of $\varphi_\epsilon(T_1, (C_{m_u}^u)^\exit) \cap \{w \in 
[0.06094493968776, 0.0609654725410667]\}$.
\par
(b-3) : Validation of $\varphi_\epsilon(T_1, (C_{m_u}^u)^\exit) \cap \{w \in 
[0.06679680288016, 0.0668173357334667]\}$.
\end{figure}

%
%	New Subsection
%
\subsection{Continuation of homoclinic orbits}
\label{section-homo-continuation}
Finally, we verify the $\epsilon$-continuation of homoclinic orbits validated in the previous subsection.
It is very hard to validate trajectories of (\ref{FN})$_\epsilon$ for all $\epsilon \in (0,\epsilon_0]$ with large $\epsilon_0$ at a time.
In most cases, covering relations in the fast dynamics fail.
To overcome this difficulty, we divide a closed interval $[0,\epsilon_0]$ into sub-intervals
\begin{equation*}
[0, \epsilon_0^1]\cup [\epsilon_0^1, \epsilon_0^2]\cup \cdots \cup [\epsilon_0^{m-1}, \epsilon_0^m],\quad \epsilon_0^m = \epsilon_0
\end{equation*}
and validate trajectories within each subinterval.

In our demonstrating example, we divide the interval $[0,5.0\times 10^{-5}]$ into
\begin{align*}
&[0, 5.0\times 10^{-6}] \cup [5.0\times 10^{-6}, 1.2\times 10^{-5}]\cup 
	[1.2\times 10^{-5}, 2.2\times 10^{-5}]\\
	&\cup [2.2\times 10^{-5}, 3.2\times 10^{-5}] \cup [3.2\times 10^{-5}, 4.2\times 10^{-5}]\cup [4.2\times 10^{-5}, 5.0\times 10^{-5}].
\end{align*}

Validating homoclinic orbits for each sub-interval and summarizing them, we obtain the following.
Let $p_0, q_0, p_1$ and $q_1$ be points corresponding to those stated in Computer Assisted Result \ref{car-homoclinic}.

\begin{car}
\label{car-continuation}
Consider (\ref{FN}) with $a=0.3$, $\gamma = 10.0$ and $\delta = 9.0$. 
Then for all $c \in [0.799,0.801]$ and for all $\epsilon \in (0,5.0\times 10^{-5}]$, there exists a homoclinic orbit $H_\epsilon$ of $p_0$ near $H_0$, where $H_0$ is the singular homoclinic orbit for (\ref{FN})$_0$ obtained in Computer Assisted Result \ref{car-homoclinic}. 
Several numerical data for validations of $H_\epsilon$ are listed in Table \ref{table-continuation-p0} - \ref{table-continuation-time}.
\end{car} 

% p0
\begin{center}
  \begin{table}[h]
    \begin{center}
      \begin{tabular}{|c|c|c|c|c|c|} \hline
	$\epsilon$ & $\bar w^0_{m_0}$  & $\bar h$ & $H$ & $\chi$ & $\ell_u$ \\ 
        \hline
	$[0.0, 5.0\times 10^{-6}]$ & $0.00175$ & $0.00023$ & $0.0065$ & $0.8812568505663135$ & $0.04417933$ \\
	$[5.0\times 10^{-6}, 1.2\times 10^{-5}]$ & $0.00175$ & $0.00023$ & $0.0065$ & $0.8812568505663135$ & $0.0441793$ \\
	$[1.2\times 10^{-5}, 2.2\times 10^{-5}]$ & $0.00175$ & $0.00023$ & $0.0065$ & $0.8813652126300421$ & $0.0441725$ \\
	$[2.2\times 10^{-5}, 3.2\times 10^{-5}]$ & $0.00175$ & $0.00023$ & $0.0065$ & $0.8813652126300421$ & $0.0441528$ \\
	$[3.2\times 10^{-5}, 4.2\times 10^{-5}]$ & $0.00175$ & $0.00023$ & $0.0065$ & $0.8813652126300421$ & $0.0441429$ \\ 
	$[4.2\times 10^{-5}, 5.0\times 10^{-5}]$ & $0.0019$ & $0.00025$ & $0.0065$ & $0.8809523809523810$ & $0.0441474$ \\ 
        \hline
      \end{tabular}
    \end{center}
    \caption{Main information of blocks around $p_0=(u_{m_0}^0,v_{m_0}^0, w_{m_0}^0)$ with $|w^0_{m_0}- \bar w^0_{m_0}| < \bar h$.}
    \label{table-continuation-p0}
    The sharpness $m_u$ of unstable cones is an identical value : $m_u = 50$.
    The ratio $d_a, d_b$ in (\ref{shadow}) are set as identical values : $d_a=0.8, d_b = 0.75$.
    The space length parameters $r_a, r_b$ of fast-saddle-type blocks are set as identical values : $r_a=0.008, r_b = 0.0085$. 
  \end{table}
\end{center}

% q0
\begin{center}
  \begin{table}[h]
    \begin{center}
      \begin{tabular}{|c|c|c|c|c|c|c|} \hline
	$\epsilon$ & $\bar w_0^1$  & $\bar h$ & $H$ & $\chi$ & $\ell_s$ \\ 
        \hline
	$[0.0, 5.0\times 10^{-6}]$ & $0.0633$ & $0.00023$ & $0.0065$ & $0.8812568505663135$ & $0.12217$ \\
	$[5.0\times 10^{-6}, 1.2\times 10^{-5}]$ & $0.0633$ & $0.00023$ & $0.0065$ & $0.8812568505663135$ & $0.12217$ \\
	$[1.2\times 10^{-5}, 2.2\times 10^{-5}]$ & $0.0635$ & $0.00023$ & $0.0065$ & $0.8813652126300421$ & $0.122164$ \\
	$[2.2\times 10^{-5}, 3.2\times 10^{-5}]$ & $0.0635$ & $0.00023$ & $0.0065$ & $0.8813652126300421$ & $0.122145$ \\
	$[3.2\times 10^{-5}, 4.2\times 10^{-5}]$ & $0.0635$ & $0.00023$ & $0.0065$ & $0.8813652126300421$ & $0.122141$ \\ 
	$[4.2\times 10^{-5}, 5.0\times 10^{-5}]$ & $0.0635$ & $0.00025$ & $0.0065$ & $0.8809523809523810$ & $0.122144$ \\ 
        \hline
      \end{tabular}
    \end{center}
    \caption{Main information of blocks around $q_0=(u_0^1,v_0^1, w_0^1)$ with $|w_0^1- \bar w_0^1| < \bar h$.}
    \label{table-continuation-q0}
    The sharpness $m_s$ of stable cones is an identical value : $m_s = 45$.
    The ratio $d_a, d_b$ in (\ref{shadow}) are set as identical values : $d_a=0.8, d_b = 0.75$.
    The space length parameters $r_a, r_b$ of fast-saddle-type blocks are set as identical values : $r_a=0.008, r_b = 0.0085$. 
  \end{table}
\end{center}

% p1
\begin{center}
  \begin{table}[h]
    \begin{center}
      \begin{tabular}{|c|c|c|c|c|c|} \hline
	$\epsilon$ & $\bar w^1_{m_1}$  & $\bar h$ & $H$ & $\chi$ & $\ell_u$\\ 
        \hline
	$[0.0, 5.0\times 10^{-6}]$ & $0.0636$ & $0.00022$ & $0.0066$ & $0.8792535675082328$ & $0.0370798$ \\
	$[5.0\times 10^{-6}, 1.2\times 10^{-5}]$ & $0.0636$ & $0.00022$ & $0.0066$ & $0.8792535675082328$ & $0.0370798$ \\
	$[1.2\times 10^{-5}, 2.2\times 10^{-5}]$ & $0.0636$ & $0.00022$ & $0.0066$ & $0.8792535675082328$ & $0.0370719$ \\
	$[2.2\times 10^{-5}, 3.2\times 10^{-5}]$ & $0.0636$ & $0.00022$ & $0.0066$ & $0.8792535675082328$ & $0.0370534$ \\
	$[3.2\times 10^{-5}, 4.2\times 10^{-5}]$ & $0.0636$ & $0.00022$ & $0.0066$ & $0.8792535675082328$ & $0.0370464$ \\ 
	$[4.2\times 10^{-5}, 5.0\times 10^{-5}]$ & $0.0634$ & $0.0003$ & $0.0065$ & $0.8786764705882354$ & $0.0370888$ \\ 
        \hline
      \end{tabular}
    \end{center}
    \caption{Main information of blocks around $p_1=(u_{m_1}^1,v_{m_1}^1, w_{m_1}^1)$ with $|w^1_{m_1}- \bar w^1_{m_1}| < \bar h$.}
    \label{table-continuation-p1}
    The sharpness $m_u$ of unstable cones is an identical value : $m_u = 52$.
    The ratio $d_a, d_b$ in (\ref{shadow}) are set as identical values : $d_a=0.8, d_b = 0.8$ for $\epsilon \in (0.0, 4.2\times 10^{-4}]$, $d_a=0.75, d_b = 0.75$ for $\epsilon \in [4.2\times 10^{-4}, 5.0\times 10^{-5}]$.
    The space length parameters $r_a, r_b$ of fast-saddle-type blocks are set as identical values : $r_a=0.008, r_b = 0.0078$. 
  \end{table}
\end{center}

% q1
\begin{center}
  \begin{table}[h]
    \begin{center}
      \begin{tabular}{|c|c|c|c|c|c|} \hline
	$\epsilon$ & $\bar w_0^0$  & $\bar h$ & $H$ & $\chi$ & $\ell_s$ \\ 
        \hline
	$[0.0, 5.0\times 10^{-6}]$ & $0.0019$ & $0.00022$ & $0.0066$ & $0.8792535675082328$ & $0.0984568$ \\
	$[5.0\times 10^{-6}, 1.2\times 10^{-5}]$ & $0.0019$ & $0.00022$ & $0.0066$ & $0.8792535675082328$ & $0.0984568$ \\
	$[1.2\times 10^{-5}, 2.2\times 10^{-5}]$ & $0.0019$ & $0.00022$ & $0.0066$ & $0.8792535675082328$ & $0.0984452$ \\
	$[2.2\times 10^{-5}, 3.2\times 10^{-5}]$ & $0.0019$ & $0.00022$ & $0.0066$ & $0.8792535675082328$ & $0.098424$ \\
	$[3.2\times 10^{-5}, 4.2\times 10^{-5}]$ & $0.0019$ & $0.00022$ & $0.0066$ & $0.8792535675082328$ & $0.0984201$ \\ 
	$[4.2\times 10^{-5}, 5.0\times 10^{-5}]$ & $0.0018$ & $0.0003$ & $0.0065$ & $0.8786764705882354$ & $0.0984281$ \\ 
        \hline
      \end{tabular}
    \end{center}
    \caption{Main information of blocks around $q_1=(u_0^0,v_0^0, w_0^0)$ with $|w_0^0- \bar w_0^0| < \bar h$.}
    \label{table-continuation-q1}
    The sharpness $m_s$ of stable cones is an identical value : $m_s = 55$.
    The ratio $d_a, d_b$ in (\ref{shadow}) are set as identical values : $d_a=0.8, d_b = 0.8$ for $\epsilon \in (0.0, 4.2\times 10^{-4}]$, $d_a=0.75, d_b = 0.75$ for $\epsilon \in [4.2\times 10^{-4}, 5.0\times 10^{-5}]$.
    The space length parameters $r_a, r_b$ of fast-saddle-type blocks are set as identical values : $r_a=0.008, r_b = 0.0078$. 
  \end{table}
\end{center}

\begin{center}
  \begin{table}[h]
    \begin{center}
      \begin{tabular}{|c|c|c|c|c|c|} \hline
	$\epsilon$ & $N_\epsilon^{0,\exit} \overset{\varphi_\epsilon(T_0,\cdot )}{\Longrightarrow} N_\epsilon^{1,0}$
	 & $N_\epsilon^{1,\exit} \overset{\varphi_\epsilon(T_1,\cdot )}{\Longrightarrow} N_\epsilon^{0,0}$ \\ 
        \hline
	$[0.0, 5.0\times 10^{-6}]$ & $61$ min. $59$ sec. & $73$ min. $57$ sec.  \\
	$[5.0\times 10^{-6}, 1.2\times 10^{-5}]$ & $61$ min. $27$ sec. & $75$ min. $57$ sec. \\
	$[1.2\times 10^{-5}, 2.2\times 10^{-5}]$ & $61$ min. $38$ sec. & $74$ min. $55$ sec. \\
	$[2.2\times 10^{-5}, 3.2\times 10^{-5}]$ & $59$ min. $48$ sec. & $71$ min. $45$ sec. \\
	$[3.2\times 10^{-5}, 4,2\times 10^{-5}]$ & $58$ min. $57$ sec. & $71$ min. $18$ sec. \\ 
	$[4.2\times 10^{-5}, 5.0\times 10^{-5}]$ & $59$ min. $31$ sec. & $73$ min. $53$ sec. \\ \hline
	$[0.0, 5.0\times 10^{-5}]$ & (Total) $363$ min. $20$ sec. & (Total) $441$ min. $45$ sec. \\ 
        \hline
      \end{tabular}
    \end{center}
    \caption{Computation times of {\bf Drop} in Computer Assisted Result \ref{car-continuation}.}
    \label{table-continuation-time}
 It has taken about $13.4$ hours for validating {\bf Drop} in $\{\epsilon \in [0,5.0\times 10^{-5}]\}$ in our computation environments.
 As for the slow shadowing and {\bf Jump}, it totally takes less than $30$ seconds to validate.
  \end{table}
\end{center}

%
%	New Section
%
\bigskip
\section*{Conclusion}
\label{section-conclusion}
We have proposed a topological methodology for validating singularly perturbed periodic, homoclinic and heteroclinic orbits for fast-slow systems with its applicability to concrete systems with computer assistance.
The main features of our proposing methodology are the following.
\begin{itemize}
\item Our central strategy consists of applications of well-known topological tools, covering relations, cones and isolating blocks, with taking the singular perturbation structure of normally hyperbolic invariant manifolds into account. 
This is one of different points from preceding works such as \cite{AK2015}.
\item All our procedures can be validated via rigorous numerics (e.g. interval arithmetics). 
As a consequence, one can validate the continuation of singular limit orbits for all $\epsilon\in (0,\epsilon_0]$ with a given $\epsilon_0 > 0$ with computer assistance.
\end{itemize}

A main concept for realizing the above points simultaneously is the covering-exchange; the singular perturbation's version of covering relations. 
We have proposed not only its primitive form but also its generalization; a collection of local behavior named \lq\lq slow shadowing", \lq\lq drop" and \lq\lq jump". 
The notion of covering-exchange describes the behavior of trajectories in the full system which shadow normally hyperbolic invariant manifolds as well as their stable and unstable manifolds, even for sufficiently small $\epsilon$.
In particular, this notion enables us to validate such trajectories {\em without solving any differential equations in practical computations}.
Moreover, as mentioned above, it exceeds the limit of the multiple timescale parameter range: \lq\lq sufficiently small $\epsilon$", to $\epsilon$ with an explicit range in practical applications with computer assistance.

Of course, further applications of topological tools such as covering relations and the Conley index enable us to prove the existence of trajectories with complicated behavior. 
We believe that all these ideas will overcome difficulties of a broad class of singular perturbation problems.

\bigskip
We end this paper proposing further directions of our arguments.

\begin{description}
\item[Bridging validation results to \lq\lq direct" approach]
\end{description}
Our methodology concentrates on validations of trajectories for $\epsilon\in (0,\epsilon_0]$ on the basis of geometric singular perturbation theory.
One of challenges concerning with continuation of trajectories is the integration of singularly perturbed trajectories with those by a {\em direct approach}; namely, validation of trajectories without taking account of singular perturbation structure of systems for positive $\epsilon$ bounded away from $0$ (e.g. \cite{AK2015}).
In our validation examples, under a specific choice of parameter values, solution trajectories has been validated for $\epsilon \in (0,5.0\times 10^{-5}]$ in the case of homoclinic orbits for the FitzHugh-Nagumo system (\ref{FN}).
In the case of (\ref{FN}), validations of each branch of slow manifolds are completed in less than one second.

This parameter range is still far from validation ranges with direct approaches; $\epsilon = 0.01$ in \cite{AK2015}, for example.
Validations of slow shadowing for further range of $\epsilon$ are not easy, which is mainly because there is a trade-off between the slow shadowing condition (\ref{shadow}) and the covering relation in Assumption (SS5), as discussed in Section \ref{section-demo-shadow}.
Both conditions are essential to describing slow shadowing, and hence our task for overcoming this situation will be developments of an improved topological and numerical method, like {\em multi-step methods} in numerical initial value problems, keeping the essence of slow shadowing phenomena. 
Obviously, if one can validate branches of normally hyperbolic slow manifolds by {\em one} blocks, there is no problem in this direction, in which case we can apply the primitive form of covering-exchange.
After improvements of validations of slow shadowing or blocks containing slow manifolds with computer assistance, there will be a possibility that validated trajectories can be further continued in $\epsilon$ and, we hope, be connected to validated ones via direct approach under appropriate choice of various parameters or solvers.
Further works in this direction from the viewpoint of topological and numerical analysis are ones of our future studies.

\begin{description}
\item[Uniqueness and stability of validated orbits]
\end{description}
In our topological theorems (Theorems \ref{thm-periodic-1}, \ref{thm-heteroclinic-1}, Corollaries  \ref{cor-periodic-2}, \ref{cor-periodic-3}, \ref{cor-heteroclinic-2} and \ref{cor-heteroclinic-3}), only the existence of desiring orbits is stated. 
As for the uniqueness of such orbits, we have no answers yet. 
More precisely, these theorems only require topological transversality in terms of covering relations,  which is weaker than the transversality in terms of differential manifolds. 
If we can even verify transversality in the sense of differential manifolds, then Exchange Lemma gives us the local uniqueness of singularly perturbed orbits. Transversality is also essential to discuss stability of connecting orbits. 
Jones \cite{Jones1984} discusses stability of homoclinic orbits for (\ref{FN}) as the traveling wave solutions of FitzHugh-Nagumo PDE (\ref{FN-PDE}) from geometric viewpoints. 
For this stability arguments, some information about the nature of transversality between two manifolds play a central role.
\par
One of standard approaches to transversality of manifolds is to use {\em Melnikov integrals} via solving variational equations.
If we take rigorous numerics into account, the $C^1$-Lohner method \cite{ZLoh} provides the variational information of trajectories as well as solutions themselves and will give an answer to solve transversality problems.

\begin{description}
\item[Covering-Exchange for fast-slow systems with multi-dimensional slow variables]
\end{description}
Our validation arguments can be applied only to fast-slow systems with one-dimensional slow variable at present. 
It is natural to question whether our concept is applicable to fast-slow systems with multi-dimensional slow variables. 
The key problem concerning this extension is how we should track true trajectories near slow manifolds in terms of covering-exchange properties. 
Systems with multi-dimensional slow variables have rich behavior even on slow manifolds. 
The validation of slow manifolds by one fast-saddle-type block as well as concrete vector fields is thus more difficult than systems with one-dimensional slow variables because of nonlinearity of manifolds. 
It is not thus realistic to extend Theorem \ref{thm-periodic-1} and \ref{thm-heteroclinic-1} to systems with multi-dimensional slow variable directly from the viewpoint of practical computations. 

By the way, existence theorems of trajectories with covering relations imply that true trajectories shadow reference orbits. 
In the case of singular perturbation problems, reference orbits correspond to singular orbits consisting of heteroclinic orbits and trajectories on slow manifolds. 
Slow shadowing discussed in Section \ref{section-show-shadowing} reflects the aspect of shadowing trajectories in fast-slow systems with one-dimensional slow variable.
This expectation will be valid even for systems with multi-dimensional slow variables. 
Such an extended sequence will yield the existence of true trajectories which shadow those on slow manifolds. 
Higher dimensional extension of slow shadowing condition or its analogue will give an explicit criterion for validations with rigorous numerics in reasonable processes.

\begin{description}
\item[Extension of procedures for slow manifolds with non-hyperbolic points]
\end{description}
Our validation arguments are based on the assumption that limiting critical manifolds are normally hyperbolic everywhere in consideration. 
Our procedure in this paper can be never applied if there is a point which is non-hyperbolic, like fold points, inside limiting critical manifolds. 
Nevertheless, there are a lot of cases that the interesting phenomena in singular perturbation problems are involved by such non-hyperbolic points. 
One of famous examples is {\em canard solutions},  which was discovered and first analyzed by Benoit, Callot, Diener and Diener \cite{BCDD}. 
A canard solution is a solution of a singular perturbed system which is contained in the intersection of an attracting slow manifold and a repelling one. 
Canard solutions provide a rich phenomenon such as {\em canard explosion}, namely, a transition from a small limit cycle to a relaxation oscillation through a sequence of {\em canard cycles} \cite{DR}. 
Since canard cycles involve fold points as jump points, our current implementations can never validate canard cycles. 
Krupa and Szmolyan \cite{KS} discuss the extension of geometric singular perturbation theory for slow manifolds including fold and canard points via the technique of blow-up. 
Simultaneously, Liu \cite{L} discusses the Exchange Lemma in case that the critical manifold contains loss-of-stability turning points, corresponding the phenomenon called {\em the delay of stability loss}. 
In \cite{L}, the Fenichel-type coordinate which overcomes non-normal hyperbolicity is constructed by the other type of invariant manifold theorem proved by Chow-Liu-Yi \cite{CLY}.
We also mention sequential works by Schecter and Szmolyan \cite{SS2004, S2008_1, S2008_2, SS2009}, which discuss the singularly perturbed Riemann-Dafermos solutions as the small perturbation of composite waves consisting of constant waves, shock waves and rarefaction waves in systems of conservation laws.
There, non-hyperbolic points called gain-of-stability turning points naturally arise in invariant manifolds.
They derive the General Exchange Lemma for describing behavior of tracking manifolds near such non-hyperbolic manifolds with the help of blow-ups.

We have to take such mathematical techniques into account, if we deal with invariant manifolds which are not normally hyperbolic with computer assistance.

\section*{Acknowledgements}
This research was partially supported by Coop with Math Program, a commissioned project by MEXT, Japan. 
The author thanks to reviewers for providing him with helpful suggestions about contents and organizations of this paper.

\bibliographystyle{plain}
\bibliography{fast_slow_1dim}

\def\cprime{$'$}
\begin{thebibliography}{10}

\bibitem{AK2015}
G.~Arioli and H.~Koch.
\newblock Existence and stability of traveling pulse solutions of the
  {F}itz{H}ugh-{N}agumo equation.
\newblock {\em Nonlinear Analysis: Theory, Methods \& Applications},
  113:51--70, 2015.


\bibitem{BCDD} E. Benoit, J.L. Callot, F. Diener and M. Diener. Chasse au canards, {\it Collect. Math.}, {\bf 31}(1981), 37--119.

\bibitem{CAPD} CAPD: {\it Computer Assisted Proof of Dynamics software}, {\tt http://capd.ii.uj.edu.pl}

\bibitem{C}
G.A. Carpenter.
\newblock A geometric approach to singular perturbation problems with
  applications to nerve impulse equations.
\newblock {\em J. Differential Equations}, 23(3):335--367, 1977.

\bibitem{CL}
R.~Castelli and J.-P. Lessard.
\newblock Rigorous numerics in {F}loquet theory: computing stable and unstable
  bundles of periodic orbits.
\newblock {\em SIAM J. Appl. Dyn. Syst.}, 12(1):204--245, 2013.

\bibitem{CLY}
S.-N. Chow, W.~Liu, and Y.~Yi.
\newblock Center manifolds for smooth invariant manifolds.
\newblock {\em Trans. Amer. Math. Soc.}, 352(11):5179--5211, 2000.

\bibitem{Con}
C.~Conley.
\newblock {\em Isolated invariant sets and the {M}orse index}, volume~38 of
  {\em CBMS Regional Conference Series in Mathematics}.
\newblock American Mathematical Society, Providence, R.I., 1978.

\bibitem{D2}
B.~Deng.
\newblock The existence of infinitely many traveling front and back waves in
  the {F}itz{H}ugh-{N}agumo equations.
\newblock {\em SIAM J. Math. Anal.}, 22(6):1631--1650, 1991.

\bibitem{DR}
F.~Dumortier and R.~Roussarie.
\newblock Canard cycles and center manifolds.
\newblock {\em Mem. Amer. Math. Soc.}, 121(577):x+100, 1996.
\newblock With an appendix by Cheng Zhi Li.

\bibitem{F}
N.~Fenichel.
\newblock Geometric singular perturbation theory for ordinary differential
  equations.
\newblock {\em J. Differential Equations}, 31(1):53--98, 1979.

\bibitem{GGKKMO}
M.~Gameiro, T.~Gedeon, W.~Kalies, H.~Kokubu, K.~Mischaikow, and H.~Oka.
\newblock Topological horseshoes of traveling waves for a fast-slow
  predator-prey system.
\newblock {\em J. Dynam. Differential Equations}, 19(3):623--654, 2007.

\bibitem{GS}
R.~Gardner and J.~Smoller.
\newblock The existence of periodic travelling waves for singularly perturbed
  predator-prey equations via the {C}onley index.
\newblock {\em J. Differential Equations}, 47(1):133--161, 1983.

\bibitem{GKMO}
T.~Gedeon, H.~Kokubu, K.~Mischaikow, and H.~Oka.
\newblock The {C}onley index for fast--slow systems. {II}. {M}ultidimensional
  slow variable.
\newblock {\em J. Differential Equations}, 225(1):242--307, 2006.

\bibitem{GKMOR}
T.~Gedeon, H.~Kokubu, K.~Mischaikow, H.~Oka, and J.F. Reineck.
\newblock The {C}onley index for fast-slow systems. {I}. {O}ne-dimensional slow
  variable.
\newblock {\em J. Dynam. Differential Equations}, 11(3):427--470, 1999.

\bibitem{GJM2012}
J.~Guckenheimer, T.~Johnson, and P.~Meerkamp.
\newblock Rigorous enclosures of a slow manifold.
\newblock {\em SIAM Journal on Applied Dynamical Systems}, 11(3):831--863,
  2012.

\bibitem{GK}
J.~Guckenheimer and C.~Kuehn.
\newblock Computing slow manifolds of saddle type.
\newblock {\em SIAM J. Appl. Dyn. Syst.}, 8(3):854--879, 2009.

\bibitem{Jones1984}
C.K.R.T. Jones.
\newblock Stability of the travelling wave solution of the
  {F}itz{H}ugh-{N}agumo system.
\newblock {\em Trans. Amer. Math. Soc.}, 286(2):431--469, 1984.

\bibitem{Jones}
C.K.R.T. Jones.
\newblock Geometric singular perturbation theory.
\newblock In {\em Dynamical systems ({M}ontecatini {T}erme, 1994)}, volume 1609
  of {\em Lecture Notes in Math.}, pages 44--118. Springer, Berlin, 1995.

\bibitem{JKK}
C.K.R.T. Jones, T.J. Kaper, and N.~Kopell.
\newblock Tracking invariant manifolds up to exponentially small errors.
\newblock {\em SIAM J. Math. Anal.}, 27(2):558--577, 1996.

\bibitem{JK}
C.K.R.T. Jones and N.~Kopell.
\newblock Tracking invariant manifolds with differential forms in singularly
  perturbed systems.
\newblock {\em J. Differential Equations}, 108(1):64--88, 1994.

\bibitem{KS}
M.~Krupa and P.~Szmolyan.
\newblock Extending geometric singular perturbation theory to nonhyperbolic
  points---fold and canard points in two dimensions.
\newblock {\em SIAM journal on mathematical analysis}, 33(2):286--314, 2001.

\bibitem{L}
W.~Liu.
\newblock Exchange lemmas for singular perturbation problems with certain
  turning points.
\newblock {\em J. Differential Equations}, 167(1):134--180, 2000.

\bibitem{Mat}
K.~Matsue.
\newblock Rigorous numerics for stationary solutions of dissipative
  {PDE}s-{E}xistence and local dynamics.
\newblock {\em Nonlinear Theory and Its Applications, IEICE}, 4(1):62--79,
  2013.

\bibitem{Mc}
{C}.{K}. Mc{C}ord.
\newblock Mappings and homological properties in the {C}onley index theory.
\newblock {\em Ergodic Theory and Dynamical Systems}, 8(8*):175--198, 1988.

\bibitem{Mis}
K.~Mischaikow.
\newblock Conley index theory.
\newblock In {\em Dynamical systems ({M}ontecatini {T}erme, 1994)}, volume 1609
  of {\em Lecture Notes in Math.}, pages 119--207. Springer, Berlin, 1995.

\bibitem{Rob}
C.~Robinson.
\newblock {\em Dynamical systems}.
\newblock Studies in Advanced Mathematics. CRC Press, Boca Raton, FL, second
  edition, 1999.
\newblock Stability, symbolic dynamics, and chaos.

\bibitem{S2008_1}
S.~Schecter.
\newblock Exchange lemmas 1: {D}eng's lemma.
\newblock {\em Journal of Differential Equations}, 245(2):392--410, 2008.

\bibitem{S2008_2}
S.~Schecter.
\newblock Exchange lemmas 2: {G}eneral exchange lemma.
\newblock {\em Journal of Differential Equations}, 245(2):411--441, 2008.

\bibitem{SS2004}
S.~Schecter and P.~Szmolyan.
\newblock Composite waves in the {D}afermos regularization.
\newblock {\em Journal of Dynamics and Differential Equations}, 16(3):847--867,
  2004.

\bibitem{SS2009}
S.~Schecter and P.~Szmolyan.
\newblock Persistence of rarefactions under {D}afermos regularization: blow-up
  and an exchange lemma for gain-of-stability turning points.
\newblock {\em SIAM Journal on Applied Dynamical Systems}, 8(3):822--853, 2009.

\bibitem{Smo}
J.~Smoller.
\newblock {\em Shock waves and reaction-diffusion equations}, volume 258 of
  {\em Grundlehren der Mathematischen Wissenschaften [Fundamental Principles of
  Mathematical Sciences]}.
\newblock Springer-Verlag, New York, second edition, 1994.

\bibitem{S}
P.~Szmolyan.
\newblock Transversal heteroclinic and homoclinic orbits in singular
  perturbation problems.
\newblock {\em J. Differential Equations}, 92(2):252--281, 1991.

\bibitem{TKJ}
S.-K. Tin, N.~Kopell, and C.K.R.T. Jones.
\newblock Invariant manifolds and singularly perturbed boundary value problems.
\newblock {\em SIAM J. Numer. Anal.}, 31(6):1558--1576, 1994.

\bibitem{BMLM}
J.B. van~den Berg, J.D. Mireles-James, J.-P. Lessard, and K.~Mischaikow.
\newblock Rigorous numerics for symmetric connecting orbits: even homoclinics
  of the {G}ray-{S}cott equation.
\newblock {\em SIAM J. Math. Anal.}, 43(4):1557--1594, 2011.

\bibitem{W}
D.~Wilczak.
\newblock The existence of {S}hilnikov homoclinic orbits in the {M}ichelson
  system: a computer assisted proof.
\newblock {\em Found. Comput. Math.}, 6(4):495--535, 2006.

\bibitem{W2}
D.~Wilczak.
\newblock Abundance of heteroclinic and homoclinic orbits for the hyperchaotic
  {R}\"ossler system.
\newblock {\em Discrete Contin. Dyn. Syst. Ser. B}, 11(4):1039--1055, 2009.

\bibitem{WZ}
D.~Wilczak and P.~Zgliczy\'{n}ski.
\newblock Topological method for symmetric periodic orbits for maps with a
  reversing symmetry.
\newblock {\em Discrete Contin. Dyn. Syst.}, 17(3):629--652 (electronic), 2007.

\bibitem{ZLoh}
P.~Zgliczy\'{n}ski.
\newblock {$C^1$}-{L}ohner algorithm.
\newblock {\em Found. Comput. Math.}, 2(4):429--465, 2002.

\bibitem{ZCov}
P.~Zgliczy\'{n}ski.
\newblock Covering relations, cone conditions and the stable manifold theorem.
\newblock {\em J. Differential Equations}, 246(5):1774--1819, 2009.

\bibitem{Z3}
P.~Zgliczy\'{n}ski.
\newblock Rigorous numerics for dissipative {PDE}s {III}. an effective
  algorithm for rigorous integration of dissipative {PDE}s.
\newblock {\em Topological Methods in Nonlinear Analysis}, 36:197--262, 2010.

\bibitem{ZG}
P.~Zgliczy\'{n}ski and M.~Gidea.
\newblock Covering relations for multidimensional dynamical systems.
\newblock {\em J. Differential Equations}, 202(1):32--58, 2004.

\bibitem{ZM}
P.~Zgliczy\'{n}ski and K.~Mischaikow.
\newblock Rigorous numerics for partial differential equations: the
  {K}uramoto-{S}ivashinsky equation.
\newblock {\em Found. Comput. Math.}, 1(3):255--288, 2001.

\end{thebibliography}

%
%	New Section
%
\appendix
\section{Concrete terms for (\ref{FN}) according to Section \ref{section-block-pred-corr}}
\label{appendix-concrete-FN}
Here we show concrete forms of error terms and Jacobian matrices calculated from the vector field (\ref{FN})
 for validating fast-saddle-type blocks, and cone conditions with the predictor-corrector approach (Section \ref{section-block-pred-corr}). 
\subsection{Fast-saddle-type blocks}

We list various terms in Section \ref{section-block-pred-corr} in the case of FitzHugh-Nagumo system for readers' accessibility.
Let $(\bar u, 0, \bar w)$ be a numerical zero of $-f(u)+w$, where $f(u) = u(u-a)(1-u), a\in (0,1/2)$.

Use the coordinate $(\tilde u, \tilde v, \tilde w)$ given by the following:
\begin{equation}
\label{transform-FN}
u = \tilde u + (\bar u + f_u(\bar u)^{-1}\tilde w),\quad v = \tilde v,\quad w = \tilde w + \bar w,
\end{equation}
where $(\bar u, 0, \bar w)$ is a (numerical) root of $\{v=0, cv -f(u) + w = 0\}$ and $f_u(\bar u)$ is assumed to be invertible.
The transformed vector field as the extended system is then 
\begin{align*}
\tilde u' &= u' - f_u(\bar u)^{-1}\tilde w'\\
	&= v - f_u(\bar u)^{-1}\epsilon c^{-1}(u-\gamma w)\\
	&= \tilde v - \epsilon c^{-1} f_u(\bar u)^{-1} (\tilde u + (f_u(\bar u)^{-1} - \gamma )\tilde w + \bar u   -\gamma \bar w),\\
\tilde v' &= \delta^{-1}(cv - f(u) + w)\\
	&= \delta^{-1}\left\{ c\tilde v - f \left(\tilde u + \bar u + f_u(\bar u)^{-1}\tilde w \right) + \tilde w + \bar w \right\},\\
\tilde w' &= \epsilon c^{-1} (\tilde u + (f_u(\bar u)^{-1} - \gamma )\tilde w + \bar u   -\gamma \bar w),\\
\eta' &= 0,
\end{align*}
where we introduced an auxiliary variable $\eta$ to be $\epsilon = \epsilon_0 \eta$, $\eta \in [0,1]$, with constant $\epsilon_0$.

The fast component of the vector field is 
\begin{align*}
\begin{pmatrix}
\tilde u'  \\ \tilde v' 
\end{pmatrix} & = 
\begin{pmatrix}
\tilde v - \epsilon c^{-1} f_u(\bar u)^{-1} (\tilde u + (f_u(\bar u)^{-1} - \gamma )\tilde w + \bar u   -\gamma \bar w)\\
\delta^{-1}\left\{ c\tilde v - f \left(\tilde u + \bar u + f_u(\bar u)^{-1}\tilde w \right) + \tilde w + \bar w \right\},\\
\end{pmatrix}\\
&= 
\begin{pmatrix}
0 & 1 \\
-\delta^{-1} f_u(\bar u) & \delta^{-1} c
\end{pmatrix}
\begin{pmatrix}
\tilde u \\ \tilde v
\end{pmatrix}
+
\begin{pmatrix}
-\epsilon c^{-1} f_u(\bar u)^{-1} (\tilde u + (f_u(\bar u)^{-1} - \gamma )\tilde w + \bar u   -\gamma \bar w)\\
\delta^{-1}\left\{ - f (u) + f(\bar u)+f_u(\bar u)\tilde u + \tilde w \right\}\\
\end{pmatrix}.
\end{align*}

\bigskip
Denote (\ref{FN}) by
\begin{equation*}
X = F(X,Y,\epsilon),\quad Y=\epsilon G(X,Y,\epsilon),
\end{equation*}
where $X=(u,v)^T$ and $Y=w$; namely, 
\begin{equation*}
F(X,Y,\epsilon) = \begin{pmatrix}
v \\ \delta^{-1}(cv-f(u)+w)
\end{pmatrix},\quad 
G(X,Y,\epsilon) = c^{-1}(u-\gamma w).
\end{equation*}
The Jacobian matrices at $(\bar u, 0, \bar w)$ are
\begin{align*}
&F_X\mid_{(\bar u, 0, \bar w)} = 
 \begin{pmatrix}
0 & 1 \\ -\delta^{-1}f_u(\bar u) & \delta^{-1}c
\end{pmatrix},\quad 
\left(F_X\mid_{(\bar u, 0, \bar w)}\right)^{-1} = 
\frac{\delta}{f_u(\bar u)} \begin{pmatrix}
\delta^{-1}c & -1 \\ \delta^{-1}f_u(\bar u) & 0
\end{pmatrix},\\
&F_Y\mid_{(\bar u, 0, \bar w)} = 
\begin{pmatrix}
0 \\ \delta^{-1}
\end{pmatrix},\quad \left(F_X^{-1} \circ F_Y\right) \mid_{(\bar u, 0, \bar w)} = \begin{pmatrix}
-1 /f_u(\bar u) \\ 0
\end{pmatrix}.
\end{align*}
Around the numerical equilibrium $(\bar X, \bar Y)\equiv ((\bar u, 0), w) = (\bar u, 0, w)$, we introduce the affine transformation $T: (Z,W)\mapsto (X,Y)$ of the predictor-corrector form by
\begin{align*}
(X,Y) &= T(Z,W) := \left( PZ + \bar X -\left(F_X^{-1} \circ F_Y\right) \mid_{(\bar u, 0, \bar w)} W , W+\bar Y \right).
\end{align*}
where $P$ is the nonsingular matrix such that $P^{-1} F_X\mid_{(\bar u, 0, \bar w)} P = \Lambda = \diag(\lambda_1,\lambda_2)$.
Then, denoting $Z=(z_1,z_2)^T$,
\begin{align*}
\begin{pmatrix}
z_1' \\ z_2'
\end{pmatrix}
&=
\begin{pmatrix}
\lambda_1 z_1 \\ \lambda_2 z_2
\end{pmatrix}
+  P^{-1}\left( \begin{pmatrix}
0 \\ -\delta^{-1}W
\end{pmatrix}  + \hat F(X,Y,\epsilon) + \epsilon \begin{pmatrix}
- G(X,Y,\epsilon)/f_u(\bar u) \\ 0
\end{pmatrix} \right)\\
&=\begin{pmatrix}
\lambda_1 z_1 \\ \lambda_2 z_2
\end{pmatrix}
+  P^{-1}\left( \begin{pmatrix}
0 \\ -\delta^{-1}W
\end{pmatrix}  + \delta^{-1}\begin{pmatrix}
0 \\ -(f(u) - f_u(\bar u)(u-\bar u))+W
\end{pmatrix} + \epsilon \begin{pmatrix}
-G(X,Y,\epsilon)/f_u(\bar u) \\ 0
\end{pmatrix} \right)\\
&=\begin{pmatrix}
\lambda_1 z_1 \\ \lambda_2 z_2
\end{pmatrix}
+  P^{-1}\begin{pmatrix}
- \epsilon G(X,Y,\epsilon)/f_u(\bar u) \\ -\delta^{-1}(f(u) - f_u(\bar u)(u-\bar u))
\end{pmatrix}
\end{align*}
Note that the linear order term of $W$-variable in the error term is eliminated.

Let $N\subset \mathbb{R}^3$ be an $h$-set (in the $(u,v,w)$-coordinate) containing $(\bar u, 0, \bar w)$. 
Then the Mean Value Theorem implies that the higher order term $f(u) - f_u(\bar u)(u-\bar u)$ is included in the enclosure
\begin{equation*}
\left\{\frac{f''(\tilde u)}{2}(u-\bar u)^2 \mid u,\tilde u \text{ such that }(u,v,w), (\tilde u, v,w)\in N\right\}.
\end{equation*}
Thanks to the term $(u-\bar u)^2$, this set is in general very small if $N$ is small.
In our example, $f''(u) = -6u+2(a+1)$.
Obviously, the first component $-\epsilon G(X,Y,\epsilon)/f_u(\bar u)$ is very small if $\epsilon > 0$ is sufficiently small.
Note that the denominator $f_u(\bar u)$ is bounded away from $0$ since we are focusing on validation of normally hyperbolic invariant manifolds.
Finally, for given compact set $N$ and $\epsilon_0 > 0$, we obtain enclosures of the error term
\begin{align*}
P^{-1}\begin{pmatrix}
\left[ -\epsilon c^{-1}(u-\gamma w)/f_u(\bar u)\mid (u,v,w)\in N,\ \epsilon\in [0,\epsilon_0] \right]
 \\ \left[ \delta^{-1}\{3\tilde u-(a+1)\}(u-\bar u)^2 \mid u,\tilde u \text{ such that }(u,v,w), (\tilde u, v,w)\in N \right]
\end{pmatrix}\subset
\begin{pmatrix}
[\delta_1^-, \delta_1^+]\\
[\delta_2^-, \delta_2^+] 
\end{pmatrix},
\end{align*}
which can be computed by interval arithmetics and, in principle, returns very small enclosures for small $N$ and $\epsilon_0$.

\subsection{Jacobian matrices for cone conditions}

%%%%		Jacobian matrix		%%%%%%%%%%%
The Jacobian matrix in this coordinate at $(\tilde u, \tilde v, \tilde w, \eta)$ ignoring the differential of $\eta'$-term is
\begin{equation}
\label{Jacobian-FN}
\begin{pmatrix}
-\epsilon (c  f_u(\bar u))^{-1}  & 1 & \epsilon (c f_u(\bar u))^{-1}  ( - f_u(\bar u)^{-1} + \gamma ) & - \epsilon_0 (c f_u(\bar u))^{-1} (\tilde u + (f_u(\bar u)^{-1} - \gamma )\tilde w + \bar u  -\gamma \bar w) \\
-\delta^{-1} f_u (u) & \delta^{-1} c &  (\delta f_u(\bar u))^{-1} \left\{ - f_u(u) + f_u(\bar u)\right\} & 0 \\
\epsilon c^{-1} & 0 & \epsilon c^{-1} (f_u(\bar u)^{-1} - \gamma ) & \epsilon_0 c^{-1} (\tilde u + (f_u(\bar u)^{-1} - \gamma )\tilde w + \bar u   -\gamma \bar w)
\end{pmatrix},
\end{equation}
where the variable $u$ in the second row corresponds to (\ref{transform-FN}).
We additionally transform $(\tilde u, \tilde v)$ linearly to $(z_1, z_2)$ by $(\tilde u, \tilde v) = P(z_1, z_2)$, where
$P$ is the nonsingular matrix such that
\begin{equation*}
P^{-1} \begin{pmatrix}
0 & 1 \\
-\delta^{-1} f_u (\bar u) & \delta^{-1} c 
\end{pmatrix} P = 
\begin{pmatrix}
\lambda_1 & 0 \\
0 & \lambda_2 
\end{pmatrix}.
\end{equation*}
The Jacobian matrix (\ref{Jacobian-FN}) in the new coordinate $(z_1,z_2,\tilde w)$ is then written as 
\begin{align}
\notag
&\begin{pmatrix}
\lambda_1  & 0 & 0 & 0 \\
0 & \lambda_2 & 0 & 0 \\
0 & 0 & 0 & 0 \\
\end{pmatrix}\\
\label{Jacobian-FN-normal}
&+\begin{pmatrix}
F_{11}  & F_{12} & \epsilon (c f_u(\bar u))^{-1}  ( - f_u(\bar u)^{-1} + \gamma ) & - \epsilon_0 (c f_u(\bar u))^{-1} (\tilde u + (f_u(\bar u)^{-1} - \gamma )\tilde w + \bar u   -\gamma \bar w) \\
F_{21} & F_{22} &  (\delta f_u(\bar u))^{-1} \left\{ - f_u(u) + f_u(\bar u)\right\} & 0 \\
\epsilon c^{-1} & 0 & \epsilon c^{-1} (f_u(\bar u)^{-1} - \gamma ) & \epsilon_0 c^{-1} (\tilde u + (f_u(\bar u)^{-1} - \gamma )\tilde w + \bar u   -\gamma \bar w)
\end{pmatrix},\\
\notag
&\begin{pmatrix}
F_{11}  & F_{12} \\
F_{21}  & F_{22}
\end{pmatrix} = P^{-1}\begin{pmatrix}
-\epsilon (c  f_u(\bar u))^{-1} & 0 \\
\delta^{-1} (f_u(\bar u)-f_u(u)) & 0
\end{pmatrix}P.
\end{align}
In order to verify cone conditions in a given $h$-set $N$ containing $(\bar u, 0, \bar w)$, we compute corresponding maximal singular values in $N$.
All such calculations are done by interval arithmetics.
As for calculations of the enclosure of $f_u(\bar u)-f_u(u)$ in $N$, it is useful to use the mean value form
\begin{align*}
f_u(\bar u)-f_u(u) &\in \{-f_{uu}(\hat u)(u-\bar u)\mid u, \hat u \text{ such that }(u,v,w), (\hat u,v,w)\in N\}\\
	&= \{(6\hat u^2 - 2(a+1))(u-\bar u)\mid u, \hat u \text{ such that }(u,v,w), (\hat u,v,w)\in N\}.
\end{align*}

We apply the matrix (\ref{Jacobian-FN-normal}) as the Fr\'{e}chet differential of the vector field in (\ref{abstract-form}) to validating cone conditions discussed in Section \ref{section-inv-mfd}. 
The $m$-cone conditions can be treated in the similar manner.

\end{document}